\documentclass[a4, 12pt]{amsart}
\usepackage{amssymb,amstext,amscd,amsthm,amsfonts,mathdots,mathrsfs}
\usepackage{pdflscape}
\usepackage{hyperref}
\usepackage{pdflscape}
\usepackage{setspace}
\usepackage[all]{xy}
\usepackage{enumerate}
\usepackage{indentfirst}
\usepackage{ytableau}
\usepackage{wasysym}
\usepackage{color}
\usepackage{colortbl}
\usepackage{amsmath}
\usepackage{tikz}
\usetikzlibrary{shapes.geometric, arrows}
\usetikzlibrary{calc}
\numberwithin{equation}{section}
\tikzset{
  every node/.style={text=olive}
}

\usepackage[thicklines]{cancel}

\usepackage[capitalize]{cleveref}
\crefname{equation}{}{}
\crefname{alphatheorem}{Theorem}{Theorems}

\newtheorem{theorem}{Theorem}[section]
\newtheorem{lemma}[theorem]{Lemma}
\newtheorem{proposition}[theorem]{Proposition}
\newtheorem{definition}[theorem]{Definition}

\newtheorem{corollary}[theorem]{Corollary}
\newtheorem{example}[theorem]{Example}

\theoremstyle{remark}
\newtheorem{rem}[theorem]{Remark}

\newtheorem{alphatheorem}{\bf Theorem}
\newtheorem{alphacorollary}[alphatheorem]{\bf Corollary}

\newtheorem{alphadefinition}[alphatheorem]{\bf Definition}

\newcommand{\clr}{rgb:black,2;blue,2;red,0}

\tikzset{anchorbase/.style={baseline={([yshift=-0.5ex]current bounding box.center)}}}
\tikzset{wipe/.style={white,line width=4pt}}
\tikzset{
    circ/.style={
        draw,
        circle,
        minimum size=0.4em,
        inner sep=0pt,
    }
}

\newcommand{\ballstick}{
  \begin{tikzpicture}[anchorbase,scale=.6,line width=1pt]
    \draw (0,0) --(0,.5);
    \node[circ,fill=white] at (0,.5){};
  \end{tikzpicture}
}
\DeclareFontFamily{OT1}{pzc}{}
\DeclareFontShape{OT1}{pzc}{m}{it}{ <-> s*[1.2] pzcmi7t }{}
\DeclareMathAlphabet{\mathpzc}{OT1}{pzc}{m}{it}

\newcommand{\Wb}{\texttt{B}\!\mathpzc{Web}}
\newcommand{\WCb}{\texttt{B}\!\mathpzc{Web}^{\prime}}
\newcommand{\AWb}{\texttt{B}\!\mathpzc{Web}^\bullet}  
\newcommand{\AWCb}{\texttt{B}\!\mathpzc{Web}^{\bullet\prime}}

\newcommand{\SMat}{\text{SMat}}

\newcommand{\PSMat}{\text{ParSMat}}
\newcommand{\bSMat}{\texttt{B}\text{SMat}}
\newcommand{\bPSMat}{\texttt{B}\text{ParSMat}}
\newcommand{\Par}{\text{Par}}
\newcommand{\EPar}{\text{EPar}}

\newcommand\C{\mathbb{C}}
\newcommand\Z{\mathbb{Z}}
\newcommand\Q{\mathbb{Q}}

\newcommand\N{\mathbb{N}}
\newcommand\kk{\Bbbk}
\newcommand\la{\lambda}
\newcommand{\Hom}{{\rm Hom}}
\newcommand{\End}{{\rm End}}
\newcommand{\lol}{\ballstick}

\newcommand{\rot}{\rotatebox[origin=c]{180}}

\newcommand{\dotbubble}{\mbox{$\bigodot$}}

\newcommand{\arxiv}[1]{\href{http://arxiv.org/abs/#1}{\tt arXiv:\nolinkurl{#1}}}

\def\n{\mathfrak n}

\def\t{\mathfrak t}

\def\unit{\mathfrak{1}}

\newcommand{\braidup}{to[out=up,in=down]}
\newcommand{\braiddown}{to[out=down,in=up]}

\newcommand{\capa}{
\begin{tikzpicture}[xscale=-1,anchorbase, scale=0.3, color=\clr]
\draw[-,line width=1.2pt] (0,0) to [out=up,in=left] (1,1) to [out=right,in=up] (2,0);
\node at (0,-.3){$\scriptstyle a$};
\node at (2,-.3){$\scriptstyle a$};
\end{tikzpicture}}

\newcommand{\cupa}{
\begin{tikzpicture}[xscale=-1,anchorbase, scale=0.3, color=\clr]
\draw[-,line width=1.2pt] (1,1) to [out=down,in=left] (2,0) to [out=right,in=down] (3,1);
\node at (1,1.3){$\scriptstyle a$};
\node at (3,1.3){$\scriptstyle a$};
\end{tikzpicture}}

\newcommand{\bubblea}{
\begin{tikzpicture}[anchorbase, scale=1.2, color=\clr,line width=1.2pt]
					\draw[-] (0.2,0.2) to[out=90,in=0] (0,.4);
					\draw[-] (0,0.4) to[out=180,in=90] (-.2,0.2);
					\draw[-] (-.2,0.2) to[out=-90,in=180] (0,0);
					\draw[-] (0,0) to[out=0,in=-90] (0.2,0.2);
				\node at (-0.3,0.2) {$\scriptstyle{a}$};
			\end{tikzpicture}}
\newcommand{\bubble}{
\begin{tikzpicture}[anchorbase, scale=1.5, color=\clr,thick]
					\draw[-] (0.2,0.2) to[out=90,in=0] (0,.4);
					\draw[-] (0,0.4) to[out=180,in=90] (-.2,0.2);
					\draw[-] (-.2,0.2) to[out=-90,in=180] (0,0);
					\draw[-] (0,0) to[out=0,in=-90] (0.2,0.2);
				\node at (-0.3,0.2) {$\scriptstyle{1}$};
			\end{tikzpicture}}
   
\newcommand{\stra}{\begin{tikzpicture}[baseline = 10pt, scale=0.5, color=\clr]
            \draw[-,line width=1.6pt] (0,0.5)to[out=up,in=down](0,1.7);
            \draw (0,0.2) node{$\scriptstyle a$};
\end{tikzpicture} 
}

\newcommand{\bdot}{ node[circle, draw, fill=\clr, thick, inner sep=0pt, minimum width=2.8pt]{}}

\newcommand{\squ}{ node[draw, fill=white,scale=.35]{}}

\newcommand{\dia}{ node[draw, fill=\clr,scale=.35]{}}

\newcommand{\merge}
{\begin{tikzpicture}[baseline = -.5mm,color=\clr]
	\draw[-,line width=1pt] (0.28,-.3) to (0.08,0.04);
	\draw[-,line width=1pt] (-0.12,-.3) to (0.08,0.04);
	\draw[-,line width=1.6pt] (0.08,.4) to (0.08,0);
        \node at (-0.22,-.4) {$\scriptstyle a$};
        \node at (0.35,-.4) {$\scriptstyle b$};\node at (0,.55){$\scriptstyle a+b$};\end{tikzpicture} }
        
\newcommand{\splits}
{\begin{tikzpicture}[baseline = -.5mm,color=\clr]
	\draw[-,line width=1.6pt] (0.08,-.3) to (0.08,0.04);
	\draw[-,line width=1pt] (0.28,.4) to (0.08,0);
	\draw[-,line width=1pt] (-0.12,.4) to (0.08,0);
        \node at (-0.22,.5) {$\scriptstyle a$};
        \node at (0.36,.5) {$\scriptstyle b$};
        \node at (0.1,-.45){$\scriptstyle a+b$};
\end{tikzpicture}}

\newcommand{\dotgen}
{\begin{tikzpicture}[anchorbase, scale=0.5, color=\clr]
\draw[-,thick] (0,0) to[out=up, in=down] (0,1.4);
\draw(0,0.6) \bdot;
\node at (0,-.3) {$\scriptstyle 1$};
\end{tikzpicture}}

\newcommand{\crossing}{\begin{tikzpicture}[baseline=-.5mm,scale=.8,color=\clr]
	\draw[-,line width=1.2pt] (-0.3,-.3) to (.3,.4);
	\draw[-,line width=1.2pt] (0.3,-.3) to (-.3,.4);
        \node at (0.3,-.4) {$\scriptstyle b$};
        \node at (-0.3,-.4) {$\scriptstyle a$};
         \node at (0.3,.55) {$\scriptstyle a$};
        \node at (-0.3,.55) {$\scriptstyle b$};
\end{tikzpicture}}

\newcommand{\wkdota}{\begin{tikzpicture}[anchorbase, scale=0.5, color=\clr]
\draw[-,line width=1.4pt] (0,0) to[out=up, in=down] (0,1.4);
\draw(0,0.6) \bdot; 
\draw (0.7,0.6) node {$\scriptstyle \omega_r$};
\node at (0,-.3) {$\scriptstyle a$};
\end{tikzpicture} }

\newcommand{\wkdotaa}{\begin{tikzpicture}[anchorbase, scale=0.5, color=\clr]
\draw[-,line width=1.4pt] (0,0) to[out=up, in=down] (0,1.4);
\draw(0,0.6) \bdot; 
\node at (0,-.3) {$\scriptstyle a$};
\end{tikzpicture} }

\newcommand{\wkdotawmu}{\begin{tikzpicture}[anchorbase, scale=0.5, color=\clr]
\draw[-,line width=1.4pt] (0,0) to[out=up, in=down] (0,1.4);
\draw(0,0.6) \bdot; 
\draw (0.7,0.6) node {$\scriptstyle \omega_{\mu}$};
\node at (0,-.3) {$\scriptstyle {a}$};
\end{tikzpicture} }

\newcommand{\wkdotawmue}{\begin{tikzpicture}[anchorbase, scale=0.5, color=\clr]
\draw[-,line width=1.4pt] (0,0) to[out=up, in=down] (0,1.4);
\draw(0,0.6) \bdot; 
\draw (0.7,0.6) node {$\scriptstyle \omega_{\mu}$};
\node at (0,-.3) {$\scriptstyle {2a}$};
\end{tikzpicture} }

\newcommand{\lollipopa}{\begin{tikzpicture}[anchorbase, scale=0.5, color=\clr]
\draw[-,line width=1.4pt] (0,0) to[out=up, in=down] (0,1.4);
\draw (0,1.4) \bdot;
\node at (0,-.3) {$\scriptstyle {2a}$};
\end{tikzpicture} }

\newcommand{\wkdotamu}{\begin{tikzpicture}[anchorbase, scale=0.5, color=\clr]
\draw[-,line width=1.4pt] (0,0) to[out=up, in=down] (0,1.4);
\draw(0,0.6) \bdot; 
\draw (0,1.4) \bdot;
\draw (0.7,0.6) node {$\scriptstyle \omega_{\mu}$};
\node at (0,-.3) {$\scriptstyle {2a}$};
\end{tikzpicture} }

\newcommand{\wkdotamurev}{\begin{tikzpicture}[anchorbase, scale=0.5, yscale=-1 ,color=\clr]
\draw[-,line width=1.4pt] (0,0) to[out=up, in=down] (0,1.4);
\draw(0,0.6) \bdot; 
\draw (0,1.4) \bdot;
\draw (0.7,0.6) node {$\scriptstyle \omega_{\mu}$};
\node at (0,-.3) {$\scriptstyle {2a}$};
\end{tikzpicture} }

\newcommand{\wkdotamurevbox}{\begin{tikzpicture}[anchorbase, scale=0.5, yscale=-1 ,color=\clr]
\draw[-,line width=1.4pt] (0,0) to[out=up, in=down] (0,1.4);
\draw(0,0.6) \dia; 
\draw (0,1.4) \bdot;
\node at (0,-.3) {$\scriptstyle {2a}$};
\end{tikzpicture} }

\newcommand{\wkdotsnu}{\begin{tikzpicture}[anchorbase, scale=0.5, color=\clr]
\draw[-,line width=1.4pt] (0,0) to[out=up, in=down] (0,1.4);
\draw(0,0.6) \bdot; 
\draw (0,1.4) \bdot;
\draw (0.7,0.6) node {$\scriptstyle \omega_{\nu}$};
\node at (0,-.3) {$\scriptstyle {2s}$};
\end{tikzpicture} }

\newcommand{\wkdotsnurev}{\begin{tikzpicture}[anchorbase, scale=0.5, yscale=-1, color=\clr]
\draw[-,line width=1.4pt] (0,0) to[out=up, in=down] (0,1.4);
\draw(0,0.6) \bdot; 
\draw (0,1.4) \bdot;
\draw (0.7,0.6) node {$\scriptstyle \omega_{\nu}$};
\node at (0,-.3) {$\scriptstyle {2s}$};
\end{tikzpicture} }

 \newcommand{\bubbledota}{
\begin{tikzpicture}[anchorbase, scale=1.2, color=\clr,line width=1.2pt]
					\draw[-] (0.2,0.2) to[out=90,in=0] (0,.4);
					\draw[-] (0,0.4) to[out=180,in=90] (-.2,0.2);
					\draw[-] (-.2,0.2) to[out=-90,in=180] (0,0);
					\draw[-] (0,0) to[out=0,in=-90] (0.2,0.2);
				\node at (0,-.1) {$\scriptstyle{a}$};
         \draw (-.2,.2) \bdot;
			\end{tikzpicture}}

 \newcommand{\bubbledotat}{
\begin{tikzpicture}[anchorbase, line width=1.2pt,scale=1.2, color=\clr]
					\draw[-] (0.2,0.2) to[out=90,in=0] (0,.4);
					\draw[-] (0,0.4) to[out=180,in=90] (-.2,0.2);
					\draw[-] (-.2,0.2) to[out=-90,in=180] (0,0);
					\draw[-] (0,0) to[out=0,in=-90] (0.2,0.2);
			\node at (0,-.1) {$\scriptstyle{a}$};
         \draw (-.2,.2) \bdot;
     \node at (-.4,.2) {$\scriptstyle{t\omega_{a}}$};
			\end{tikzpicture}}

\newcommand{\bubbledottt}{
\begin{tikzpicture}[anchorbase, scale=1.2, line width=1.2pt, color=\clr]
					\draw[-] (0.2,0.2) to[out=90,in=0] (0,.4);
					\draw[-] (0,0.4) to[out=180,in=90] (-.2,0.2);
					\draw[-] (-.2,0.2) to[out=-90,in=180] (0,0);
					\draw[-] (0,0) to[out=0,in=-90] (0.2,0.2);
			\node at (0,-.1) {$\scriptstyle{t}$};
         \draw (-.2,.2) \bdot;
     \node at (-.4,.2) {$\scriptstyle{\omega_t}$};
			\end{tikzpicture}}

 \newcommand{\bubbledott}{
\begin{tikzpicture}[anchorbase, scale=1.2, color=\clr]
					\draw[-] (0.2,0.2) to[out=90,in=0] (0,.4);
					\draw[-] (0,0.4) to[out=180,in=90] (-.2,0.2);
					\draw[-] (-.2,0.2) to[out=-90,in=180] (0,0);
					\draw[-] (0,0) to[out=0,in=-90] (0.2,0.2);
			\node at (0,-.1) {$\scriptstyle{1}$};
         \draw (-.2,.2) \bdot;
     \node at (-.3,.2) {$\scriptstyle{t}$};
			\end{tikzpicture}}

\newcommand{\wkdotavar}[1]{\begin{tikzpicture}[anchorbase, scale=0.5, color=\clr]
\draw[-,line width=1.4pt] (0,0) to[out=up, in=down] (0,1.4);
\draw(0,0.6) \bdot; 
\draw (0.8,0.6) node {$\scriptstyle \omega_{#1}$};
\node at (0,-.3) {$\scriptstyle {a}$};
\end{tikzpicture}}

\newcommand{\wxdota}{\begin{tikzpicture}[anchorbase, scale=0.5, color=\clr]
\draw[-,line width=1.6pt] (0,0) to[out=up, in=down] (0,1.4);
\draw(0,0.6) \bdot; 
\draw (0.7,0.6) node {$\scriptstyle \omega_1$};
\node at (0,-.3) {$\scriptstyle a$};
\end{tikzpicture}}

\newcommand{\sh}{\text{Shape}}

\newcommand{\coupon}[2]{ 
    \draw (#1) node[inner sep=2pt,draw,fill=white,rounded corners] {$\scriptstyle{#2}$}
}
\newcommand{\strandlabel}[1]{$\scriptstyle{#1}$}

\newcommand{\toplabel}[1]{node[anchor=south] {\strandlabel{#1}}}

\newcommand{\BoxRow}[2]{ 
    \foreach \i in {0, ..., #1} {
        \draw (#2+\i, 0) rectangle (#2+\i+1, 1);
    }
}

\newcommand{\hm}{\hat{m}}

\newcommand{\g}{\mathfrak g}
\renewcommand{\n}{\mathfrak n}
\newcommand{\h}{\mathfrak h}
\newcommand{\borel}{\mathfrak b}
\newcommand{\so}{\mathfrak{so}}
\newcommand{\U}{\mathbf{U}}
\newcommand{\emod}{\mathcal{E}nd}
\newcommand{\cG}{\mathcal{G}}
\newcommand{\cF}{\mathcal{F}}
\newcommand{\id}{\text{Id}}

\newcommand{\modf}{\operatorname{-mod}\nolimits}
\newcommand{\Mod}{\operatorname{-Mod}\nolimits}

\begin{document}
\setlength{\baselineskip}{17pt}
\title{Brauer webs of finite and affine type}

\author{Yaolong Shen}
\address{School of Mathematical Sciences, Key Laboratory of MEA (Ministry of Education)
\& Shanghai Key Laboratory of PMMP, East China Normal University, Shanghai 200241,
China}
\email{ylshen@math.ecnu.edu.cn}

\author{Linliang Song}
\address{School of Mathematical Science, Tongji University, Shanghai, 200092, China} \email{llsong@tongji.edu.cn}

\author{Weiqiang Wang}
 \address{Department of Mathematics, University of Virginia, Charlottesville, VA 22904, USA} \email{ww9c@virginia.edu}



\begin{abstract}
We introduce and develop an integral diagrammatic framework for two new monoidal categories, the finite and affine B-web categories. They arise naturally from representation categories of type BCD Lie algebras, and can be viewed as a Brauer type generalization of the affine web category which we introduced earlier. Integral bases consisting of bubbled elementary diagrams are established for the finite and affine B-webs. A simplified presentation of these categories over a field of characteristic zero is also obtained. 
\end{abstract}

\maketitle

\tableofcontents

%
\section{Introduction}

\subsection{The goal}

A monoidal category called the affine web category was introduced by two of the authors in \cite{SWweb}
and also in \cite{DKM25}. With additional thick dot generating morphisms, it was built on the finite web category (also called Schur category) from \cite{BEEO}, which can in turn be  viewed as a simplified $\mathfrak{gl}_\infty$-version of the web category from \cite{CKM}. The affine web category has been further developed with connections to Yangians in \cite{BI25}. The affine web category can be naturally enlarged to another monoidal category known as affine Schur category \cite{SWSchur}. These web and Schur categories and their cyclotomic quotients admit natural $q$-deformations \cite{SSW25} (see also \cite{Bru25} for finite type). 

All these web and Schur categories should be regarded as of type $A$ in the sense that they admit actions on various module categories of type $A$ Lie algebras or quantum groups. The path algebras of these categories and their cyclotomic quotients provide new diagrammatic realizations of variants of degenerate affine type $A$ Schur algebras and cyclotomic Schur algebras. New quantum algebras arise from cyclotomic $q$-web categories \cite{SSW25} and they provide a natural setting for categorification generalizing the classic  works of Lascoux-Leclerc-Thibon, Ariki, and Varagnolo-Vasserot; see \cite{SWCat} and the references therein. 

This is the first of a series of papers on development of a theory of new monoidal diagrammatic categories,  Brauer web categories (which will be referred to as B-web categories for short) of finite and affine type; 
we focus on the classical setting in this paper and address the quantum case in a sequel. The B-web categories act naturally on various module categories of Lie algebras of BCD type (so B-web stands both for BCD and for Brauer). These B-Webs and their $q$-variants are expected to be closely related to categorification of canonical bases on iquantum groups and modules \cite{BW18KL}. 

A major diagrammatic difference from the web categories of type A is that the B-webs contain (thick) caps, cups, and bubbles, which make diagrammatic computations more involved.

\subsection{The main results}

Let us explain in detail the main constructions and results of this paper. 

\subsubsection{}

Introduce the following commutative rings (as subrings of $\Q[x]$): 
\[
\mathbb J :=\{f(x)\in \Q[x]\mid f(n)\in\Z, \forall n\in\Z\},
\qquad
\mathbb J' :=\{f(x)\in \Q[x]\mid f(n)\in\Z[{\scriptstyle \frac12}], \forall n\in\Z\}.
\]
It is known that $\{\binom{x}{a}\mid a\ge 0\}$ forms a $\Z$-basis for $\mathbb J$ and also a $\Z[\frac12]$-basis for $\mathbb J'$.

By \cref{def-Bweb}, the B-web category $\Wb_{\mathbb J}$ (or simply $\Wb$) is the strict $\mathbb J$-linear monoidal category generated by objects $\stra$ for $a\in \mathbb Z_{\ge1}$. 
The generating morphisms are merges, splits, (thick) crossings, cups and caps depicted as, for $a,b \ge 1$:
\begin{align*}
\begin{tikzpicture}[anchorbase,color=\clr]
	\draw[-,line width=1pt] (0.28,-.3) to (0.08,0.04);
	\draw[-,line width=1pt] (-0.12,-.3) to (0.08,0.04);
	\draw[-,line width=1.4pt] (0.08,.4) to (0.08,0);
        \node at (-0.22,-.4) {$\scriptstyle a$};
        \node at (0.35,-.4) {$\scriptstyle b$};\node at (0,.55){$\scriptstyle a+b$};\end{tikzpicture} 
&:(a,b) \rightarrow (a+b),
\quad
\begin{tikzpicture}[anchorbase,color=\clr]
	\draw[-,line width=1.4pt] (0.08,-.3) to (0.08,0.04);
	\draw[-,line width=1pt] (0.28,.4) to (0.08,0);
	\draw[-,line width=1pt] (-0.12,.4) to (0.08,0);
        \node at (-0.22,.5) {$\scriptstyle a$};
        \node at (0.36,.5) {$\scriptstyle b$};
        \node at (0.1,-.45){$\scriptstyle a+b$};
\end{tikzpicture}
:(a+b)\rightarrow (a,b),
\quad
\crossing :(a,b) \rightarrow (b,a),
\\
\cupa &: \unit \rightarrow (a,a), 
\qquad \qquad
\capa  : (a,a) \rightarrow \unit
\;,  
\end{align*}
where $\unit$ denotes the unit object of the category. These morphisms satisfy the relations \cref{webassoc}--\cref{bubblecap}, that is, the same relations 
\eqref{webassoc}--\eqref{splitmerge} among merges, splits and crossings as for the type A web (or Schur) category in \cite{BEEO}, and the additional relations \eqref{zigzag}--\eqref{bubblecap} involving cups/caps, that is,
%
\begin{align*}
\begin{tikzpicture}[baseline = -1mm,scale=.8,color=\clr]
        \draw[-,line width=1.2pt] (0,-0.5) to (0,0) to
        [out=up,in=left] (.25,.25) to [out=right,in=up] (.5,0) 
        to [out=down,in=left] (.75,-0.25) to [out=right,in=down] (1,0)
        to (1,.5);
        \node at (0,-.65){$\scriptstyle{a}$};
    \end{tikzpicture}
    =
       \begin{tikzpicture}[baseline = -1mm,scale=.8,color=\clr]
        \draw[-,line width=1.2pt] (0,-.5) to (0,.5);
        \node at (0,-.65){$\scriptstyle{a}$};
    \end{tikzpicture} ,
    \quad 
\begin{tikzpicture}[baseline = -1mm,scale=.8,color=\clr]
        \draw[-,line width=1.2pt] (0,-0.5) to (0,0) to
        [out=up,in=right] (-.25,.25) to [out=left,in=up] (-.5,0) 
        to [out=down,in=right] (-.75,-0.25) to [out=left,in=down] (-1,0)
        to (-1,.5);
        \node at (0,-.65){$\scriptstyle{a}$};
    \end{tikzpicture}
   & =
       \begin{tikzpicture}[baseline = -1mm,scale=.8,color=\clr]
        \draw[-,line width=1.2pt] (0,-.5) to (0,.5);
        \node at (0,-.65){$\scriptstyle{a}$};
    \end{tikzpicture},
    \qquad \qquad
    %
    	\begin{tikzpicture}[anchorbase,color=\clr,line width=1pt]
	  \draw[-](.6,.4) to (.1,-.3);
		\draw[-] (0.6,-0.3) to[out=140, in=0] (0.1,0.2);
		\draw[-] (0.1,0.2) to[out = -180, in = 90] (-0.2,-0.3);
        \node at (-.2,-.45){$\scriptstyle{a}$};
        \node at (.6,-.45){$\scriptstyle{a}$};
        \node at (.1,-.45){$\scriptstyle{b}$};
	\end{tikzpicture}
	 =
	\begin{tikzpicture}[anchorbase,color=\clr,line width=1pt]
		\draw[-](-.5,.4) to (0,-.3);
		\draw[-] (0.3,-0.3) to[out=90, in=0] (0,0.2);
		\draw[-] (0,0.2) to[out = -180, in = 40] (-0.5,-0.3);
        \node at (-.5,-.45){$\scriptstyle{a}$};
        \node at (.3,-.45){$\scriptstyle{a}$};
        \node at (0,-.45){$\scriptstyle{b}$};
	\end{tikzpicture} ,
 \quad
	\begin{tikzpicture}[anchorbase,color=\clr,line width=1pt]
		\draw[-](-.5,-.3) to (0,.4);
		\draw[-] (0.3,0.4) to[out=-90, in=0] (0,-0.1);
		\draw[-] (0,-0.1) to[out = 180, in = -40] (-0.5,0.4);
        \node at (-.5,.5){$\scriptstyle{a}$};
        \node at (.3,.5){$\scriptstyle{a}$};
        \node at (.0,.52){$\scriptstyle{{b}}$};
	\end{tikzpicture}
	=
	\begin{tikzpicture}[anchorbase,color=\clr,line width=1pt]
		\draw[-](.6,-.3) to (.1,.4);
	  \draw[-] (0.6,0.4) to[out=-140, in=0] (0.1,-0.1);
		\draw[-] (0.1,-0.1) to[out = 180, in = -90] (-0.2,0.4);
        \node at (-.2,.5){$\scriptstyle{a}$};
        \node at (.6,.5){$\scriptstyle{a}$};
        \node at (.1,.52){$\scriptstyle{{b}}$};
	\end{tikzpicture} ,
    \\
        \begin{tikzpicture}[anchorbase,color=\clr]
	\draw[-,line width=1.4pt] (0.08,-.3) to (0.08,0.04);
	\draw[-,line width=1pt] (0.28,.4) to (0.08,0);
	\draw[-,line width=1pt] (-0.12,.4) to (0.08,0);
    \draw[-,line width=1pt] (.28,.4) to [out=up,in=up] (.7,.4) to (.7,-.3);
    \draw[-,line width=1pt] (-.12,.4) to [out=up,in=up] (1,.4) to (1,-.3);
     \node at (1,-.45) {$\scriptstyle {b}$};
     \node at (.7,-.45) {$\scriptstyle {a}$};
\end{tikzpicture}
 &=
\begin{tikzpicture}[anchorbase,color=\clr]
    \draw[-,line width=1.2pt] (0,-.3) to (0,.4);
    \draw[-,line width=1.2pt] (0,.4) to [out=up,in=up] (.6,.4) to (.6,0);
    \draw[-,line width=1pt] (.6,0) to (.8,-.3) ; 
    \draw[-,line width=1pt] (.6,0) to (.4,-.3); 
     \node at (.8,-.45) {$\scriptstyle {b}$};
      \node at (.4,-.45){$\scriptstyle {a}$};
\end{tikzpicture},
\qquad
 \begin{tikzpicture}[xscale=-1,anchorbase,color=\clr]
	\draw[-,line width=1.4pt] (0.08,-.3) to (0.08,0.04);
	\draw[-,line width=1pt] (0.28,.4) to (0.08,0);
	\draw[-,line width=1pt] (-0.12,.4) to (0.08,0);
    \draw[-,line width=1pt] (.28,.4) to [out=up,in=up] (.7,.4) to (.7,-.3);
    \draw[-,line width=1pt] (-.12,.4) to [out=up,in=up] (1,.4) to (1,-.3);
     \node at (1,-.45) {$\scriptstyle {a}$};
     \node at (.7,-.45) {$\scriptstyle {b}$};
\end{tikzpicture}
=
\begin{tikzpicture}[xscale=-1,anchorbase,color=\clr]
    \draw[-,line width=1.2pt] (0,-.3) to (0,.4);
    \draw[-,line width=1.2pt] (0,.4) to [out=up,in=up] (.6,.4) to (.6,0);
    \draw[-,line width=1pt] (.6,0) to (.8,-.3) ; 
    \draw[-,line width=1pt] (.6,0) to (.4,-.3); 
     \node at (.8,-.45) {$\scriptstyle {a}$};
      \node at (.4,-.45){$\scriptstyle {b}$};
\end{tikzpicture} \:\: ,
    \\
  \bubblea &= \binom{x }{a} ,
  \qquad
\begin{tikzpicture}[yscale=-1, baseline = -1mm,scale=.8,color=\clr]
\draw[-,line width=1.6pt] (0.1,.3) to (0.1,.6);
\draw[-,thick] (0.1,-.51) to [out=right,in=-45] (0.1,.31);
\draw[-,thick] (0.1,-.51) to [out=left,in=-135] (0.1,.31);
\node at (-.33,-.05) {$\scriptstyle a$};
\node at (.45,-.05) {$\scriptstyle a$};
\end{tikzpicture} =0,
\qquad
\begin{tikzpicture}[anchorbase, scale=1, color=\clr]
 \draw[line width=1.2pt] (-0.15,0) arc(180:0:0.15) to[out=down,in=60] (-0.15,-0.4);
        \draw[line width=1.2pt] (0.15,-0.4) to[out=120,in=down] (-0.15,0);
        \node at (-.15,-.5) {$\scriptstyle a$};
        \node at (.15,-.5) {$\scriptstyle a$};
\end{tikzpicture}
= (-1)^a\!\capa .
\end{align*} 
There is a natural functor from the $A$-web category to the B-web category, which sends objects $a$ to $a$ and the generating morphisms (splits, merges, and crossings) to the same ones.

\subsubsection{}
Note that $\bubble=x$ in $\Wb$. When developing representations of the B-web category, $x$ often acts as a fixed scalar $N$, and it is convenient to consider the variant of $\Wb$ after a base change $\Wb_\C(N) =\Wb_{\mathbb J} \otimes_{\mathbb J} \C$. 
Assuming further that $N\in \Z_{\ge 1}$, we denote by $\so_N\text{-mod}$ the category of finite-dimensional modules of the complex orthogonal Lie algebra $\so_N$. Let $V$ denote the natural representation of $\so_N$. We refer to \eqref{eq:Yab}--\eqref{cuparep} 
for morphisms \rm{\rot{Y}}$_{a,b}$, 
   ${\rm Y}_{a,b}$, ${\rm X}_{a,b}$, $\cap_a$ and $\cup_a$.

\begin{alphatheorem}  [\cref{thm:Worep}]
    \label{thm:A}
There is a monoidal functor $\mathcal F_N: \Wb_{\C}(N) \rightarrow \so_N\text{-mod}$ such that 
it sends $a$ to $\wedge^a V$ and sends the 
 generating morphisms $\merge$, $\splits$, $\crossing$, $\capa$ and $\cupa$ to \rm{\rot{Y}}$_{a,b}$, 
   ${\rm Y}_{a,b}$, ${\rm X}_{a,b}$, $\cap_a$ and $\cup_a$ respectively. 
\end{alphatheorem}

\subsubsection{}

Introduce the shorthand lollipop diagrams
\begin{align}
    \label{def:huochai}
\begin{tikzpicture}[anchorbase,scale=.8,color=\clr]
\draw[-,line width=1.4pt] (0.08,-.2) to (0.08,.4);\draw (0.08,-.2)\bdot;
\node at (0.08,.6) {$\scriptstyle 2s$};
\end{tikzpicture}
& :=~\frac{1}{2^s}~
\begin{tikzpicture}[anchorbase,scale=.8,color=\clr]
\draw[-,line width=1.4pt] (0.08,.3) to (0.08,.6);
\draw[-,thick] (0.1,-.51) to [out=right,in=-45] (0.1,.31);
\draw[-,thick] (0.1,-.51) to [out=left,in=-135] (0.1,.31);
\node at (0.08,.8) {$\scriptstyle 2s$};
\draw (-.1,-.2)\bdot;
\end{tikzpicture},
\qquad \qquad \qquad
\begin{tikzpicture}[yscale=-1,anchorbase,scale=.8,color=\clr]
\draw[-,line width=1.4pt] (0.08,-.2) to (0.08,.4);\draw (0.08,-.2)\bdot;\node at (0.08,.6) {$\scriptstyle 2s$};
\end{tikzpicture}
:=~\frac{1}{2^s}~
\begin{tikzpicture}[yscale=-1, anchorbase,scale=.8,color=\clr]
\draw[-,line width=1.4pt] (0.08,.3) to (0.08,.6);
\draw[-,thick] (0.1,-.51) to [out=right,in=-45] (0.1,.31);
\draw[-,thick] (0.1,-.51) to [out=left,in=-135] (0.1,.31);
\node at (0.08,.8) {$\scriptstyle 2s$};
\draw (-.1,-.2)\bdot;
\end{tikzpicture}. 
\end{align}

A basic contribution of this paper is the formulation of the following new monoidal category with favorable properties.

\begin{alphadefinition}
\label{def:affineBWeb}
    The affine B-web category $\AWb$ is the $\mathbb J'$-linear monoidal category with generating objects $\stra$ (for $a\ge 1$) and generating morphisms (for $a,b \ge 1$)
 \begin{align}
  \splits, & \quad \merge,\quad \crossing,\quad \cupa,\quad \capa, \quad
\label{dotgenerator}
 \begin{tikzpicture}[baseline = 3pt, scale=0.5, color=\clr]
\draw[-,line width=1.4pt] (0,0) to[out=up, in=down] (0,1.4);
\draw(0,0.6) \bdot;
\node at (0,-.3) {$\scriptstyle a$};
\end{tikzpicture}.
\end{align}
These morphisms, excluding $\!\begin{tikzpicture}[baseline = 3pt, scale=0.5, color=\clr]
\draw[-,line width=1.4pt] (0,0) to[out=up, in=down] (0,1.4);
\draw(0,0.6) \bdot;
\node at (0,-.3) {$\scriptstyle a$};
\end{tikzpicture}\!$,  satisfy the same finite-type relations \cref{webassoc}--\cref{bubblecap} as for $\Wb$ and $\!\begin{tikzpicture}[baseline = 3pt, scale=0.5, color=\clr]
\draw[-,line width=1.4pt] (0,0) to[out=up, in=down] (0,1.4);
\draw(0,0.6) \bdot;
\node at (0,-.3) {$\scriptstyle a$};
\end{tikzpicture}\!$ satisfies the following additional relations \cref{dotmovecrossing1}--\cref{dotmovecupcaps}:
\begin{gather}
\label{dotmovecrossing1}
\sum_{0\leq t \leq \min(a,b)}t!~
\begin{tikzpicture}[baseline = 7.5pt,scale=.8, color=\clr]
\draw[-,thick] (0.58,-.2) to (0.58,0) to (-.48,.8) to (-.48,1);
\draw[-,thick] (-.48,-.2) to (-.48,0) to (.58,.8) to (.58,1);
\draw[-,thick] (-.48,.95) to [out=-46,in=left](0.1,.7) to [out=right, in =45] (.58,.95);
\draw[-,thick] (-.48,.-.15) to [out=46,in=left](0.1,.1) to [out=right, in =-45] (.58,-.15);
\draw[-,thick] (-.5,-.2) to (-.5,0);
\draw[-,thick] (-.5,0.8) to (-.5,1);
\draw[-,thick] (0.605,-.2) to (0.605,0);
\draw[-,thick] (0.605,.8) to (0.605,1);
\draw(-.2,0.6)\bdot;
\node at (-.5,-.3) {$\scriptstyle a$};
\node at (0.6,-.3) {$\scriptstyle b$};
\node at (0.1,.9) {$\scriptstyle t$};
\node at (0.1,-.1) {$\scriptstyle t$};
\end{tikzpicture} 
= 
 \sum_{0 \leq t \leq \min(a,b)}t!~
\begin{tikzpicture}[baseline = 7.5pt,scale=.8, color=\clr]
\draw[-,thick] (0.58,-.2) to (0.58,0) to (-.48,.8) to (-.48,1);
\draw[-,thick] (-.48,-.2) to (-.48,0) to (.58,.8) to (.58,1);
\draw[-,thick] (-.5,-.2) to (-.5,1);
\draw[-,thick] (0.605,-.2) to (0.605,1);
\draw(.35,0.2)\bdot;
\node at (-.5,-.3) {$\scriptstyle a$};
\node at (0.6,-.3) {$\scriptstyle b$};
\node at (0.77,.5) {$\scriptstyle t$};
\end{tikzpicture},
\\
 \label{updotmovesplits+merge}
\begin{tikzpicture}[baseline = -.5mm,scale=1,color=\clr]
\draw[-,line width=1.2pt] (0.08,-.5) to (0.08,0.04);
\draw[-,line width=1pt] (0.34,.5) to (0.08,0);
\draw[-,line width=1pt] (-0.2,.5) to (0.08,0);
\node at (-0.22,.6) {$\scriptstyle a$};
\node at (0.36,.65) {$\scriptstyle b$};
\draw (0.08,-.2) \bdot;
\end{tikzpicture} 
=
\sum_{0\leq s\leq \min(a,b)}s!~
\begin{tikzpicture}[baseline = -.5mm,scale=1,color=\clr]
\draw[-,line width=1.2pt] (0.08,-.5) to (0.08,0.04);
\draw[-,line width=1.2pt] (0.08,-.5) to (-.1,-.1);
\draw (-.1,-.1) \bdot;
\node at (-0.2,-.3) {$\scriptstyle 2s$};
\draw[-,line width=1pt] (0.44,.6) to (0.08,0);
\draw[-,line width=1pt] (-0.3,.6) to (0.08,0);
\draw[-, line width=1pt](-0.3,.6)  to[out=down,in=down] (0.44,.6);
\node at (-0.3,.7) {$\scriptstyle a$};
\node at (0.44,.75) {$\scriptstyle b$};
\draw (-.05,.24) \bdot;
\node at (0.08,.5) {$\scriptstyle s$};
\draw (.22,.24) \bdot;
\end{tikzpicture},
\qquad \:
\begin{tikzpicture}[yscale=-1,baseline = -.5mm,scale=1,color=\clr]
\draw[-,line width=1.2pt] (0.08,-.5) to (0.08,0.04);
\draw[-,line width=1pt] (0.34,.5) to (0.08,0);
\draw[-,line width=1pt] (-0.2,.5) to (0.08,0);
\node at (-0.22,.6) {$\scriptstyle a$};
\node at (0.36,.65) {$\scriptstyle b$};
\draw (0.08,-.2) \bdot;
\end{tikzpicture} 
=
\sum_{0\leq s\leq \min(a,b)}s!~
\begin{tikzpicture}[yscale=-1,baseline = -.5mm,scale=1.1,color=\clr]
\draw[-,line width=1.2pt] (0.08,-.5) to (0.08,0.04);
\draw[-,line width=1.2pt] (0.08,-.5) to (-.1,-.1);
\draw (-.1,-.1) \bdot;
\node at (-0.2,-.3) {$\scriptstyle 2s$};
\draw[-,line width=1pt] (0.44,.6) to (0.08,0);
\draw[-,line width=1pt] (-0.3,.6) to (0.08,0);
\draw[-, line width=1pt](-0.3,.6)  to[out=down,in=down] (0.44,.6);
\node at (-0.3,.7) {$\scriptstyle a$};
\node at (0.44,.75) {$\scriptstyle b$};
\draw (-.05,.24) \bdot;
\node at (0.08,.5) {$\scriptstyle s$};
\draw (.22,.24) \bdot;
\end{tikzpicture},
\\
\label{dotmovecupcaps}
\begin{tikzpicture}[anchorbase, scale=0.3, color=\clr]
\draw[-,line width=1.2pt] (0,0) to [out=up,in=left] (1,1) to [out=right,in=up] (2,0);
\node at (0,-.3){$\scriptstyle a$};
\node at (2,-.3){$\scriptstyle a$};
\draw (.25,.58)\bdot;
\end{tikzpicture} \:
=(-1)^a
\begin{tikzpicture}[anchorbase, scale=0.3, color=\clr]
\draw[-,line width=1.2pt] (0,0) to [out=up,in=left] (1,1) to [out=right,in=up] (2,0);
\node at (0,-.3){$\scriptstyle a$};
\node at (2,-.3){$\scriptstyle a$};
\draw (1.75,.58)\bdot;
\end{tikzpicture} \:\:, 
\qquad \qquad  \quad
\begin{tikzpicture}[yscale=-1,anchorbase, scale=0.3, color=\clr]
\draw[-,line width=1.2pt] (0,0.5) to [out=up,in=left] (1,1.5) to [out=right,in=up] (2,0.5);
\node at (0,.2){$\scriptstyle a$};
\node at (2,.2){$\scriptstyle a$};
\draw (.1,.88)\bdot;
\end{tikzpicture} \:
=(-1)^a
\begin{tikzpicture}[yscale=-1,anchorbase, scale=0.3, color=\clr]
\draw[-,line width=1.2pt] (0,.5) to [out=up,in=left] (1,1.5) to [out=right,in=up] (2,0.5);
\node at (0,.2){$\scriptstyle a$};
\node at (2,.2){$\scriptstyle a$};
\draw (1.95,.88)\bdot;
\end{tikzpicture} .
\end{gather}
\end{alphadefinition}

In contrast to the finite type, the above relations in $\AWb$ moving dots over crossings, merges and splits differ from their type A counterparts in \cite{SWweb} and \cite{DKM25}. On the other hand, dropping the summands involving cups/caps in the B-relations (i.e., dropping the summands for $t\neq 0$ on the LHS of the relation \eqref{dotmovecrossing1} or the summands for $s\neq 0$ from 
\eqref{updotmovesplits+merge}) recovers the type $A$ relations; to a large extent, our previous work on affine webs was foundational and preparatory for the current B-constructions. As we shall see in \cref{thm:F}, the nontrivial relations \cref{dotmovecrossing1}--\cref{dotmovecupcaps} are determined by the familiar and simpler relations in Definition \ref{def:affineBwebC} which suffice for $\AWb$ over $\C[x]$. In particular, the thin-strand subcategory of $\AWb$ can be identified with the affine Brauer category \cite{RS2019}.

\subsubsection{}

Let $\emod(\so_N\Mod)$ denote the monoidal category of the endofunctors of the category $\so_N$-Mod
of $U(\so_N)$-modules. The functor $\cG_N$ below extends a functor from the affine Brauer category to $\emod(\so_N\Mod)$ in \cite{RS2019}.

\begin{alphatheorem} [\cref{AWbrep}]
\label{thm:C}
    There is a monoidal functor 
    \[
    \cG_N: \AWb_{\mathbb C}(N)\longrightarrow \emod(\so_N\Mod)
    \]
    such that $\cG_N(a)=-\otimes \bigwedge^a V$ and the images of $\cG_N$ on the generating morphisms of $\AWb_{\mathbb C}(N)$ are explicitly given in \eqref{cG}.
\end{alphatheorem}

The objects in $\AWb$ are naturally identified with strict compositions $\mu \in \Lambda_{\text{st}} =\cup_{m\ge 0} \Lambda_{\text{st}}(m)$. 
Recall the combinatorial sets $\SMat_{\lambda,\mu}$ from \eqref{eq:Mat} and $\PSMat_{\lambda,\mu}$ from \eqref{eq:MSMat}; elements in these sets can be naturally visualized as diagrammatic morphisms composed from the generating morphisms of $\AWb$. 

Denote by $\Xi =\Hom_{\AWb}(\unit,\unit)$ the ring of bubbles. Note that $\Xi$ acts naturally on each hom-space $\Hom_{\AWb}(\mu,\lambda)$ by horizontal concatenation from the left.
We introduce in \eqref{bubble+dot+par} the shorthand notation for thick bubbles $\dotbubble_a:=\bubbledota$, for $a\ge 1$, and set $\dotbubble_0=1$. 

\begin{alphatheorem} [Basis 
\cref{thm:spanAWb,thm:basisAWb}]
\label{thm:D}
\qquad
\begin{enumerate}
    \item $\Xi$ is a free commutative $\mathbb J'$-algebra in the even bubbles $\{\dotbubble_{2a} \mid a\in \Z_{\ge 1}\}$.
    \item 
    The morphism space $\Hom_{\AWb}(\mu,\lambda)$ is a free left (or right) $\Xi$-module with basis $\PSMat_{\lambda,\mu}$, for $\lambda, \mu\in \Lambda_{\text{st}}$.
\end{enumerate}
\end{alphatheorem}

The proof of this basis theorem consists of two parts: the spanning part requires an in-depth study of lollipops and bubbles (i.e., morphisms in $\Hom_{\AWb}(\unit,\mu)$ and $\Hom_{\AWb}(\unit,\unit)$), which are new type $B$ features. We establish a number of additional relations involving lollipops and bubbles which allow us to reduce general diagrams to the canonical ones appearing in $\PSMat_{\lambda,\mu}$, and the whole reduction procedure is much more technically challenging than the type $A$ case in \cite{SWweb}; see Section \ref{sec:spanning} and Appendix~ \ref{app:lollipop}. The linear independence part is based on Theorem \ref{thm:C} in an essential way. 

While there are many variants of web categories in the literature, it is often challenging to formulate (even conjecturally) basis results such as \cref{thm:D}. We view our basis theorem (as for affine web categories in the previous works \cite{SWweb,SSW25}) as a supporting evidence that our affine B-web category is natural and fundamental. 

\begin{alphacorollary} [\cref{cor:basis:Wb}]
    \label{cor:E}
    The morphism space $\Hom_{\Wb}(\mu,\lambda)$ is a free $\mathbb J$-module with basis $\SMat_{\lambda,\mu}$, for $\lambda, \mu\in \Lambda_{\text{st}}$.  
\end{alphacorollary}

\cref{thm:A} can be used to give a more direct argument for the linear independence part of \cref{cor:E}.

\subsubsection{}

Over a field of characteristic zero, a presentation of the B-web category can be simplified. We formulate a $\C[x]$-linear variant of the affine B-web category $\AWCb$ in Definition \ref{def:affineBwebC}. After base change, we also regard \(\AWb\) as a \(\mathbb C[x]\)-linear category below.

\begin{alphatheorem} [Theorem
\ref{thm:AWiso}]
\label{thm:F}
 The $\C[x]$-linear monoidal categories $\AWCb$ and $\AWb$ are isomorphic by matching the generating objects and generating morphisms in the same symbols. 
\end{alphatheorem}

It is much easier to check the relations in $\AWCb$ in practice, and Theorem \ref{thm:F} allows us to derive the relations for $\AWb$ from $\AWCb$. The affine B-web category $\AWCb$ contains a finite B-web category $\WCb$ over $\C$ as a subcategory; an isomorphic variant of $\WCb$
can be obtained by taking the $q=1$ limit from a (non-monoidal) quantum web category constructed in \cite{ST19}.

The main part of the proof of Theorem \ref{thm:F} is to derive the  intricate defining relations \cref{dotmovecrossing1}--\cref{dotmovecupcaps} for $\AWb$ from the simple and standard relations for $\AWCb$ (which arise from type A web and Brauer categories). This is how we were able to write down in the first place the defining relations for $\AWb$ in Definition \ref{def:affineBWeb}. Note that the presentation for $\AWb$ uses crucially the (redundant) thick crossings as part of the generating morphisms as the thin crossing appears already in $\AWCb$ and in the Brauer category. The strategy of starting with a simpler presentation over $\C$ this way for type A, B and beyond was first spelled out in \cite{SWweb}.

\subsubsection{}

The basis Theorem \ref{thm:D} indicates that only the even bubbles are algebraically independent. Recall a sequence of unsigned Genocchi numbers $\{G_{2k}\mid k\ge 1\}$ from Remark~ \ref{rem:Genocchi}; see also \cref{recursiondm,identity2}. It is remarkable that such a classic number sequence appears naturally in the rich combinatorial setting of B-bubbles. 

\begin{alphatheorem}  [\cref{RelationBubble,bubbleOdd}]
    \label{thm:G}
    For $m\in \N$, we have
    \begin{align*} 
        \dotbubble_{2m+1} = \sum_{k=0}^m (-1)^k(2k)!\, G_{2k+2}\, \binom{x-2m+2k}{2k+2} \binom{x-2m+2k-1}{2k} \dotbubble_{2m-2k}.
    \end{align*}
\end{alphatheorem}
This formula provides necessary and sufficient admissibility conditions for representations of $\AWb$ where the bubbles act as scalars.

As a part of our long-term program initiated in \cite{SWweb, SWSchur, SSW25}, this work has opened up several research directions: 
\begin{itemize}
   \item 
   Define and develop the representation theory of cyclotomic B-web categories. As $\AWb$ admits caps/cups and bubbles of various thickness, this is highly nontrivial; in particular, it will lead to intimate connections between bubbles and centers of the universal enveloping algebras of classical Lie algebras. 
\item 
It is also natural to ask for a $q$-deformed version of the finite and affine B-web categories, generalizing our previous construction in \cite{SSW25}; see \cite{BW25} for some related type $B$ work. 
\item 
Enlarge $\AWb$ (and its cyclotomic quotients) to an affine B-Schur category (and its cyclotomic quotients), generalizing the type $A$ constructions in \cite{SWSchur}.  
\item 
Develop homological properties for all variants of B-web and B-Schur categories and their endomorphism algebras, such as triangular bases and quasi-hereditary algebra structures. They provide new settings for Schur type dualities and Schur functors. 
\item 
Develop the representation theories of these B-web and B-Schur categories, which are expected to be closely related to categorification of modules of iquantum groups; for  categorification in the type A setting, see \cite{SWCat} and references therein. 
\end{itemize}
We plan to follow up in future work on these topics. 


\subsection{The organization}

The paper is organized as follows. 
In Section \ref{sec:BWeb}, we introduce the B-web category $\Wb$, and derive a number of secondary relations. Then we construct a representation of $\Wb_\C(N)$ on $\so_N$-mod. In Section \ref{sec:affine_BWeb}, we formulate the affine B-web category $\AWb$ and establish new relations which will be needed for proving the spanning result in Section \ref{sec:spanning}. 
Section \ref{sec:lol:bubbles} reduces lollipops/bubbles of general type to simpler forms, with some reduction steps postponed to Appendix \ref{app:lollipop}. Section \ref{sec:spanning} uses these reductions and various relations formulated in Section \ref{sec:affine_BWeb} to establish a spanning set for an arbitrary morphism space in $\AWb$. In Section \ref{sec:basis} and Appendix \ref{app:independence}, we show the spanning set from the previous section is indeed a basis for the morphism spaces of $\AWb$, using the functor $\mathcal G_N$ from Theorem \ref{thm:C}. 
In Section \ref{sec:char0}, we provide a simpler presentation for $\AWb$ over $\C[x]$, and the proof of the equivalence can be found in Appendix \ref{app:Wiso}. The bases for morphism spaces of $\AWb$ indicate that only the even bubbles are algebraically independent. In Section \ref{sec:thick bubbles}, we identify the odd bubbles as polynomials in the even bubbles. 


%

\section{B-webs}
  \label{sec:BWeb}
    
In this section we introduce a diagrammatic monoidal category, the Brauer web (or B-web) category. 

\subsection{Definition of the B-web category} 

For any \( m \in \mathbb{N} \), a composition (respectively, partition) \( \mu = (\mu_1, \mu_2, \ldots, \mu_n) \) of \( m \) is defined as a sequence of (respectively, weakly decreasing) non-negative integers such that \( \sum_{i \geq 1} \mu_i = m \). A composition \( \mu \) is called \textit{strict} if \( \mu_i > 0 \) for all \( i \). The length of \( \mu \) is denoted by \( l(\mu) \). Let \( \Lambda_{\text{st}}(m) \) (respectively, \( \Par(m) \)) denote the set of all strict compositions (respectively, partitions) of \( m \). We further denote \( \Lambda_{\text{st}} =\cup_{m\ge 0}\Lambda_{\text{st}}(m) \) and \( \Par=\cup_{m\ge 0}\Par(m)\).

Throughout this paper, let $\kk$ be a commutative ring with $1$. All categories and functors will be $\kk$-linear without further mention. Introduce the commutative ring
\begin{equation}
    \label{eq:J}
    \mathbb J:=\{f(x)\in \Q[x]\mid f(n)\in\Z \text{ for all }n\in\Z\}\subseteq\Q[x].
\end{equation}

\begin{definition}
\label{def-Bweb}
The Brauer web category (or B-web category, for short) $\Wb =\Wb_{\mathbb J}$ is the strict $\mathbb J$-linear monoidal category generated by objects $a\in \mathbb Z_{\ge1}$. The object $a$ and its identity morphism  will be drawn as $\stra.$ 
The generating morphisms are merges, splits, (thick) crossings, cups and caps depicted as (for $a,b \ge 1$)
\begin{align}
\label{merge+split+crossing}
\begin{tikzpicture}[anchorbase,color=\clr]
	\draw[-,line width=1pt] (0.28,-.3) to (0.08,0.04);
	\draw[-,line width=1pt] (-0.12,-.3) to (0.08,0.04);
	\draw[-,line width=1.4pt] (0.08,.4) to (0.08,0);
        \node at (-0.22,-.4) {$\scriptstyle a$};
        \node at (0.35,-.4) {$\scriptstyle b$};\node at (0,.55){$\scriptstyle a+b$};\end{tikzpicture} 
&:(a,b) \rightarrow (a+b),
\quad
\begin{tikzpicture}[anchorbase,color=\clr]
	\draw[-,line width=1.4pt] (0.08,-.3) to (0.08,0.04);
	\draw[-,line width=1pt] (0.28,.4) to (0.08,0);
	\draw[-,line width=1pt] (-0.12,.4) to (0.08,0);
        \node at (-0.22,.5) {$\scriptstyle a$};
        \node at (0.36,.5) {$\scriptstyle b$};
        \node at (0.1,-.45){$\scriptstyle a+b$};
\end{tikzpicture}
:(a+b)\rightarrow (a,b),
\quad
\crossing :(a,b) \rightarrow (b,a),
\\
\label{gen:cupcap}
\cupa &: \unit \rightarrow (a,a), 
\qquad \qquad
\capa  : (a,a) \rightarrow \unit
\;,  
\end{align}
subject to the following relations \eqref{webassoc}--\eqref{bubblecap}, for $a,b,c,d \in \Z_{\ge 1}$ with $d-a=c-b$:
\begin{gather}
\label{webassoc}
\begin{tikzpicture}[baseline = 0,color=\clr]
	\draw[-,thick] (0.35,-.3) to (0.08,0.14);
	\draw[-,thick] (0.1,-.3) to (-0.04,-0.06);
	\draw[-,line width=1pt] (0.085,.14) to (-0.035,-0.06);
	\draw[-,thick] (-0.2,-.3) to (0.07,0.14);
	\draw[-,line width=1.4pt] (0.08,.45) to (0.08,.1);
        \node at (0.45,-.41) {$\scriptstyle c$};
        \node at (0.07,-.4) {$\scriptstyle b$};
        \node at (-0.28,-.41) {$\scriptstyle a$};
\end{tikzpicture}
=
\begin{tikzpicture}[baseline = 0, color=\clr]
	\draw[-,thick] (0.36,-.3) to (0.09,0.14);
	\draw[-,thick] (0.06,-.3) to (0.2,-.05);
	\draw[-,line width=1pt] (0.07,.14) to (0.19,-.06);
	\draw[-,thick] (-0.19,-.3) to (0.08,0.14);
	\draw[-,line width=1.4pt] (0.08,.45) to (0.08,.1);
        \node at (0.45,-.41) {$\scriptstyle c$};
        \node at (0.07,-.4) {$\scriptstyle b$};
        \node at (-0.28,-.41) {$\scriptstyle a$};
\end{tikzpicture}\:,
\qquad
\begin{tikzpicture}[baseline = -1mm, color=\clr]
	\draw[-,thick] (0.35,.3) to (0.08,-0.14);
	\draw[-,thick] (0.1,.3) to (-0.04,0.06);
	\draw[-,line width=1pt] (0.085,-.14) to (-0.035,0.06);
	\draw[-,thick] (-0.2,.3) to (0.07,-0.14);
	\draw[-,line width=1.4pt] (0.08,-.45) to (0.08,-.1);
        \node at (0.45,.4) {$\scriptstyle c$};
        \node at (0.07,.42) {$\scriptstyle b$};
        \node at (-0.28,.4) {$\scriptstyle a$};
\end{tikzpicture}
=\begin{tikzpicture}[baseline = -1mm, color=\clr]
	\draw[-,thick] (0.36,.3) to (0.09,-0.14);
	\draw[-,thick] (0.06,.3) to (0.2,.05);
	\draw[-,line width=1pt] (0.07,-.14) to (0.19,.06);
	\draw[-,thick] (-0.19,.3) to (0.08,-0.14);
	\draw[-,line width=1.4pt] (0.08,-.45) to (0.08,-.1);
        \node at (0.45,.4) {$\scriptstyle c$};
        \node at (0.07,.42) {$\scriptstyle b$};
        \node at (-0.28,.4) {$\scriptstyle a$};
\end{tikzpicture}\:,
\\
  \label{splitmerge}
\begin{tikzpicture}[baseline = -1mm,scale=.8,color=\clr]
\draw[-,line width=1.6pt] (0.08,-.8) to (0.08,-.5);
\draw[-,line width=1.6pt] (0.08,.3) to (0.08,.6);
\draw[-,thick] (0.1,-.51) to [out=45,in=-45] (0.1,.31);
\draw[-,thick] (0.06,-.51) to [out=135,in=-135] (0.06,.31);
\node at (-.33,-.05) {$\scriptstyle a$};
\node at (.45,-.05) {$\scriptstyle b$};
\end{tikzpicture}
= 
{\displaystyle \binom{a+b}{a} }
\begin{tikzpicture}[baseline = -1mm,scale=.8,color=\clr]
	\draw[-,line width=1.4pt] (0.08,-.8) to (0.08,.6);
        \node at (.08,-1) {$\scriptstyle a+b$};
\end{tikzpicture}, 
\qquad
\begin{tikzpicture}[baseline = 7.5pt,scale=1, color=\clr]
	\draw[-,line width=1pt] (0,0) to (.29,.3) to (.29,.7) to (0,1);
	\draw[-,line width=1pt] (.6,0) to (.31,.3) to (.31,.7) to (.6,1);
        \node at (0,1.13) {$\scriptstyle b$};
        \node at (0.63,1.13) {$\scriptstyle d$};
        \node at (0,-.1) {$\scriptstyle a$};
        \node at (0.63,-.1) {$\scriptstyle c$};
\end{tikzpicture}
=
\sum_{\substack{0 \leq s \leq \min(a,b)\\0 \leq t \leq \min(c,d)\\t-s=d-a}}
\begin{tikzpicture}[baseline = 7.5pt,scale=1, color=\clr]
	\draw[-,thick] (0.58,0) to (0.58,.2) to (.02,.8) to (.02,1);
	\draw[-,thick] (0.02,0) to (0.02,.2) to (.58,.8) to (.58,1);
	\draw[-,thick] (0,0) to (0,1);
	\draw[-,line width=1pt] (0.61,0) to (0.61,1);
        \node at (0,1.13) {$\scriptstyle b$};
        \node at (0.6,1.13) {$\scriptstyle d$};
        \node at (0,-.1) {$\scriptstyle a$};
        \node at (0.6,-.1) {$\scriptstyle c$};
        \node at (-0.1,.5) {$\scriptstyle s$};
        \node at (0.77,.5) {$\scriptstyle t$};
\end{tikzpicture},
\\
    \label{zigzag}
\begin{tikzpicture}[baseline = -1mm,scale=.8,color=\clr]
        \draw[-,line width=1.2pt] (0,-0.5) to (0,0) to
        [out=up,in=left] (.25,.25) to [out=right,in=up] (.5,0) 
        to [out=down,in=left] (.75,-0.25) to [out=right,in=down] (1,0)
        to (1,.5);
        \node at (0,-.65){$\scriptstyle{a}$};
    \end{tikzpicture}
    =
       \begin{tikzpicture}[baseline = -1mm,scale=.8,color=\clr]
        \draw[-,line width=1.2pt] (0,-.5) to (0,.5);
        \node at (0,-.65){$\scriptstyle{a}$};
    \end{tikzpicture} ,
    \quad 
\begin{tikzpicture}[baseline = -1mm,scale=.8,color=\clr]
        \draw[-,line width=1.2pt] (0,-0.5) to (0,0) to
        [out=up,in=right] (-.25,.25) to [out=left,in=up] (-.5,0) 
        to [out=down,in=right] (-.75,-0.25) to [out=left,in=down] (-1,0)
        to (-1,.5);
        \node at (0,-.65){$\scriptstyle{a}$};
    \end{tikzpicture}
    =
       \begin{tikzpicture}[baseline = -1mm,scale=.8,color=\clr]
        \draw[-,line width=1.2pt] (0,-.5) to (0,.5);
        \node at (0,-.65){$\scriptstyle{a}$};
    \end{tikzpicture},
    \qquad\;\,
    %
    	\begin{tikzpicture}[anchorbase,color=\clr,line width=1pt]
	  \draw[-](.6,.4) to (.1,-.3);
		\draw[-] (0.6,-0.3) to[out=140, in=0] (0.1,0.2);
		\draw[-] (0.1,0.2) to[out = -180, in = 90] (-0.2,-0.3);
        \node at (-.2,-.45){$\scriptstyle{a}$};
        \node at (.6,-.45){$\scriptstyle{a}$};
        \node at (.1,-.45){$\scriptstyle{b}$};
	\end{tikzpicture}
	 =
	\begin{tikzpicture}[anchorbase,color=\clr,line width=1pt]
		\draw[-](-.5,.4) to (0,-.3);
		\draw[-] (0.3,-0.3) to[out=90, in=0] (0,0.2);
		\draw[-] (0,0.2) to[out = -180, in = 40] (-0.5,-0.3);
        \node at (-.5,-.45){$\scriptstyle{a}$};
        \node at (.3,-.45){$\scriptstyle{a}$};
        \node at (0,-.45){$\scriptstyle{b}$};
	\end{tikzpicture} ,
 \quad
	\begin{tikzpicture}[anchorbase,color=\clr,line width=1pt]
		\draw[-](-.5,-.3) to (0,.4);
		\draw[-] (0.3,0.4) to[out=-90, in=0] (0,-0.1);
		\draw[-] (0,-0.1) to[out = 180, in = -40] (-0.5,0.4);
        \node at (-.5,.5){$\scriptstyle{a}$};
        \node at (.3,.5){$\scriptstyle{a}$};
        \node at (.0,.52){$\scriptstyle{{b}}$};
	\end{tikzpicture}
	=
	\begin{tikzpicture}[anchorbase,color=\clr,line width=1pt]
		\draw[-](.6,-.3) to (.1,.4);
	  \draw[-] (0.6,0.4) to[out=-140, in=0] (0.1,-0.1);
		\draw[-] (0.1,-0.1) to[out = 180, in = -90] (-0.2,0.4);
        \node at (-.2,.5){$\scriptstyle{a}$};
        \node at (.6,.5){$\scriptstyle{a}$};
        \node at (.1,.52){$\scriptstyle{{b}}$};
	\end{tikzpicture} ,
    \\
    \label{mergesplitmove}
        \begin{tikzpicture}[anchorbase,color=\clr]
	\draw[-,line width=1.4pt] (0.08,-.3) to (0.08,0.04);
	\draw[-,line width=1pt] (0.28,.4) to (0.08,0);
	\draw[-,line width=1pt] (-0.12,.4) to (0.08,0);
    \draw[-,line width=1pt] (.28,.4) to [out=up,in=up] (.7,.4) to (.7,-.3);
    \draw[-,line width=1pt] (-.12,.4) to [out=up,in=up] (1,.4) to (1,-.3);
     \node at (1,-.45) {$\scriptstyle {b}$};
     \node at (.7,-.45) {$\scriptstyle {a}$};
\end{tikzpicture}
 =
\begin{tikzpicture}[anchorbase,color=\clr]
    \draw[-,line width=1.2pt] (0,-.3) to (0,.4);
    \draw[-,line width=1.2pt] (0,.4) to [out=up,in=up] (.6,.4) to (.6,0);
    \draw[-,line width=1pt] (.6,0) to (.8,-.3) ; 
    \draw[-,line width=1pt] (.6,0) to (.4,-.3); 
     \node at (.8,-.45) {$\scriptstyle {b}$};
      \node at (.4,-.45){$\scriptstyle {a}$};
\end{tikzpicture},
\qquad
 \begin{tikzpicture}[xscale=-1,anchorbase,color=\clr]
	\draw[-,line width=1.4pt] (0.08,-.3) to (0.08,0.04);
	\draw[-,line width=1pt] (0.28,.4) to (0.08,0);
	\draw[-,line width=1pt] (-0.12,.4) to (0.08,0);
    \draw[-,line width=1pt] (.28,.4) to [out=up,in=up] (.7,.4) to (.7,-.3);
    \draw[-,line width=1pt] (-.12,.4) to [out=up,in=up] (1,.4) to (1,-.3);
     \node at (1,-.45) {$\scriptstyle {a}$};
     \node at (.7,-.45) {$\scriptstyle {b}$};
\end{tikzpicture}
=
\begin{tikzpicture}[xscale=-1,anchorbase,color=\clr]
    \draw[-,line width=1.2pt] (0,-.3) to (0,.4);
    \draw[-,line width=1.2pt] (0,.4) to [out=up,in=up] (.6,.4) to (.6,0);
    \draw[-,line width=1pt] (.6,0) to (.8,-.3) ; 
    \draw[-,line width=1pt] (.6,0) to (.4,-.3); 
     \node at (.8,-.45) {$\scriptstyle {a}$};
      \node at (.4,-.45){$\scriptstyle {b}$};
\end{tikzpicture} \:\: ,
    \\
\bubblea = \binom{x}{a} ,
  \qquad
\begin{tikzpicture}[yscale=-1, baseline = -1mm,scale=.8,color=\clr]
\draw[-,line width=1.6pt] (0.1,.3) to (0.1,.6);
\draw[-,thick] (0.1,-.51) to [out=right,in=-45] (0.1,.31);
\draw[-,thick] (0.1,-.51) to [out=left,in=-135] (0.1,.31);
\node at (-.33,-.05) {$\scriptstyle a$};
\node at (.45,-.05) {$\scriptstyle a$};
\end{tikzpicture}
=0 ,
\qquad
    \begin{tikzpicture}[anchorbase, scale=1, color=\clr]
 \draw[line width=1.2pt] (-0.15,0) arc(180:0:0.15) to[out=down,in=60] (-0.15,-0.4);
        \draw[line width=1.2pt] (0.15,-0.4) to[out=120,in=down] (-0.15,0);
        \node at (-.15,-.5) {$\scriptstyle a$};
        \node at (.15,-.5) {$\scriptstyle a$};
\end{tikzpicture}
= (-1)^a\!\capa .
  \label{bubblecap} 
\end{gather} 
\end{definition}
We often choose to omit some of the labels when they can be read off from others. For example, the thick strands in \eqref{webassoc} are implicitly labeled by $a+b+c$, the thick strand in the middle relation in \eqref{bubblecap} is implicitly labeled by $2a$, and so on. 

By base change, sometimes we view $\Wb$ as $\Q[x]$-linear or $\C[x]$-linear. 
We consider some other base changes later in this section.

By the associativity \eqref{webassoc}, we can define 3-fold merge and split as follows:
\begin{align}
\begin{tikzpicture}[baseline = 0, color=\clr]
	\draw[-,thick] (0.35,-.3) to (0.09,0.14);
	\draw[-,thick] (0.08,-.3) to (.08,0.1);
	\draw[-,thick] (-0.2,-.3) to (0.07,0.14);
	\draw[-,line width=1.6pt] (0.08,.45) to (0.08,.1);
        \node at (0.45,-.41) {$\scriptstyle c$};
        \node at (0.07,-.4) {$\scriptstyle b$};
        \node at (-0.28,-.41) {$\scriptstyle a$};
\end{tikzpicture}
:=
\begin{tikzpicture}[baseline = 0, color=\clr]
	\draw[-,thick] (0.35,-.3) to (0.08,0.14);
	\draw[-,thick] (0.1,-.3) to (-0.04,-0.06);
	\draw[-,line width=1pt] (0.085,.14) to (-0.035,-0.06);
	\draw[-,thick] (-0.2,-.3) to (0.07,0.14);
	\draw[-,line width=1.6pt] (0.08,.45) to (0.08,.1);
        \node at (0.45,-.41) {$\scriptstyle c$};
        \node at (0.07,-.4) {$\scriptstyle b$};
        \node at (-0.28,-.41) {$\scriptstyle a$};
\end{tikzpicture}
&=
\begin{tikzpicture}[baseline = 0, color=\clr]
	\draw[-,thick] (0.36,-.3) to (0.09,0.14);
	\draw[-,thick] (0.06,-.3) to (0.2,-.05);
	\draw[-,line width=1pt] (0.07,.14) to (0.19,-.06);
	\draw[-,thick] (-0.19,-.3) to (0.08,0.14);
	\draw[-,line width=1.6pt] (0.08,.45) to (0.08,.1);
        \node at (0.45,-.41) {$\scriptstyle c$};
        \node at (0.07,-.4) {$\scriptstyle b$};
        \node at (-0.28,-.41) {$\scriptstyle a$};
\end{tikzpicture}\:,&
\begin{tikzpicture}[baseline = 0, color=\clr]
	\draw[-,thick] (0.35,.3) to (0.09,-0.14);
	\draw[-,thick] (0.08,.3) to (.08,-0.1);
	\draw[-,thick] (-0.2,.3) to (0.07,-0.14);
	\draw[-,line width=1.6pt] (0.08,-.45) to (0.08,-.1);
        \node at (0.45,.41) {$\scriptstyle c$};
        \node at (0.07,.43) {$\scriptstyle b$};
        \node at (-0.28,.41) {$\scriptstyle a$};
\end{tikzpicture}
:=
\begin{tikzpicture}[baseline = 0, color=\clr]
	\draw[-,thick] (0.35,.3) to (0.08,-0.14);
	\draw[-,thick] (0.1,.3) to (-0.04,0.06);
	\draw[-,line width=1pt] (0.085,-.14) to (-0.035,0.06);
	\draw[-,thick] (-0.2,.3) to (0.07,-0.14);
	\draw[-,line width=1.6pt] (0.08,-.45) to (0.08,-.1);
        \node at (0.45,.41) {$\scriptstyle c$};
        \node at (0.07,.43) {$\scriptstyle b$};
        \node at (-0.28,.41) {$\scriptstyle a$};
\end{tikzpicture}
&=
\begin{tikzpicture}[baseline = 0, color=\clr]
	\draw[-,thick] (0.36,.3) to (0.09,-0.14);
	\draw[-,thick] (0.06,.3) to (0.2,.05);
	\draw[-,line width=1pt] (0.07,-.14) to (0.19,.06);
	\draw[-,thick] (-0.19,.3) to (0.08,-0.14);
	\draw[-,line width=1.6pt] (0.08,-.45) to (0.08,-.1);
        \node at (0.45,.41) {$\scriptstyle c$};
        \node at (0.07,.43) {$\scriptstyle b$};
        \node at (-0.28,.41) {$\scriptstyle a$};
\end{tikzpicture}\:.
\end{align}    
More general $r$-fold merges and splits, for any $r$, can be similarly defined.
 
We list some well-known implied relations in $\Wb$ (which follow from \eqref{webassoc}--\eqref{splitmerge}, same as the defining relations for type $A$ webs), cf. \cite{CKM, BEEO, SWweb}:
\begin{align}
\begin{tikzpicture}[anchorbase,scale=.7,color=\clr]
	\draw[-,line width=1.6pt] (0.08,.3) to (0.08,.5);
\draw[-,thick] (-.2,-.8) to [out=45,in=-45] (0.1,.31);
\draw[-,thick] (.36,-.8) to [out=135,in=-135] (0.06,.31);
        \node at (-.3,-.95) {$\scriptstyle a$};
        \node at (.45,-.95) {$\scriptstyle b$};
\end{tikzpicture}
&=
\begin{tikzpicture}[anchorbase,scale=.7,color=\clr]
	\draw[-,line width=1.6pt] (0.08,.1) to (0.08,.5);
\draw[-,thick] (.46,-.8) to [out=100,in=-45] (0.1,.11);
\draw[-,thick] (-.3,-.8) to [out=80,in=-135] (0.06,.11);
        \node at (-.3,-.95) {$\scriptstyle a$};
        \node at (.43,-.95) {$\scriptstyle b$};
\end{tikzpicture},
& 
\begin{tikzpicture}[anchorbase,scale=.7,color=\clr]
	\draw[-,line width=1.6pt] (0.08,-.3) to (0.08,-.5);
\draw[-,thick] (-.2,.8) to [out=-45,in=45] (0.1,-.31);
\draw[-,thick] (.36,.8) to [out=-135,in=135] (0.06,-.31);
        \node at (-.3,.95) {$\scriptstyle a$};
        \node at (.45,.95) {$\scriptstyle b$};
\end{tikzpicture}
&=
\begin{tikzpicture}[anchorbase,scale=.7,color=\clr]
	\draw[-,line width=1.6pt] (0.08,-.1) to (0.08,-.5);
\draw[-,thick] (.46,.8) to [out=-100,in=45] (0.1,-.11);
\draw[-,thick] (-.3,.8) to [out=-80,in=135] (0.06,-.11);
        \node at (-.3,.95) {$\scriptstyle a$};
        \node at (.43,.95) {$\scriptstyle b$};
\end{tikzpicture},
\label{swallows}
\\
\begin{tikzpicture}[anchorbase,scale=0.7,color=\clr]
	\draw[-,thick] (0.4,0) to (-0.6,1);
	\draw[-,thick] (0.08,0) to (0.08,1);
	\draw[-,thick] (0.1,0) to (0.1,.6) to (.5,1);
        \node at (0.6,1.13) {$\scriptstyle c$};
        \node at (0.1,1.16) {$\scriptstyle b$};
        \node at (-0.65,1.13) {$\scriptstyle a$};
\end{tikzpicture}
\!\!=\!\!
\begin{tikzpicture}[anchorbase,scale=0.7,color=\clr]
	\draw[-,thick] (0.7,0) to (-0.3,1);
	\draw[-,thick] (0.08,0) to (0.08,1);
	\draw[-,thick] (0.1,0) to (0.1,.2) to (.9,1);
        \node at (0.9,1.13) {$\scriptstyle c$};
        \node at (0.1,1.16) {$\scriptstyle b$};
        \node at (-0.4,1.13) {$\scriptstyle a$};
\end{tikzpicture},\:
\begin{tikzpicture}[anchorbase,scale=0.7,color=\clr]
	\draw[-,thick] (-0.4,0) to (0.6,1);
	\draw[-,thick] (-0.08,0) to (-0.08,1);
	\draw[-,thick] (-0.1,0) to (-0.1,.6) to (-.5,1);
        \node at (0.7,1.13) {$\scriptstyle c$};
        \node at (-0.1,1.16) {$\scriptstyle b$};
        \node at (-0.6,1.13) {$\scriptstyle a$};
\end{tikzpicture}
 &=\!\!
\begin{tikzpicture}[anchorbase,scale=0.7,color=\clr]
	\draw[-,thick] (-0.7,0) to (0.3,1);
	\draw[-,thick] (-0.08,0) to (-0.08,1);
	\draw[-,thick] (-0.1,0) to (-0.1,.2) to (-.9,1);
        \node at (0.4,1.13) {$\scriptstyle c$};
        \node at (-0.1,1.16) {$\scriptstyle b$};
        \node at (-0.95,1.13) {$\scriptstyle a$};
\end{tikzpicture},
&
\:\begin{tikzpicture}[baseline=-3.3mm,scale=0.7,color=\clr]
	\draw[-,thick] (0.4,0) to (-0.6,-1);
	\draw[-,thick] (0.08,0) to (0.08,-1);
	\draw[-,thick] (0.1,0) to (0.1,-.6) to (.5,-1);
        \node at (0.6,-1.13) {$\scriptstyle c$};
        \node at (0.07,-1.13) {$\scriptstyle b$};
        \node at (-0.6,-1.13) {$\scriptstyle a$};
\end{tikzpicture}
&=\!\!
\begin{tikzpicture}[baseline=-3.3mm,scale=0.7,color=\clr]
	\draw[-,thick] (0.7,0) to (-0.3,-1);
	\draw[-,thick] (0.08,0) to (0.08,-1);
	\draw[-,thick] (0.1,0) to (0.1,-.2) to (.9,-1);
        \node at (1,-1.13) {$\scriptstyle c$};
        \node at (0.1,-1.13) {$\scriptstyle b$};
        \node at (-0.4,-1.13) {$\scriptstyle a$};
\end{tikzpicture},\:
\begin{tikzpicture}[baseline=-3.3mm,scale=0.7,color=\clr]
	\draw[-,thick] (-0.4,0) to (0.6,-1);
	\draw[-,thick] (-0.08,0) to (-0.08,-1);
	\draw[-,thick] (-0.1,0) to (-0.1,-.6) to (-.5,-1);
        \node at (0.6,-1.13) {$\scriptstyle c$};
        \node at (-0.1,-1.13) {$\scriptstyle b$};
        \node at (-0.6,-1.13) {$\scriptstyle a$};
\end{tikzpicture}
\!\!=\!\!
\begin{tikzpicture}[baseline=-3.3mm,scale=0.7,color=\clr]
	\draw[-,thick] (-0.7,0) to (0.3,-1);
	\draw[-,thick] (-0.08,0) to (-0.08,-1);
	\draw[-,thick] (-0.1,0) to (-0.1,-.2) to (-.9,-1);
        \node at (0.34,-1.13) {$\scriptstyle c$};
        \node at (-0.1,-1.13) {$\scriptstyle b$};
        \node at (-0.95,-1.13) {$\scriptstyle a$};
\end{tikzpicture},
\label{sliders}
\\
\label{braid}
\mathord{
\begin{tikzpicture}[baseline = -1mm,scale=0.8,color=\clr]
	\draw[-,thick] (0.28,0) to[out=90,in=-90] (-0.28,.6);
	\draw[-,thick] (-0.28,0) to[out=90,in=-90] (0.28,.6);
	\draw[-,thick] (0.28,-.6) to[out=90,in=-90] (-0.28,0);
	\draw[-,thick] (-0.28,-.6) to[out=90,in=-90] (0.28,0);
        \node at (0.3,-.75) {$\scriptstyle b$};
        \node at (-0.3,-.75) {$\scriptstyle a$};
\end{tikzpicture}
}
&=
\mathord{
\begin{tikzpicture}[baseline = -1mm,scale=0.8,color=\clr]
	\draw[-,thick] (0.2,-.6) to (0.2,.6);
	\draw[-,thick] (-0.2,-.6) to (-0.2,.6);
        \node at (0.2,-.75) {$\scriptstyle b$};
        \node at (-0.2,-.75) {$\scriptstyle a$};
\end{tikzpicture}
}\:,
& 
\mathord{
\begin{tikzpicture}[baseline = -1mm,scale=0.8,color=\clr]
	\draw[-,thick] (0.45,.6) to (-0.45,-.6);
	\draw[-,thick] (0.45,-.6) to (-0.45,.6);
        \draw[-,thick] (0,-.6) to[out=90,in=-90] (-.45,0);
        \draw[-,thick] (-0.45,0) to[out=90,in=-90] (0,0.6);
        \node at (0,-.77) {$\scriptstyle b$};
        \node at (0.5,-.77) {$\scriptstyle c$};
        \node at (-0.5,-.77) {$\scriptstyle a$};
\end{tikzpicture}
}
&=
\mathord{
\begin{tikzpicture}[baseline = -1mm,scale=0.8,color=\clr]
	\draw[-,thick] (0.45,.6) to (-0.45,-.6);
	\draw[-,thick] (0.45,-.6) to (-0.45,.6);
        \draw[-,thick] (0,-.6) to[out=90,in=-90] (.45,0);
        \draw[-,thick] (0.45,0) to[out=90,in=-90] (0,0.6);
        \node at (0,-.77) {$\scriptstyle b$};
        \node at (0.5,-.77) {$\scriptstyle c$};
        \node at (-0.5,-.77) {$\scriptstyle a$};
\end{tikzpicture}
}\:.
\end{align}

The defining relations for $\Wb$ are chosen to be convenient but they are not optimal. For example, the second relation in \eqref{mergesplitmove} follows from the first one and the fourth relation in \eqref{zigzag} follows from the third one using the zigzag relations (i.e., the first two relations in \eqref{zigzag}). 

\begin{rem}
    Our convention on the crossings differs by a sign on the thin crossing from the (affine) Brauer category in \cite{RS2019}. 
\end{rem}

\subsection{More implied relations in $\Wb$}

\begin{lemma}
    The following relations hold in $\Wb$, for $a\ge 1$: 
\begin{align}
 \begin{tikzpicture}[anchorbase, yscale=-1, color=\clr]
 \draw[line width=1.2pt] (-0.15,0) arc(180:0:0.15) to[out=down,in=60] (-0.15,-0.4);
        \draw[line width=1.2pt] (0.15,-0.4) to[out=120,in=down] (-0.15,0);
        \node at (-.15,-.5) {$\scriptstyle a$};
        \node at (.15,-.5) {$\scriptstyle a$};
\end{tikzpicture}
=(-1)^a\!\cupa ,
\qquad
    \begin{tikzpicture}[baseline = -1mm,scale=.8,color=\clr]
\draw[-,line width=1.6pt] (0.1,.3) to (0.1,.6);
\draw[-,thick] (0.1,-.51) to [out=right,in=-45] (0.1,.31);
\draw[-,thick] (0.1,-.51) to [out=left,in=-135] (0.1,.31);
\node at (-.33,-.05) {$\scriptstyle a$};
\node at (.45,-.05) {$\scriptstyle a$};
\end{tikzpicture}
=0.
\end{align}
\end{lemma}

\begin{proof}
    This follows from \eqref{zigzag} and \eqref{bubblecap}.
\end{proof}

\begin{lemma}
The following relations hold in $\Wb$, for $a,b \ge 1$: 
      \begin{align}
    \label{mergesplitslideleft}
\begin{tikzpicture}[yscale=-1,anchorbase,scale=.8,color=\clr]
	\draw[-,line width=1.4pt] (0.08,-.3) to (0.08,0.04);
	\draw[-,line width=1pt] (0.28,.4) to (0.08,0);
	\draw[-,line width=1pt] (-0.12,.4) to (0.08,0);
    \draw[-,line width=1pt] (.28,.4) to [out=up,in=up] (.7,.4) to (.7,-.3);
    \draw[-,line width=1pt] (-.12,.4) to [out=up,in=up] (1,.4) to (1,-.3);
     \node at (1,-.45) {$\scriptstyle {{b}}$};
     \node at (.7,-.45) {$\scriptstyle {{a}}$};
\end{tikzpicture}=\begin{tikzpicture}[yscale=-1,anchorbase,color=\clr]
    \draw[-,line width=1.2pt] (0,-.3) to (0,.4);
    \draw[-,line width=1.2pt] (0,.4) to [out=up,in=up] (.6,.4) to (.6,0);
    \draw[-,line width=1pt] (.6,0) to (.8,-.3) ; 
    \draw[-,line width=1pt] (.6,0) to (.4,-.3); 
     \node at (.8,-.45) {$\scriptstyle {{b}}$};
      \node at (.4,-.45){$\scriptstyle {{a}}$};
\end{tikzpicture},       
\qquad
  \begin{tikzpicture}[xscale=-1,yscale=-1,anchorbase,scale=.8,color=\clr]
	\draw[-,line width=1.4pt] (0.08,-.3) to (0.08,0.04);
	\draw[-,line width=1pt] (0.28,.4) to (0.08,0);
	\draw[-,line width=1pt] (-0.12,.4) to (0.08,0);
    \draw[-,line width=1pt] (.28,.4) to [out=up,in=up] (.7,.4) to (.7,-.3);
    \draw[-,line width=1pt] (-.12,.4) to [out=up,in=up] (1,.4) to (1,-.3);
     \node at (1,-.45) {$\scriptstyle {{a}}$};
     \node at (.7,-.45) {$\scriptstyle {{b}}$};
\end{tikzpicture}=\begin{tikzpicture}[xscale=-1,yscale=-1,anchorbase,color=\clr]
    \draw[-,line width=1.2pt] (0,-.3) to (0,.4);
    \draw[-,line width=1.2pt] (0,.4) to [out=up,in=up] (.6,.4) to (.6,0);
    \draw[-,line width=1pt] (.6,0) to (.8,-.3) ; 
    \draw[-,line width=1pt] (.6,0) to (.4,-.3); 
     \node at (.8,-.45) {$\scriptstyle {{a}}$};
      \node at (.4,-.45){$\scriptstyle {{b}}$};
\end{tikzpicture} .
    \end{align}
\end{lemma}

\begin{proof}
    This follows from \eqref{zigzag} and \eqref{mergesplitmove}.
  \end{proof}

\begin{lemma}
    The following relations hold in $\Wb$, for $a\in \Z_{\geq 1}$:
    \begin{align}
        \label{curlremove}
        \begin{tikzpicture}[anchorbase,color=\clr,line width=1pt]
            	\draw[-] (0,0.6) to (0,0.3);
					\draw[-] (0,0.3) to [out=-90,in=0] (-.3,-0.2);
					\draw[-] (-0.3,-0.2) to [out=180,in=-90](-.5,0);
					\draw[-] (-0.5,0) to [out=90,in=180](-.3,0.2);
					\draw[-] (-0.3,.2) to [out=0,in=90](0,-0.3);
					\draw[-] (0,-0.3) to (0,-0.6);
                    \node at (0,-.75){$\scriptstyle{a}$};
        \end{tikzpicture}
        =(-1)^a
        \begin{tikzpicture}[anchorbase,color=\clr,line width=1pt]
            	\draw[-] (0,-0.6) to (0,0.4);
             \node at (0,-.75){$\scriptstyle{a}$};
        \end{tikzpicture}
        =\begin{tikzpicture}[anchorbase,color=\clr,line width=1pt]
            		\draw[-] (0,0.6) to (0,0.3);
					\draw[-] (0,0.3) to [out=-90,in=180] (.3,-0.2);
					\draw[-] (0.3,-0.2) to [out=0,in=-90](.5,0);
					\draw[-] (0.5,0) to [out=90,in=0](.3,0.2);
					\draw[-] (0.3,.2) to [out=180,in=90](0,-0.3);
					\draw[-] (0,-0.3) to (0,-0.6);
     \node at (0,-.75){$\scriptstyle{a}$};
        \end{tikzpicture}.
    \end{align}
\end{lemma}
\begin{proof}
    This follows from \eqref{zigzag} and \eqref{bubblecap}.
\end{proof}

\begin{lemma}
 The following relations hold in $\Wb$: 
      \begin{align}
    \label{dumbbellremove}
\begin{tikzpicture}[anchorbase,scale=0.2,yscale=0.7, color=\clr]
\draw [-,line width=1.4pt] (0,-1) to (0,.75);
\draw [-,line width=1pt] (0,.75) to [out=30,in=270] (1,2.5);
\draw [-,line width=1pt] (0,.75) to [out=150,in=270] (-1,2.5); 
\draw [-,line width=1pt] (1,-2.75) to [out=90,in=330] (0,-1);
\draw [-,line width=1pt] (-1,-2.75) to [out=90,in=210] (0,-1);
\draw [-,line width=1pt] (1,2.5) to [in=180, out=90] (2,3.5) to [in=90, out=0](3,2.5);
\draw [-,line width=1pt] (1,-2.75) to [in=180, out=270] (2,-3.75) to [in=270, out=360](3,-2.75);
\draw [-,line width=1pt] (-1,-3.75) to (-1,-2.75);
\draw [-,line width=1pt] (3,-2.75) to (3,2.5);
\draw [-,line width=1pt] (-1,2.5) to (-1,3.5);
\node at (-1,4) {$\scriptstyle{a}$};
\node at (-1,-4.5) {$\scriptstyle{a}$};
\node at (3.5,0) {$\scriptstyle{b}$};
\end{tikzpicture}
&= \binom{x-a}{b}
\begin{tikzpicture}[anchorbase,scale=0.2,yscale=0.7, color=\clr]
\draw [very thick, -] (0,-2.5) to (0,2.5);
\node at (0,-3.1) {$\scriptstyle{a}$};
\end{tikzpicture}
=
\begin{tikzpicture}[anchorbase,scale=0.2,yscale=0.7, color=\clr]
\draw [-,line width=1.4pt] (0,-1) to (0,.75);
\draw [-,line width=1pt] (0,.75) to [out=30,in=270] (1,2.5);
\draw [-,line width=1pt] (0,.75) to [out=150,in=270] (-1,2.5); 
\draw [-,line width=1pt] (1,-2.75) to [out=90,in=330] (0,-1);
\draw [-,line width=1pt] (-1,-2.75) to [out=90,in=210] (0,-1);
\draw [-,line width=1pt] (-3,2.5) to [in=180, out=90] (-2,3.5) to [in=90, out=0](-1,2.5);
\draw [-,line width=1pt] (-1,-2.75) to [in=0, out=270] (-2,-3.75) to [in=270, out=180](-3,-2.75);
\draw [-,line width=1pt] (1,-3.75) to (1,-2.75);
\draw[-,line width=1pt] (-3,-2.75) to (-3,2.5);
\draw[-,line width=1pt] (1,2.5) to (1,3.5);
\node at (1,4) {$\scriptstyle{a}$};
\node at (1,-4.5) {$\scriptstyle{a}$};
\node at (-3.5,0) {$\scriptstyle{b}$};
\end{tikzpicture}, 
\\
\label{sidewaydumbell}
\begin{tikzpicture}
    [anchorbase,color=\clr,scale=.6]
    \draw[-,line width=1.2pt] (0,0) to [out=up,in=left] (.5,.5) to [out=right,in=up] (1,0);
    \draw[-] (1,0) -- (1.25,-.25)
             (0,0) to [out=-45,in=-135] (1,0)
             (0,0) -- (-.25,-.25);
    \node at (.5,.7){$\scriptstyle a+b$};
    \node at (.5,-.5){$\scriptstyle b$};
\end{tikzpicture}
&=\binom{x-a}{b}~
\begin{tikzpicture}
    [anchorbase,color=\clr,scale=.5]
    \draw[-,line width=1pt] (0,0) to [out=up,in=left] (.5,.5) to [out=right,in=up] (1,0);
    \node at (0,-.2){$\scriptstyle a$};
     \node at (1,-.2){$\scriptstyle a$};
\end{tikzpicture},\qquad
\begin{tikzpicture}
    [anchorbase,color=\clr,scale=.6,yscale=-1]
    \draw[-,line width=1.2pt] (0,0) to [out=up,in=left] (.5,.5) to [out=right,in=up] (1,0);
    \draw[-] (1,0) -- (1.25,-.25)
             (0,0) to [out=-45,in=-135] (1,0)
             (0,0) -- (-.25,-.25);
    \node at (.5,.7){$\scriptstyle a+b$};
    \node at (.5,-.5){$\scriptstyle b$};
\end{tikzpicture}
~=\binom{x-a}{b}~
\begin{tikzpicture}
    [anchorbase,color=\clr,yscale=-1,scale=.5]
    \draw[-,line width=1pt] (0,0) to [out=up,in=left] (.5,.5) to [out=right,in=up] (1,0);
    \node at (0,-.2){$\scriptstyle a$};
     \node at (1,-.2){$\scriptstyle a$};
\end{tikzpicture} .
\end{align} 
\end{lemma}

\begin{proof}
    Note that \cref{sidewaydumbell} follows from \cref{dumbbellremove} via \cref{mergesplitslideleft} and \cref{zigzag} directly. Hence it suffices to prove \cref{dumbbellremove}. To that end, we prove both equalities in \cref{dumbbellremove} simultaneously by induction on $a+b$. For $b=1$ we have
\begin{align*}
    \begin{tikzpicture}[baseline = .5pt,scale=0.3,yscale=0.7, color=\clr]
        \draw [-,line width=2pt] (0,-1) to (0,.75);
\draw [-,line width=1pt] (0,.75) to [out=30,in=270] (1,2.5);
\draw [-,line width=1pt] (0,.75) to [out=150,in=270] (-1,2.5); 
\draw [-,line width=1pt] (1,-2.75) to [out=90,in=330] (0,-1);
\draw [-,line width=1pt] (-1,-2.75) to [out=90,in=210] (0,-1);
\draw [-,line width=1pt] (-3,2.5) to [in=180, out=90] (-2,3.5) to [in=90, out=0](-1,2.5);
\draw [-,line width=1pt] (-1,-2.75) to [in=0, out=270] (-2,-3.75) to [in=270, out=180](-3,-2.75);
\draw [-,line width=1pt] (1,-3.75) to (1,-2.75);
\draw [-,line width=1pt] (-3,-2.75) to (-3,2.5);
\draw [-,line width=1.2pt] (1,2.5) to (1,3.5);
\node at (1,4) {$\scriptstyle{a}$};
\node at (1,-4.5) {$\scriptstyle{a}$};
\node at (-3.4,0) {$\scriptstyle{1}$};
\node at (1.2,0) {$\scriptstyle{a+1}$};
    \end{tikzpicture}\overset{\eqref{splitmerge}}{=}  
    \begin{tikzpicture}[baseline = .5pt,scale=0.3,yscale=0.7, color=\clr]
\draw [-,line width=1pt] (-1,-2.75) to  (1,2.5);
\draw [-,line width=1pt] (1,-2.75) to (-1,2.5); 
\draw [-,line width=1.2pt] (1,-2.75) to (1,2.5);
\draw [-,line width=1pt] (-3,2.5) to [in=180, out=90] (-2,3.5) to [in=90, out=0](-1,2.5);
\draw [-,line width=1pt] (-1,-2.75) to [in=0, out=270] (-2,-3.75) to [in=270, out=180](-3,-2.75);
\draw [-,line width=2pt] (1,-3.75) to (1,-2.75);
\draw [-,line width=1pt] (-3,-2.75) to (-3,2.5);
\draw [-,line width=2pt] (1,2.5) to (1,3.5);
\node at (1,4) {$\scriptstyle{a}$};
\node at (1,-4.5) {$\scriptstyle{a}$};
\node at (-3.4,0) {$\scriptstyle{1}$};
\node at (2.2,0) {$\scriptstyle{a-1}$};
    \end{tikzpicture}
    ~+~    \bubble\stra  
\overset{\eqref{curlremove},\eqref{splitmerge},\eqref{bubblecap}}{=}(x-a)\stra.
\end{align*}
The argument for the other relation with $b=1$ is entirely similar. 

Now assume that all the relations in \cref{dumbbellremove}  hold after replacing $a,b$ by any $t,s$ such that $t+s< a+b$. Then we have
\begin{align*}
    \begin{tikzpicture}[baseline = .5pt,scale=0.3,yscale=0.7, color=\clr]
\draw [-,line width=2pt] (0,-1) to (0,.75);
\draw [-,line width=1pt] (0,.75) to [out=30,in=270] (1,2.5);
\draw [-,line width=1pt] (0,.75) to [out=150,in=270] (-1,2.5); 
\draw [-,line width=1pt] (1,-2.75) to [out=90,in=330] (0,-1);
\draw [-,line width=1pt] (-1,-2.75) to [out=90,in=210] (0,-1);
\draw [-,line width=1pt] (-3,2.5) to [in=180, out=90] (-2,3.5) to [in=90, out=0](-1,2.5);
\draw [-,line width=1pt] (-1,-2.75) to [in=0, out=270] (-2,-3.75) to [in=270, out=180](-3,-2.75);
\draw [-,line width=1pt] (1,-3.75) to (1,-2.75);
\draw [-,line width=1pt] (-3,-2.75) to (-3,2.5);
\draw [-,line width=1pt] (1,2.5) to (1,3.5);
\node at (1,4) {$\scriptstyle{a}$};
\node at (1,-4.5) {$\scriptstyle{a}$};
\node at (1.2,0) {$\scriptstyle{a+b}$};
\end{tikzpicture}\overset{\eqref{splitmerge}}{=}
\sum_{t-s=a-b}
 \begin{tikzpicture}[baseline = .5pt,scale=0.3,yscale=0.7, color=\clr]
 \draw[-,line width=1pt] (-1,-2.75) to (-1,2.5);
  \draw[-,line width=1pt] (1,-2.75) to (1,2.5);
  \draw[-,line width=1pt] (-1,-2.5) to (1,2.2);
  \draw[-,line width=1pt] (1,-2.5) to (-1,2.2);
\draw [-,line width=1pt] (-3,2.5) to [in=180, out=90] (-2,3.5) to [in=90, out=0](-1,2.5);
\draw [-,line width=1pt] (-1,-2.75) to [in=0, out=270] (-2,-3.75) to [in=270, out=180](-3,-2.75);
\draw [-,line width=1pt] (1,-3.75) to (1,-2.75);
\draw [-,line width=1pt] (-3,-2.75) to (-3,2.5);
\draw [-,line width=1pt] (1,2.5) to (1,3.5);
\node at (1,4) {$\scriptstyle{a}$};
\node at (1,-4.5) {$\scriptstyle{a}$};
\node at (-1.2,-4.5) {$\scriptstyle{b}$};
\node at (1.25,.6) {$\scriptstyle{t}$};
\node at (-1.25,.6) {$\scriptstyle{s}$};
\end{tikzpicture}\overset{\eqref{mergesplitslideleft},\eqref{zigzag}}{=} 
\sum_{t-s=a-b}
\begin{tikzpicture}[baseline = .5pt,scale=0.3,yscale=0.7, color=\clr]
 \draw[-,line width=1pt] (-1,-2.75) to (-1,2.5);
  \draw[-,line width=1pt] (1,-2.75) to (1,2.5);
  \draw[-,line width=1pt] (1,2.2) to (-.6,-2.5) to (-.6,-3.2)to [out=down,in=down] (-4,-3.2) to (-4,-2.75) to (-3.5,-1) to (-3,-2.75);
  \draw[-,line width=1pt] (1,-2.5) to (-.6,2.2) to (-.6,3) to [out=up,in=up] (-4,3) to (-4,2.5) to (-3.5,1) to (-3,2.5);
  \draw[-,line width=2pt] (-3.5,1) to (-3.5,-1);
\draw [-,line width=1pt] (-3,2.5) to [in=180, out=90] (-2,3.5) to [in=90, out=0](-1,2.5);
\draw [-,line width=1pt] (-1,-2.75) to [in=0, out=270] (-2,-3.75) to [in=270, out=180](-3,-2.75);
\draw [-,line width=1pt] (1,-3.75) to (1,-2.75);
\draw [-,line width=1pt] (1,2.5) to (1,3.5);
\node at (1,4) {$\scriptstyle{a}$};
\node at (1,-4.5) {$\scriptstyle{a}$};
\node at (-4.1,0) {$\scriptstyle{b}$};
\node at (1.25,.6) {$\scriptstyle{t}$};
\node at (-1.3,.6) {$\scriptstyle{s}$};
\end{tikzpicture}.
\end{align*}
Thus by the inductive hypothesis we have
\begin{align*}
    \begin{tikzpicture}[baseline = .5pt,scale=0.3,yscale=0.7, color=\clr]
\draw [-,line width=2pt] (0,-1) to (0,.75);
\draw [-,line width=1pt] (0,.75) to [out=30,in=270] (1,2.5);
\draw [-,line width=1pt] (0,.75) to [out=150,in=270] (-1,2.5); 
\draw [-,line width=1pt] (1,-2.75) to [out=90,in=330] (0,-1);
\draw [-,line width=1pt] (-1,-2.75) to [out=90,in=210] (0,-1);
\draw [-,line width=1pt] (-3,2.5) to [in=180, out=90] (-2,3.5) to [in=90, out=0](-1,2.5);
\draw [-,line width=1pt] (-1,-2.75) to [in=0, out=270] (-2,-3.75) to [in=270, out=180](-3,-2.75);
\draw [-,line width=1pt] (1,-3.75) to (1,-2.75);
\draw [-,line width=1pt] (-3,-2.75) to (-3,2.5);
\draw [-,line width=1pt] (1,2.5) to (1,3.5);
\node at (1,4) {$\scriptstyle{a}$};
\node at (1,-4.5) {$\scriptstyle{a}$};
\node at (1.2,0) {$\scriptstyle{a+b}$};
\end{tikzpicture}
=&\sum_{t-s=a-b}\binom{x-a+t}{s}\begin{tikzpicture}[baseline = .5pt,scale=0.3,yscale=0.7, color=\clr]
 \draw[-,line width=1pt] (1,-2.75) to (1,2.75);
  \draw[-,line width=1pt] (1,2.2) to (-.6,-2.2) to [out=240,in=120] (-.6,2.2) to (1,-2.2);
\node at (1,4) {$\scriptstyle{a}$};
\node at (1,-4.5) {$\scriptstyle{a}$};
\node at (1.25,.6) {$\scriptstyle{t}$};
\end{tikzpicture}
\overset{\eqref{curlremove},\eqref{splitmerge}}{=}\sum_{t-s=a-b}(-1)^{a-t}\binom{x-a+t}{s}\binom{a}{t}\stra.
\end{align*}

By a change of variable
$k=b-s=a-t$, we have
\begin{align*}
\sum_{t-s=a-b}(-1)^{a-t}\binom{x-a+t}{s}\binom{a}{t}
    =&\sum_{k=0}^{b}(-1)^k\binom{x-k}{b-k}\binom{a}{a-k}
    \\
    =&\sum_{k=0}^{b}(-1)^{b}\binom{b-x-1}{b-k}\binom{a}{k}
    \\
    =&(-1)^b\binom{b+a-x-1}{b}=\binom{x-a}{b}.
\end{align*}
This completes the inductive proof for the first relation in \cref{dumbbellremove}.

The inductive step for the other relation in \cref{dumbbellremove} is similar and skipped. 
\end{proof}

\begin{lemma}
 The following relations hold in $\Wb$: 
   \begin{align}
   \label{thickcapandthin}
   \begin{tikzpicture}[anchorbase,scale=0.16,color=\clr]
\draw [-,line width=1.4pt] (0,-.5) to (0,.75);
\draw [-,line width=1pt] (0,.75) to [out=30,in=270] (1,2.5);
\draw [-,line width=1pt] (0,.75) to [out=150,in=270] (-1,2.5); 
\draw [-,line width=1pt] (1,2.5) to [out=90,in=180] (2,3.5) to [out=0,in=90] (3,2.5);
\draw [-,line width=1pt] (-1,2.5) to [out=90,in=180] (2,5.5) to [out=0,in=90] (5,2.5);
\draw [-,line width=1.4pt] (4,-.5) to (4,.75);
\draw [-,line width=1pt] (4,.75) to [out=30,in=270] (5,2.5);
\draw [-,line width=1pt] (4,.75) to [out=150,in=270] (3,2.5);
\node at (2,6.5) {$\scriptstyle{b}$};
\node at (0,-1) {$\scriptstyle{a}$};
\node at (4,-1) {$\scriptstyle{a}$};
\end{tikzpicture}
=\binom{a}{b}\capa, 
\qquad  
\begin{tikzpicture}[yscale=-1,anchorbase,scale=0.16,color=\clr]
\draw [-,line width=1.4pt] (0,-.5) to (0,.75);
\draw [-,line width=1pt] (0,.75) to [out=30,in=270] (1,2.5);
\draw [-,line width=1pt] (0,.75) to [out=150,in=270] (-1,2.5); 
\draw [-,line width=1pt] (1,2.5) to [out=90,in=180] (2,3.5) to [out=0,in=90] (3,2.5);
\draw [-,line width=1pt] (-1,2.5) to [out=90,in=180] (2,5.5) to [out=0,in=90] (5,2.5);
\draw [-,line width=1.4pt] (4,-.5) to (4,.75);
\draw [-,line width=1pt] (4,.75) to [out=30,in=270] (5,2.5);
\draw [-,line width=1pt] (4,.75) to [out=150,in=270] (3,2.5);
\node at (0,-1) {$\scriptstyle{a}$};
\node at (4,-1) {$\scriptstyle{a}$};
\node at (2,6.5) {$\scriptstyle{b}$};
\end{tikzpicture}
=\binom{a}{b}\cupa.
\end{align}
\end{lemma}
\begin{proof}
     This follows from \eqref{zigzag} and \eqref{mergesplitslideleft}.
\end{proof}

\subsection{ A representation of $\Wb$}
\label{sec:repWo}
In this subsection, we work over $\kk=\mathbb C$. For any $N\in\N$, we define
\[
[N]\colon=\begin{cases}
    \{-n,\ldots,-1,0,1,\ldots,n-1,n\}, & \text{ for }N=2n+1,\\
     \{-n,\ldots,-1,1,\ldots,n-1,n\}, & \text{ for } N=2n.
\end{cases}
\]
Let $V$ be the $N$-dimensional $\C$-vector space with a standard basis $\{v_i\mid i\in [N]\}$. Let $(\cdot,\cdot)$ denote the non-degenerate symmetric bilinear form on $V$ defined by 
$(v_i,v_j)=\delta_{i,-j}, \forall i,j\in[N].$
We identify the orthogonal Lie algebra as
\[
\g=\so_N=\{g\in \End(V)\mid (gx,y)+(x,gy)=0\}.
\]
Let $E_{i,j}$ be the usual matrix unit with respect to the basis $\{v_i\mid i\in [N]\}$. Set $F_{i,j}:=E_{i,j}-E_{-j,-i}$. Then $\so_N$ has a basis given by
\begin{align} \label{eq:basis:g}
\begin{cases}
    \{F_{i,i}\mid 1\leq i\leq n\}\cup \{F_{\pm i,\pm j}\mid 1\leq i<j\leq n\}\cup \{F_{0,\pm i}\mid 1\leq i \leq n\}, & \text{ for }N=2n+1,\\
     \{F_{i,i}\mid 1\leq i\leq n\}\cup \{F_{\pm i,\pm j}\mid 1\leq i<j\leq n\}, & \text{ for }N=2n,
\end{cases}
\end{align}
where the Cartan subalgebra $\mathfrak h$ has a basis $\{h_i:=F_{i,i}\mid 1\leq i\leq n\}$. There is a triangular decomposition $\g=\n^+\oplus \h\oplus \n^-$ such that the positive part $\n^+$ has a basis 
\[
\begin{cases}
    \{F_{i,\pm j}\mid 1\leq i<j\leq n\}\cup \{F_{0,-i}\mid 1\leq i \leq n\}, & \text{ for }N=2n+1, \\
    \{F_{i,\pm j}\mid 1\leq i<j\leq n\}, & \text{ for }N=2n.
\end{cases}
\]

Let $\U(\g)$ be the universal enveloping algebra of $\g$ and $\g$-mod be the category of all finite-dimensional $\g$-modules or equivalently $\U(\g)$-modules. Then $\g$-mod is a strict symmetric monoidal category with the braiding 
\[
\sigma_{M,L}: M\otimes L \longrightarrow L \otimes M, \qquad m\otimes l\mapsto l\otimes m.
\]
For later use, we also define

 \begin{equation}
 \label{casimir}
     \Omega=\frac{1}{2}\sum_{i,j\in[N]} F_{ij}\otimes F_{ji}.
 \end{equation}

To give a representation of $\Wb$, we define the following 
morphisms in $\so_N$-mod. First, we have the following operators for merges, splits, and thick crossings, which coincide with the counterparts in type A (cf. \cite{CKM, BEEO, SWweb}):
\begin{align}
\text{Y}_{a,b} : &
\bigwedge\nolimits^{a+b} V  
\longrightarrow 
\bigwedge\nolimits ^a V\otimes \bigwedge\nolimits^b V,
\label{eq:Yab} \\  
 & v_{i_1}\wedge \ldots \wedge v_{i_{a+b}} 
\mapsto  
\sum_{w} (-1)^{\ell(w)}v_{i_{w(1)}}\wedge\ldots\wedge v_{i_{w(a)}}\otimes v_{i_{w(a+1)}}\wedge \ldots\wedge v_{i_{w(a+b)}}, 
\notag \\
\rot{Y}_{a,b}: &
\bigwedge\nolimits ^aV\otimes \bigwedge\nolimits ^bV
 \longrightarrow 
 \bigwedge\nolimits^{a+b} V,
 \label{eq:rotYab} \\
& v_{i_1}\wedge \ldots\wedge v_{i_a} \otimes v_{i_{a+1}}\wedge \ldots\wedge v_{i_{a+b}} 
  \mapsto v_{i_1}\wedge \ldots \wedge v_{i_{a+b}},
 \notag \\
 \text{X}_{a,b}: &
 \bigwedge\nolimits ^aV\otimes \bigwedge\nolimits ^b V
  \longrightarrow 
 \bigwedge\nolimits ^b V\otimes \bigwedge\nolimits^a V, \quad x\otimes y \mapsto (-1)^{ab} y\otimes x,
 \label{eq:Xab}
\end{align}
where the summation over $w$ in \eqref{eq:Yab} runs over the set $(\mathfrak S_{a+b}/\mathfrak S_a\times\mathfrak S_b)_{\text{min}}$ of minimal length coset representatives.

Associated to the thin cap and cup, we define the morphisms

\[ 
\cup_1: \C \longrightarrow V^{\otimes 2}, \qquad 1\mapsto \sum_{i\in [N]}v_i\otimes v_{-i} , 
\]
 and 
 \[
 \cap_1: V^{\otimes 2} \longrightarrow \C, \qquad  v_i\otimes v_j \mapsto \delta_{i,-j}. 
 \]
For the thick cups and caps $\cupa$ and $\capa$, we can compute the corresponding morphisms $\cup_a$ and $\cap_a$ inductively via \cref{thickcapandthin}. Indeed, we have
\begin{align}
\label{caparep}
      &\cap_a: \bigwedge\nolimits ^a V\otimes \bigwedge\nolimits ^a V\rightarrow \C,
      \quad v_{i_1}\wedge\ldots\wedge v_{i_a} \otimes v_{j_a}\wedge\ldots\wedge v_{j_1}\mapsto \prod_{k=1}^a \delta_{i_k,-j_k},\atop\qquad\qquad \text{ for }i_1<\cdots<i_a, j_1>\cdots>j_a\\
\label{cuparep}
    &\cup_a: \C\rightarrow \bigwedge\nolimits ^a V\otimes \bigwedge\nolimits ^a V,\quad 1\mapsto \sum_{ i_1<i_2<\ldots< i_a} v_{i_1}\wedge\ldots \wedge v_{i_a}\otimes v_{-i_a} \wedge \ldots \wedge v_{-i_1}.
\end{align}

Recall that the category $\Wb$ is defined over $\mathbb J$. Denote by $\mathbb J_N$ the quotient ring of $\mathbb J$ by the ideal generated by $x-N$, and identify $\mathbb J_N\cong \Z$. We note that all the defining relations \cref{webassoc}--\eqref{bubblecap} are {\em integral} when we replace $x$ by $N$. Thus we can do a base change from $\mathbb J$ to $\Z\equiv \mathbb J_N$ to obtain the monoidal category $\Wb_\Z(N)=\Wb\otimes_{\mathbb J} \mathbb J_N$ over $\Z$. By another base change from $\Z$ to $\kk$, we obtain a category $\Wb_\kk(N)$ over an arbitrary commutative ring $\kk$ with $1$. We work with $\kk=\C$ in the remainder of this section.

\begin{theorem} 
\label{thm:Worep}
There is a monoidal functor $\mathcal F_N\colon \Wb_{\C}(N) \rightarrow \so_N\text{-mod}$ which sends $a$ to $\wedge^a V$ and sends the 
 generating morphisms $\merge$, $\splits$, $\crossing$, $\capa$ and $\cupa$ to \rm{\rot{Y}}$_{a,b}$, 
   ${\rm Y}_{a,b}$, ${\rm X}_{a,b}$, $\cap_a$ and $\cup_a$ respectively. 
\end{theorem}

\begin{proof}
    It suffices to check that $\mathcal F_N$ preserves all the relations  \cref{webassoc}--\cref{bubblecap} (with $x$ replaced by $N$). The verification of the relations \cref{webassoc}-\cref{splitmerge} which involve no cap/cup is the same as \cite[Theorem 4.14]{BEEO} in the type A setting.
    
    The remaining relations \cref{zigzag}--\cref{bubblecap} follow from direct computations using \cref{caparep} and \cref{cuparep}. We will only give the details for \cref{bubblecap} below. 
    
    The first relation in \cref{bubblecap} is preserved by $\mathcal F_N$ since
\[
    \cap_a\circ \cup_a(1):= \sum_{i_1<\cdots<i_a}\cap_a\bigl(
v_{i_1}\wedge\cdots\wedge v_{i_a}\otimes
v_{-i_a}\wedge\cdots\wedge v_{-i_1}
\bigr) = \sum_{i_1<\cdots<i_a} 1=\binom{N}{a}.
\]

The third relation in \cref{bubblecap} is preserved by $\mathcal F_N$ by the following computation:
\[
\begin{aligned}
\cap_a\circ \mathrm{X}_{a,a}\left(
v_{i_a}\wedge\cdots\wedge v_{i_1}
\otimes
v_{j_1}\wedge\cdots\wedge v_{j_a}
\right) 
 &=
(-1)^{a^2}
\prod_{k=1}^a \delta_{i_k,-j_{a+1-k}}\\
=&
(-1)^a
\cap_a\left(
v_{i_a}\wedge\cdots\wedge v_{i_1}
\otimes
v_{j_1}\wedge\cdots\wedge v_{j_a}
\right).
\end{aligned}
\]

For the second relation in \cref{bubblecap} (or rather its equivalent upside down relation after applying the zigzag relation), we compute
\[
\begin{aligned}
(\rot{Y}_{a,a}\circ \cup_a)(1)
&=
\rot{Y}_{a,a}\left(
\sum_{i_1<i_2<\cdots<i_a}
v_{i_1}\wedge\cdots\wedge v_{i_a}
\otimes
v_{-i_a}\wedge\cdots\wedge v_{-i_1}
\right) \\
&=
\sum_{i_1<i_2<\cdots<i_a}
v_{i_1}\wedge\cdots\wedge v_{i_a}
\wedge
v_{-i_a}\wedge\cdots\wedge v_{-i_1}.
\end{aligned}
\]
Thus it remains to show that
\begin{align}  \label{wedge0}
\sum_{i_1<i_2<\cdots<i_a}
v_{i_1}\wedge\cdots\wedge v_{i_a}
\wedge
v_{-i_a}\wedge\cdots\wedge v_{-i_1}
=0.
\end{align}
To that end, we start with an identity $\sum_i v_i\wedge v_{-i} =0$ (due to
\(
v_i\wedge v_{-i}+v_{-i}\wedge v_i=0
\)).
Taking the \(a\)-fold exterior product of this sum gives us
\[
0=
\left(\sum_i v_i\wedge v_{-i}\right)^{\wedge a}
=
a!(-1)^{\frac{a(a-1)}2}
\sum_{i_1<i_2<\cdots<i_a}
v_{i_1}\wedge\cdots\wedge v_{i_a}
\wedge
v_{-i_a}\wedge\cdots\wedge v_{-i_1},
\]
and \eqref{wedge0} follows.
Consequently, we have $(\rot{Y}_{a,a}\circ \cup_a)(1)=0.$
\end{proof}

\begin{rem}
   Let $N$ be a positive even integer. An analogous monoidal functor $\cF_{-N}\colon \Wb_{\mathbb C}(-N)\to \mathfrak{sp}_N\modf$ can be constructed, where the object $a$ is sent to  the symmetric power $S^aV$ of the natural $\mathfrak{sp}_N$-module $V$.
\end{rem}

\section{Affine B-webs}
  \label{sec:affine_BWeb}

In this section, we first derive a number of new relations in the affine B-web category $\AWb$ defined in Definition \ref{def:affineBWeb}. Then we work out various general relations valid only modulo lower degree terms, which will be used in Section \ref{sec:spanning} in establishing spanning sets for the morphism spaces of $\AWb$.

\subsection{More relations in $\AWb$}

By the defining relations of the affine B-web category $\AWb$ in Definition \ref{def:affineBWeb} (and relations obtained from these by the zigzag identities), there is an isomorphism of strict monoidal categories
\begin{equation}
\label{anti-auto}
\div : \AWb\longrightarrow {\AWb}^{\text{op}}    
\end{equation}
by reflecting each diagram along a horizontal axis.
We define the following morphisms in $\AWb$
\begin{equation}
\label{equ:r=0andr=a}
 \qquad\quad (1\le r\le a).
\end{equation}
We adopt the convention throughout the paper that
\[ %
.
 \]

\begin{lemma}
    The following relation holds in $\AWb$, for $a,b\ge 1$:
\begin{align}    
\label{dotmovecrossing2}
\sum_{0\leq t \leq \min(a,b)}t!~
%
.
\end{align}
\end{lemma}

\begin{proof}
    It follows from \eqref{dotmovecrossing1} via the zigzag relations in \eqref{zigzag}.
\end{proof}

\begin{lemma}
    The following relation holds in $\AWb$, for $a\ge 1$:
\begin{align}
\label{lollipopr}
%
 .
\end{align}
\end{lemma}

\begin{proof}
    We prove by induction on $a$. The base case $a=1$ is trivial. Assume $a>1$. We have 
    \begin{align*}
  a!~ %
    \end{align*}
by the inductive hypothesis.
\end{proof}

Recall  
$
%
 
$
 from \eqref{def:huochai}. 

\begin{lemma} 
The following relations hold in $\AWb$ for $b, s\ge 0$:
\[ 
%
.
\]
\end{lemma}

\begin{proof}
    We only prove the first relation as the second one follows similarly. Indeed we have
    \begin{equation}
    \label{eq:stickswitch}
    \begin{aligned}
       2^s%
\ .
    \end{aligned}
    \end{equation}
Now to complete the proof it suffices to show that each summand in the second and third sums on the right hand side of \cref{eq:stickswitch} equals $0$.  We proceed by induction on $s$. When $s=1$, by \cref{eq:stickswitch} we see that
\[
2%
.
\]
Now for general $s$ and any $1\leq t\leq \min(s,b)$, by the inductive hypothesis we have
\[ 
%
~\overset{\cref{bubblecap}}{=}0.
\]
Similarly the summands in the third term of \cref{eq:stickswitch} are all zero. This concludes the proof.
\end{proof}

\begin{lemma}
\label{Lem:newdotsss}
For any $a,b \ge 1$, we have 
\[ 
%
.
\]    
\end{lemma}

\begin{proof}
    We only prove the first one as the second one is similar. In fact, we have
    \[
    %
.
\]
This concludes the proof.
\end{proof}

\begin{lemma}
\label{lem:222s}
For any $s \ge 1$, we have 
\[ 
%
.
\]        
\end{lemma}

\begin{proof}
   This follows from applying \cref{Lem:newdotsss} $s$ times to the LHS of the above equations.
\end{proof}

\begin{lemma} 
\label{bendstick}
For any $a \ge 1$, we have 
\[  
%
 \;\; .
\]
\end{lemma}
\begin{proof}
    These relations are straightforward to check via the relations \cref{zigzag} and \cref{mergesplitslideleft}.
\end{proof}

\begin{lemma}
\label{Lem:dotmovesplitmerge}
  The following relations hold in $\AWb$, for $0\le r \le a+b$:
 \begin{equation}\label{dotmovesplitss}
%
 .
\end{equation}
\end{lemma}

\begin{proof}
    We only show \cref{dotmovemergess}, as the identity 
    \cref{dotmovesplitss} follows by symmetry. 
    We have 
    \begin{align*}
%
.
    \end{align*}
    The diagrams in the summand may be rewritten as 
\begin{align*}
 %
.
\end{align*}

The lemma is proved.
\end{proof}

Given $a\ge 1$, let $E_a$ be the subalgebra of $\End_{\AWb}(a)$ generated by the elements $\wkdota$, for $0\le r \le a$. 

\begin{lemma}
\label{lem:commute}
The following commutativity relation holds in $\AWb$:
\begin{equation}
    \label{commutato}
    %
,
\end{equation}    
for any admissible $r,t,a$. In particular, $E_a$ is a commutative subalgebra of $\End_{\AWb}(a)$. 
\end{lemma}

\begin{proof}
   We may assume $1\le r\le t\le a$. We prove by induction on $r+t+a$. The case $r=t=1=a$ is trivial. Without loss of generality we may assume that $1\le r <t\le a$.
We have 
\begin{align*}
  %
.
\end{align*}
The new diagrams may be rewritten as 
\begin{align*}
 %
.
\end{align*}
Then the identity \cref{commutato} follows since each diagram in the summation above is fixed under $\div$ by the induction hypothesis.
\end{proof}

Let $\Par_a$ be the set of  partitions $\lambda=(\lambda_1,\ldots,\lambda_k)$ with  each part  $\lambda_i\le a$. Let $\EPar_a$ be the subset of $\Par_a$ consisting of even partitions (i.e., with all parts even). For any partition $\lambda=(\lambda_1,\lambda_2,\ldots, \lambda_k)\in \Par_a$, the {\em elementary dot packet (of thickness $a$)} is defined to be 
\begin{equation}
%
.   
\end{equation}

\subsection{Relations modulo lower terms}

Let $\unrhd$ denote the dominance order on partitions. We write $\lambda\rhd\mu$ if $\lambda\unrhd \mu$ and $\lambda\neq \mu$. We also denote by $\lambda+\mu$ the partition $(\lambda_1+\mu_1,\lambda_2+\mu_2,\cdots)$ for $\la=(\la_1,\la_2,\ldots)$ and $\mu=(\mu_1,\mu_2,\ldots)$, where we set $\lambda_i=\mu_j=0$ for $i>\ell(\lambda)$ and $ j>\ell(\mu)$. We define 
\begin{equation} \label{eq:Ea}
    E_{a}^{\lambda}:=\mathbb J'\text{-span} \big\{%
\mid \nu\unrhd \lambda \big\}.
\end{equation}

By definition, any morphism of $\AWb$ is a $\mathbb J'$-linear combination of {\em dotted B-web diagrams}, i.e., the dotted string diagrams generated by the generating morphisms \cref{dotgenerator}.  

By assigning degree $r$ to $\wkdota$ (for $r\le a$) and degree $0$ to other generating morphisms, we define the degree of any dotted B-web diagram additively by summing up degrees of local generating morphisms. Let $\Hom_{\AWb}(\mu,\lambda)_{\le k}$ (and respectively, $\Hom_{\AWb}(\mu,\lambda)_{< k}$) be the $\mathbb J'$-span
of all dotted B-web diagrams of type $\mu\rightarrow \lambda$ with degree $\le k$ (and respectively, $<k$). For any two dotted B-web diagrams $D$ and $D'$  in  $\Hom_{\AWb}(\mu,\lambda)_{\le k}$, we write 
\begin{align}  \label{modLOT}
D\equiv D' \quad \text{ if } D=D'  \mod \Hom_{\AWb}(\mu,\lambda)_{< k}.
\end{align}
In this case, we say $D$ equals $D'$ up to lower degree terms.

\cref{dotmovefreely} below tells us that we may move the dots through crossings, splits and merges freely up to lower degree terms. This result will be used frequently without explicit reference.

\begin{lemma} In $\AWb$, the following $\equiv$-relations hold:
\label{dotmovefreely}
 \begin{enumerate}
  \item 
    $\begin{tikzpicture}[anchorbase, scale=0.5, color=\clr]
\draw[-,line width=1.4pt] (0,-.2) to  (1,2.2);
\draw[-,line width=1.4pt] (1,-.2) to  (0,2.2);
\draw(0.2,1.6)\bdot;
\draw(-.4,1.6) node {$\scriptstyle \omega_r$};
\node at (0, -.5) {$\scriptstyle a$};
\node at (1, -.5) {$\scriptstyle b$};
\end{tikzpicture}
~\equiv~ 
\begin{tikzpicture}[anchorbase, scale=0.5, color=\clr]
\draw[-, line width=1.4pt] (0,-.2) to (1,2.2);
\draw[-,line width=1.4pt] (1,-.2) to (0,2.2);
\draw(.8,0.3)\bdot;
\draw(1.35,0.3) node {$\scriptstyle \omega_r$};
\node at (0, -.5) {$\scriptstyle a$};
\node at (1, -.5) {$\scriptstyle b$};
\end{tikzpicture}, 
\qquad 
\begin{tikzpicture}[anchorbase, scale=0.5, color=\clr]
\draw[-, line width=1.4pt] (0,-.2) to (1,2.2);
\draw[-,line width=1.4pt] (1,-.2) to(0,2.2);
\draw(.2,0.2)\bdot;
\draw(-.4,.2)node {$\scriptstyle \omega_r$};
\node at (0, -.5) {$\scriptstyle b$};
\node at (1, -.5) {$\scriptstyle a$};
\end{tikzpicture}
~\equiv~
\begin{tikzpicture}[anchorbase, scale=0.5, color=\clr]
\draw[-,line width=1.4pt] (0,-.2) to  (1,2.2);
\draw[-,line width=1.4pt] (1,-.2) to  (0,2.2);
\draw(0.8,1.6)\bdot;
\draw(1.4,1.6)node {$\scriptstyle \omega_r$};
\node at (0, -.5) {$\scriptstyle b$};
\node at (1, -.5) {$\scriptstyle a$};
\end{tikzpicture},
$ 
\item
$
\begin{tikzpicture}[baseline = -.5mm,scale=1.2,color=\clr]
\draw[-,line width=1.4pt] (0.08,-.5) to (0.08,0.04);
\draw[-,line width=1pt] (0.34,.5) to (0.08,0);
\draw[-,line width=1pt] (-0.2,.5) to (0.08,0);
\node at (-0.22,.6) {$\scriptstyle a$};
\node at (0.36,.65) {$\scriptstyle b$};
\draw (0.08,-.2) \bdot;
\draw (0.35,-.2) node{$\scriptstyle \omega_r$};
\end{tikzpicture} 
\equiv~
\sum\limits_{c+d=r} ~
\begin{tikzpicture}[baseline = -.5mm,scale=1.2,color=\clr]
\draw[-,line width=1.4pt] (0.08,-.5) to (0.08,0.04);
\draw[-,line width=1pt] (0.34,.5) to (0.08,0);
\draw[-,line width=1pt] (-0.2,.5) to (0.08,0);
\node at (-0.22,.6) {$\scriptstyle a$};
\node at (0.36,.65) {$\scriptstyle b$};
\draw (-.05,.24) \bdot;
\draw (-0.26,.2) node{$\scriptstyle \omega_{c}$};
\draw (0.48,.2) node{$\scriptstyle \omega_{d}$};
\draw (.22,.24) \bdot;
\end{tikzpicture},
\qquad 
\begin{tikzpicture}[baseline = -.5mm, scale=1.2, color=\clr]
\draw[-,line width=1pt] (0.3,-.5) to (0.08,0.04);
\draw[-,line width=1pt] (-0.2,-.5) to (0.08,0.04);
\draw[-,line width=1.4pt] (0.08,.6) to (0.08,0);
\node at (-0.22,-.6) {$\scriptstyle a$};
\node at (0.35,-.6) {$\scriptstyle b$};
\draw (0.08,.2) \bdot;
\draw (0.35,.2) node{$\scriptstyle \omega_r$};  
\end{tikzpicture}
 ~\equiv~
\sum\limits_{c+d=r}~
\begin{tikzpicture}[baseline = -.5mm,scale=1.2, color=\clr]
\draw[-,line width=1pt] (0.3,-.5) to (0.08,0.04);
\draw[-,line width=1pt] (-0.2,-.5) to (0.08,0.04);
\draw[-,line width=1.4pt] (0.08,.6) to (0.08,0);
\node at (-0.22,-.6) {$\scriptstyle a$};
\node at (0.35,-.6) {$\scriptstyle b$};
\draw (-.08,-.3) \bdot; \draw (.22,-.3) \bdot;
\draw (-0.31,-.3) node{$\scriptstyle \omega_{c}$};
\draw (0.48,-.3) node{$\scriptstyle \omega_{d}$};
\end{tikzpicture}
$. 
\end{enumerate}  
\end{lemma}

\begin{proof}
By definition, we replace the dot on the left-hand side of the first equation in (1) with the diagram from the right-hand side of \eqref{equ:r=0andr=a}. First, we slide the split through the crossing using \eqref{sliders}. Next, we move the dot through the crossing via \eqref{dotmovecrossing1}, working modulo lower terms. Finally, applying \eqref{sliders} again to move the merge through the crossing yields the right-hand side. The proof of the second equation is analogous.

    (2) follows from Lemma \ref{Lem:dotmovesplitmerge}. 
\end{proof}

\begin{lemma}
\label{wkdotsspan}
For any $\lambda$ and $\mu$, we have $\begin{tikzpicture}
[anchorbase,scale=1,color=\clr]
\draw[-,line width=1.5pt] (0.08,-.8) to (0.08,-.5);
\draw[-,line width=1.5pt] (0.08,.3) to (0.08,.6);
\draw[-,thick] (0.1,-.51) to [out=45,in=-45] (0.1,.31);
\draw[-,thick] (0.06,-.51) to [out=135,in=-135] (0.06,.31);
\draw(-.1,0) \bdot;
\draw(-.4,0)node {$\scriptstyle \omega_\lambda$};
\draw(.25,0) \bdot;
\draw(.55,0)node {$\scriptstyle \omega_\mu$};
\node at (-.22,-.35) {$\scriptstyle a$};
\node at (.42,-.35) {$\scriptstyle b$};
\end{tikzpicture}\in E^{\lambda+\mu}_{a+b}$ up to lower degree terms.
    \end{lemma}
    
\begin{proof} 
The arguments below are similar to \cite[Lemma 2.13]{SWweb}.

    Let $k=\ell(\lambda)$ and $l=\ell(\mu)$.
  We may assume that $a\geq b$. Moreover, we set $\overline{\lambda}=(\lambda_1,\ldots,\lambda_{k-1})$ when $k\ge 2$ and $\tilde{\lambda}=(\lambda_2,\ldots,\lambda_{k-1})$ when $k> 2$.
  
  We first treat the case when $b=1$. Then $\mu=(1^{l})$, and we proceed via an induction on $l$. In the case when $l=0$, we use induction on $k$ and $\lambda_k$. If $k=1$, then there is nothing to prove. If $k>1$ and $\lambda_k=1$, then we have
  \[
.
  \]
Thus it follows from the induction on $k$ that $%
\in E_{a+1}^{\lambda}$.

Now suppose $k>1$ and $\lambda_k>1$, then we have
\[
%
.
\]
Thus it follows from the induction on $k$ and $\lambda_k$ that $%
\in E_{a+1}^{\lambda}$. This concludes the proof of the case when $b=1$ and $l=0$.

For $b=1$ and $l\geq 1$, we have
\[ 
%
.
\]
Thus it follows from the induction on $l$ that $%
\in B^{\lambda+\mu}_{a+1}$.
This concludes the proof for the case when $b=1$.

Now consider the case when $b>1$. We proceed by induction on $l$ again. We first consider the case when $l=0$. In this case, when $k=1$ there is nothing to prove. When $k>1$, we have
\[
%
\; .
\]
The first summand here lies in $E_{a+b}^{\lambda}$ by induction on $k$. The second summand lies in $E_{a+b}^{\lambda}$ by induction on $b$ since $%
$ for any $0<s<b$. For the last summand, we see that
\[
%
 \; .
\]
The first summand here lies in $E_{a+b}^\lambda$ by induction on $k$ while the second summand lies in $E_{a+b}^\lambda$ by induction on $b$. This concludes the proof of the case when $l=0$ and $b>1$.

Now suppose that $l\geq 1$ and $b>1$, and let $\overline{\mu}=(\mu_1,\ldots,\mu_{l-1})$. If $\mu_l<b$, then we have
\[
%
\]
where $\nu^{(1)}\in \Par_{\mu_l}$ and $\nu^{(2)}\in \Par_{b-\mu_l}$;
then we have 
$%
\in E_{a+b}^{\lambda+\mu}$ by induction on $b$. Hence we only need to consider the case when $\mu_l=b$, where $\mu=(b^l)$. In this case we have
\[
%
.
\]
The first summand lies in $E_{a+b}^{\lambda+\mu}$ by induction on $l$. For the second summand, the terms for $0<u<b$ lie in $E_{a+b}^{\lambda+\mu}$ by induction on $b$ while the term for $u=b$ lies in $E_{a+b}^{\lambda+\mu}$ by induction on $l$. This concludes the proof of the case when $l\geq 1$ and hence for the case $b>1$.
\end{proof}

\section{Relations among lollipops and bubbles}
  \label{sec:lol:bubbles}

  In this section, we provide reduction procedures to rewrite arbitrary morphisms in $\Hom_\AWb(\unit,\mu)$ (called lollipops) and $\Hom_\AWb(\unit,\unit)$ (called bubbles) in terms of some canonical ones. This is preparatory for the spanning sets of the morphism spaces of $\AWb$ in Section \ref{sec:spanning}.

\subsection{Lollipops}

We aim to simplify the diagram 
\begin{align}   \label{eq:ballstick}
\lol_{a,\lambda} :=
\begin{tikzpicture}[yscale=-1, anchorbase,scale=.8,color=\clr]
\draw[-,line width=1.4pt] (0.08,.3) to (0.08,.6);
\draw[-,thick] (0.1,-.51) to [out=right,in=-45] (0.1,.31);
\draw[-,thick] (0.1,-.51) to [out=left,in=-135] (0.1,.31);
\node at (0.08,.8) {$\scriptstyle 2a$};
\draw (-.1,-.2)\bdot;
\node at (-.4,-.2) {$\scriptstyle \omega_\lambda$};
\end{tikzpicture}
\end{align}
for $\lambda\in \Par_a$. Let $\venus_{2a}$ be the subspace of $\Hom_{\AWb}(2a,0)$ spanned by $\wkdotamu$, for all  $\mu\in\EPar_{2a}$ (recall $\EPar_{2a}$ denotes the set of even partitions with each part $\le 2a$). 

\begin{proposition}
\label{lollipopspan}
    Let $a,b\in \N$ with $b\le a$. We have
    \begin{enumerate}
        \item $\lol_{a,\lambda}\in \venus_{2a}$ modulo lower degree terms, for any $\lambda\in\Par_a$.
\item
$\begin{tikzpicture}[anchorbase,color=\clr]
    \draw[line width=1.4pt] (0,-.3) -- (0,0);
    \draw[line width=1pt] (0,0) -- (.3,.6)
                          (0,0) -- (-.3,.6);
    \draw (.3,.6)\bdot (-.3,.6)\bdot (.15,.3)\bdot (-.15,.3)\bdot;
    \node at (-.45,.35){$\scriptstyle \omega_\lambda$};
    \node at (.5,.3){$\scriptstyle \omega_\mu$};
    \node at (-.6,.1){$\scriptstyle 2a-2b$};
     \node at (.35,.1){$\scriptstyle 2b$};
\end{tikzpicture} 
\in  \venus_{2a}$ modulo lower degree terms, for any $\lambda\in \EPar_{2a-2b}$ and $\mu\in \EPar_{2b}$.
    \end{enumerate}
\end{proposition}

The proof of \cref{lollipopspan} can be found in Appendix \ref{app:lollipop}. 

For any $a,b\in\N$ with $a+b\in 2\N$, let $U_{a,b}$ be the subspace of $\Hom_{\AWb}(\unit,(a,b))$ spanned by
\[
\begin{tikzpicture}[anchorbase, scale=0.45, color=\clr]
\draw [-,line width=1.4pt] (-0.02,-.1) to (-0.05,.4);
\draw [-,line width=1.4pt] (2,-.1) to (2,.4);
\draw[-,line width=1.1pt] (0,0) to [out=down,in=left] (1,-1) to [out=right,in=down] (2,0);
\draw[-,line width=1pt] (0,0) to (-.5,-.5) -- (-1,-1);
\draw[-,line width=1pt] (2,0) to (2.5,-.5) -- (3,-1);
\node at (0,.7){$\scriptstyle a$};
\node at (1.4,-1.3){$\scriptstyle t$};
\node at (2,.7){$\scriptstyle b$};
\node at (-1.05, -.4) {$\scriptstyle \omega_{\lambda}$};
\node at (3.1, -.4) {$\scriptstyle \omega_{\mu}$};
\draw (-.5,-.5)\bdot (-1,-1)\bdot;
\draw (.1,-.5)\bdot (3,-1)\bdot;
\node at (.05,-.9){$\scriptstyle \omega_{\nu}$};
\draw (2.5,-.5)\bdot;
\end{tikzpicture}
\]
for all \( t\in\N \) such that \( a - t, b - t \in 2\mathbb{N} \),  \( \lambda \in \EPar_{a-t} \), \( \mu \in \EPar_{b-t} \) and \( \nu \in \Par_t \).

\begin{proposition}
\label{lollipop-merge-split}  
For any $a,b\in\N$ with $a+b\in 2\N$ and $\lambda\in \Par_{a+b}$,  modulo lower degree terms we have 
$
\begin{tikzpicture}[anchorbase,scale=.8,color=\clr]
\draw[-,line width=1.2pt] (0.08,-.5) to (0.08,0.04);
\draw[-,line width=1pt] (0.34,.5) to (0.08,0);
\draw[-,line width=1pt] (-0.2,.5) to (0.08,0);
\node at (0.36,.65) {$\scriptstyle b$};
\node at (-.2,.65) {$\scriptstyle a$};
\draw (0.08,-.5) \bdot;
\end{tikzpicture} \in U_{a,b}
$ 
and 
$
\begin{tikzpicture}[anchorbase,scale=.8,color=\clr]
\draw[-,line width=1.2pt] (0.08,-.5) to (0.08,0.04);
\draw[-,line width=1pt] (0.34,.5) to (0.08,0);
\draw[-,line width=1pt] (-0.2,.5) to (0.08,0);
\node at (0.36,.65) {$\scriptstyle b$};
\node at (-.2,.65) {$\scriptstyle a$};
\node at (.4,-.2){$\scriptstyle \omega_{\lambda}$};
\draw (0.08,-.5) \bdot (0.08,-.2)\bdot;
\end{tikzpicture}
\in U_{a,b}$.
\end{proposition}

\begin{proof} 
Write $a+b=2r$.
We have
  \begin{align*}
   \begin{tikzpicture}[anchorbase,scale=.8,color=\clr]
\draw[-,line width=1.2pt] (0.08,-.5) to (0.08,0.04);
\draw[-,line width=1pt] (0.34,.5) to (0.08,0);
\draw[-,line width=1pt] (-0.2,.5) to (0.08,0);
\node at (0.36,.65) {$\scriptstyle b$};
\node at (-.2,.65) {$\scriptstyle a$};
\draw (0.08,-.5) \bdot;
\end{tikzpicture}
= \frac{1}{2^r}
  \begin{tikzpicture}[anchorbase,scale=.8,color=\clr]
\draw[-,line width=1.2pt] (0.08,-.2) to (0.08,0.04);
\draw[-,line width=1pt] (0.34,.5) to (0.08,0);
\draw[-,line width=1pt] (-0.2,.5) to (0.08,0);
\draw[-,line width=1pt] (0.34,-.5) to (0.08,-.2);
\draw[-,line width=1pt] (-0.2,-.5) to (0.08,-.2);
\draw[-,line width=1pt] (-.2,-.5) to [out=down,in=left] (0.08,-.8) to [out=right,in=down] (.34,-.5);
\node at (0.36,.65) {$\scriptstyle b$};
\node at (-.2,.65) {$\scriptstyle a$};
\draw (-.2,-.5) \bdot;
\end{tikzpicture}
\overset{\eqref{splitmerge}}=
\frac{1}{2^r}\sum_{\substack{0 \leq s \leq \min(a,r)\\0 \leq t \leq \min(r,b)\\t-s=b-r}}
\begin{tikzpicture}[anchorbase,scale=.9, color=\clr]
	\draw[-,thick] (0.58,0) to (0.58,.2) to (.02,.8) to (.02,1);
	\draw[-,thick] (0.02,0) to (0.02,.2) to (.58,.8) to (.58,1);
	\draw[-,thick] (0,0) to (0,1);
	\draw[-,line width=1pt] (0.61,0) to (0.61,1);
 \draw[-,thick](.3,-.4) to [out=left,in=up] (0,0);
 \draw[-,thick](.3,-.4) to [out=right,in=up] (0.6,0);
 \draw (.1,-.2) \bdot;
        \node at (0,1.16) {$\scriptstyle a$};
        \node at (0.6,1.16) {$\scriptstyle b$};
        \node at (-0.1,.5) {$\scriptstyle s$};
        \node at (0.77,.5) {$\scriptstyle t$};
\end{tikzpicture}
\overset{\eqref{updotmovesplits+merge}}\equiv
\frac{1}{2^r}\sum_{\substack{0 \leq s \leq \min(a,r)\\0 \leq t \leq \min(r,b)\\t-s=b-r}}
\begin{tikzpicture}[anchorbase,scale=.9, color=\clr]
	\draw[-,thick] (0.58,0) to (0.58,.2) to (.02,.8) to (.02,1);
	\draw[-,thick] (0.02,0) to (0.02,.2) to (.58,.8) to (.58,1);
	\draw[-,thick] (0,0) to (0,1);
	\draw[-,line width=1pt] (0.61,0) to (0.61,1);
 \draw[-,thick](.3,-.4) to [out=left,in=up] (0,0);
 \draw[-,thick](.3,-.4) to [out=right,in=up] (0.6,0);
 \draw (0,.5) \bdot;
  \draw (0.2,.4) \bdot;
        \node at (0,1.16) {$\scriptstyle a$};
        \node at (0.6,1.16) {$\scriptstyle b$};
        \node at (-0.2,.3) {$\scriptstyle s$};
        \node at (0.77,.5) {$\scriptstyle t$};
\end{tikzpicture}.
  \end{align*}
The summand can be rewritten as 
\begin{align*}
    \begin{tikzpicture}[anchorbase,scale=.9, color=\clr]
	\draw[-,thick] (0.58,0) to (0.58,.2) to (.02,.8) to (.02,1);
	\draw[-,thick] (0.02,0) to (0.02,.2) to (.58,.8) to (.58,1);
	\draw[-,thick] (0,0) to (0,1);
	\draw[-,line width=1pt] (0.61,0) to (0.61,1);
 \draw[-,thick](.3,-.4) to [out=left,in=up] (0,0);
 \draw[-,thick](.3,-.4) to [out=right,in=up] (0.6,0);
 \draw (0,.5) \bdot;
  \draw (0.2,.4) \bdot;
        \node at (0,1.16) {$\scriptstyle a$};
        \node at (0.6,1.16) {$\scriptstyle b$};
        \node at (-0.2,.3) {$\scriptstyle s$};
        \node at (0.77,.5) {$\scriptstyle t$};
\end{tikzpicture}
&\overset{\eqref{dotmovecrossing2}}\equiv
\begin{tikzpicture}[anchorbase,scale=.9, color=\clr]
	\draw[-,thick] (0.58,0) to (0.58,.2) to (.02,.8) to (.02,1);
	\draw[-,thick] (0.02,0) to (0.02,.2) to (.58,.8) to (.58,1);
	\draw[-,thick] (0,0) to (0,1);
	\draw[-,line width=1pt] (0.61,0) to (0.61,1);
 \draw[-,thick](.3,-.4) to [out=left,in=up] (0,0);
 \draw[-,thick](.3,-.4) to [out=right,in=up] (0.6,0);
 \draw (0,.5) \bdot;
  \draw (0.4,.6) \bdot;
        \node at (0,1.16) {$\scriptstyle a$};
        \node at (0.6,1.16) {$\scriptstyle b$};
        \node at (-0.2,.3) {$\scriptstyle s$};
        \node at (0.77,.5) {$\scriptstyle t$};
\end{tikzpicture}
\overset{\eqref{mergesplitslideleft}}= 
\begin{tikzpicture}[anchorbase,scale=.9, color=\clr]
	\draw[-,thick] (0.4,0) to (0.4,.2) to (.02,.8) to (.02,1);
	\draw[-,thick] (0.02,0) to (0.02,.2) to (.58,.8) to (.58,1);
	\draw[-,thick] (0,0) to (0,1);
    \draw[-,thick] (0,0) to (0,-.3);
      \draw[-,thick] (0,0) to [out=down,in=left] (0.2,-.2) to [out=right,in=down] (0.4,0);
       \draw[-,thick] (0,-.3) to [out=down,in=left] (0.3,-.5) to [out=right,in=down] (0.61,-.3);
	\draw[-,line width=1pt] (0.61,-.3) to (0.61,1);
 \draw (0,.5) \bdot;
  \draw (0.4,.6) \bdot;
        \node at (0,1.16) {$\scriptstyle a$};
        \node at (0.6,1.16) {$\scriptstyle b$};
        \node at (-0.2,.3) {$\scriptstyle s$};
        \node at (0.77,.5) {$\scriptstyle t$};
\end{tikzpicture}\\
&\overset{\eqref{splitmerge}}=
\sum_{k,l} ~ 
\begin{tikzpicture}[anchorbase,scale=.9, color=\clr]
\draw[-,thick] (0,1) to (0,0);
\draw[-,thick] (1,1) to (1,0);
\draw[-,thick] (0,.5) to (0.4,0);
\draw[-,thick] (0.4,.4) to (0.4,0);
\draw[-,thick] (0,0) to (1,.9);
 \draw[-,thick] (0.4,0) to [out=down,in=left] (0.7,-.2) to [out=right,in=down] (0.9,0);
 \draw[-,thick](0.9,0) to (0,.9);
 \node at (0,1.16) {$\scriptstyle a$};
        \node at (1,1.16) {$\scriptstyle b$};
 \draw[-,thick] (0,0) to [out=down,in=left] (0.5,-.4) to [out=right,in=down] (1,0);
 \draw (0,.6) \bdot;
  \draw (0.75,.7) \bdot;
  \node at (-0.2,.3) {$\scriptstyle k$};
        \node at (0.55,.15) {$\scriptstyle l$};
\end{tikzpicture}=
\sum_{k,l} ~ 
\begin{tikzpicture}[anchorbase,scale=.9, color=\clr]
\draw[-,thick] (0.2,1) to (0.2,0.8);
\draw[-,thick](0,.6) to (0,.2);
\draw[-,thick] (0.2,.8) to [out=left,in=up] (0,.6);
\draw[-,thick] (0.2,.8) to [out=right,in=up] (0.4,.6);
\draw[-,thick] (1.2,1) to (1.2,0.8);
\draw[-,thick] (1.2,.8) to [out=left,in=up] (1,.6);
\draw[-,thick] (1.2,.8) to [out=right,in=up] (1.4,.6);
\draw[-,thick](1.4,.6) to (1.4,.2);
 \draw[-,thick] (0,0.2) to [out=down,in=left] (0.7,-.2) to [out=right,in=down] (1.4,0.2);
  \draw[-,thick] (0.4,0.6) to [out=down,in=left] (0.7,.2) to [out=right,in=down] (1,0.6);
  \draw[-,thick] (1,0.6) to [out=down,in=left] (1.2,.4) to [out=right,in=down] (1.4,0.6);
   \draw[-,thick] (0,0.6) to [out=down,in=left] (.2,.4) to [out=right,in=down] (.4,0.6);
    \draw (0,.65) \bdot;
  \draw (1,.7) \bdot;
  \node at (-0.2,.3) {$\scriptstyle k$};
        \node at (0.55,.15) {$\scriptstyle l$};
\end{tikzpicture}\\
&\overset{\eqref{updotmovesplits+merge}}
\equiv\sum_{k,l} ~ 
\begin{tikzpicture}[anchorbase,scale=.9, color=\clr]
\draw[-,thick] (0.2,1) to (0.2,0.8);
\draw[-,thick](0,.6) to (0,.2);
\draw[-,thick] (0.2,.8) to [out=left,in=up] (0,.6);
\draw[-,thick] (0.2,.8) to [out=right,in=up] (0.4,.6);
\draw[-,thick] (1.2,1) to (1.2,0.8);
\draw[-,thick] (1.2,.8) to [out=left,in=up] (1,.6);
\draw[-,thick] (1.2,.8) to [out=right,in=up] (1.4,.6);
\draw[-,thick](1.4,.6) to (1.4,.2);
 \draw[-,thick] (0,0.2) to [out=down,in=left] (0.7,-.2) to [out=right,in=down] (1.4,0.2);
  \draw[-,thick] (0.4,0.6) to [out=down,in=left] (0.7,.2) to [out=right,in=down] (1,0.6);
  \draw[-,thick] (1,0.6) to [out=down,in=left] (1.2,.4) to [out=right,in=down] (1.4,0.6);
   \draw[-,thick] (0,0.6) to [out=down,in=left] (.2,.4) to [out=right,in=down] (.4,0.6);
    \draw (0.2,.4) \bdot;
      \draw (0,.3) \bdot;
  \draw (1.2,.4) \bdot;
   \draw (.9,.3) \bdot;
  \node at (-0.2,.3) {$\scriptstyle k$};
        \node at (0.55,.15) {$\scriptstyle l$};
\end{tikzpicture}
= 2^{r-t}\sum_{k+l =t}(-1)^l \binom{t}{k}
\begin{tikzpicture}[anchorbase, scale=0.4, color=\clr]
\draw[-,line width=1.2pt] (0,0) to [out=down,in=left] (1,-1) to [out=right,in=down] (2,0);
\draw[-,line width=1pt] (0,0) to (-.5,-.5);
\draw[-,line width=1pt] (2,0) to (2.5,-.5);
\draw[line width=1.3pt] (0,0) -- (0,.4)
                      (2,0) -- (2,.4);
\node at (0,.65){$\scriptstyle a$};
\node at (1.4,-1.2){$\scriptstyle t$};
\node at (2,.65){$\scriptstyle b$};
\draw (-.5,-.5)\bdot;
\draw (.1,-.5)\bdot;
\draw (2.5,-.5)\bdot;
\end{tikzpicture}\in U_{a,b}.
\end{align*}
 This proves the first claim.

We have
\begin{align*}
    \begin{tikzpicture}[anchorbase,scale=.8,color=\clr]
\draw[-,line width=1.2pt] (0.08,-.5) to (0.08,0.04);
\draw[-,line width=1pt] (0.34,.5) to (0.08,0);
\draw[-,line width=1pt] (-0.2,.5) to (0.08,0);
\node at (0.36,.65) {$\scriptstyle b$};
\node at (-.2,.65) {$\scriptstyle a$};
\node at (.4,-.2){$\scriptstyle \omega_{\lambda}$};
\draw (0.08,-.5) \bdot (0.08,-.2)\bdot;
\end{tikzpicture}
~\overset{\eqref{dotmovesplitss}}{\equiv}~
\sum_{\vartheta+\varsigma=\lambda}
    \begin{tikzpicture}[anchorbase,scale=.8,color=\clr]
\draw[-,line width=1.2pt] (0.08,-.5) to (0.08,0.04);
\draw[-,line width=1pt] (0.34,.5) to (0.08,0);
\draw[-,line width=1pt] (-0.2,.5) to (0.08,0);
\node at (0.36,.65) {$\scriptstyle b$};
\node at (-.2,.65) {$\scriptstyle a$};
\draw (0.08,-.5) \bdot;
\draw (.2,.24)\bdot (-.1,.24)\bdot;
\node at (.55,.24){$\scriptstyle \omega_{\varsigma}$};
\node at (-.45,.24){$\scriptstyle \omega_{\vartheta}$};
\end{tikzpicture}
~\overset{\cref{lollipop-merge-split}}{\equiv}~
\sum h_{x,y,z}^{\vartheta,\varsigma}
\begin{tikzpicture}[anchorbase, scale=0.45, color=\clr]
\draw[-,line width=1.2pt] (0,0) to [out=down,in=left] (1,-1) to [out=right,in=down] (2,0);
\draw[-,line width=1pt] (0,0) to (-.5,-.5) -- (-1,-1);
\draw[-,line width=1pt] (2,0) to (2.5,-.5) -- (3,-1);
\draw[line width=1.3pt] (0,0) -- (0,.6)
                      (2,0) -- (2,.6);
\draw (0,.3)\bdot (2,.3)\bdot;
\node at (-.6,.3){$\scriptstyle \omega_{\vartheta}$};
\node at (2.6,.3){$\scriptstyle \omega_{\varsigma}$};
\node at (0,.8){$\scriptstyle a$};
\node at (1.4,-1.2){$\scriptstyle t$};
\node at (2,.8){$\scriptstyle b$};
\node at (-1.05, -.25) {$\scriptstyle \omega_{x}$};
\node at (3.1, -.25) {$\scriptstyle \omega_{z}$};
\draw (-.5,-.5)\bdot (-1,-1)\bdot;
\draw (.1,-.5)\bdot (3,-1)\bdot;
\node at (.05,-.95){$\scriptstyle \omega_y$};
\draw (2.5,-.5)\bdot;
\end{tikzpicture},
\end{align*}
for $h_{x,y,z}^{\vartheta,\varsigma}\in \mathbb J'.$
Through \eqref{dotmovemergess}, we may move the dots $\omega_{\vartheta}$ and $\omega_{\varsigma}$ down to the legs. Then the second claim follows from \cref{lollipopspan}.
\end{proof}

\subsection{Bubbles}

\begin{lemma}
\label{lem:nested}
    Any nested or intersecting bubbles in $\End_{\AWb}(\unit)$ is a linear combination of products of mutually disjoint, non-nested bubbles.
\end{lemma}

\begin{proof}
We proceed by induction on the total degree, showing that every nested or intersecting configuration can be rewritten as a linear combination of unnested, nonintersecting bubbles, modulo lower degree terms.

For the base case when the degree is zero (i.e., the bubbles contain no dots), the relations \eqref{zigzag}, \eqref{braid}, and \eqref{curlremove} imply that strands can be moved freely via isotopies to eliminate any nesting and intersection.

For the inductive step, \cref{dotmovefreely} allows us to slide dots along the bubbles modulo lower degree terms. Specifically, working modulo lower-degree terms, we can systematically shift each bubble to the left to eliminate nesting, and subsequently straighten each strand to remove crossings. The lemma follows. 
\end{proof}

For any partition $\vartheta=(\vartheta_1,\ldots, \vartheta_k)\in \Par_a$, we define the following dotted bubbles
\begin{align}
    \label{bubble+dot+par}    \dotbubble_{\vartheta,a}:=\begin{tikzpicture}[yscale=.6,anchorbase, scale=.8, color=\clr] 
\draw[-,line width=1.2pt]  (0.5,2) to[out=left,in=up] (0,1)to[out=down,in=left] (0.5,0);      
\draw[-,line width=1.2pt] (0.5,0)to[out=right,in=down] (1,1)to[out=up,in=right] (0.5,2);
\node at (-.2,.9){$\scriptstyle \vdots$};
\draw (0.1,1.7) \bdot;
\node at (-.3,1.7) {$\scriptstyle \omega_{\vartheta_1}$}; 
\draw (0,1.2) \bdot; 
\node at (-0.4,1.2) {$\scriptstyle \omega_{\vartheta_2}$};
\draw (0.05,.4) \bdot; 
\node at (-0.35,.2) {$\scriptstyle \omega_{\vartheta_k}$};
\node at (0.3,-.3) {$\scriptstyle a$};
\end{tikzpicture} \: ,
\qquad \qquad
\dotbubble_a:=\dotbubble_{(a),a} =\bubbledota \: .
\end{align}
Also set $\dotbubble_0=1$.

Let $\Xi$ be the commutative subalgebra of $\End_{\AWb}(\unit)$ generated by the even bubbles $\bubbledota$, for $a \in 2\N$. (It will follow from \cref{thm:basisAWb} that $\Xi =\End_{\AWb}(\unit)$ and that it is a polynomial algebra.)

\begin{lemma}
\label{Deltaaodd}
We have $\dotbubble_a\equiv 0$ and $\dotbubble_a\in \Xi$, for $a$ odd.
\end{lemma}

\begin{proof}
Assume that $a$ is odd. By \eqref{zigzag} we have 
    \[
    \dotbubble_a=(-1)^a
    \begin{tikzpicture}
        [anchorbase,scale=.5,color=\clr]
        \draw[thick] (.5,-.5) \braidup (-.5,.5) to [out=up,in=left] (0,1) to [out=right,in=up] (.5,.5) \braiddown (-.5,-.5) to [out=down,in=left] (0,-1) to [out=right,in=down] (.5,-.5);
        \node at (0,-1.25){$\scriptstyle a$};
        \draw (-.5,.5)\bdot;
    \end{tikzpicture}
    ~\overset{\eqref{dotmovecrossing1}}{\equiv}(-1)^a~
       \begin{tikzpicture}
        [anchorbase,scale=.5,color=\clr]
        \draw[thick] (.5,-.5) \braidup (-.5,.5) to [out=up,in=left] (0,1) to [out=right,in=up] (.5,.5) \braiddown (-.5,-.5) to [out=down,in=left] (0,-1) to [out=right,in=down] (.5,-.5);
        \node at (0,-1.25){$\scriptstyle a$};
        \draw (.5,-.5)\bdot;
    \end{tikzpicture}
    =
    \begin{tikzpicture}
        [anchorbase,color=\clr,scale=.5]
        \draw[thick] (.5,0) to [out=up,in=right] (0,.5) to
                [out=left,in=up] (-.5,0) to
                [out=down,in=left] (0,-.5) to
                [out=right,in=down] (.5,0);
        \draw (.5,0)\bdot;
        \node at (0,-.75){$\scriptstyle a$};
    \end{tikzpicture}
    \overset{\eqref{dotmovecupcaps}}{=} (-1)^a\dotbubble_a
    = -\dotbubble_a.
    \]
    Since we work over \(\mathbb J'\), the scalar 2 is invertible. Thus the congruence \(\dotbubble_a\equiv -\dotbubble_a\) expresses \(\dotbubble_a\) as a linear combination of lower-degree bubble diagrams. The second claim follows from the first one and \cref{lem:nested}. 
\end{proof}

Our goal is to show that $\dotbubble_{\vartheta,a}\in\Xi$ by induction on $a$. The following lemma is the first step of the induction (for $a=1$).
For the thin bubble, we use $\bubbledott$ to denote the diagram with  $t$ dots on the left of the bubble. 
\begin{lemma}
\label{lem:bubble1+dott}
For any $t\in \Z_{\geq 1}$, we have   $\bubbledott\in \Xi$.    
\end{lemma}

\begin{proof}
   Suppose $t$ is odd. Then 
     \[
   \bubbledott=-
    \begin{tikzpicture}
        [anchorbase,scale=.5,color=\clr]
        \draw[thick] (.5,-.5) \braidup (-.5,.5) to [out=up,in=left] (0,1) to [out=right,in=up] (.5,.5) \braiddown (-.5,-.5) to [out=down,in=left] (0,-1) to [out=right,in=down] (.5,-.5);
        \node at (-.75,.5){$\scriptstyle t$};
        \draw (-.5,.5)\bdot;
    \end{tikzpicture}
    ~\overset{\eqref{dotmovecrossing1}}{\equiv}-~
       \begin{tikzpicture}
        [anchorbase,scale=.5,color=\clr]
        \draw[thick] (.5,-.5) \braidup (-.5,.5) to [out=up,in=left] (0,1) to [out=right,in=up] (.5,.5) \braiddown (-.5,-.5) to [out=down,in=left] (0,-1) to [out=right,in=down] (.5,-.5);
        \node at (0.75,-.5){$\scriptstyle a$};
        \draw (.5,-.5)\bdot;
    \end{tikzpicture}
    =
    \begin{tikzpicture}
        [anchorbase,color=\clr,scale=.5]
        \draw[thick] (.5,0) to [out=up,in=right] (0,.5) to
                                [out=left,in=up] (-.5,0) to
                                [out=down,in=left] (0,-.5) to
                                [out=right,in=down] (.5,0);
        \draw (.5,0)\bdot;
        \node at (0.75,0){$\scriptstyle t$};
    \end{tikzpicture}
    \overset{\eqref{dotmovecupcaps}}{=} -\bubbledott .
    \]
    Hence it suffices to prove the result when $t$ is even.
    
    Assuming $t$ is even, we have 
    \begin{align*}
        t \bubbledottt
        =&\begin{tikzpicture}[yscale=0.8,anchorbase,scale=.7,color=\clr]
\draw[-,thick] (0.08,-.8) to (0.08,-.5);
\draw[-,thick] (0.08,.3) to (0.08,.6) to [out=up,in=up] (1,.6) to (1,-.8) to [out=down,in=down] (.08,-.8);
\draw[-,line width=1pt] (0.1,-.51) to [out=45,in=-45] (0.1,.31);
\draw[-,line width=1pt] (0.06,-.51) to [out=135,in=-135] (0.06,.31);
\draw(-.1,0) \bdot;
\draw(.25,0) \bdot;
\node at (.45,-.35) {$\scriptstyle 1$};
\node at (.5,-1.3){$\scriptstyle t$};
\end{tikzpicture}
~\overset{\eqref{dumbbellremove}}{=}~
\begin{tikzpicture}[anchorbase,scale=.7,color=\clr]
\draw[-,line width=1pt] (-.1,-.51) to (-.1,.31);
\draw[-,line width=1pt] (.25,-.51) to (.25,.31);
\draw[-,line width=1pt] (.25,.31) to [out=up,in=up] (.75,.31) to (1,.1); 
\draw[-,line width=1pt] (-.1,.31) to [out=up,in=up] (1.25,.31) to (1,.1); 
\draw[-,line width=1pt] (.25,-.51) to [out=down,in=down] (.75,-.51) to (1,-.3); 
\draw[-,line width=1pt] (-.1,.-.51) to [out=down,in=down] (1.25,-.51) to (1,-.3);
\draw[-,thick] (1,-.3) to (1,.1);
\draw(-.1,0) \bdot;
\draw(.25,0) \bdot;
\node at (.45,-.35) {$\scriptstyle 1$};
\node at (.5,-1.1){$\scriptstyle t-1$};
\end{tikzpicture}
~\overset{\eqref{splitmerge}}{=}
\begin{tikzpicture}[anchorbase,scale=.7,color=\clr]
\draw[-,line width=1pt] (-.1,-.51) to (-.1,.31);
\draw[-,line width=1pt] (.25,-.51) to (.25,.31);
\draw[-,line width=1pt] (.25,.31) to [out=up,in=up] (.75,.31) to (.75,-.51); 
\draw[-,line width=1pt] (-.1,.31) to [out=up,in=up] (1.25,.31) to (1.25,-.51); 
\draw[-,line width=1pt] (.25,-.51) to [out=down,in=down] (.75,-.51); 
\draw[-,line width=1pt] (-.1,.-.51) to [out=down,in=down] (1.25,-.51);
\draw(-.1,0) \bdot;
\draw(.25,0) \bdot;
\node at (.45,-.35) {$\scriptstyle 1$};
\node at (.5,-1.1){$\scriptstyle t-1$};
\end{tikzpicture}
~+~
\begin{tikzpicture}[anchorbase,scale=.7,color=\clr]
\draw[-,line width=1pt] (-.1,-.51) to (-.1,.31);
\draw[-,line width=1pt] (.25,-.51) to (.25,.31);
\draw[-,line width=1pt] (.25,.31) to [out=up,in=up] (.75,.31) to (1.25,-.51); 
\draw[-,line width=1pt] (-.1,.31) to [out=up,in=up] (1.25,.31) to (.75,-.51); 
\draw[-,line width=1pt] (.25,-.51) to [out=down,in=down] (.75,-.51); 
\draw[-,line width=1pt] (-.1,.-.51) to [out=down,in=down] (1.25,-.51) to (1.25,.31);
\draw(-.1,0) \bdot;
\draw(.25,0) \bdot;
\node at (.45,-.35) {$\scriptstyle 1$};
\node at (.5,-1.1){$\scriptstyle t-1$};
\end{tikzpicture}
~\overset{\eqref{dotmovecrossing1}}{\underset{\eqref{curlremove}}{\equiv}}~
\dotbubble_1\dotbubble_{t-1} +~ 
\begin{tikzpicture}[yscale=0.8,anchorbase,scale=.7,xscale=-1,color=\clr]
\draw[-,thick] (0.08,-.8) to (0.08,-.5);
\draw[-,thick] (0.08,.3) to (0.08,.6) to [out=up,in=up] (1,.6) to (1,-.8) to [out=down,in=down] (.08,-.8);
\draw[-,line width=1pt] (0.1,-.51) to [out=45,in=-45] (0.1,.31);
\draw[-,line width=1pt] (0.06,-.51) to [out=135,in=-135] (0.06,.31);
\draw(1,0) \bdot;
\draw(.25,0) \bdot;
\node at (.45,-.35) {$\scriptstyle 1$};
\node at (.5,-1.3){$\scriptstyle t-1$};
\end{tikzpicture}\\
~\overset{\eqref{dotmovecupcaps}}{=}~
&\dotbubble_1\dotbubble_{t-1}+ (-1)^{t-1}
\begin{tikzpicture}[yscale=0.8,anchorbase,scale=.7,xscale=-1,color=\clr]
\draw[-,thick] (0.08,-.8) to (0.08,-.5);
\draw[-,thick] (0.08,.3) to (0.08,.6) to [out=up,in=up] (1,.6) to (1,-.8) to [out=down,in=down] (.08,-.8);
\draw[-,line width=1pt] (0.1,-.51) to [out=45,in=-45] (0.1,.31);
\draw[-,line width=1pt] (0.06,-.51) to [out=135,in=-135] (0.06,.31);
\draw(-.1,0) \bdot;
\draw(.25,0) \bdot;
\draw(.25,-.2)\bdot;
\node at (.45,-.35) {$\scriptstyle 1$};
\node at (.5,-1.3){$\scriptstyle t-1$};
\end{tikzpicture}
~\overset{\eqref{dumbbellremove}}{=}~\dotbubble_1\dotbubble_{t-1}+ (-1)^{t-1}
\begin{tikzpicture}[anchorbase,scale=.7,xscale=-1,color=\clr]
\draw[-,line width=1pt] (-.1,-.51) to (-.1,.31);
\draw[-,line width=1pt] (.25,-.51) to (.25,.31);
\draw[-,line width=1pt] (.25,.31) to [out=up,in=up] (.75,.31) to (1,.1); 
\draw[-,line width=1pt] (-.1,.31) to [out=up,in=up] (1.25,.31) to (1,.1); 
\draw[-,line width=1pt] (.25,-.51) to [out=down,in=down] (.75,-.51) to (1,-.3); 
\draw[-,line width=1pt] (-.1,.-.51) to [out=down,in=down] (1.25,-.51) to (1,-.3);
\draw[-,thick] (1,-.3) to (1,.1);
\draw(-.1,0) \bdot;
\draw(.25,0) \bdot;
\draw(.25,-.2)\bdot;
\node at (.45,-.35) {$\scriptstyle 1$};
\node at (.5,-1.1){$\scriptstyle t-2$};
\end{tikzpicture}
\\
~\overset{\eqref{splitmerge}}{\equiv}~&\dotbubble_1\dotbubble_{t-1} -
\begin{tikzpicture}[anchorbase, scale=1.4, color=\clr]
					\draw[-] (0.2,0.2) to[out=90,in=0] (0,.4);
					\draw[-] (0,0.4) to[out=180,in=90] (-.2,0.2);
					\draw[-] (-.2,0.2) to[out=-90,in=180] (0,0);
					\draw[-] (0,0) to[out=0,in=-90] (0.2,0.2);
			\node at (0,-.1) {$\scriptstyle{1}$};
         \draw (-.2,.2) \bdot;
     \node at (-.3,.2) {$\scriptstyle{2}$};
			\end{tikzpicture}\dotbubble_{t-2}
~-~
\begin{tikzpicture}[yscale=.8,anchorbase,scale=.7,color=\clr]
\draw[-,thick] (0.08,-.8) to (0.08,-.5);
\draw[-,thick] (0.08,.3) to (0.08,.6) to [out=up,in=up] (1,.6) to (1,-.8) to [out=down,in=down] (.08,-.8);
\draw[-,line width=1pt] (0.1,-.51) to [out=45,in=-45] (0.1,.31);
\draw[-,line width=1pt] (0.06,-.51) to [out=135,in=-135] (0.06,.31);
\draw(.08,.6) \bdot;
\draw(.25,0) \bdot;
\draw(.25,-.2) \bdot;
\node at (.45,-.35) {$\scriptstyle 1$};
\node at (.5,-1.3){$\scriptstyle t-2$};
\end{tikzpicture}
\\
\equiv & \sum_{k=1}^{s}\Big((-1)^{k-1}\begin{tikzpicture}[anchorbase, scale=1.4, color=\clr]
					\draw[-] (0.2,0.2) to[out=90,in=0] (0,.4);
					\draw[-] (0,0.4) to[out=180,in=90] (-.2,0.2);
					\draw[-] (-.2,0.2) to[out=-90,in=180] (0,0);
					\draw[-] (0,0) to[out=0,in=-90] (0.2,0.2);
			\node at (0,-.1) {$\scriptstyle{1}$};
         \draw (-.2,.2) \bdot;
     \node at (-.5,.2) {$\scriptstyle{2k-1}$};
			\end{tikzpicture}\dotbubble_{t-2k+1}
    ~+~(-1)^k
  \begin{tikzpicture}[anchorbase, scale=1.4, color=\clr]
					\draw[-] (0.2,0.2) to[out=90,in=0] (0,.4);
					\draw[-] (0,0.4) to[out=180,in=90] (-.2,0.2);
					\draw[-] (-.2,0.2) to[out=-90,in=180] (0,0);
					\draw[-] (0,0) to[out=0,in=-90] (0.2,0.2);
			\node at (0,-.1) {$\scriptstyle{1}$};
         \draw (-.2,.2) \bdot;
     \node at (-.4,.2) {$\scriptstyle{2k}$};
			\end{tikzpicture}\dotbubble_{t-2k}\Big).
    \end{align*}
 Thus the lemma follows from the above computation and an induction on $t=2s$.
\end{proof}

\begin{lemma}
    For any partition $\vartheta\in \Par_a$, we have (modulo lower degree terms):
    \begin{align}
    \label{dumbell+dots}
     \begin{tikzpicture}[xscale=-1,yscale=.8,anchorbase,scale=.7,color=\clr] 
\draw[-,line width=1.2pt]  (0.5,2) to[out=left,in=up] (0,1)to[out=down,in=left] (0.5,0);      
\draw[-,line width=1.2pt] (0.5,0)to[out=right,in=-120] (1,.5);
\draw[-,line width=1.2pt] (1,1.5)to[out=60,in=left] (1.5,2);
\draw[-,line width=1.2pt] (0.5,2)to[out=right,in=120] (1,1.5);
\draw[-,line width=1.2pt] (1,.5)to[out=-60,in=left] (1.5,0);
\draw[-,line width=1.4pt] (1,.5) to (1,1.5);
\draw (0,1.2) \bdot; 
\node at (-0.5,1.2) {$\scriptstyle \omega_{\vartheta}$};
\node at (0.2,-.1) {$\scriptstyle b$};
\node at (1.25,-.1) {$\scriptstyle a$};
\end{tikzpicture},
\quad
\begin{tikzpicture}[yscale=.8,anchorbase, scale=.7, color=\clr] 
\draw[-,line width=1.2pt]  (0.5,2) to[out=left,in=up] (0,1)to[out=down,in=left] (0.5,0);      
\draw[-,line width=1.2pt] (0.5,0)to[out=right,in=-120] (1,.5);
\draw[-,line width=1.2pt] (1,1.5)to[out=60,in=left] (1.5,2);
\draw[-,line width=1.2pt] (0.5,2)to[out=right,in=120] (1,1.5);
\draw[-,line width=1.2pt] (1,.5)to[out=-60,in=left] (1.5,0);
\draw[-,line width=1.4pt] (1,.5) to (1,1.5);
\draw (0,1.2) \bdot; 
\node at (-0.4,1.2) {$\scriptstyle \omega_{\vartheta}$};
\node at (0.2,-.1) {$\scriptstyle b$};
\node at (1.25,-.1) {$\scriptstyle a$};
\end{tikzpicture}
~\in \mathbb J'\text{-span}~\left\{\dotbubble_{c,\nu}\otimes\wkdotavar{\varsigma} \colon 0\leq c\leq a, \nu+\varsigma  \unrhd \vartheta\right\} .
\end{align}
\end{lemma}

\begin{proof}
    We proceed by induction on $a+b$ and $|\vartheta|$ simultaneously. When $a=b=1$, the statement for any $\vartheta$ follows from \eqref{splitmerge} and \eqref{curlremove} directly. Suppose the lemma is true for any $c,d$ such that $c+d<a+b$. Now we only prove for $\begin{tikzpicture}[yscale=.8,anchorbase, scale=.7, color=\clr] 
\draw[-,line width=1.2pt]  (0.5,2) to[out=left,in=up] (0,1)to[out=down,in=left] (0.5,0);      
\draw[-,line width=1.2pt] (0.5,0)to[out=right,in=-120] (1,.5);
\draw[-,line width=1.2pt] (1,1.5)to[out=60,in=left] (1.5,2);
\draw[-,line width=1.2pt] (0.5,2)to[out=right,in=120] (1,1.5);
\draw[-,line width=1.2pt] (1,.5)to[out=-60,in=left] (1.5,0);
\draw[-,line width=1.4pt] (1,.5) to (1,1.5);
\draw (0,1.2) \bdot; 
\node at (-0.4,1.2) {$\scriptstyle \omega_{\vartheta}$};
\node at (0.2,-.1) {$\scriptstyle b$};
\node at (1.25,-.1) {$\scriptstyle a$};
\end{tikzpicture}$ 
as the other case is proved similarly. 
    
    Suppose $a\ge b$. In this case we have
\begin{align}
\label{sumdumbell}
    \begin{tikzpicture}[yscale=.8,anchorbase, scale=.7, color=\clr] 
\draw[-,line width=1.2pt]  (0.5,2) to[out=left,in=up] (0,1)to[out=down,in=left] (0.5,0);      
\draw[-,line width=1.2pt] (0.5,0)to[out=right,in=-120] (1,.5);
\draw[-,line width=1.2pt] (1,1.5)to[out=60,in=left] (1.5,2);
\draw[-,line width=1.2pt] (0.5,2)to[out=right,in=120] (1,1.5);
\draw[-,line width=1.2pt] (1,.5)to[out=-60,in=left] (1.5,0);
\draw[-,line width=1.4pt] (1,.5) to (1,1.5);
\draw (0,1.2) \bdot; 
\node at (-0.43,1.23) {$\scriptstyle \omega_{\vartheta}$};
\node at (0.2,-.1) {$\scriptstyle b$};
\node at (1.25,-.1) {$\scriptstyle a$};
\end{tikzpicture}
\equiv
\sum_{s=0}^b\begin{tikzpicture}[yscale=.8,anchorbase, scale=.7, color=\clr] 
\draw[-,line width=1.2pt]  (0.5,2) to[out=left,in=up] (0,1)to[out=down,in=left] (0.5,0);      
\draw[-,line width=1.2pt] (0.5,0)to[out=right,in=-120] (1,.5);
\draw[-,line width=1.2pt] (1.5,1.5)to[out=60,in=left] (2,2);
\draw[-,line width=1.2pt] (0.5,2)to[out=right,in=120] (1,1.5);
\draw[-,line width=1.2pt] (1.5,.5)to[out=-60,in=left] (2,0);
\draw[-,line width=1pt] (1,.5) to (1,1.5);
\draw[-,line width=1pt] (1.5,.5) to (1.5,1.5);
\draw[-,line width=1pt] (1,.5) to (1.5,1.5);
\draw[-,line width=1pt] (1,1.5) to (1.5,.5);
\node at (0.85,1) {$\scriptstyle s$};
\node at (2.2,1) {$\scriptstyle s+a-b$};
\draw (0,1.2) \bdot; 
\node at (-0.43,1.23) {$\scriptstyle \omega_{\vartheta}$};
\node at (0.2,-.1) {$\scriptstyle b$};
\node at (1.4,-.1) {$\scriptstyle a$};
\end{tikzpicture}
\pmod {\End_{\AWb}(a)_{<|\vartheta|}} .
\end{align}
In the above summation, the summand for $s=0$ is equal to $(-1)^b\begin{tikzpicture}
[anchorbase,scale=1,color=\clr]
\draw[-,line width=1.5pt] (0.08,-.8) to (0.08,-.5);
\node at (.08,-1){$\scriptstyle a$};
\draw[-,line width=1.5pt] (0.08,.3) to (0.08,.6);
\draw[-,thick] (0.1,-.51) to [out=45,in=-45] (0.1,.31);
\draw[-,thick] (0.06,-.51) to [out=135,in=-135] (0.06,.31);
\draw(-.1,0) \bdot;
\draw(-.4,0)node {$\scriptstyle \omega_{\vartheta}$};
\node at (-.22,-.35) {$\scriptstyle b$};
\end{tikzpicture}$. By \cref{wkdotsspan} we see this term belongs to $E_a^\vartheta$ up to lower degree terms. The summand for $s=b$ is equal to $\dotbubble_{\vartheta,b}\otimes \stra$. When $0<s<b$, we have
\[
\begin{tikzpicture}[yscale=.8,anchorbase, scale=.7, color=\clr] 
\draw[-,line width=1.2pt]  (0.5,2) to[out=left,in=up] (0,1)to[out=down,in=left] (0.5,0);      
\draw[-,line width=1.2pt] (0.5,0)to[out=right,in=-120] (1,.5);
\draw[-,line width=1.2pt] (1.5,1.5)to[out=60,in=left] (2,2);
\draw[-,line width=1.2pt] (0.5,2)to[out=right,in=120] (1,1.5);
\draw[-,line width=1.2pt] (1.5,.5)to[out=-60,in=left] (2,0);
\draw[-,line width=1pt] (1,.5) to (1,1.5);
\draw[-,line width=1pt] (1.5,.5) to (1.5,1.5);
\draw[-,line width=1pt] (1,.5) to (1.5,1.5);
\draw[-,line width=1pt] (1,1.5) to (1.5,.5);
\node at (0.85,1) {$\scriptstyle s$};
\node at (2.2,1) {$\scriptstyle s+a-b$};
\draw (0,1.2) \bdot; 
\node at (-0.43,1.23) {$\scriptstyle \omega_{\vartheta}$};
\node at (0.2,-.1) {$\scriptstyle b$};
\node at (1.4,-.1) {$\scriptstyle a$};
\end{tikzpicture}\overset{\eqref{zigzag},\eqref{mergesplitslideleft}}{=}\begin{tikzpicture}[anchorbase,scale=.8,scale=0.3,yscale=0.7, color=\clr]
 \draw[-,line width=1pt] (-1,-2.75) to (-1,2.5);
  \draw[-,line width=1pt] (1,-2.75) to (1,2.5);
  \draw[-,line width=1pt] (1,2.2) to (-.6,-2.5) to (-.6,-3.2)to [out=down,in=down] (-4,-3.2) to (-4,-2.75) to (-3.5,-1) to (-3,-2.75);
  \draw[-,line width=1pt] (1,-2.5) to (-.6,2.2) to (-.6,3) to [out=up,in=up] (-4,3) to (-4,2.5) to (-3.5,1) to (-3,2.5);
  \draw[-,line width=1.4pt] (-3.5,1) to (-3.5,-1);
  \draw (-3.5,.3)\bdot;
\draw [-,line width=1pt] (-3,2.5) to [in=180, out=90] (-2,3.5) to [in=90, out=0](-1,2.5);
\draw [-,line width=1pt] (-1,-2.75) to [in=0, out=270] (-2,-3.75) to [in=270, out=180](-3,-2.75);
\draw [-,line width=1pt] (1,-3.75) to (1,-2.75);
\draw [-,line width=1pt] (1,2.5) to (1,3.5);
\node at (1,4) {$\scriptstyle{a}$};
\node at (1,-4.5) {$\scriptstyle{a}$};
\node at (-4.5,.3) {$\scriptstyle{\omega_{\vartheta}}$};
\node at (-1.4,.6) {$\scriptstyle{s}$};
\end{tikzpicture}
~\equiv~
\sum_{\varsigma+\varsigma'=\vartheta}
\begin{tikzpicture}[anchorbase,scale=0.25,yscale=0.7, color=\clr]
 \draw[-,line width=1pt] (-1,-2.75) to (-1,2.5);
  \draw[-,line width=1pt] (1,-2.75) to (1,2.5);
  \draw[-,line width=1pt] (1,2.2) to (-.6,-2.5) to (-.6,-3.2)to [out=down,in=down] (-4,-3.2) to (-4,-2.75) to (-3.5,-1) to (-3,-2.75);
  \draw[-,line width=1pt] (1,-2.5) to (-.6,2.2) to (-.6,3) to [out=up,in=up] (-4,3) to (-4,2.5) to (-3.5,1) to (-3,2.5);
  \draw[-,line width=1.4pt] (-3.5,1) to (-3.5,-1);
  \draw (-3,2.5)\bdot;
  \draw (-4,2.5)\bdot;
\draw [-,line width=1pt] (-3,2.5) to [in=180, out=90] (-2,3.5) to [in=90, out=0](-1,2.5);
\draw [-,line width=1pt] (-1,-2.75) to [in=0, out=270] (-2,-3.75) to [in=270, out=180](-3,-2.75);
\draw [-,line width=1pt] (1,-3.75) to (1,-2.75);
\draw [-,line width=1pt] (1,2.5) to (1,3.5);
\node at (1,4) {$\scriptstyle{a}$};
\node at (1,-4.5) {$\scriptstyle{a}$};
\node at (-4.9,2.5) {$\scriptstyle{\omega_{\varsigma}}$};
\node at (-2.1,2) {$\scriptstyle{\omega_{\varsigma'}}$};
\node at (-1.6,.6) {$\scriptstyle{s}$};
\end{tikzpicture}.
\]
 Then applying the inductive assumption to each summand we conclude that they belong to the $\mathbb J'$-span in \eqref{dumbell+dots}. The case when $a<b$ can be proved similarly. 
\end{proof}

\begin{lemma}
    For any partition $\vartheta\in \Par_a$, we have (modulo lower degree terms):
    \begin{equation}
        \label{sidewaydumbbell}
        \begin{tikzpicture}
            [anchorbase,color=\clr]
            \draw[line width=1.4pt] (0,-.3) -- (0,0)
                                    (1,-.3) -- (1,0)
                                    (0,0) to [out=45,in=135] (1,0);
            \draw[line width=1pt]   (0,-.3) to [out=-45,in=-135] (1,-.3)
            (0,-.3) -- (-.2,-.5) (1,-.3) -- (1.2,-.5);
            \draw (.2,-.45)\bdot;
            \node at (.3,-.25){$\scriptstyle \omega_{\vartheta}$};
            \node at (.5,-.6){$\scriptstyle a$};
        \end{tikzpicture}
        \in \text{$\mathbb J'$-span} \Big\{
        \dotbubble_{c,\nu}\otimes \begin{tikzpicture}
            [anchorbase,color=\clr]
            \draw[line width=1pt] (0,0) to [out=up,in=left] (.4,.4) to[out=right,in=up] (.8,0);
            \draw (.2,.35)\bdot;
            \node at (.2,.55){$\scriptstyle \omega_{\varsigma}$};
            \node at (0,-.2){$\scriptstyle a$};
             \node at (.8,-.2){$\scriptstyle a$};
        \end{tikzpicture}
        ~\colon~
        0\leq c\leq a, \nu+\varsigma\unrhd \vartheta
        \Big\} .
    \end{equation}
\end{lemma}
\begin{proof}
    This follows from \cref{mergesplitslideleft} and \cref{dumbell+dots} directly.
\end{proof}

\begin{proposition}
\label{bubblealgebra}
    For any partition $\vartheta\in \Par_a$, we have $\dotbubble_{\vartheta,a}\in \Xi$.
\end{proposition}

\begin{proof}
    By \cref{Deltaaodd}, it suffices to show that $\dotbubble_{\vartheta,a}\in \widetilde{\Xi}$ where $\widetilde{\Xi}$ is the subalgebra of $\End_{\AWb}(\unit)$ generated by $\dotbubble_a$ for all $a\in\N$. We proceed by induction on $a$. The case when $a=1$ follows from \cref{lem:bubble1+dott}. 

    By inductive assumption, we have $\dotbubble_{\vartheta,b} \in\widetilde{\Xi}$ whenever $b<a$. Write $\vartheta=(\vartheta_1,\ldots, \vartheta_t)$ with $\vartheta_1\le a$. We divide the following proof into two cases:

\underline{Case 1, $\vartheta\neq (a^t)$}: In this case, $\vartheta_t<a$, and we have  
\begin{align*}
   \dotbubble_{\vartheta,a}{=}
   \begin{tikzpicture}[anchorbase, scale=.8, color=\clr] 
\draw[-,line width=1.4pt]  (0.5,2) to[out=left,in=up] (0,1)to[out=down,in=left] (0.5,-.5);      
\draw[-,line width=1.4pt] (0.5,-.5)to[out=right,in=down] (1,1)to[out=up,in=right] (0.5,2);
\node at (-.2,.9){$\scriptstyle \vdots$};
\draw (0.1,1.7) \bdot;
\node at (-.3,1.7) {$\scriptstyle \omega_{\vartheta_1}$}; 
\draw (0,1.2) \bdot; 
\node at (-0.4,1.2) {$\scriptstyle \omega_{\vartheta_2}$};
\node at (.8,-.65) {$\scriptstyle a$};
\draw[-,line width=1pt] (0,.4) to [out=-45,in=45] (.2,-.4);
\draw (.2,.15)\bdot;
\node at (.6,.2){$\scriptstyle{\omega_{\vartheta_t}}$};
\end{tikzpicture}
\equiv
\sum_{\varsigma_i+\varsigma_i'=\vartheta_i\atop \varsigma_t=0}\begin{tikzpicture}[yscale=.8,anchorbase, scale=.7, color=\clr] 
\draw[-, line width=1.4pt] (0.5,2) to (0.5,2.5);
\draw[-, line width=1.4pt] (0.5,0) to (0.5,-.5);
\draw[-,line width=1.4pt]  (0.5,2) to[out=left,in=up] (0,1)to[out=down,in=left] (0.5,0);      
\draw[-,line width=1pt] (0.5,0)to[out=right,in=down] (1,1)to[out=up,in=right] (0.5,2);
\draw[-,line width=1.4pt] (.5,2.5) to [out=up,in=up]  (2,2.5) to (2,-.5) to [out=down,in=down] (.5,-.5);
\node at (-.2,.9){$\scriptstyle \vdots$};
\draw (0.1,1.7) \bdot;
\node at (-.3,1.7) {$\scriptstyle \omega_{\varsigma_1}$}; 
\draw (0,1.2) \bdot; 
\node at (-0.4,1.2) {$\scriptstyle \omega_{\varsigma_2}$};
\draw (0.05,.4) \bdot; 
\node at (-0.3,.3) {$\scriptstyle \omega_{\varsigma_t}$};
\draw (1,1.2) \bdot;
\draw (.95,.4) \bdot; 
\node at (1.5,.1) {$\scriptstyle \omega_{\varsigma_t'}$};
\draw (.9,1.7) \bdot; 
\node at (1.3,1.7) {$\scriptstyle \omega_{\varsigma_1'}$};
\node at (1.5,1.1) {$\scriptstyle \omega_{\varsigma_2'}$};
\node at (1.2,.9){$\scriptstyle \vdots$};
\node at (.9,-0.2) {$\scriptstyle \vartheta_t$};
\end{tikzpicture}
=
\sum_{\varsigma_i+\varsigma_i'=\vartheta_i\atop \varsigma_t=0}\begin{tikzpicture}[yscale=.8,anchorbase, scale=.7, color=\clr] 
\draw[-,line width=1.4pt]  (0,2) to (0,0);      
\draw[-,line width=1pt] (.5,0)to (.5,2);
\draw[-,line width=1pt] (.5,2) to [out=up,in=up] (1.5,2) to (1.75,1.25);
\draw[-,line width=1.4pt] (0,2) to [out=up,in=up] (2,2) to (1.75,1.25);
\draw[-,line width=1pt] (.5,0) to [out=down,in=down] (1.5,0) to (1.75,.75);
\draw[-,line width=1.4pt] (0,0) to [out=down,in=down] (2,0) to (1.75,.75);
\draw[-,line width=1.4pt] (1.75,1.25) to (1.75,.75);
\node at (-.3,.9){$\scriptstyle \vdots$};
\draw (0,1.7) \bdot;
\node at (-.4,1.7) {$\scriptstyle \omega_{\varsigma_1}$}; 
\draw (0,1.2) \bdot; 
\node at (-0.4,1.2) {$\scriptstyle \omega_{\varsigma_2}$};
\draw (0.,.4) \bdot; 
\node at (-0.4,.3) {$\scriptstyle \omega_{\varsigma_t}$};
\draw (.5,1.2) \bdot;
\draw (.5,.4) \bdot; 
\node at (1,.1) {$\scriptstyle \omega_{\varsigma_t'}$};
\draw (.5,1.7) \bdot; 
\node at (.9,1.7) {$\scriptstyle \omega_{\varsigma_1'}$};
\node at (.9,1.1) {$\scriptstyle \omega_{\varsigma_2'}$};
\node at (.8,.9){$\scriptstyle \vdots$};
\node at (1,-.85) {$\scriptstyle a-\vartheta_t$};
\end{tikzpicture}
\pmod {\End_{\AWb}(\varnothing)_{<|\vartheta|}}.
\end{align*}
Note that by the commutativity \eqref{commutato}, we can always assume that $\varsigma$ and $\varsigma'$ above are partitions. Then by \eqref{dumbell+dots}, the inductive assumption (since $\vartheta_k<a$ and $a-\vartheta_k<a$) and \cref{lem:nested}, each summand above is equivalent to a $\mathbb J'$-linear combination of elements in $\Xi$.

\underline{Case 2, $\vartheta=(a^t)$}: 
In this case, we have  $\dotbubble_{\vartheta,a}=\bubbledotat$. Note that Lemma~\ref{lem:bubble1+dott} is a special case (for $a=1$) of this case. Now suppose the lemma is true for any $\dotbubble_{\vartheta,b}$ where $b<at$ and $|\vartheta|<at$ (the base case again follows from Lemma~\ref{lem:bubble1+dott}).

To deal with $\bubbledotat$, we shall consider $\dotbubble_{at}$. (The case for $a=1$ was treated in Lemma~\ref{lem:bubble1+dott}.) In fact, we have the following equality modulo $\End_{\AWb}(\varnothing)_{<at}$:
\begin{equation}
    \label{eq:ataat}
\begin{aligned}
    & \binom{at}{a}\dotbubble_{at}=\begin{tikzpicture}[yscale=0.8,anchorbase,scale=.8,color=\clr]
\draw[-,line width=1.2pt] (0.08,-.8) to (0.08,-.5);
\draw[-,line width=1.2pt] (0.08,.3) to (0.08,.6) to [out=up,in=up] (1,.6) to (1,-.8) to [out=down,in=down] (.08,-.8);
\draw[-,line width=1pt] (0.1,-.51) to [out=45,in=-45] (0.1,.31);
\draw[-,line width=1pt] (0.06,-.51) to [out=135,in=-135] (0.06,.31);
\draw(-.1,0) \bdot;
\draw(.25,0) \bdot;
\draw(.6,0)node {$\scriptstyle \omega_{a}$};
\draw(.45,-.3)node {$\scriptstyle a$};
\node at (.5,-1.3){$\scriptstyle a(t-1)$};
\end{tikzpicture}
~=~
\begin{tikzpicture}[anchorbase,scale=.8,color=\clr]
\draw[-,line width=1pt] (-.1,-.51) to (-.1,.31);
\draw[-,line width=1pt] (.25,-.51) to (.25,.31);
\draw[-,line width=1pt] (.25,.31) to [out=up,in=up] (.75,.31) to (1,.1); 
\draw[-,line width=1pt] (-.1,.31) to [out=up,in=up] (1.25,.31) to (1,.1); 
\draw[-,line width=1pt] (.25,-.51) to [out=down,in=down] (.75,-.51) to (1,-.3); 
\draw[-,line width=1pt] (-.1,.-.51) to [out=down,in=down] (1.25,-.51) to (1,-.3);
\draw[-,line width=1.3pt] (1,-.3) to (1,.1);
\draw(-.1,0) \bdot;
\draw(.25,0) \bdot;
\draw(.6,0)node {$\scriptstyle \omega_{a}$};
\draw(.55,-.4)node {$\scriptstyle a$};
\node at (.5,-1.1){$\scriptstyle a(t-1)$};
\end{tikzpicture}
~\overset{\cref{dotmovecupcaps},\cref{dumbell+dots},\cref{sumdumbell}}{\equiv}~\sum_{{y+z=a\atop z\neq a}}f_y\begin{tikzpicture}[yscale=.6,anchorbase, scale=.8, color=\clr] 
\draw[-,line width=1.2pt]  (0.5,2) to[out=left,in=up] (0,1)to[out=down,in=left] (0.5,0);      
\draw[-,line width=1.2pt] (0.5,0)to[out=right,in=down] (1,1)to[out=up,in=right] (0.5,2);
\draw (0.1,1.7) \bdot;
\draw (0.05,.4) \bdot; 
\node at (-0.3,.4) {$\scriptstyle \omega_{z}$};
\node at (0.4,-.4) {$\scriptstyle a(t-1)$};
\end{tikzpicture}
~+
\begin{tikzpicture}[yscale=.8,anchorbase,scale=.8,color=\clr]
\draw[-,line width=1.2pt] (0.08,-.8) to (0.08,-.5);
\draw[-,line width=1.2pt] (0.08,.3) to (0.08,.6) to [out=up,in=up] (1,.6) to (1,-.8) to [out=down,in=down] (.08,-.8);
\draw[-,line width=1pt] (0.1,-.51) to [out=45,in=-45] (0.1,.31);
\draw[-,line width=1pt] (0.06,-.51) to [out=135,in=-135] (0.06,.31);
\draw(.08,.6) \bdot;
\draw(.25,0) \bdot;
\draw(.6,.0)node {$\scriptstyle \omega_{a}$};
\draw(.45,-.3)node {$\scriptstyle a$};
\node at (.5,-1.3){$\scriptstyle a(t-1)$};
\end{tikzpicture} ,
\end{aligned}
\end{equation}
where $f_y=\sum_{0\le c\le a}h_{c,y}\dotbubble_{c,y}\in \Xi$ by the inductive assumption, and the second summand on RHS \eqref{eq:ataat} comes from \eqref{sumdumbell} with $s=0$.

The first summand on RHS \eqref{eq:ataat}  belongs to $\Xi$ because of the proof of Case 1 and the inductive assumption. For the second summand, we observe that
\[
\begin{tikzpicture}[yscale=.8,anchorbase,scale=.8,color=\clr]
\draw[-,line width=1.2pt] (0.08,-.8) to (0.08,-.5);
\draw[-,line width=1.2pt] (0.08,.3) to (0.08,.6) to [out=up,in=up] (1,.6) to (1,-.8) to [out=down,in=down] (.08,-.8);
\draw[-,line width=1pt] (0.1,-.51) to [out=45,in=-45] (0.1,.31);
\draw[-,line width=1pt] (0.06,-.51) to [out=135,in=-135] (0.06,.31);
\draw(.08,.6) \bdot;
\draw(.25,0) \bdot;
\draw(.6,.0)node {$\scriptstyle \omega_{a}$};
\draw(.45,-.3)node {$\scriptstyle a$};
\node at (.45,-.35) {$\scriptstyle $};
\node at (.5,-1.3){$\scriptstyle a(t-1)$};
\end{tikzpicture}
~=~
\begin{tikzpicture}[yscale=.8,anchorbase,scale=.8,color=\clr]
\draw[-,line width=1.2pt] (0.08,-.8) to (0.08,-.5);
\draw[-,line width=1.2pt] (0.08,.3) to (0.08,.6) to [out=up,in=up] (1,.6) to (1,-.8) to [out=down,in=down] (.08,-.8);
\draw[-,line width=1pt] (0.1,-.51) to [out=45,in=-45] (0.1,.31);
\draw[-,line width=1pt] (0.06,-.51) to [out=135,in=-135] (0.06,.31);
\draw(-.1,0) \bdot;
\draw(.25,0) \bdot;
\draw(.6,.0)node {$\scriptstyle \omega_{a}$};
\draw(.25,-.3) \bdot;
\draw(.6,-.3)node {$\scriptstyle \omega_{a}$};
\node at (.5,-1.3){$\scriptstyle a(t-1)$};
\end{tikzpicture}
~=~
\begin{tikzpicture}[anchorbase,scale=.8,color=\clr]
\draw[-,line width=1pt] (-.1,-.51) to (-.1,.31);
\draw[-,line width=1pt] (.25,-.51) to (.25,.31);
\draw[-,line width=1pt] (.25,.31) to [out=up,in=up] (.75,.31) to (1,.1); 
\draw[-,line width=1pt] (-.1,.31) to [out=up,in=up] (1.25,.31) to (1,.1); 
\draw[-,line width=1pt] (.25,-.51) to [out=down,in=down] (.75,-.51) to (1,-.3); 
\draw[-,line width=1pt] (-.1,.-.51) to [out=down,in=down] (1.25,-.51) to (1,-.3);
\draw[-,line width=1.3pt] (1,-.3) to (1,.1);
\draw(-.1,0) \bdot;
\draw(.25,0) \bdot;
\draw(.7,0)node {$\scriptstyle 2\omega_{a}$};
\draw(.55,-.3)node {$\scriptstyle a$};
\node at (.5,-1.1){$\scriptstyle a(t-2)$};
\end{tikzpicture} \: .
\]
Hence we can repeat the process in \eqref{eq:ataat}. By the inductive assumption and \cref{lem:nested} we have 
\[
\binom{at}{a}\dotbubble_{at}\equiv \bubbledotat+\text{ elements in $\Xi$} 
\pmod {\End_{\AWb}(\varnothing)_{<at}}.
\]
Hence we have $\bubbledotat\in \Xi$.
\end{proof}

\begin{lemma}
    \label{stick-stick}
    We have
    $\begin{tikzpicture}
        [anchorbase,color=\clr]
        \draw[-,line width=1pt] (0,-.4) -- (0,.4);
        \draw (0,-.4)\bdot (0,.4)\bdot (0,0)\bdot;
        \node at (.3,0){$\scriptstyle \omega_\lambda$};
         \node at (-.2,-.2){$\scriptstyle a$};
    \end{tikzpicture}\in \Xi$, for any $a\in 2\N$ and $\lambda\in \EPar_a$.
\end{lemma}
\begin{proof}
    We have
    \begin{align*}
        \begin{tikzpicture}
        [anchorbase,color=\clr]
        \draw[-,line width=1pt] (0,-.4) -- (0,.4);
        \draw (0,-.4)\bdot (0,.4)\bdot (0,0)\bdot;
        \node at (.3,0){$\scriptstyle \omega_\lambda$};
        \node at (-.2,-.2){$\scriptstyle a$};
    \end{tikzpicture}
    ~\overset{\cref{bendstick}}{=}~
    \begin{tikzpicture}
        [anchorbase,color=\clr]
        \draw[-,line width=1pt] (0,-.2) to [out=up,in=left] (.2,.2) to [out=right,in=up] (.4,-.2);
        \draw (0,-.2)\bdot (.4,-.2)\bdot (.05,.1)\bdot;
        \node at (-.35,.1){$\scriptstyle \omega_\lambda$};
         \node at (.2,.4){$\scriptstyle a$};
    \end{tikzpicture}
    ~\equiv \sum_{\mu+\nu=\lambda}~
\begin{tikzpicture}[yscale=.8,anchorbase,scale=.8,color=\clr]
\draw[-,line width=1.2pt] (0.08,-.8) to (0.08,-.5);
\draw[-,line width=1.2pt] (0.08,.3) to (0.08,.6) to [out=up,in=up] (1,.6) to (1,-.8) to [out=down,in=down] (.08,-.8);
\draw[-,line width=1pt] (0.1,-.51) to [out=45,in=-45] (0.1,.31);
\draw[-,line width=1pt] (0.06,-.51) to [out=135,in=-135] (0.06,.31);
\draw(.25,0) \bdot;
\draw (-.08,.1)\bdot;
\node at (-.4,.1){$\scriptstyle \omega_\mu$};
\node at (.55,-.3){$\scriptstyle \frac{a}{2}$};
\draw (-.12,-.15) \bdot (-.08,-.35)\bdot;
\node at (-.4,-.45){$\scriptstyle \omega_\nu$};
\node at (.5,-1.3){$\scriptstyle a$};
\end{tikzpicture}
~\overset{\cref{wkdotsspan}}{\equiv}\sum_{\nu}a_\vartheta~
\begin{tikzpicture}
    [anchorbase,color=\clr,scale=1.2]
    \draw [-,line width=1pt] (0,0) to [out=up,in=left] (.2,.2) to [out=right,in=up] (.4,0) to [out=down,in=right] (.2,-.2) to [out=left,in=down] (0,0);
    \draw (0,0)\bdot;
    \node at (.6,0){$\scriptstyle a$};
     \node at (-.3,0){$\scriptstyle \omega_{\vartheta}$};
\end{tikzpicture}
\end{align*}
for $a_\vartheta\in\mathbb J'$. Thus the lemma follows from \cref{bubblealgebra}.
\end{proof}

\begin{lemma}
    \label{strand-stick-stick}
    Let $a,b\in\N$ and $c,d\in 2\N$ with $a+c=b+d$. Then, for any ${\lambda}\in\EPar_c$ and ${\mu}\in\EPar_d$, we have (modulo lower degree terms)
\[
\begin{tikzpicture}[anchorbase,scale=1, color=\clr]
	\draw[-,line width=1pt] (0,0) to (.29,.3) to (.29,.7) to (0,1);
	\draw[-,line width=1pt] (.6,0) to (.31,.3) to (.31,.7) to (.6,1);
    \draw (.6,1)\bdot (.6,0)\bdot (.45,.15)\bdot (.45,.85)\bdot;
    \node at (.7,.25){$\scriptstyle \omega_{\mu}$};
     \node at (.7,.75){$\scriptstyle \omega_{\lambda}$};
        \node at (0,1.2) {$\scriptstyle b$};
        \node at (0.63,1.2) {$\scriptstyle d$};
        \node at (0,-.2) {$\scriptstyle a$};
        \node at (0.63,-.2) {$\scriptstyle c$};
\end{tikzpicture}
~\in\text{$\mathbb J'$-span}~
\Bigg\{
\begin{tikzpicture}[anchorbase,scale=1, color=\clr]
       \draw[-,line width=1pt] (0,-.7) -- (0,.7)
                              (0,.5) -- (-.2,.2)
                               (0,-.5) -- (-.2,-.2);
      \draw (-.2,.2)\bdot (-.2,-.2)\bdot (-.1,-.35)\bdot (-.1,.35)\bdot
     (0,0)\bdot;
     \node at (-.4,.35){$\scriptstyle \omega_{\nu}$};
     \node at (-.4,-.35){$\scriptstyle \omega_{\gamma}$};
     \node at (.3,0){$\scriptstyle \omega_{\vartheta}$};
     \node at (.2,-.2){$\scriptstyle t$};
     \node at (0,-.9){$\scriptstyle a$};
      \node at (0,.9){$\scriptstyle b$};
\end{tikzpicture}~\colon~ a-t,b-t\in 2\N, \vartheta\in \Par_t, \nu\in \EPar_{b-t},\gamma\in \EPar_{a-t}
\Bigg\}.
\]
\end{lemma}
\begin{proof}
   We proceed by induction on $a+c=b+d$. The base case $a+c=1$ follows from \cref{stick-stick}. We have
    \begin{align*}
        \begin{tikzpicture}[anchorbase,scale=1, color=\clr]
	\draw[-,line width=1pt] (0,0) to (.29,.3) to (.29,.7) to (0,1);
	\draw[-,line width=1pt] (.6,0) to (.31,.3) to (.31,.7) to (.6,1);
    \draw (.6,1)\bdot (.6,0)\bdot (.45,.15)\bdot (.45,.85)\bdot;
    \node at (.7,.25){$\scriptstyle \omega_{\mu}$};
     \node at (.7,.75){$\scriptstyle \omega_{\lambda}$};
        \node at (0,1.2) {$\scriptstyle b$};
        \node at (0.63,1.2) {$\scriptstyle d$};
        \node at (0,-.2) {$\scriptstyle a$};
        \node at (0.63,-.2) {$\scriptstyle c$};
\end{tikzpicture}
~\overset{\cref{splitmerge}}{=}
    \sum_{{0\leq s\leq \min(a,b)\atop 0\leq t \leq \min(c,d)} \atop t-s=d-a}~
    \begin{tikzpicture}
        [anchorbase,color=\clr]
        \draw[-,line width=1pt] (0,0) -- (0,1)
                                (0,.2) -- (.5,.8)
                                (0,.8) -- (.5,.2) -- (.5,.8)
                                (.5,.8) -- (.7,.9) -- (.9,1)
                                (.5,.2) -- (.7,.1) -- (.9,0);
        \draw (.7,.9)\bdot (.9,1)\bdot (.7,.1)\bdot (.9,0)\bdot;
        \node at (-.2,.5){$\scriptstyle s$};
        \node at (.7,.5){$\scriptstyle t$};
        \node at (0,-.2){$\scriptstyle a$};
        \node at (0,1.2){$\scriptstyle b$};
        \node at (.6,1.1){$\scriptstyle \omega_{\lambda}$};
         \node at (.6,-.1){$\scriptstyle \omega_{\mu}$};
    \end{tikzpicture} \: .
    \end{align*}
The summand above can be rewritten as
\begin{align*}
      \begin{tikzpicture}
        [anchorbase,color=\clr]
        \draw[-,line width=1pt] (0,0) -- (0,1)
                                (0,.2) -- (.5,.8)
                                (0,.8) -- (.5,.2) -- (.5,.8)
                                (.5,.8) -- (.7,.9) -- (.9,1)
                                (.5,.2) -- (.7,.1) -- (.9,0);
        \draw (.7,.9)\bdot (.9,1)\bdot (.7,.1)\bdot (.9,0)\bdot;
        \node at (-.2,.5){$\scriptstyle s$};
        \node at (.7,.5){$\scriptstyle t$};
        \node at (0,-.2){$\scriptstyle a$};
        \node at (0,1.2){$\scriptstyle b$};
        \node at (.6,1.1){$\scriptstyle \omega_{\lambda}$};
         \node at (.6,-.1){$\scriptstyle \omega_{\mu}$};
    \end{tikzpicture}
    ~&\equiv~
    \begin{tikzpicture}
        [anchorbase,color=\clr]
        \draw[-,line width=1pt]  (0,-.5) -- (0,1)
                                 (0,-.3) -- (.2,-.1) -- (.4,-.3) -- (.6,-.1) to [out=45,in=-45] (0,.8)
                                 (.2,-.1) -- (.2,.3)
                                 (.4,-.3) -- (.4,-.7);
        \draw (.2,.1)\bdot (.2,.3)\bdot (.4,-.5)\bdot (.4,-.7)\bdot;
        \node at (-.2,.25){$\scriptstyle s$};
        \node at (0,-.7){$\scriptstyle a$};
        \node at (0,1.2){$\scriptstyle b$};
        \node at (.45,.1){$\scriptstyle \omega_{\lambda}$};
         \node at (.65,-.5){$\scriptstyle \omega_{\mu}$};
         \node at (.35,-.1){$\scriptstyle t$};
    \end{tikzpicture}
    ~\overset{\cref{lollipop-merge-split}}{\equiv}
    \text{ a $\mathbb J'$-span of }~
        \begin{tikzpicture}
        [anchorbase,color=\clr]
        \draw[-,line width=1pt]  (0,-.5) -- (0,1)
                                 (0,-.3) -- (.2,-.1) to [out=45,in=135] (.6,-.1) -- (.7,-.2)
                                 to [out=-45,in=-135] (1.1,-.2) -- (1.3,0) to [out=45,in=-45] (0,.8);
        \draw[-,line width=1pt]  (.2,-.1) -- (.1,.3)
                                 (.6,-.1) -- (.7,.3)
                                 (.7,-.2) -- (.6,-.6)
                                 (1.1,-.2) -- (1.2,-.6);
        \draw (.65,.1)\bdot (.7,.3)\bdot (.15,.1)\bdot (.1,.3)\bdot
              (.65,-.4)\bdot (.6,-.6)\bdot (1.15,-.4)\bdot (1.2,-.6)\bdot (.3,-.05)\bdot (.8,-.25)\bdot;
        \node at (0,-.7){$\scriptstyle a$};
        \node at (0,1.2){$\scriptstyle b$};
        \node at (-.2,.5){$\scriptstyle s$};
    \end{tikzpicture}
    =\text{ a $\mathbb J'$-span of }\begin{tikzpicture}
        [anchorbase,color=\clr]
         \draw[-,line width=1pt]  (0,-.5) -- (0,1)
                                 (0,-.3) -- (.5,.1) -- (.8,-.2)
                                 (.5,.1) -- (.5,.4) -- (0,.8)
                                 (.5,.4) -- (.8,.7);
        \draw (.65,.55)\bdot (.8,.7)\bdot (.65,-.05)\bdot (.8,-.2)\bdot
             (.2,.65)\bdot (.2,-.15)\bdot (-.2,.65)\bdot (-.4,.5)\bdot
             (-.2,-.15)\bdot (-.4,0)\bdot;
        \draw[-,line width=1pt] (0,.8) -- (-.2,.65) -- (-.4,.5)
                                (0,-.3) -- (-.2,-.15) -- (-.4,0);
        \node at (0,-.7){$\scriptstyle a$};
        \node at (0,1.2){$\scriptstyle b$};
        \node at (-.2,.25){$\scriptstyle s$};
        \node at (.3,.25){$\scriptstyle t$};
        \node at (-.25,.9){$\scriptstyle \omega_{\nu_1}$};
        \node at (-.25,-.4){$\scriptstyle \omega_{\nu_2}$};
        \node at (.25,.85){$\scriptstyle \omega_{\nu_3}$};
        \node at (.25,-.4){$\scriptstyle \omega_{\nu_4}$};
         \node at (.85,.35){$\scriptstyle \omega_{\nu_5}$};
        \node at (.8,.15){$\scriptstyle \omega_{\nu_6}$};
        \node at (.65,.85){$\scriptstyle x$};
        \node at (.65,-.35){$\scriptstyle y$};
    \end{tikzpicture}
\end{align*}
where $x,y\in 2\N$, $\nu_1\in \EPar_{b-t-s+x}, \nu_2\in \EPar_{a-s-t+y}, \nu_3\in \Par_{t-x},\nu_4\in \Par_{t-y},\nu_5\in \EPar_x,\nu_6\in \EPar_y$. Then the lemma follows from the inductive assumption, \cref{wkdotsspan} and \cref{lollipopspan}.
\end{proof}

\section{A spanning set}
\label{sec:spanning}

In this section, we establish a diagrammatic spanning set consisting of bubbled elementary diagrams for each morphism space of $\AWb$ by establishing a reduction procedure of arbitrary diagrams to these distinguished diagrams.

\subsection{Elementary diagrams of $\AWb$}

Suppose $\lambda\in\Lambda_{\text{st}}(m)$ and $\mu\in \Lambda_{\text{st}}(n)$. By \cref{def:affineBWeb}, we have
\[
\Hom_{\AWb}(\lambda,\mu)=0, \qquad \text{ unless } m\equiv n\mod 2.
\]
Suppose that $m\equiv n\mod 2$ below. A (dotted) B-diagram is a (dotted) string diagram formed by tensor products and compositions of the generating diagrams \eqref{dotgenerator}. 

\begin{definition} \label{OwebCFD}
   A $\lambda\times \mu$  chicken foot B-diagram is an undotted B-diagram in $\Hom_{\AWb}(\mu,\la)$ consisting of three horizontal parts such that
   \begin{itemize} 
     \item the bottom part consists of splits, caps and possibly crossings involving some cap,
    \item the top part consists of merges, cups and possibly crossings involving some cup,
    \item the middle part consists of only crossings of the thinner strands (i.e., legs) which are not cups or caps,
    \end{itemize}
\end{definition}
A $\lambda\times \mu$ chicken foot B-diagram is said to be {\em reduced} if there is at most one intersection or join between every pair of legs.

\begin{example}
Let $\lambda=(1,2,4,1)$ and $\mu=(4,2)$. The following are three $\lambda\times \mu$ chicken foot B-diagrams, where the first one is reduced while the second and the third one are not reduced:
\begin{align}
    \label{OwebCFDexample}
     \begin{tikzpicture}[anchorbase,scale=.6,color=\clr]
     \draw[-,thick] (0,0) to (0.5,1);
     \draw[-,thick] (1.5,1) to (.5,0);
    \draw[-,line width=1.2pt] (1.5,0) to (.5,1);
    \draw[-,line width=1.5pt] (.5,1) to (.5,1.8);
    \draw[-,line width=1.2pt] (1.5,1.5) to (1.5,1.8);
        \draw[-,thick] (1.5,1.6) to (1.5,1);
    \draw[-,thick] (0,-.8) to (0,-.6);
    \draw[-,thick] (0,-.6) to (0,0);
    \draw[-,thick] (.5,-.6) to (.5,0);
    \draw[-,line width=1.2pt] (.5,-.8) to (.5,-.5);
    \draw[-,line width=1.5pt] (1.5,-.8) to (1.5,-.5);
    \draw[-,line width=1.2pt] (1.5,-.6) to (1.5,0);
    \draw[-,thick] (2.5,-.8) to (2.5,-.6);
    \draw[-,thick] (1.5,1.6) to [out=down,in=down] (.5,1.6);
    \draw[-,thick] (1.5,-.6) to [out=up,in=up] (.5,-.6);
    \draw[-,thick] (1.5,-.6) to [out=up,in=up] (2.5,-.6);
    \node at (1,1.5){$\scriptstyle{1}$};
    \node at (1,-.65){$\scriptstyle{1}$};
    \node at (2,-.65){$\scriptstyle{1}$};
\end{tikzpicture},\qquad\qquad
\begin{tikzpicture}[anchorbase,scale=.6,color=\clr]
     \draw[-,thick] (0,0) to (0.5,1);
     \draw[-,thick] (1.5,1) to (.5,0);
     \draw[-,thick] (1.5,0) to (.5,1);
    \draw[-,thick] (1.5,0) to [out=170,in=down] (.5,1);
    \draw[-,line width=1.4pt] (.5,1) to (.5,1.8);
    \draw[-,line width=1.2pt] (1.5,1.5) to (1.5,1.8);
        \draw[-,thick] (1.5,1.6) to (1.5,1);
    \draw[-,thick] (0,-.8) to (0,0);
    \draw[-,thick] (.5,-.6) to (.5,0);
    \draw[-,line width=1.2pt] (.5,-.8) to (.5,-.5);
    \draw[-,line width=1.4pt] (1.5,-.8) to (1.5,-.5);
    \draw[-,line width=1.2pt] (1.5,-.6) to (1.5,0);
    \draw[-,thick] (2.5,-.8) to (2.5,-.6);
    \draw[-,thick] (1.5,1.6) to [out=down,in=down] (.5,1.6);
    \draw[-,thick] (1.5,-.6) to [out=up,in=up] (.5,-.6);
    \draw[-,thick] (1.5,-.6) to [out=up,in=up] (2.5,-.6);
    \node at (1,1.5){$\scriptstyle{1}$};
    \node at (1,-.65){$\scriptstyle{1}$};
    \node at (2,-.65){$\scriptstyle{1}$};
\end{tikzpicture},\qquad\qquad
\begin{tikzpicture}[anchorbase,scale=.6,color=\clr]
     \draw[-,thick] (0,0) to (0.5,1);
     \draw[-,thick] (1.5,1) to (.5,0);
    \draw[-,line width=1.2pt] (1.5,0) to (.5,1);
    \draw[-,line width=1.4pt] (.5,1) to (.5,1.8);
    \draw[-,line width=1.2pt] (1.5,1.8) to (1.5,1);
    \draw[-,thick] (0,-.8) to (0,0);
    \draw[-,thick] (.5,-.6) to (.5,0);
    \draw[-,line width=1.2pt] (.5,-.8) to (.5,-.5);
    \draw[-,line width=1.4pt] (1.5,-.8) to (1.5,-.5);
    \draw[-,line width=1.2pt] (1.5,-.6) to (1.5,0);
    \draw[-,thick] (2.5,-.8) to (2.5,-.6);
    \draw[-,thick] (1.5,1) to (1.5,0.9) to [out=down,in=down] (.5,.9) to (.5,1);
    \draw[-,thick] (1.5,-.6) to [out=up,in=up] (.5,-.6);
    \draw[-,thick] (1.5,-.6) to [out=up,in=up] (2.5,-.6);
    \node at (1.8,.8){$\scriptstyle{1}$};
    \node at (1,-.65){$\scriptstyle{1}$};
    \node at (2,-.65){$\scriptstyle{1}$};
\end{tikzpicture}.
\end{align}
\end{example}

For any partition $\nu\in \Par_a$, the {\em elementary dot packet (of thickness $a$ and shape $\nu$)} is defined to be 
\[\omega_{a,\nu}:=\wkdotavar{\nu}.\]
We write $\omega_\nu=\omega_{a,\nu}$ if $a$ is clear from the context and draw it as 
$
\begin{tikzpicture}[baseline = 3pt, scale=0.5, color=\clr]
\draw[-,line width=1.6pt] (0,-.2) to[out=up, in=down] (0,1.2);
\draw(0,0.5) \bdot; \node at (0.6,0.5) {$ \scriptstyle \nu$};
\node at (0,-.4) {$\scriptstyle a$};
\end{tikzpicture}$. To enhance the readability of our dotted B-diagrams, we may omit the subscript \( \nu \) and instead represent it as  
\(
\begin{tikzpicture}[baseline = 3pt, scale=0.5, color=\clr]
\draw[-,line width=1.6pt] (0,-.2) to[out=up, in=down] (0,1.2);
\draw(0,0.5) \squ; 
\node at (0,-.4) {$\scriptstyle a$};
\end{tikzpicture}
\)
when the information of \( \nu \) does not need to be explicitly tracked. When \( \nu \) is an even partition at the leg of a lollipop, we distinguish it by drawing  
\(
\begin{tikzpicture}[baseline = 3pt, scale=0.5, color=\clr]
\draw[-,line width=1.6pt] (0,-.2) to[out=up, in=down] (0,1.2);
\draw(0,0.5) \dia; 
\node at (0,-.4) {$\scriptstyle a$};
\end{tikzpicture}.
\)

\begin{definition} \label{elementaryD}
    A $\lambda\times \mu$ elementary B-diagram is a $\lambda\times \mu$ reduced chicken foot B-diagram with 
    \begin{enumerate}
        \item an elementary dot packet $\omega_{\nu}$ (for some $\nu\in \Par_a$) attached to each leg of thickness $a$ in the following way:
    \begin{itemize}
        \item If the leg is neither a cup nor a cap, then the elementary dot packet is attached at the bottom of that leg,
        \item If the leg is a cup or a cap, then the elementary dot packet is attached at the leftmost boundary of that leg;
    \end{itemize}
        \item elementary lollipops $\wkdotsnu$ (resp. $\wkdotsnurev$), for $\nu\in \EPar_{2s}$,  attached to thick legs on the bottom (resp. top);
        \item crossing-free and nest-free bubbles $\bubbledota$, for $a\in 2\N$, arranged to the left of all other strands.
    \end{enumerate}
\end{definition}

 \begin{example} 
Let $\lambda=(15,12),\mu=(11,4,14)$. The following is a $\lambda\times \mu$ elementary diagram in $\Hom_{\AWb}(\mu,\la)$:
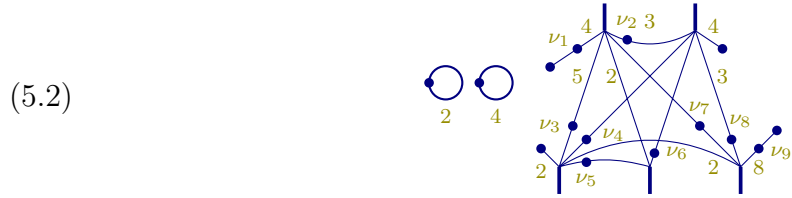
\begin{equation}
\label{exelementaryD}
    \begin{tikzpicture}
        [anchorbase,scale=1.1,color=\clr]
        \draw[thick] (0,0) to[out=down,in=left] (.2,-.2) to[out=right,in=down] (.4,0)
        to[out=up,in=right] (.2,.2) to[out=left,in=up] (0,0);
        \draw (0,0) \bdot;
        \node at (.2,-.4){$\scriptstyle 2$};
    \end{tikzpicture}\ 
     \begin{tikzpicture}
        [anchorbase,scale=1.1,color=\clr]
        \draw[thick] (0,0) to[out=down,in=left] (.2,-.2) to[out=right,in=down] (.4,0)
        to[out=up,in=right] (.2,.2) to[out=left,in=up] (0,0);
        \draw (0,0) \bdot;
        \node at (.2,-.4){$\scriptstyle 4$};
    \end{tikzpicture}\ 
    \begin{tikzpicture}
        [anchorbase,scale=1.2, color=\clr]
        \draw[line width=1.4pt] (0,.5) -- (0,.8)
                              (1,.5) -- (1,.8)
                              (-.5,-1) -- (-.5,-1.3)
                              (.5,.-1) -- (.5,-1.3)
                              (1.5,-1) -- (1.5,-1.3);  
        \draw (0,.5) to[out=-30,in=-150] (1,.5)
               (-.5,-1) to[out=30,in=150] (1.5,-1)
               (-.5,-1) to[out=15,in=165] (.5,-1);
        \draw (0,.5) -- (.5,-1) -- (1,.5)
              (0,.5) -- (-.5,-1)
              (1,.5) -- (1.5,-1)
              (0,.5) -- (1.5,-1)
              (1,.5) -- (-.5,-1);
        \draw (0,.5) -- (-.3,.3) -- (-.6,.1) 
                (-.3,.3)\bdot (-.6,.1)\bdot
              (1,.5) -- (1.3,.3) (1.3,.3)\bdot
              (-.5,-1) -- (-.7,-.8) (-.7,-.8)\bdot
              (1.5,-1) -- (1.7,-.8) -- (1.9,-.6) (1.7,-.8)\bdot (1.9,-.6)\bdot;
        \node at (-.2,.55){$\scriptstyle 4$};
         \node at (1.2,.55){$\scriptstyle 4$};
        \node at (-.7,-1.05){$\scriptstyle 2$};
        \node at (1.7,-1){$\scriptstyle 8$};
        \draw (-.35,-.55) \bdot
              (-.2,-.7) \bdot
              (-.2,-.95)\bdot
              (.55,-.85)\bdot
              (1.4,-.7)\bdot
              (1.05,-.55)\bdot
              (.25,.4)\bdot;
         \node at (-.5,.4){$\scriptstyle \nu_1$};
        \node at (.25,.6){$\scriptstyle \nu_2$};
        \node at (-.6,-.55){$\scriptstyle \nu_3$};
        \node at (0.1,-.65){$\scriptstyle \nu_4$};
        \node at (-.2,-1.15){$\scriptstyle \nu_5$};
        \node at (.8,-.85){$\scriptstyle \nu_6$};
        \node at (1.05,-.35){$\scriptstyle \nu_7$};
        \node at (1.5,-.5){$\scriptstyle \nu_8$};
         \node at (1.95,-.85){$\scriptstyle \nu_9$};
        \node at (-.3,0){$\scriptstyle 5$};
        \node at (.075,0){$\scriptstyle 2$};
        \node at (1.3,0){$\scriptstyle 3$};
        \node at (.5,.6){$\scriptstyle 3$};
        \node at (1.2,-1){$\scriptstyle 2$};
    \end{tikzpicture}
\end{equation}
\end{example}
If the information of $\nu_1,\ldots,\nu_9$ are not explicitly tracked, we may draw \cref{exelementaryD} as 
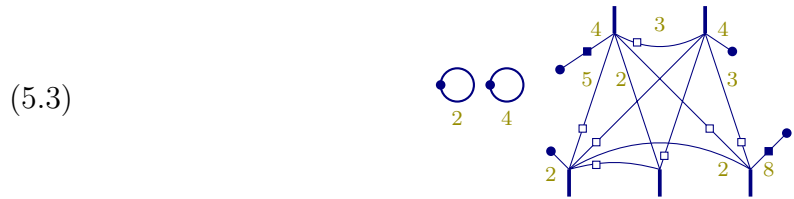
\begin{equation}
    \label{simplifiedD}
        \begin{tikzpicture}
        [anchorbase,scale=1.1,color=\clr]
        \draw[thick] (0,0) to[out=down,in=left] (.2,-.2) to[out=right,in=down] (.4,0)
        to[out=up,in=right] (.2,.2) to[out=left,in=up] (0,0);
        \draw (0,0) \bdot;
        \node at (.2,-.4){$\scriptstyle 2$};
    \end{tikzpicture}\ 
     \begin{tikzpicture}
        [anchorbase,scale=1.1,color=\clr]
        \draw[thick] (0,0) to[out=down,in=left] (.2,-.2) to[out=right,in=down] (.4,0)
        to[out=up,in=right] (.2,.2) to[out=left,in=up] (0,0);
        \draw (0,0) \bdot;
        \node at (.2,-.4){$\scriptstyle 4$};
    \end{tikzpicture}\ 
    \begin{tikzpicture}
        [anchorbase,scale=1.2, color=\clr]
        \draw[line width=1.4pt] (0,.5) -- (0,.8)
                              (1,.5) -- (1,.8)
                              (-.5,-1) -- (-.5,-1.3)
                              (.5,.-1) -- (.5,-1.3)
                              (1.5,-1) -- (1.5,-1.3);  
        \draw (0,.5) to[out=-30,in=-150] (1,.5)
               (-.5,-1) to[out=30,in=150] (1.5,-1)
               (-.5,-1) to[out=15,in=165] (.5,-1);
        \draw (0,.5) -- (.5,-1) -- (1,.5)
              (0,.5) -- (-.5,-1)
              (1,.5) -- (1.5,-1)
              (0,.5) -- (1.5,-1)
              (1,.5) -- (-.5,-1);
        \draw (0,.5) -- (-.3,.3) -- (-.6,.1) 
                (-.3,.3)\dia (-.6,.1)\bdot
              (1,.5) -- (1.3,.3) (1.3,.3)\bdot
              (-.5,-1) -- (-.7,-.8) (-.7,-.8)\bdot
              (1.5,-1) -- (1.7,-.8) -- (1.9,-.6) (1.7,-.8)\dia (1.9,-.6)\bdot;
        \node at (-.2,.55){$\scriptstyle 4$};
         \node at (1.2,.55){$\scriptstyle 4$};
        \node at (-.7,-1.05){$\scriptstyle 2$};
        \node at (1.7,-1){$\scriptstyle 8$};
        \draw (-.35,-.55) \squ
              (-.2,-.7) \squ
              (-.2,-.95)\squ
              (.55,-.85)\squ
              (1.4,-.7)\squ
              (1.05,-.55)\squ
              (.25,.4)\squ;
        \node at (-.3,0){$\scriptstyle 5$};
        \node at (.075,0){$\scriptstyle 2$};
        \node at (1.3,0){$\scriptstyle 3$};
        \node at (.5,.6){$\scriptstyle 3$};
        \node at (1.2,-1){$\scriptstyle 2$};
    \end{tikzpicture}
\end{equation}
The complete data of a bubble-free $\lambda\times \mu$ elementary diagram can be recorded by a pair of matrices $(A,P)$ as follows. Suppose $\lambda=(\lambda_1,\ldots,\lambda_s)$ and $\mu=(\mu_1,\ldots,\mu_t)$. Let 
$$
k=l(\lambda)+l(\mu)=s+t,
\qquad
\lambda\cup\mu:=(\lambda_1,\ldots,\lambda_s,\mu_1,\ldots,\mu_t).
$$ 
Then $A=(a_{ij})_{1\leq i,j\leq k}$ is the matrix where $2^{\delta_{ij}}a_{ij}$ records the thickness of the leg connecting $(\lambda\cup\mu)_i$ and $(\lambda\cup\mu)_j$ while $P=(p_{ij})_{1\leq i,j\leq k}$ is the matrix such that $p_{ij}$ records the dot packet on the corresponding leg. In particular, $2a_{ii}\in 2\N$ counts the thickness of lollipop at $(\lambda\cup\mu)_i$.

Note that $A$ and $P$ are symmetric by definition.
We say the bubble-free $\lambda\times \mu$ elementary diagram is of type $(A,P)$. Note that by the relations \eqref{zigzag} and \eqref{sliders}--\eqref{braid} we see that elementary diagrams of the same type are equal as morphisms in $\AWb$.
\begin{example}
    Let $\lambda=(15,12),\mu=(11,4,14)$. The elementary diagram \eqref{exelementaryD} with bubbles removed is recorded by
    \[
    A=\begin{pmatrix}
        2& 3& 5&2 & 1\\
        3& 2& 1& 1& 3\\
        5& 1&1 & 1& 2\\
        2& 1& 1& 0& 0\\
        1& 3& 2& 0& 4
    \end{pmatrix},\qquad
    P=\begin{pmatrix}
        \nu_1 & \nu_2& \nu_3& (\varnothing) & \nu_7\\
        \nu_2& (\varnothing)& \nu_4& \nu_6& \nu_8\\
        \nu_3& \nu_4&(\varnothing) &\nu_5& (\varnothing)\\
        (\varnothing)& \nu_6& \nu_5& (\varnothing)& (\varnothing)\\
        \nu_7& \nu_8& (\varnothing)& (\varnothing)& \nu_9
    \end{pmatrix} .
    \]
\end{example}

Denote by $\SMat_{k\times k}(\N)$ the set of $k\times k$ symmetric matrices with non-negative integer entries. Denote
\begin{equation}\label{eq:Mat}
 \SMat_{\lambda,\mu}:=\left\{ A=(a_{ij}) \in \text{SMat}_{k\times k}(\N)~\Big\vert~ {\sum_{1\le h\le k}2^{\delta_{ih}}a_{ih}=(\la\cup\mu)_i, 
 \atop 
 \sum_{1\le h\le k}2^{\delta_{jh}}a_{jh}=(\la\cup\mu)_j, \forall i,j}
 \right\}.   
\end{equation}

For each elementary diagram from $\mu$ to $\lambda$, we encode its information with two symmetric matrices $A=(a_{ij})\in \SMat_{\lambda,\mu}$ and $P=(\vartheta_{ij}), \vartheta_{ij}\in \Par_{2^{\delta_{ij}}a_{ij}}$. Denote  
\begin{equation}
 \label{eq:MSMat}
 \PSMat_{\lambda,\mu}:=\left\{ (A, P)~\Big\vert~ { A=(a_{ij})\in \SMat_{\lambda,\mu}, P=(\vartheta_{ij})\in\SMat_{k\times k}, \atop \vartheta_{ij}\in \Par_{a_{ij}}\text{ for } i\neq j \text{ and } \vartheta_{ii}\in \EPar_{2a_{ii}} }\right\}.   
\end{equation}

We will identify the set of all elementary diagrams from $\mu$ to $\lambda$ with $\PSMat_{\lambda,\mu}$.

\subsection{A spanning set of $\AWb$}

Again we fix $\lambda=(\lambda_1,\ldots,\lambda_s)$, $\mu=(\mu_1,\ldots,\mu_t)$ and $k=s+t$ such that $\lambda \vdash m,\mu\vdash n$ and $m\equiv n \mod 2$. We shall construct a spanning set for each morphism space $\Hom_{\AWb}(\mu,\lambda)$. 

By definition $\Hom_{\AWb}(\mu,\lambda)$ is spanned by all dotted B-web diagrams from $\mu$ to $\lambda$. These dotted web diagrams may contain bubbles with dots.  If such a bubble does not intersect with any other strands, we say it is a {\em floating} bubble. It follows from \cref{bubblealgebra} that the algebra generated by all floating bubbles is just $\Xi$. Recall $\dotbubble_a$ from \eqref{bubble+dot+par}. Denote
\begin{align} \label{eq:BPMat}
\bPSMat_{\lambda,\mu} &:=\{\dotbubble_2^{k_2}\dotbubble_4^{k_4}\cdots f \mid f\in \PSMat_{\lambda,\mu}, k_{2i}\in \N,\forall i\geq 1\text{ and }\sum_i k_{2i}<\infty\}.
\end{align}
Note that $\Xi$ acts naturally on each morphism space $\Hom_{\AWb}(\mu,\lambda)$ by horizontal concatenation from the left.

\begin{theorem}
\label{thm:spanAWb}
    Let $\lambda, \mu\in \Lambda_{\text{st}}$. Then the morphism space $\Hom_{\AWb}(\mu,\lambda)$ is spanned by 
    $\bPSMat_{\lambda,\mu}.$ Equivalently, 
    $\Xi=\Hom_{\AWb}(\unit,\unit)$, and as a $\Xi$-module, $\Hom_{\AWb}(\mu,\lambda)$ is spanned by $\PSMat_{\lambda,\mu}$.
\end{theorem}

\subsection{Proof of \cref{thm:spanAWb} }

 To prove the theorem, it suffices to show that $fg$ can be written as a linear combination of elementary diagrams, for an arbitrary elementary diagram $g$ (up to lower degree terms) and a specific generating morphism $f$ of the following five types, where $1_*$ denotes suitable identity morphisms (cf. \cite{SWweb,SWSchur}); see Table \ref{tab1}. (The crossing generators can be expressed in terms of merges and splits by \cite[(4.36)]{BEEO}.)
 
\begin{table}[h]  
	\begin{center}
		\centering
	\begin{tabular}{|m{1.6cm}<{\centering}|m{2.2cm}<{\centering}|m{2.2cm}<{\centering}|m{2.2cm}<{\centering}
    |m{2.2cm}<{\centering}|}
	\hline
	(O) & (I) & (II) & (III) & (IV) \\
	\hline
	$\cup$ & \rot{Y} & Y & $\cap$ & \wkdotaa
		\\
	\hline
	  	 
   $1_*\cupa 1_*$ & $1_*\merge 1_*$ & $1_* \splits 1_*$ & $1_*\capa 1_*$ & $1_* \wkdotaa 1_*$
\\\hline
\end{tabular} 
\end{center}
	\vspace{0.5cm}
\caption{Five types of generating morphisms}
	\label{tab1}
\end{table}


The strategy is to proceed by induction on the degree \( k \) for \( \Hom_{\AWb}(\mu,\lambda)_{\leq k} \).
  The core of the argument involves demonstrating that for any generating morphism $f$ of the five types in Table~\ref{tab1} and an arbitrary elementary diagram $g$, the composition $fg$ can be rewritten as a linear combination of elementary diagrams, working modulo terms of lower degrees. This reduction is achieved by systematically applying the defining and implied relations of the category to untangle the resulting diagrams into their reduced forms.

The base case \( k = 0 \) follows from similar arguments as those used for types (O)--(IV) below. We divide the proof according to the types of $f$. 

The verification of type $\cup$ is straightforward as the composition $fg$ must be reduced. This takes care of Case (O). 

Suppose that $f$ is of type $\rot{Y}$ and we can join $f$ to the $i$-th and $(i+1)$-th strand at the top of $g$. Then $g$ has an $s$-fold merge $\rot{Y}_i$ and $t$-fold merge $\rot{Y}_{i+1}$ on its $i$-th and $(i+1)$-st vertices, respectively. If a cup in the top part of $g$ is a shared leg of the $u$-th and the $v$-th strands such that $\{u,v\}\neq \{i,i+1\}$, then we can move this cup upwards so that it only intersects with the head of $f$. All the caps in the bottom part of $g$ can be moved downward via \eqref{sliders} so that they intersect with the head of the splits only. Hence these cups and caps do not interfere with our further operations.

For example, suppose 
$f=1_*
\begin{tikzpicture}[anchorbase,color=\clr]
	\draw[-,line width=1pt] (0.28,-.3) to (0.08,0.04);
	\draw[-,line width=1pt] (-0.12,-.3) to (0.08,0.04);
	\draw[-,line width=1.4pt] (0.08,.3) to (0.08,0);
        \node at (-0.22,-.4) {$\scriptstyle 9$};
        \node at (0.4,-.45) {$\scriptstyle 10$};
\end{tikzpicture}1_*$ 
    and 
\begin{equation}
            \label{exofg}
            g=    \begin{tikzpicture}
        [anchorbase,scale=.8, color=\clr]
       \draw [-,line width=1.4pt] (0,0) -- (0,-.4)
                                  (1,0) -- (1,-.4)
                                  (2,0) -- (2,-.4)
                                  (-1,0) -- (-1,-.4)
                                  (0,2) -- (0,2.4)
                                  (2,2) -- (2,2.4)
                                  (3,2) -- (3,2.4);
        \draw[-]   (0,0) -- (0,2)
                   (0,0) -- (-.3,.3)
                   (0,0) to [out=30,in=150] (2,0)
                   (-1,0) to [out=60,in=120] (1,0)
                   (1,0) -- (0,2)
                   (1,0) -- (2,2)
                   (2,0) -- (2,2)
                   (2,0) -- (2.3,.3)
                   (0,2) -- (-.3,1.7)
                   (0,2) to [out=-30,in=-150] (2,2)
                   (0,2) to [out=-60,in=-120] (3,2)
                   (2,2) -- (2.3,1.7);
        \draw (-.3,1.7) \bdot (-.3,.3) \bdot (2.3,.3)\bdot (2.3,1.7)\bdot;
        \draw (-.15,1.85)\dia (-.15,.15)\dia 
        (2.15,.15)\dia (2.15,1.85)\dia;
        \draw (.2,.11) \squ (.3,1.85)\squ (-.85,.2)\squ (1.08,.16)\squ (0,.2)\squ
        (.3,1.65)\squ (2,.3)\squ;
        \node at (1,2){$\scriptstyle 4$};
        \node at (-.4,2){$\scriptstyle 2$};
        \node at (2.4,2){$\scriptstyle 4$};
        \node at (-.4,0){$\scriptstyle 4$};
        \node at (2.4,0){$\scriptstyle 6$};
    \end{tikzpicture}
    \quad\Rightarrow\quad
    fg\equiv
    \begin{tikzpicture}
        [anchorbase,scale=.8, color=\clr]
       \draw [-,line width=1.4pt] (0,0) -- (0,-.4)
                                  (1,.5) -- (1,-.4)
                                  (2,0) -- (2,-.4)
                                  (-1,-.3) -- (-1,-.4)
                                  (0,2) -- (0,2.2) -- (1,2.6)-- (1,3)
                                  (2,2) -- (2,2.2) -- (1,2.6)
                                  (3,2.8) -- (3,3);
        \draw[-]   (0,0) -- (0,2)
                   (0,0) -- (-.3,.3)
                   (0,0) to [out=30,in=150] (2,0)
                   (-1,-.3) to [out=15,in=165] (1,-.3)
                   (1,.5) -- (0,2)
                   (1,.5) -- (2,2)
                   (2,0) -- (2,2)
                   (2,0) -- (2.3,.3)
                   (0,2) -- (-.3,1.7)
                   (0,2) to [out=-30,in=-150] (2,2)
                   (1,2.8) to [out=-30,in=-150] (3,2.8)
                   (2,2) -- (2.3,1.7);
        \draw (-.3,1.7) \bdot (-.3,.3) \bdot (2.3,.3)\bdot (2.3,1.7)\bdot;
        \draw (-.15,1.85)\dia (-.15,.15)\dia 
        (2.15,.15)\dia (2.15,1.85)\dia;
        \draw (.2,.11) \squ (.3,1.85)\squ (-.85,-.25)\squ (1.2,.7)\squ (0,.2)\squ
        (1.3,2.65)\squ (2,.3)\squ;
        \node at (1,2){$\scriptstyle 4$};
        \node at (-.4,2){$\scriptstyle 2$};
        \node at (2.4,2){$\scriptstyle 4$};
        \node at (-.4,0){$\scriptstyle 4$};
        \node at (2.4,0){$\scriptstyle 6$};
    \end{tikzpicture}
\end{equation}

After moving all those cups and caps, the remaining cups can only share legs between $\rot{Y}_i$ and $\rot{Y}_{i+1}$. Then by \cref{lollipopspan} and the symmetry \cref{anti-auto} we can write each lollipop of the shape $\begin{tikzpicture}[anchorbase,scale=.8,color=\clr]
\draw[-,line width=1.4pt] (0.08,.3) to (0.08,.6);
\draw[-,thick] (0.1,-.51) to [out=right,in=-45] (0.1,.31);
\draw[-,thick] (0.1,-.51) to [out=left,in=-135] (0.1,.31);
\node at (0.08,.8) {$\scriptstyle 2a$};
\draw (-.1,-.2)\bdot;
\node at (-.4,-.2) {$\scriptstyle \omega_\lambda$};
\end{tikzpicture}$ as a $\mathbb J'$-linear combination of lollipops of the form $\wkdotamurev$ where $\mu$ is an even partition. After that, we may move all the lollipops of the form $\wkdotamurev$ attached to $\rot{Y}_i$ and $\rot{Y}_{i+1}$ upwards so that they become legs of the head of $f$. At last, we apply \cref{lollipopspan} to the lollipops attached to the head of $f$. 

For example, continuing from the right hand side of \cref{exofg}, we have
\begin{equation}
    \label{exmergelollipop}
    fg\equiv 
     \begin{tikzpicture}
        [anchorbase,scale=.8, color=\clr]
       \draw [-,line width=1.4pt] (0,0) -- (0,-.4)
                                  (1,.5) -- (1,-.4)
                                  (2,0) -- (2,-.4)
                                  (-1,-.3) -- (-1,-.4)
                                  (0,2) -- (0,2.2) -- (1,2.6)-- (1,3.2)
                                  (2,2) -- (2,2.2) -- (1,2.6)
                                  (3,2.8) -- (3,3.2);
        \draw[-]   (0,0) -- (0,2)
                   (0,0) -- (-.3,.3)
                   (0,0) to [out=30,in=150] (2,0)
                   (-1,-.3) to [out=15,in=165] (1,-.3)
                   (1,.5) -- (0,2)
                   (1,.5) -- (2,2)
                   (2,0) -- (2,2)
                 (2,0) -- (2.3,.3)
                   (0,2) to [out=-30,in=-150] (2,2)
                   (1,2.8) to [out=-30,in=-150] (3,2.8)
                   (1,3) -- (.7,2.7)
                   (1,3) -- (1.5,3);
        \draw (-.3,.3) \bdot (.7,2.7) \bdot (2.3,.3)\bdot (1.5,3)\bdot;
        \draw (.85,2.85)\dia (-.15,.15)\dia 
        (2.15,.15)\dia (1.25,3)\dia;
        \draw (.2,.11) \squ (.3,1.85)\squ (-.85,-.25)\squ (1.2,.7)\squ (0,.2)\squ
        (1.3,2.65)\squ (2,.3)\squ;
        \node at (1,2){$\scriptstyle 4$};
        \node at (.6,3){$\scriptstyle 2$};
        \node at (1.25,3.3){$\scriptstyle 4$};
        \node at (-.4,0){$\scriptstyle 4$};
        \node at (2.4,0){$\scriptstyle 6$};
    \end{tikzpicture}
    \equiv
    \text{ a $\mathbb J'$-span of }
    \begin{tikzpicture}
         [anchorbase,scale=.8, color=\clr]
          \draw [-,line width=1.4pt] (0,0) -- (0,-.5)
                                  (1,.5) -- (1,-.5)
                                  (2,0) -- (2,-.5)
                                  (-1,-.3) -- (-1,-.5)
                                  (1,1.3) -- (1,1.9)
                                  (2,1.5) -- (2,1.9);
            \draw[-]   (0,0) -- (1,1.5)
                   (0,0) -- (-.3,.3)
                   (0,0) to [out=30,in=150] (2,0)
                   (-1,-.3) to [out=15,in=165] (1,-.3)
                   (1,.5) to [out=70,in=-70] (1,1.3)
                   (1,.5) to [out=110,in=-110] (1,1.3)
                   (2,0) -- (1,1.5)
                 (2,0) -- (2.3,.5)
                   (1,1.5) to [out=-30,in=-150] (2,1.5)
                        (1,1.5) -- (.5,1.2);
        \draw (-.3,.3) \bdot (.5,1.2) \bdot (2.3,.5)\bdot;
        \draw (.75,1.35)\dia (-.15,.15)\dia 
        (2.15,.25)\dia;
        \draw (.3,.11) \squ (-.85,-.25)\squ (1.1,.8)\squ (.3,.5)\squ (1.3,1.35)\squ;
        \node at (-.4,0){$\scriptstyle 4$};
        \node at (2.4,0){$\scriptstyle 6$};
         \node at (.3,1.5){$\scriptstyle 14$};
    \end{tikzpicture} \: .
\end{equation}

Note that the resulting diagrams in this case might still not be reduced since there may exist two legs joining twice in the bottom part. For each pair of strands which join twice (i.e., they join both on the top and the bottom), we use relations \eqref{webassoc} and \eqref{swallows}-\eqref{braid} to pull one of the strands freely such that the non-reduced pair of strands becomes disjoint lanterns. Moreover, using \eqref{swallows} and \eqref{sliders} we can move the lanterns such that no other strand passes through the thin strands of the lantern. This allows us to use \cref{wkdotsspan} to write each lantern as a $\mathbb J'$-linear combination of $\wkdotawmu$. We then move back those cups and caps to appropriate positions afterwards via \eqref{sliders} and \cref{dotmovefreely} such that the resulting diagrams are elementary. This concludes the proof of Case (I).

Suppose $f$ is of type Y and we can join $f$ to the $i$-th strand at the top of $g$. Then $g$ has a $h$-fold merge $\rot{Y}_i$ in the $i$-th vertex. Similar to case (I), see  \cref{exofg}, we can move the cups and caps to appropriate positions such that they do not interfere our further operations. Then we can use \eqref{webassoc}, \eqref{splitmerge} and \eqref{sliders}
to rewrite the composition of the split $f$ and the merge $\rot{Y}_i$ as a sum of reduced  diagrams. 

For example, suppose  $f=1_*$
\begin{tikzpicture}[anchorbase,color=\clr]
	\draw[-,line width=1.4pt] (0.08,-.3) to (0.08,0.04);
	\draw[-,line width=1pt] (0.28,.4) to (0.08,0);
	\draw[-,line width=1pt] (-0.12,.4) to (0.08,0);
        \node at (-0.22,.6) {$\scriptstyle 5$};
        \node at (0.36,.6) {$\scriptstyle 4$};
\end{tikzpicture}$1_*$ and
\begin{equation}
    \label{exofgsplit}
    g=
    \begin{tikzpicture}
        [anchorbase,color=\clr]
        \draw[-,line width=1.4pt] (0,1) -- (0,1.3) 
                                  (1,1) -- (1,1.3)
                                  (.5,0) -- (.5,-.3)
                                  (1.5,0) -- (1.5,-.3);
        \draw[-,line width=1pt]   (0,1) -- (.5,0)
                                  (0,1) to [out=-30,in=-150] (1,1)
                                  (0,1) -- (-.2,.5)
                                  (.5,0) -- (1,1)
                                  (.5,0) to [out=30,in=150] (1.5,0)
                                  (1.5,0) -- (1,1)
                                  (1,1) -- (1.5,.8);
        \draw (-.2,.5)\bdot (1.5,.8)\bdot;
        \draw (-.1,.75)\dia (1.25,.9)\dia;
        \draw (.4,.2)\squ (.6,.2)\squ (.7,.1)\squ (.2,.9)\squ (1.3,.4)\squ;
        \node at (-.2,1){$\scriptstyle 2$};
        \node at (.2,.3){$\scriptstyle 3$};
        \node at (.5,1){$\scriptstyle 4$};
        \node at (1.5,1.1){$\scriptstyle 4$};
        \node at (1,0){$\scriptstyle 2$};
    \end{tikzpicture} \: .
\end{equation}
Locally we have
\begin{equation}\label{examplenewsplit}
    \begin{aligned}
    \begin{tikzpicture}[anchorbase,color=\clr]
	\draw[-,line width=1.4pt] (0.08,0) to (0.08,0.2);
 \draw[-,line width=1pt] (0.08,0) to (-.4,-.6);\node at (-0.4,-.8) {$\scriptstyle 2$};
 \draw[-,line width=1pt] (0.08,0) to (0.08,-.6);\node at (0.08,-.8) {$\scriptstyle 3$};
 \draw[-,line width=1pt] (0.08,0) to (0.48,-.6);\node at (0.48,-.8) {$\scriptstyle 4$};
	\draw[-,line width=1pt] (0.28,.6) to (0.08,0.2);
	\draw[-,line width=1pt] (-0.12,.6) to (0.08,0.2);
        \node at (-0.22,.7) {$\scriptstyle 5$};
        \node at (0.36,.7) {$\scriptstyle 4$};
\end{tikzpicture}&
=\begin{tikzpicture}[anchorbase,color=\clr]
 \draw[-,line width=1pt]  (-.4,-.6)to (.6,.6) ;
 \draw[-,line width=1pt]  (0.08,-.6) to(.6,.6);
 \node at (-0.4,-.8) {$\scriptstyle 2$};
 \draw[-,line width=1pt]  (0.08,-.6) to (-.4,.6);
 \node at (0.08,-.8) {$\scriptstyle 3$}; \node at (-0.4,.2) {$\scriptstyle 1$};\node at (0.55,.2) {$\scriptstyle 2$};
 \draw[-,line width=1pt]  (0.6,-.6) to(-.4,.6)  ;
 \node at (0.6,-.8) {$\scriptstyle 4$};
%
%
        \node at (-0.4,0.85) {$\scriptstyle 5$};
        \node at (0.6,.85) {$\scriptstyle 4$};
\end{tikzpicture}
+ \begin{tikzpicture}[anchorbase,color=\clr]
 \draw[-,line width=1pt]  (-.4,-.6)to (.6,.6) ;
 \draw[-,line width=1pt]  (0.08,-.6) to(.6,.6);
 \node at (-0.4,-.8) {$\scriptstyle 2$};
 \draw[-,line width=1pt]  (0.08,-.6) to (-.4,.6); \draw[-,line width=1pt]  (0.6,-.6) to(.6,.6)  ;
 \node at (0.08,-.8) {$\scriptstyle 3$}; \node at (-0.4,.2) {$\scriptstyle 2$};
 \node at (0.7,.2) {$\scriptstyle 1$};
 \draw[-,line width=1pt]  (0.6,-.6) to(-.4,.6)  ;\node at (0.6,-.8) {$\scriptstyle 4$};
%
%
        \node at (-0.4,0.85) {$\scriptstyle 5$};
        \node at (0.6,.85) {$\scriptstyle 4$};
\end{tikzpicture}
+ \begin{tikzpicture}[anchorbase,color=\clr]
 \draw[-,line width=1pt]  (-.4,-.6)to (.6,.6) ;
 \node at (-0.4,-.8) {$\scriptstyle 2$};
 \draw[-,line width=1pt]  (0.08,-.6) to (-.4,.6); \draw[-,line width=1pt]  (0.6,-.6) to(.6,.6)  ;
 \node at (0.08,-.8) {$\scriptstyle 3$}; 
 \node at (0.7,.2) {$\scriptstyle 2$};
 \draw[-,line width=1pt]  (0.6,-.6) to(-.4,.6)  ;\node at (0.6,-.8) {$\scriptstyle 4$};
%
%
        \node at (-0.4,0.85) {$\scriptstyle 5$};
        \node at (0.6,.85) {$\scriptstyle 4$};
\end{tikzpicture}
~-~\begin{tikzpicture}[anchorbase,color=\clr]
 \draw[-,line width=1pt]  (-.4,-.6)to (.6,.6) ;
 \draw[-,line width=1pt]  (0.08,-.6) to(.6,.6);
 \node at (-0.4,-.8) {$\scriptstyle 2$};
 \draw[-,line width=1pt]  (-0.4,-.6) to (-.4,.6); 
 \node at (0.08,-.8) {$\scriptstyle 3$}; 
 \draw[-,line width=1pt]  (0.6,-.6) to(-.4,.6)  ;\node at (0.6,-.8) {$\scriptstyle 4$};
%
%
        \node at (-0.4,0.85) {$\scriptstyle 5$};
        \node at (0.6,.85) {$\scriptstyle 4$};
\end{tikzpicture}
~+~\begin{tikzpicture}[anchorbase,color=\clr]
 \draw[-,line width=1pt]  (-.4,-.6)to (.6,.6) ;
 \draw[-,line width=1pt]  (-.4,-.6)to (-.4,.6) ;
  \draw[-,line width=1pt]  (0.08,-.6) to(.6,.6);
 \node at (-0.4,-.8) {$\scriptstyle 2$};
 \draw[-,line width=1pt]  (0.08,-.6) to (-.4,.6); \draw[-,line width=1pt]  (0.6,-.6) to(.6,.6)  ;
 \node at (0.08,-.8) {$\scriptstyle 3$}; 
 \node at (0.7,.2) {$\scriptstyle 1$};
 \draw[-,line width=1pt]  (0.6,-.6) to(-.4,.6)  ;\node at (0.6,-.8) {$\scriptstyle 4$};
%
%
        \node at (-0.4,0.85) {$\scriptstyle 5$};
        \node at (0.6,.85) {$\scriptstyle 4$};
\end{tikzpicture}
~-~\begin{tikzpicture}[anchorbase,color=\clr]
 \draw[-,line width=1pt]  (-.4,-.6)to (.6,.6) ;
 \draw[-,line width=1pt]  (-.4,-.6)to (-.4,.6) ;
  \draw[-,line width=1pt]  (0.08,-.6) to(.6,.6);
 \node at (-0.4,-.8) {$\scriptstyle 2$};
 \draw[-,line width=1pt]  (0.08,-.6) to (-.4,.6); \draw[-,line width=1pt]  (0.6,-.6) to(.6,.6)  ;
 \node at (0.08,-.8) {$\scriptstyle 3$}; 
 \node at (0.7,.2) {$\scriptstyle 2$};
 \draw[-,line width=1pt]  (0.6,-.6) to(-.4,.6)  ;\node at (0.6,-.8) {$\scriptstyle 4$};
%
%
        \node at (-0.4,0.85) {$\scriptstyle 5$};
        \node at (0.6,.85) {$\scriptstyle 4$};
\end{tikzpicture}\\
&+\begin{tikzpicture}[anchorbase,color=\clr]
 \draw[-,line width=1pt]  (-.4,-.6)to (.6,.6) ;
 \draw[-,line width=1pt]  (-.4,-.6)to (-.4,.6) ;
 \node at (-0.4,-.8) {$\scriptstyle 2$};
 \draw[-,line width=1pt]  (0.08,-.6) to (-.4,.6); \draw[-,line width=1pt]  (0.6,-.6) to(.6,.6)  ;
 \node at (0.08,-.8) {$\scriptstyle 3$}; 
 \node at (0.7,.2) {$\scriptstyle 3$};
 \draw[-,line width=1pt]  (0.6,-.6) to(-.4,.6)  ;\node at (0.6,-.8) {$\scriptstyle 4$};
%
%
        \node at (-0.4,0.85) {$\scriptstyle 5$};
        \node at (0.6,.85) {$\scriptstyle 4$};
\end{tikzpicture} 
+\begin{tikzpicture}[anchorbase,color=\clr]
 \draw[-,line width=1pt]  (-.4,-.6)to (-.4,.6) ;
 \draw[-,line width=1pt]  (0.08,-.6) to(.6,.6);
 \node at (-0.4,-.8) {$\scriptstyle 2$};
 \draw[-,line width=1pt]  (0.08,-.6) to (-.4,.6); \draw[-,line width=1pt]  (0.6,-.6) to(.6,.6)  ;
 \node at (0.08,-.8) {$\scriptstyle 3$}; 
 \node at (0.7,.2) {$\scriptstyle 2$};
 \draw[-,line width=1pt]  (0.6,-.6) to(-.4,.6)  ;\node at (0.6,-.8) {$\scriptstyle 4$};
    %
%
        \node at (-0.4,0.85) {$\scriptstyle 5$};
        \node at (0.6,.85) {$\scriptstyle 4$};
\end{tikzpicture}
+\begin{tikzpicture}[anchorbase,color=\clr]
 \draw[-,line width=1pt]  (-.4,-.6)to (-.4,.6) ;
 \draw[-,line width=1pt]  (0.08,-.6) to(.6,.6);
 \node at (-0.4,-.8) {$\scriptstyle 2$};
 \draw[-,line width=1pt]  (0.08,-.6) to (-.4,.6); \draw[-,line width=1pt]  (0.6,-.6) to(.6,.6)  ;
 \node at (0.08,-.8) {$\scriptstyle 3$}; 
 \node at (0.7,.2) {$\scriptstyle 3$};
 \draw[-,line width=1pt]  (0.6,-.6) to(-.4,.6)  ;\node at (0.6,-.8) {$\scriptstyle 4$};
%
%
        \node at (-0.4,0.85) {$\scriptstyle 5$};
        \node at (0.6,.85) {$\scriptstyle 4$};
\end{tikzpicture}
+\begin{tikzpicture}[anchorbase,color=\clr]
 \draw[-,line width=1pt]  (-.4,-.6)to (-.4,.6) ;
 \draw[-,line width=1pt]  (0.08,-.6) to(.6,.6);
 \node at (-0.4,-.8) {$\scriptstyle 2$};
 \draw[-,line width=1pt]  (0.6,-.6) to(.6,.6)  ;
 \node at (0.08,-.8) {$\scriptstyle 3$}; 
 \node at (0.7,.2) {$\scriptstyle 1$};
 \draw[-,line width=1pt]  (0.6,-.6) to(-.4,.6)  ;\node at (0.6,-.8) {$\scriptstyle 4$};
  %
%
        \node at (-0.4,0.85) {$\scriptstyle 5$};
        \node at (0.6,.85) {$\scriptstyle 4$};
\end{tikzpicture}
+\begin{tikzpicture}[anchorbase,color=\clr]
 \draw[-,line width=1pt]  (-.4,-.6)to (-.4,.6) ;
 \node at (-0.4,-.8) {$\scriptstyle 2$};
 \draw[-,line width=1pt]  (0.08,-.6) to (-.4,.6);
 \draw[-,line width=1pt]  (0.56,-.6) to(.56,.6)  ;
 \node at (0.08,-.8) {$\scriptstyle 3$}; 
 \node at (0.6,-.8) {$\scriptstyle 4$};
%
        \node at (-0.4,0.85) {$\scriptstyle 5$};
        \node at (0.6,.85) {$\scriptstyle 4$};
\end{tikzpicture}.
\end{aligned} \end{equation}
Next, we compose these diagrams with the remainder of \( fg \). The resulting diagrams may not be in reduced form, as new splits may appear in the top part. The heads of these new splits are connected in one of the three ways: 
\begin{enumerate}
    \item  to the top part of \( fg \) by a cup,
  \item  to the bottom of the diagram \( fg \) by a vertical strand, 
    \item to a lollipop of the form \( \wkdotamurevbox \).
\end{enumerate}

We first handle the new splits whose heads are connected to the top part of \( fg \) by a cup. In this case, we apply \eqref{sliders} and \eqref{mergesplitmove} to transform these splits into merges in the top part. Next, if the head of a new split is connected to the bottom of \( fg \) by a vertical strand, we use \eqref{swallows}--\eqref{braid} to move it downward. Finally, we apply  \cref{webassoc}  and \cref{lollipop-merge-split} to rewrite the local diagrams in the top part of the form
 $
\begin{tikzpicture}[anchorbase,scale=.8,color=\clr]
\draw[-,line width=1.2pt] (0.08,-.5) to (0.08,0.04);
\draw[-,line width=1pt] (0.34,.5) to (0.08,0);
\draw[-,line width=1pt] (-0.2,.5) to (0.08,0);
\node at (0.36,.65) {$\scriptstyle b$};
\node at (-.2,.65) {$\scriptstyle a$};
\draw (0.08,-.5) \bdot (0.08,-.2)\dia;
\end{tikzpicture}$. 
In the resulting diagrams, we combine all the lollipops attached to the same thick strand on the top part via \cref{lollipopspan} such that they become reduced.

For example, if we compose the fifth term in \cref{examplenewsplit} with the remaining of $fg$, then
\begin{align*}
\begin{tikzpicture}[anchorbase,color=\clr]
 \draw[-,line width=1pt]  (-.5,0) to (.5,.5) ;
 \draw[-,line width=1pt]  (-.5,0)to (-.5,.5) ;
  \draw[-,line width=1pt]  (0,0) to(.5,.5);
 \draw[-,line width=1pt]  (0,0) to (-.5,.5); 
 \draw[-,line width=1pt]  (0.5,0) to(.5,.5)  ;
 \draw[-,line width=1pt]  (0.5,0) to(-.5,.5);
 \draw[-,line width=1.4pt] (.5,.5) -- (.5,.8)
                           (-.5,.5) -- (-.5,.8)
                           (1.2,.5) -- (1.2,.8)
                           (.5,-.5) -- (.5,-.8)
                           (1.5,-.5) -- (1.5,-.8);
\draw[-,line width=1pt]   (.5,-.5) to [out=30,in=150] (1.5,-.5)
                          (.5,-.5) -- (1.2,.5)
                          (1.5,-.5) -- (1.2,.5)
                          (.5,-.5) -- (0,0)
                          (1.2,.5) -- (1.5,.3)
                          (.5,0) to [out=-20,in=70] (1.2,.5)
                          (-.5,0) -- (-.6,-.5);
\draw (-.6,-.5)\bdot (1.5,.3)\bdot;
\draw (-.55,-.25)\dia (1.35,.4)\dia;
\draw (.4,-.4)\squ  (.7,-.25)\squ (.7,-.42)\squ (1.42,-.2)\squ (.7,0.05)\squ;
\node at (-.35,.6){$\scriptstyle 3$};
\node at (-.75,-.25){$\scriptstyle 2$};
\node at (.1,-.3){$\scriptstyle 3$};
\node at (.7,.25){$\scriptstyle 4$};
\node at (1,-.5){$\scriptstyle 2$};
\node at (.2,0){$\scriptstyle 2$};
\node at (1.35,.6){$\scriptstyle 4$};
\end{tikzpicture}
~&\equiv\text{ a $\mathbb J'$-span of}~
\begin{tikzpicture}[anchorbase,color=\clr]
     \draw[-,line width=1.4pt] (.5,.5) -- (.5,.8)
                           (-.5,.5) -- (-.5,.8)
                           (1.2,.5) -- (1.2,.8)
                           (.5,-.5) -- (.5,-.8)
                           (1.5,-.5) -- (1.5,-.8);
    \draw[-,line width=1pt]   (.5,-.5) to [out=30,in=150] (1.5,-.5)
                          (.5,-.5) -- (1.2,.5)
                          (1.5,-.5) -- (1.2,.5)
                          (.5,-.5) -- (-.5,.5)
                          (.5,-.5) -- (.5,.5)
                          (1.2,.5) -- (1.5,.3)
                          (-.5,0) -- (-.6,-.5)
                          (-.5,.5) -- (-.5,0)
                          (-.5,0) -- (.5,.5)
                          (-.5,.55) to [out=-40,in=-140] (1.2,.5)
                          (.5,.5) to [out=-20,in=-160] (1.2,.5);
    \draw (-.6,-.5)\bdot (1.5,.3)\bdot;
\draw (-.55,-.25)\dia (1.35,.4)\dia;
\draw (.3,-.3)\squ  (.7,-.25)\squ (.7,-.42)\squ (1.42,-.2)\squ (.5,-.3)\squ (-.3,.45)\squ (.6,.45)\squ;
\node at (1,-.5){$\scriptstyle 2$};
\node at (-.75,-.25){$\scriptstyle 2$};
\node at (-.25,.65){$\scriptstyle 3$};
\node at (.6,0){$\scriptstyle 2$};
\node at (1.35,.6){$\scriptstyle 4$};
\end{tikzpicture}
~\equiv\text{ a $\mathbb J'$-span of}~
\begin{tikzpicture}[anchorbase,color=\clr]
     \draw[-,line width=1.4pt] (.5,.5) -- (.5,.8)
                           (-.5,.5) -- (-.5,.8)
                           (1.2,.5) -- (1.2,.8)
                           (.5,-.5) -- (.5,-.8)
                           (1.5,-.5) -- (1.5,-.8);
    \draw[-,line width=1pt]   (.5,-.5) to [out=30,in=150] (1.5,-.5)
                          (.5,-.5) -- (1.2,.5)
                          (1.5,-.5) -- (1.2,.5)
                          (.5,-.5) -- (.5,.5)
                          (1.2,.5) -- (1.5,.3)
                        (-.5,.5) -- (0,.4) -- (.5,.5)
                        (0,.4) -- (0,.1)
                        (-.5,.5) to [out=-60,in=left] (0,-.1) to [out=right,in=-135] (1.2,.5)
                        (-.5,.5) -- (-.5,.3) to [out=down,in=135] (.5,-.5)
                          (.5,.5) to [out=-20,in=-160] (1.2,.5);
    \draw  (1.5,.3)\bdot (0,.1)\bdot;
\draw  (1.35,.4)\dia (0,.25)\dia;
\draw (.7,-.25)\squ (.7,-.42)\squ (1.42,-.2)\squ (.5,-.3)\squ  (.6,.45)\squ (-.3,.15)\squ (.4,-.4)\squ;
\node at (1.35,.6){$\scriptstyle 4$};
\node at (1,-.5){$\scriptstyle 2$};
\node at (.2,.2){$\scriptstyle 2$};
\node at (-.28,.35){$\scriptstyle 3$};
\node at (.6,.25){$\scriptstyle 2$};
\end{tikzpicture} \\
&~\underset{\cref{webassoc}}{\overset{{\cref{lollipop-merge-split}}}{\equiv}}\text{ a $\mathbb J'$-span of}~\begin{tikzpicture}[anchorbase,color=\clr]
     \draw[-,line width=1.4pt] (.5,.5) -- (.5,.8)
                           (-.5,.5) -- (-.5,.8)
                           (1.2,.5) -- (1.2,.8)
                           (.5,-.5) -- (.5,-.8)
                           (1.5,-.5) -- (1.5,-.8);
    \draw[-,line width=1pt]   (.5,-.5) to [out=30,in=150] (1.5,-.5)
                          (.5,-.5) -- (1.2,.5)
                          (1.5,-.5) -- (1.2,.5)
                        (-.5,.5) to [out=-30,in=-150] (.5,.5)
                          (.5,-.5) -- (.5,.5)
                          (1.2,.5) -- (1.5,.3)
                        (-.5,.5) to [out=-60,in=left] (0,-.1) to [out=right,in=-135] (1.2,.5)
                        (-.5,.5) -- (-.5,.3) to [out=down,in=135] (.5,-.5)
                          (.5,.5) to [out=-20,in=-160] (1.2,.5)
                          (-.5,.5) -- (-.7,.2)
                          (.5,.5) -- (.3,.2);
    \draw  (1.5,.3)\bdot (.3,.2)\bdot (-.7,.2)\bdot;
\draw  (1.35,.4)\dia (-.6,.35)\dia (.4,.35)\dia;
\draw (.7,-.25)\squ (.7,-.42)\squ (1.42,-.2)\squ (.5,-.3)\squ  (.6,.45)\squ (-.3,.15)\squ (.4,-.4)\squ (-.3,.45)\squ;
\node at (1.35,.6){$\scriptstyle 4$};
\node at (1,-.5){$\scriptstyle 2$};
\node at (0,.1){$\scriptstyle 3$};
\node at (.6,.25){$\scriptstyle 2$};
\end{tikzpicture}.
\end{align*}
This concludes the proof of Case (II).

Suppose that $f$ is of type $\cap$, and we assume that $f=1_*\capa1_*$ is joined to the $i$-th and $i+1$-th strand at the top of $g$. Then $g$ has an $s$-fold merge $\rot{Y}_i$ and $t$-fold merge $\rot{Y}_{i+1}$ on its $i$-th and $(i+1)$-st vertices, respectively. Once again we may move cups and caps of $g$ to positions where they do not interfere with our operations unless they are shared legs of $\rot{Y}_{i}$ and $\rot{Y}_{i+1}$.

We first apply \cref{sidewaydumbbell} to the composition of $f,g$. We see that every cup which is a shared leg of $\rot{Y}_{i}$ and $\rot{Y}_{i+1}$ now becomes a cap on the top part. So we may assume that $\rot{Y}_{i}$ and $\rot{Y}_{i+1}$ does not have any shared leg. 

Under this assumption, we then apply \eqref{mergesplitslideleft} to \( \rot{Y}_{i} \), transforming it into an \( s \)-fold split on top of \( \rot{Y}_{i+1} \). After moving the dot packets to appropriate positions via \cref{dotmovefreely}, we use \cref{splitmerge} to rewrite the composition of the \( s \)-fold split with \( \rot{Y}_{i+1} \), cf. \cref{examplenewsplit}. 

For example, suppose $f=1_* \begin{tikzpicture}
    [anchorbase,color=\clr]
    \draw[line width=1pt] (0,0) to [out=up,in=left] (.2,.2) to [out=right,in=up] (.4,0);
    \node at (0,-.2) {$\scriptstyle 12$};
    \node at (.4,-.2) {$\scriptstyle 12$};
\end{tikzpicture}1_*$ and 
\begin{equation}
    \label{exofgcap}
   g= \begin{tikzpicture}
        [anchorbase,scale=1.2, color=\clr]
        \draw[line width=1.4pt] (0,.5) -- (0,.8)
                              (1,.5) -- (1,.8)
                              (-.5,-1) -- (-.5,-1.3)
                              (.5,.-1) -- (.5,-1.3)
                              (1.5,-1) -- (1.5,-1.3);  
        \draw 
               (-.5,-1) to[out=30,in=150] (1.5,-1)
               (-.5,-1) to[out=15,in=165] (.5,-1);
        \draw (0,.5) -- (.5,-1) -- (1,.5)
              (0,.5) -- (-.5,-1)
              (1,.5) -- (1.5,-1)
              (0,.5) -- (1.5,-1)
              (1,.5) -- (-.5,-1);
        \draw (0,.5) -- (-.3,.3) -- (-.6,.1) 
                (-.3,.3)\dia (-.6,.1)\bdot
              (1,.5) -- (1.3,.3) -- (1.6,.1) (1.3,.3)\dia (1.6,.1)\bdot
              (-.5,-1) -- (-.7,-.8) (-.7,-.8)\bdot
              (1.5,-1) -- (1.7,-.8) -- (1.9,-.6) (1.7,-.8)\dia (1.9,-.6)\bdot;
        \node at (-.2,.55){$\scriptstyle 4$};
         \node at (1.2,.55){$\scriptstyle 6$};
        \node at (-.7,-1.05){$\scriptstyle 2$};
        \node at (1.7,-1.05){$\scriptstyle 8$};
        \draw (-.35,-.55) \squ
              (-.2,-.7) \squ
              (-.2,-.95)\squ
              (.55,-.85)\squ
              (1.4,-.7)\squ
              (1.05,-.55)\squ;
        \node at (.8,.5){$\scriptstyle 2$};
        \node at (-.3,0){$\scriptstyle 5$};
        \node at (.075,0){$\scriptstyle 2$};
        \node at (1.3,0){$\scriptstyle 3$};
        \node at (1.2,-1.1){$\scriptstyle 2$};
    \end{tikzpicture}
    \quad \Rightarrow \quad fg \overset{\cref{mergesplitslideleft}}{=}
    \begin{tikzpicture}
        [anchorbase,color=\clr]
     \draw[line width=1.4pt] 
                              (1,.5) -- (1,.8)
                              (-.5,-1) -- (-.5,-1.3)
                              (.5,.-1) -- (.5,-1.3)
                              (1.5,-1) -- (1.5,-1.3);  
     \draw 
               (-.5,-1) to[out=30,in=150] (1.5,-1)
               (-.5,-1) to[out=15,in=165] (.5,-1);
     \draw    (.5,-1) -- (1,.5)
              (1,.5) -- (1.5,-1)
              (1,.5) -- (-.5,-1);
        \draw (1,.5) -- (1.3,.3) -- (1.6,.1) (1.3,.3)\dia (1.6,.1)\bdot
              (-.5,-1) -- (-.7,-.8) (-.7,-.8)\bdot
              (1.5,-1) -- (1.7,-.8) -- (1.9,-.6) (1.7,-.8)\dia (1.9,-.6)\bdot;
         \node at (1.2,.55){$\scriptstyle 6$};
        \node at (-.7,-1.05){$\scriptstyle 2$};
        \node at (1.7,-1.05){$\scriptstyle 8$};
        \draw  (-.2,-.7) \squ
               (-.2,-.95)\squ
              (.55,-.85)\squ
              (1.4,-.7)\squ;
        \node at (.6,.3){$\scriptstyle 2$};
         \node at (-.2,-.3){$\scriptstyle 2$};
        \node at (1.3,0){$\scriptstyle 3$};
                   \node at (-.7,1){$\scriptstyle 5$};
        \node at (-.2,1){$\scriptstyle 2$};
        \draw (-.5,-1) -- (-.5,.8) to [out=up,in=left] (.3,1.4) to [out=right,in=up] (1,.8);
        \draw (1,.8) to [out=110,in=70] (0,.8) to [out=-110,in=110] (.5,-1);
        \draw (1,.8) to [out=150,in=30] (.5,.8) to [out=-150,in=150] (1.5,-1);
        \draw (-.5,-.7)\squ
              (1,-.7)\squ;
        \draw (1,.8) -- (1.3,1) -- (1.6,1.2)
             (1.3,1)\dia (1.6,1.2)\bdot;
        \node at (1.3,1.2){$\scriptstyle 4$};
    \end{tikzpicture} \: .
\end{equation}
We apply \cref{strand-stick-stick} to the top part of \( fg \). In the resulting diagrams, we then move the lollipops to the bottom part using \cref{webassoc} and \cref{dotmovefreely}.

For example, continuing from \cref{exofgcap}, we have
\begin{align*}
     \begin{tikzpicture}
        [anchorbase,color=\clr]
     \draw[line width=1.4pt] 
                              (1,.5) -- (1,.8)
                              (-.5,-1) -- (-.5,-1.3)
                              (.5,.-1) -- (.5,-1.3)
                              (1.5,-1) -- (1.5,-1.3);  
     \draw 
               (-.5,-1) to[out=30,in=150] (1.5,-1)
               (-.5,-1) to[out=15,in=165] (.5,-1);
     \draw    (.5,-1) -- (1,.5)
              (1,.5) -- (1.5,-1)
              (1,.5) -- (-.5,-1);
        \draw (1,.5) -- (1.3,.3) -- (1.6,.1) (1.3,.3)\dia (1.6,.1)\bdot
              (-.5,-1) -- (-.7,-.8) (-.7,-.8)\bdot
              (1.5,-1) -- (1.7,-.8) -- (1.9,-.6) (1.7,-.8)\dia (1.9,-.6)\bdot;
              \node at (1.3,.5){$\scriptstyle 6$};
        \node at (-.7,-1.05){$\scriptstyle 2$};
        \node at (1.7,-1.05){$\scriptstyle 8$};
        \draw  (-.2,-.7) \squ
               (-.2,-.95)\squ
              (.55,-.85)\squ
              (1.4,-.7)\squ;
        \node at (.6,.3){$\scriptstyle 2$};
         \node at (-.2,-.3){$\scriptstyle 2$};
        \node at (1.3,0){$\scriptstyle 3$};
        \draw (-.5,-1) -- (-.5,.8) to [out=up,in=left] (.3,1.4) to [out=right,in=up] (1,.8);
        \draw (1,.8) to [out=110,in=70] (0,.8) to [out=-110,in=110] (.5,-1);
        \draw (1,.8) to [out=150,in=30] (.5,.8) to [out=-150,in=150] (1.5,-1);
        \draw (-.5,-.7)\squ
              (1,-.7)\squ;
        \draw (1,.8) -- (1.3,1) -- (1.6,1.2)
             (1.3,1)\dia (1.6,1.2)\bdot;
        \node at (1.3,1.2){$\scriptstyle 4$};
       \node at (-.7,1){$\scriptstyle 5$};
        \node at (-.2,1){$\scriptstyle 2$};
    \end{tikzpicture}
    ~&\overset{\cref{strand-stick-stick}}{\equiv}\text{a $\mathbb J'$-span of }~
     \begin{tikzpicture}
        [anchorbase,color=\clr]
     \draw[line width=1.4pt] 
                              (1,.5) -- (1,.8)
                              (-.5,-1) -- (-.5,-1.3)
                              (.5,.-1) -- (.5,-1.3)
                              (1.5,-1) -- (1.5,-1.3);  
     \draw 
               (-.5,-1) to[out=30,in=150] (1.5,-1)
               (-.5,-1) to[out=15,in=165] (.5,-1);
     \draw    (.5,-1) -- (1,.5)
              (1,.5) -- (1.5,-1)
              (1,.5) -- (-.5,-1);
        \draw (1,.5) -- (1.3,.55) -- (1.6,.6) (1.3,.55)\dia (1.6,.6)\bdot
              (-.5,-1) -- (-.7,-.8) (-.7,-.8)\bdot
              (1.5,-1) -- (1.7,-.8) -- (1.9,-.6) (1.7,-.8)\dia (1.9,-.6)\bdot;
        \node at (-.7,-1.05){$\scriptstyle 2$};
        \node at (1.7,-1.05){$\scriptstyle 8$};
        \draw  (-.2,-.7) \squ
               (-.2,-.95)\squ
              (.55,-.85)\squ
              (1.4,-.7)\squ
              (1,.65)\squ;
        \node at (.6,.3){$\scriptstyle 2$};
         \node at (-.2,-.3){$\scriptstyle 2$};
        \node at (1.3,0){$\scriptstyle 3$};
        \node at (-.7,1){$\scriptstyle 5$};
        \node at (-.2,1){$\scriptstyle 2$};
        \draw (-.5,-1) -- (-.5,.8) to [out=up,in=left] (.3,1.4) to [out=right,in=up] (1,.8);
        \draw (1,.8) to [out=110,in=70] (0,.8) to [out=-110,in=110] (.5,-1);
        \draw (1,.8) to [out=150,in=30] (.5,.8) to [out=-150,in=150] (1.5,-1);
        \draw (-.5,-.7)\squ
              (1,-.7)\squ;
        \draw (1,.8) -- (1.3,.75) -- (1.6,.7)
             (1.3,.75)\dia (1.6,.7)\bdot;
        \node at (1.3,.95){$\scriptstyle 8-t$};
         \node at (1.3,.35){$\scriptstyle 6-t$};
    \end{tikzpicture}
    \; (t\in2\N)
    \\
    ~&\underset{\cref{lollipopspan}}{\overset{\cref{webassoc}}{\equiv}}\text{ a $\mathbb J'$-span of }~
       \begin{tikzpicture}
        [anchorbase,color=\clr]
     \draw[line width=1.4pt] 
                              (1,.5) -- (1,.8)
                              (-.5,-1) -- (-.5,-1.3)
                              (.5,.-1) -- (.5,-1.3)
                              (1.5,-1) -- (1.5,-1.3);  
     \draw 
               (-.5,-1) to[out=30,in=150] (1.5,-1)
               (-.5,-1) to[out=15,in=165] (.5,-1);
     \draw    (.5,-1) -- (1,.5)
              (1,.5) -- (1.5,-1)
              (1,.5) -- (-.5,-1);
        \draw  (-.5,-1) -- (-.7,-.8) -- (-.9,-.6) (-.7,-.8)\dia (-.9,-.6)\bdot
              (1.5,-1) -- (1.7,-.8) -- (1.9,-.6) (1.7,-.8)\dia (1.9,-.6)\bdot;
        \node at (-.9,-1.05){$\scriptstyle 10-t$};
        \node at (1.9,-1.05){$\scriptstyle 14-t$};
        \draw  (-.2,-.7) \squ
               (-.2,-.95)\squ
              (.55,-.85)\squ
              (1.4,-.7)\squ;
        \node at (.6,.3){$\scriptstyle 2$};
         \node at (-.2,-.3){$\scriptstyle 2$};
        \node at (1.3,0){$\scriptstyle 3$};
        \node at (-.7,1){$\scriptstyle 5$};
        \node at (-.2,1){$\scriptstyle 2$};
        \draw (-.5,-1) -- (-.5,.8) to [out=up,in=left] (.3,1.4) to [out=right,in=up] (1,.8);
        \draw (1,.8) to [out=110,in=70] (0,.8) to [out=-110,in=110] (.5,-1);
        \draw (1,.8) to [out=150,in=30] (.5,.8) to [out=-150,in=150] (1.5,-1);
        \draw (-.5,-.7)\squ
              (1,-.7)\squ;
    \end{tikzpicture} \; (t\in 2\N).
\end{align*}

 Next, we rewrite the composition of the merge and split using \cref{splitmerge}; see also \cref{examplenewsplit}.  This operation results in new splits and merges in the top part of \( fg \). We observe that all new splits are connected to the bottom of \( fg \) and can thus be pushed downward as in Case (II). Similarly, the new merges are connected via a cap on top of the diagrams. Hence, we apply \cref{mergesplitmove} to transform these merges into new splits in the top part, which can then be pushed downward to the bottom as well. Similar to Case (II), the resulting diagrams will be elementary. This concludes the proof of Case (III).

Finally, suppose that \(f\) is of type \(\wkdotaa\). Then \(fg\) is obtained from
\(g\) by adding a dot to the head at the \(i\)-th vertex. Suppose that \(g\)
has an \(s\)-fold merge at this vertex. By \cref{dotmovefreely}, we may move the
new dot, modulo lower degree terms, to the legs of the \(s\)-fold merge and obtain
\[
fg\equiv \sum\limits_{1\le j\le s, r_1+r_2+\ldots+r_s=r} g_{j,r_j} 
\]
where the diagram $g_{j,r_j}$ is obtained from $g$ by adding a dot $\omega_{r_j}$ on the $j$-th leg of the $s$-fold merge at its $i$-th vertex. If the $j$-th leg is a lollipop, we use \cref{lollipopspan} to rewrite the resulting diagrams as a $\mathbb J'$-span of elementary diagrams. If the $j$-th leg is a cup in the top part of $g$, we use \cref{dotmovefreely} and \eqref{dotmovecupcaps} to move the new dot of $g_{j,r_j}$  through crossings (up to lower degree terms) along the cup until it meets the leftmost boundary of that cup. If the $j$-th leg is connected to the bottom part, we use Lemma \ref{dotmovefreely} to move the new dots of $g_{j,r_j}$ downward  through crossings, modulo lower degree terms, until it meets the existing dot packet on that leg; the resulting diagram is elementary. This finishes the proof of Case (IV).

The proof of \cref{thm:spanAWb} is completed.

%
\section{A basis theorem for the affine B-web category}
  \label{sec:basis}

In this section we construct a monoidal functor from the affine B-web category $\AWb$ to a category of endomorphisms for $\so_N$-modules. Then we use it to show that the spanning sets for morphism spaces of $\AWb$ constructed in the previous section are indeed bases.

\subsection{A representation of $\AWb$}

Recall the category $ \Wb_\C(N) $ obtained from $\Wb$ by base change and its representation from \cref{sec:repWo}. We do a similar base change for $\AWb$ (which is defined over $\mathbb J'$). Denote by $\mathbb J'_N$ the quotient ring of $\mathbb J'$ by the ideal generated by $x-N$, which is then identified with $\mathbb Z[\frac12]$. We set
\[
\AWb_{\mathbb Z[\frac12]}(N):=\AWb\otimes_{\mathbb J'}\mathbb J'_N,
\qquad
\AWb_{\mathbb C}(N):=\AWb_{\mathbb Z[\frac12]}(N)\otimes_{\mathbb Z[\frac12]}\mathbb C.
\]

Let $\g=\so_N$.
Recall the element $\Omega$ from \cref{casimir}. For any $M,L\in \g$-Mod, $\Omega$ naturally acts on $M\otimes L$ by
\[
\Omega(m\otimes l)=\frac{1}{2}\sum_{i,j\in[N]}F_{i,j}m\otimes F_{j,i}l,\quad \forall m\in M,l\in L.
\]
In particular, we have $\Omega=\frac{1}{2}(\Delta(C)-C\otimes 1-1\otimes C)$ where $\Delta$ is the co-multiplication of $\U(\g)$ and $C$ is the quadratic Casimir element. Since $C$ is central in $\U(\g)$, we have $\Omega\in \End_{\U(\g)}(M\otimes L)$.

Let $\emod(\so_N\Mod)$ denote the category of endofunctors of the category $\g$-Mod of $\U(\g)$-modules. The category $\emod(\so_N\Mod)$ has a monoidal structure as follows. If $F,G$ are two endofunctors of $\g$-Mod, then $F\otimes G:=F\circ G$. If $\xi:F\to G,\varsigma:H\to K$ are two natural transformations, then $\xi\otimes \varsigma:F\otimes H \to G\otimes K$ is defined to be the natural transformation determined by \[
   (\xi\otimes \varsigma)_M
   =
   \xi_{K(M)}\circ F(\varsigma_M)
   =
   G(\varsigma_M)\circ \xi_{H(M)}
\]
 for any $M\in \g$-Mod.

The following theorem extends the representation functor $\cF_N$ from \cref{sec:repWo}.
\begin{theorem}
    \label{AWbrep} 
    There is a monoidal functor $\cG_N: \AWb_{\mathbb C}(N)\to \emod(\so_N\Mod)$ such that $\cG_N(a)=-\otimes \bigwedge^a V$ and
    \begin{equation}
    \label{cG}
    \begin{aligned}
        &\cG_N\Big(\merge\Big)_M=\id_M \otimes \cF_N\Big(\merge\Big),
        \qquad
        \cG_N\Big(\splits\Big)_M=\id_M \otimes \cF_N\Big(\splits\Big),\\
         &\cG_N\Big(\capa\Big)_M=\id_M \otimes \cF_N\Big(\capa\Big),
        \qquad
        \cG_N\Big(\cupa\Big)_M=\id_M \otimes \cF_N\Big(\cupa\Big),\\
        &\cG_N\Big(\crossing\Big)_M=\id_M \otimes \cF_N\Big(\crossing\Big),
        \qquad \cG_N\Big(\dotgen\Big)_M=\Big(\Omega+\frac{1}{2}(N-1)\Big)|_{M\otimes V}.
    \end{aligned}
    \end{equation}
\end{theorem}

\begin{proof}
A simpler presentation for $\AWb$ over $\C[x]$ is given in Definition \ref{def:affineBwebC}. By a base change as for $\AWb_\C(N)$ we obtain ${\AWb_\C}'(N)$ from $\AWCb$. It follows that $\AWb_\C(N)$ is isomorphic to ${\AWb_\C}'(N)$ as $\C$-linear monoidal categories. Thanks to the isomorphism (see \cref{thm:AWiso}), it suffices to check that $\cG_N$ defined by \cref{cG} preserves the relations \cref{webassoc}--\cref{bubblecap} and \cref{dotmovecapcupC} (with $x$ replaced by $N$). The verification of \cref{webassoc}--\cref{bubblecap} was carried out in the proof of \cref{thm:Worep}. As the relations \cref{dotmovecapcupC} only involve thickness $1$, they are part of relations for affine Brauer category, and the fact that these relations are preserved by $\cG_N$ has been verified in \cite[Theorem 3.8]{RS2019}.
\end{proof}

Through the specialization functor from $\AWb \to \AWb_\C(N)$, we may regard the functor $\cG_N$ as a representation for $\AWb$ as well.

\begin{rem}
    An analogous monoidal functor $\cG_{-N}: \AWb_{\mathbb C}(-N)\to \emod(\mathfrak{sp}_N\Mod)$ can be constructed, for any positive even integer $N$. 
\end{rem}

Let $n=\lfloor N/2 \rfloor$. Recall the triangular decomposition $\g=\n^+\oplus \h \oplus \n^-$ from \cref{sec:repWo}. Let $\borel=\n^+\oplus \h$. We define the generic Verma module as follows:
\begin{equation}\label{defofgenericverma}
M^{\text{gen}} := U(\mathfrak{g}) \otimes_{U(\mathfrak{b})} U(\mathfrak{h}),
\end{equation}
where $U(\mathfrak{h})$ is viewed as a $U(\mathfrak{b})$-module by inflation. By the PBW theorem, $ M^{\rm gen} $ forms a free right $ \U(\h) $-module with a basis consisting of elements of the form  
\begin{equation}\label{Mgenbasis}
    \{f_1^{a_1} \cdots f_s^{a_s} \otimes 1 \mid a_k \in \mathbb{N}, \forall 1 \leq k \leq s \}.
\end{equation}  
where $\{f_1,\ldots,f_s\}$ forms a basis of $\n^-$. Recall the basis $\{h_i=F_{i,i}\mid 1\le i\le n\}$ for $\h$. We introduce a $ \mathbb{Z} $-grading on $M^{\text{gen}}$ by assigning $ \deg h_i = 1 $ and $ \deg f_k=0 $ for each $ i,  k $. Let $ \U(\h)_{\leq t} $ denote the span of monomials in $ \U(\h) $ of degree at most $ t $. This induces a filtration for $M^{\rm gen}$: 
\[
0 \subset M^{\rm gen}_0 \subset M^{\rm gen}_1 \subset \ldots \subset M^{\rm gen}_t \subset \ldots,
\]
where $ M^{\rm gen}_t $ is generated by elements of the form  
$f_1^{a_1} \cdots f_s^{a_s} \otimes \U(\h)_{\leq t}$, for all $a_k \in \mathbb{N}.$

The free right $ \U(\h) $-module $ M^{\rm gen} \otimes \bigwedge ^r V $ has a basis given by  
\begin{equation}\label{Mgenrbasis}
    f_1^{a_1} \cdots f_s^{a_s} \otimes h_1^{b_1} \cdots h_n^{b_n} \otimes v_{\mathbf{i}}, \quad \mathbf{i} \in [N]^r, \quad a_k, b_j \in \mathbb{N}, \quad 1 \leq k \leq s, \quad 1 \leq j \leq n,
\end{equation}  
where $ v_{\mathbf{i}} = v_{i_1} \wedge \dots \wedge v_{i_r} $ for $i_1<\cdots<i_r$.  Setting $(M^{\rm gen}\otimes \bigwedge ^r V)_t := M^{\rm gen}_t\otimes \bigwedge ^r V$ and assigning 
the degree $ \sum_{i=1}^{n} b_i $ to the basis element in \cref{Mgenrbasis}, we have defined a filtration for $M^{\rm gen}\otimes \bigwedge ^r V$:
\[
0 \subset \big(M^{\rm gen}\otimes \bigwedge ^r V\big)_0 \subset \big(M^{\rm gen}\otimes \bigwedge ^r V\big)_1 \subset \dots \subset \big(M^{\rm gen}\otimes \bigwedge ^r V\big)_t\subset \dots,
\]

\begin{lemma}\cite[Lemma 3.9]{RS2019}
\label{eM filtered degree 1}
    For $1\leq i,j\leq n$,  let $\phi_{\pm i,\pm j}$ (resp., $\phi_{\pm i}$): $M^{\rm gen} \rightarrow M^{\rm gen}  $ be the linear map
    such that $\phi_{\pm i,\pm j}(m)=F_{\pm i,\pm j}m$ (resp., $\phi_{\pm i}(m)=F_{0,\pm i}m$), for $m\in M^{\rm gen} $.
     Then
    \begin{enumerate}
        \item[(1)] $\phi_{-i,\pm j}$ (resp., $\phi_{i}$ ) is homogeneous of degree 0 for  $i<j$  (resp., $1\leq i\leq n$);
        \item[(2)] $\phi_{i,\pm j}$ (resp., $\phi_{-i}$, $\phi_{i,i}$) is of  filtered degree $1$ for     $i<j$  (resp., $ 1\leq i\leq n$);
        \item[(3)]  $\Omega\in \End_{\mathbb C} (M^{\rm gen} \otimes V)$ is of  filtered degree $1$.
    \end{enumerate}
\end{lemma}

\subsection{Explicit actions}
Recall the basis $\{h_i=F_{i,i}\mid i=1,\ldots,n\}$ of $\mathfrak h$ from \cref{sec:repWo}. We define $h_{-i}=-h_i$ for $i=1,\ldots,n$ and $h_0=0$. For later use, we shall write down explicit formulas for actions of several morphisms modulo lower degree terms.

\begin{lemma}
 \label{actionodotgen}
 The action of $\cG_N (\wxdota)_{M^{\text{gen}}}$
 on $M^{\text{gen}}\otimes \bigwedge^a V$ is of filtered degree 1. Moreover, up to lower degree terms, we have 
 \begin{equation*}
     \wxdota(1\otimes h\otimes v_{i_1}\wedge v_{i_2}\wedge\ldots \wedge v_{i_a})
     \equiv 1\otimes \sum_{1\le j\le a}h_{i_j}h\otimes v_{i_1}\wedge v_{i_2}\wedge \ldots \wedge v_{i_a},
 \end{equation*}  
 for any $h\in U(\mathfrak h)$ and ${\bf i}=({i}_1,\ldots\,{i}_a)\in [N]^a$ such that $ i_1<\cdots<i_a$.
 \end{lemma}
 \begin{proof}
     It follows from \cref{eM filtered degree 1}(3) directly that $\cG_N (\wxdota)_{M^{\text{gen}}}$ is of filtered degree 1. Moreover, we have
    \begin{multline*}
       \wxdota(1\otimes h\otimes v_{i_1}\wedge v_{i_2}\wedge\ldots\wedge v_{i_a})
       =\Omega(1\otimes h\otimes v_{i_1}\wedge v_{i_2}\wedge\ldots\wedge v_{i_a} )
       \\
     \equiv \sum_{1\le i\le  n} F_{i,i}(1\otimes h)\otimes F_{i,i} v_{i_1}\wedge v_{i_2}\wedge\ldots\wedge v_{i_a} 
     =
      1\otimes \sum_{1\le j\le a}h_{i_j}h\otimes v_{i_1}\wedge v_{i_2}\wedge \ldots \wedge v_{i_a},
    \end{multline*}  
    since $F_{i,i}v_j=\begin{cases}
        v_i & j=i,\\
        -v_{-i} &j=-i,\\
        0 & \text{otherwise}.
    \end{cases}$ for $i=1,\ldots,n$ and $j\in [N]$.  
\end{proof}

\begin{lemma}
\label{actionlantern}
    For any $\nu\in \Par_a$ and $h\in \U(\mathfrak h)$, we have 
 \[
 \omega_{a,\nu} (1\otimes h\otimes v_{i_1}\wedge \ldots\wedge v_{i_a})
 \equiv
 \big( 1\otimes e_{\nu}(h_{i_1},\ldots,h_{i_a})h \big) \otimes 
 v_{i_1}\wedge \ldots\wedge v_{i_a},
 \]
 where $e_\nu$ denotes the elementary symmetric polynomial associated with partition $\nu$.
\end{lemma}
\begin{proof}
     Since $\omega_{a,\nu}=\omega_{a,\nu_1}\cdots \omega_{a,\nu_k}$, it suffices to prove the case for $\nu=(r)$, which follows from a straightforward induction as in the proof of \cite[Corollary 3.11]{SWweb}.
\end{proof}

    Recall the bubbles defined in \cref{bubble+dot+par}. 
\begin{lemma}
\label{actionbubble2a}
    For any $a\in \N$ and $h\in \U(\mathfrak h)$, we have 
    \[
    \dotbubble_{2a}(1\otimes h)
    \equiv
    (-1)^a\otimes e_a( h_1^2,\ldots, h_{n}^2)h .
    \]
\end{lemma}

\begin{proof}
   We have
   \begin{align*}
        \dotbubble_{2a}(1\otimes h) 
        \overset{\cref{cuparep}}{=}&
        \begin{tikzpicture}
            [anchorbase,color=\clr]
            \draw[line width=1pt] (0,0) -- (0,.4) to [out=up,in=left] (.2,.6)
            to [out=right,in=up] (.4,.4) -- (.4,.0);
            \draw (0,.2)\bdot;
            \node at (0,-.2){$\scriptstyle 2a$};
            \node at (.4,-.2){$\scriptstyle 2a$};
         \end{tikzpicture}
          \Big( 1\otimes h \otimes \sum_{i_1<\cdots<i_{2a}} v_{i_1}\wedge \cdots \wedge v_{i_{2a}} \otimes
                                    v_{-i_{2a}}\wedge \cdots \wedge v_{-i_1}\Big)
        \\
         \overset{\cref{actionlantern}}{\equiv}& \begin{tikzpicture}
            [anchorbase,color=\clr]
            \draw[line width=1pt]  (0,.4) to [out=up,in=left] (.2,.6)
            to [out=right,in=up] (.4,.4);
            \node at (0,.2){$\scriptstyle 2a$};
            \node at (.4,.2){$\scriptstyle 2a$};
         \end{tikzpicture} 
         \Big(
         \sum_{i_1<\cdots<i_{2a}} 1\otimes e_{2a}(h_{i_1},\ldots,h_{i_{2a}})h \otimes  v_{i_1}\wedge \cdots \wedge v_{i_{2a}} \otimes
        v_{-i_{2a}}\wedge \cdots \wedge v_{-i_1}\Big) 
        \\
        \overset{\cref{caparep}}{=}&  \sum_{i_1<\cdots<i_{2a}} 1\otimes e_{2a}(h_{i_1},\ldots,h_{i_{2a}})h 
        =
        1\otimes e_{2a}(h_{-n},\ldots,h_{-1},h_{1},\ldots,h_{n})h 
        \\
       {=}&(-1)^a\otimes e_a(h_1^2,\ldots, h_{n}^2)h ,
   \end{align*}
   where the last equation follows from the fact that $h_i=-h_{-i}$ for $i=1,\ldots,n$. 
\end{proof}
Next we calculate the action of $\wkdotamu$, for $\mu\in\EPar_{2a}$. Let $\mu/2\in \Par_a$ denote the partition obtained from $\mu$ by dividing each part by $2$.

\begin{lemma}
\label{actionlollipop}
    For any $a\in \N$ and $\mu\in \EPar_{2a}$, we have 
    \begin{multline*}
            \wkdotamu \Big(1\otimes h\otimes v_{i_1}\wedge\cdots\wedge v_{i_{2a}}\Big)
    \equiv
        \prod_{j=1}^a \delta_{i_j,-i_{2a+1-j}}\otimes (-1)^{|\mu|/2}e_{\mu/2}(h_{i_1}^2,\ldots,h_{i_a}^2)h_{i_1}\cdots h_{i_a}h
    \end{multline*}
    for any $i_1<\cdots<i_{2a}$ with $i_j\in[N]$.
\end{lemma}

\begin{proof}
    We first compute the action of $\lollipopa$. In fact, for $i_1<\cdots<i_{2a}$ we have
    \begin{align*}
        &\lollipopa \Big(1\otimes h\otimes v_{i_1}\wedge\cdots\wedge v_{i_{2a}}\Big)\\
        =& \frac{1}{2^a}
        \begin{tikzpicture}
            [anchorbase,color=\clr]
            \draw[line width=1pt] (0,0) -- (0,.4) to [out=up,in=left] (.2,.6)
            to [out=right,in=up] (.4,.4) -- (.4,.0);
            \draw (0,.2)\bdot;
            \node at (0,-.2){$\scriptstyle a$};
            \node at (.4,-.2){$\scriptstyle a$};
         \end{tikzpicture}
          \Big( 1\otimes h \otimes \sum_{w\in S_{2a}/(S_a\times S_a)} (-1)^{\ell(w)}v_{i_{w(1)}}\wedge \cdots 
          \wedge v_{i_{w(a)}} \otimes
            v_{i_{w(a+1)}}\wedge \cdots 
            \wedge v_{i_{w(2a)}}\Big) \\
         \overset{\cref{actionlantern}}{\equiv}&
        \frac{1}{2^a} \capa \Big(
            \sum_{w\in S_{2a}/(S_a\times S_a)}(-1)^{\ell(w)}
            1\otimes e_a(h_{i_{w(1)}},\ldots,h_{i_{w(a)}})h \\
            &\qquad\qquad \otimes
            v_{i_{w(1)}}\wedge \cdots 
          \wedge v_{i_{w(a)}} \otimes
            v_{i_{w(a+1)}}\wedge \cdots 
            \wedge v_{i_{w(2a)}}
         \Big) .
    \end{align*}
 By \cref{caparep}, the above equation evaluates to zero unless $ i_j = -i_{2a+1-j} $ for all $ j = 1, \dots, a $. In this case, a nonzero summand appears if and only if $ i_{w(j)} = -i_{w(2a+1-j)} $ for all $ j = 1, \dots, a $. There are precisely $ 2^a $ such permutations $ w $ in $ S_{2a} / (S_a \times S_a) $.  
For these values of $ w $, we have  
$
\capa \Big(
     v_{i_{w(1)}}\wedge \cdots 
          \wedge v_{i_{w(a)}} \otimes
            v_{i_{w(a+1)}}\wedge \cdots 
            \wedge v_{i_{w(2a)}}
    \Big) = 1
$
by \cref{caparep}. %
Moreover, suppose each such $ w $ changes exactly $ n_w$ indices $ i_j $ to $ -i_j $. Every such $ w $ is generated by $\{s_{a}, s_{a-1}s_{a+1}s_a,\ldots,(s_1\cdots s_{a-1})(s_{2a-1}\cdots s_{a+1})s_a \}$. Thus we have $n_w\equiv \ell(w) \mod 2$, leading to the relation  
\[
e_a(h_{i_{w(1)}},\ldots,h_{i_{w(a)}}) = (-1)^{\ell(w)} e_a(h_{i_1},\ldots,h_{i_a}).
\]  
Thus, we conclude that  
\[
\lollipopa \Big( 1\otimes h\otimes v_{i_1}\wedge\cdots\wedge v_{i_{2a}} \Big) =    \prod_{j=1}^a \delta_{i_j,-i_{2a+1-j}}\otimes h_{i_1}\cdots h_{i_{a}}h.
\]

Now suppose $ i_j = -i_{2a+1-j} $ for all $ j = 1, \dots, a $, then we have
    \begin{multline*}
            \wkdotawmue \Big(
        1\otimes h \otimes v_{i_1}\wedge \cdots \wedge v_{i_a}\wedge v_{-i_a}\wedge \cdots \wedge v_{-i_1}
    \Big)
    \overset{\cref{actionlantern}}{\equiv} 1\otimes e_\mu(h_{i_1},\ldots, h_{i_a},h_{-i_a},\ldots, h_{-i_1})h .
    \end{multline*}
    Suppose $\mu=(\mu_1,\ldots,\mu_k)$. Since $\mu\in \EPar_{2a}$, we have $\mu_t\in 2\N$ and
    \[
    e_{\mu_t}(h_{i_1},\ldots, h_{i_a},h_{-i_a},\ldots, h_{-i_1})=(-1)^{\mu_t/2}e_{\mu_t/2}(h^2_{i_1},\ldots, h^2_{i_a}),\quad \forall t=1,\ldots,k.
    \]
    This proves the lemma.
\end{proof}
Similarly to \cref{actionbubble2a} and \cref{actionlollipop}, one can show directly that 
\begin{equation}
    \label{actionlollipoprev}
    \wkdotamurev(1\otimes h)=\sum_{i_1<\cdots <i_a}1\otimes(-1)^{|\mu|/2}e_{\mu/2}(h_{i_1}^2,\ldots,h_{i_a}^2)h_{i_1}\cdots h_{i_a}h\otimes
    v_{i_1} \wedge \cdots \wedge v_{i_a} \wedge
    v_{-i_a} \wedge \cdots \wedge v_{-i_1}
\end{equation}
for any $a\in \N$, $\mu\in \EPar_{2a}$ and $h\in \U(\mathfrak h)$.

\subsection{A basis theorem for $\Hom_{\AWb}(\mu,\lambda)$}

Recall from \cref{thm:spanAWb} that $\bPSMat_{\lambda,\mu}$ from \eqref{eq:BPMat} is a spanning set for $\Hom_{\AWb}(\mu,\lambda)$. 

\begin{theorem} \label{thm:basisAWb}
As a $\mathbb J'$-module, $\Hom_{\AWb}(\mu,\lambda)$ is free with basis $\bPSMat_{\lambda,\mu}$ for any $\lambda, \mu\in \Lambda_{\text{st}}$.
\end{theorem}

To complete the proof of \cref{thm:basisAWb} it remains to establish the linear independence of $\bPSMat_{\lambda,\mu}$ in $\Hom_{\AWb}(\mu,\lambda)$. 
The proof of this independence can be found in Appendix~\ref{app:independence}.

Recall $\SMat_{\lambda,\mu}$ from \eqref{eq:Mat}, and denote $\bSMat_{\lambda,\mu} :=\{\bubblea\cdot f \mid f\in \SMat_{\lambda,\mu}, a\in \Z_{\geq 0}\}.$ Recall also that $\bubblea=\binom{x}{a}$ from \eqref{bubblecap}.

\begin{corollary} \label{cor:basis:Wb}
    As a $\mathbb J$-module, the morphism space $\Hom_{\Wb}(\mu,\lambda)$ has a basis $\SMat_{\lambda,\mu}$. Equivalently, $\Hom_{\Wb}(\mu,\lambda)$ has a $\Z$-basis
    $\bSMat_{\lambda,\mu}.$
\end{corollary}

\begin{proof}
Consider the base change $\Wb_{\mathbb J'} =\Wb\otimes_{\mathbb J} {\mathbb J'}$. 
    There is a natural functor $\imath: \Wb_{\mathbb J'}\rightarrow\AWb$ which sends the  generating objects and morphisms to the ones in the same notation, since the defining relations for $\Wb$ are part of those for $\AWb$.

    The images of the diagrams in $\SMat_{\lambda,\mu}$ under
    $\iota$ are precisely the bubble-free elementary diagrams, hence form a subset of the $\mathbb J'$-basis $\bPSMat_{\lambda,\mu}$. Therefore $\SMat_{\lambda,\mu}$ is $\mathbb J'$-linearly independent after base change. Since $\mathbb J\hookrightarrow \mathbb J'$ is injective, it is already $\mathbb J$-linearly independent in $\Wb$. The spanning argument is the degree-zero specialization of \cref{sec:spanning}. Finally, because $\{\binom{x}{r}\mid r\ge0\}$ is a $\mathbb Z$-basis of $\mathbb J$, the set $\bSMat_{\lambda,\mu}$ is a $\mathbb Z$-basis.
\end{proof}

\section{Affine B-webs over $\C[x]$}
  \label{sec:char0}
  
In this section, we formulate a simplified presentation of the finite and affine B-web categories over $\kk[x]$, for any field $\kk$ of characteristic zero. We simply fix $\kk=\C$. 

\subsection{Definition of $\WCb$ over $\C[x]$}

\begin{definition}
\label{def:affineBwebC}
  The B-web category $\WCb$ is the strict $\C[x]$-linear monoidal category with generating objects $a\in \mathbb Z_{\ge 1}$ and generating morphisms 
  \begin{align}
\label{morphismsC}
\begin{tikzpicture}[anchorbase,color=\clr]
	\draw[-,line width=1pt] (0.28,-.3) to (0.08,0.04);
	\draw[-,line width=1pt] (-0.12,-.3) to (0.08,0.04);
	\draw[-,line width=1.4pt] (0.08,.4) to (0.08,0);
        \node at (-0.22,-.4) {$\scriptstyle k$};
        \node at (0.35,-.4) {$\scriptstyle 1$};\node at (0,.55){$\scriptstyle k+1$};\end{tikzpicture} 
&:(k,1) \rightarrow (k+1),&
\begin{tikzpicture}[anchorbase,color=\clr]
	\draw[-,line width=1.4pt] (0.08,-.3) to (0.08,0.04);
	\draw[-,line width=1pt] (0.28,.4) to (0.08,0);
	\draw[-,line width=1pt] (-0.12,.4) to (0.08,0);
        \node at (-0.22,.5) {$\scriptstyle k$};
        \node at (0.36,.5) {$\scriptstyle 1$};
        \node at (0.1,-.45){$\scriptstyle k+1$};
\end{tikzpicture}
&:(k+1)\rightarrow (k,1),&
\begin{tikzpicture}[baseline=-1mm, scale=.8, color=\clr]
	\draw[-,thick] (0.3,-.3) to (-.3,.4);
	\draw[-,thick] (-0.3,-.3) to (.3,.4);
        \node at (-0.3,-.45) {$\scriptstyle 1$};
        \node at (0.26,-.45) {$\scriptstyle 1$};
\end{tikzpicture}
 & :(1,1) \rightarrow (1,1),
 \\
\begin{tikzpicture}[xscale=-1,anchorbase, scale=0.3, color=\clr]
\draw[-,thick] (1,1) to [out=down,in=left] (2,0) to [out=right,in=down] (3,1);
\node at (1,1.5){$\scriptstyle 1$};
\node at (3,1.5){$\scriptstyle 1$};
\end{tikzpicture} &: \unit \rightarrow (1,1), &
\qquad 
\begin{tikzpicture}[xscale=-1,anchorbase, scale=0.3, color=\clr]
\draw[-,thick] (0,0) to [out=up,in=left] (1,1) to [out=right,in=up] (2,0);
\node at (0,-.5){$\scriptstyle 1$};
\node at (2,-.5){$\scriptstyle 1$};
\end{tikzpicture} & : (1,1) \rightarrow \unit, &
 & \;
 \notag
\end{align}
for $k \in \Z_{\ge 1}$,
subject to the relations \cref{webassocC}--\cref{bubbleC}:
\begin{gather}
\label{webassocC}
\begin{tikzpicture}[baseline = 0,color=\clr]
	\draw[-,thick] (0.35,-.3) to (0.08,0.14);
	\draw[-,thick] (0.1,-.3) to (-0.04,-0.06);
	\draw[-,line width=1pt] (0.085,.14) to (-0.035,-0.06);
	\draw[-,thick] (-0.2,-.3) to (0.07,0.14);
	\draw[-,line width=1.4pt] (0.08,.45) to (0.08,.1);
        \node at (0.45,-.41) {$\scriptstyle 1$};
        \node at (0.07,-.4) {$\scriptstyle 1$};
        \node at (-0.28,-.41) {$\scriptstyle k$};
\end{tikzpicture}
=
\begin{tikzpicture}[baseline = 0, color=\clr]
	\draw[-,thick] (0.36,-.3) to (0.09,0.14);
	\draw[-,thick] (0.06,-.3) to (0.2,-.05);
	\draw[-,line width=1pt] (0.07,.14) to (0.19,-.06);
	\draw[-,thick] (-0.19,-.3) to (0.08,0.14);
	\draw[-,line width=1.4pt] (0.08,.45) to (0.08,.1);
        \node at (0.45,-.41) {$\scriptstyle 1$};
        \node at (0.07,-.4) {$\scriptstyle 1$};
        \node at (-0.28,-.41) {$\scriptstyle k$};
\end{tikzpicture}\:,
\qquad \quad \;
\begin{tikzpicture}[anchorbase, color=\clr]
	\draw[-,thick] (0.35,.3) to (0.08,-0.14);
	\draw[-,thick] (0.1,.3) to (-0.04,0.06);
	\draw[-,line width=1pt] (0.085,-.14) to (-0.035,0.06);
	\draw[-,thick] (-0.2,.3) to (0.07,-0.14);
	\draw[-,line width=1.4pt] (0.08,-.45) to (0.08,-.1);
        \node at (0.45,.4) {$\scriptstyle 1$};
        \node at (0.07,.42) {$\scriptstyle 1$};
        \node at (-0.28,.4) {$\scriptstyle k$};
\end{tikzpicture}
= 
\begin{tikzpicture}[anchorbase, color=\clr]
	\draw[-,thick] (0.36,.3) to (0.09,-0.14);
	\draw[-,thick] (0.06,.3) to (0.2,.05);
	\draw[-,line width=1pt] (0.07,-.14) to (0.19,.06);
	\draw[-,thick] (-0.19,.3) to (0.08,-0.14);
	\draw[-,line width=1.4pt] (0.08,-.45) to (0.08,-.1);
        \node at (0.45,.4) {$\scriptstyle 1$};
        \node at (0.07,.42) {$\scriptstyle 1$};
        \node at (-0.28,.4) {$\scriptstyle k$};
\end{tikzpicture}\:,
\\
 \label{mergesplitC}
 \begin{tikzpicture}[anchorbase,scale=.8,color=\clr]
\draw[-,line width=1.6pt] (0.08,-.8) to (0.08,-.5);
\draw[-,line width=1.6pt] (0.08,.3) to (0.08,.6);
\draw[-,thick] (0.1,-.51) to [out=45,in=-45] (0.1,.31);
\draw[-,thick] (0.06,-.51) to [out=135,in=-135] (0.06,.31);
\node at (-.33,-.05) {$\scriptstyle k$};
\node at (.45,-.05) {$\scriptstyle 1$};
\end{tikzpicture}
= 
(k+1) \:
\begin{tikzpicture}[anchorbase,scale=.8,color=\clr]
	\draw[-,line width=1.4pt] (0.08,-.8) to (0.08,.6);
        \node at (.08,-1) {$\scriptstyle k+1$};
\end{tikzpicture}, 
\qquad \quad
    \begin{tikzpicture}[anchorbase,scale=.8, color=\clr]
	\draw[-,thick] (0,0) to (.28,.3) to (.28,.7) to (0,1);
	\draw[-,thick] (.6,0) to (.31,.3) to (.31,.7) to (.6,1);
        \node at (0,1.15) {$\scriptstyle k$};
        \node at (0.63,1.15) {$\scriptstyle 1$};
        \node at (0,-.15) {$\scriptstyle k$};
        \node at (0.63,-.15) {$\scriptstyle 1$};
\end{tikzpicture}
    =
    \begin{tikzpicture}[anchorbase,scale=.8, color=\clr]
        \draw[-,thick] (0,0) -- (0,1);
        \draw[-,thick] (.6,0) -- (.6,1);
                \node at (0,1.15) {$\scriptstyle k$};
        \node at (0.63,1.15) {$\scriptstyle 1$};
        \node at (0,-.15) {$\scriptstyle k$};
        \node at (0.63,-.15) {$\scriptstyle 1$};
    \end{tikzpicture}
    ~+~
        \begin{tikzpicture}[anchorbase,scale=.8, color=\clr]
        \draw[-,thick] (.6,0) -- (0,1);
        \draw[-,thick] (0,0) -- (.6,1);
                \draw[-,line width=1.4pt] (0,-.1) -- (0,1.1);
                \node at (0,1.25) {$\scriptstyle k$};
        \node at (0.63,1.15) {$\scriptstyle 1$};
        \node at (0,-.25) {$\scriptstyle k$};
        \node at (0.63,-.15) {$\scriptstyle 1$};
    \end{tikzpicture} ,
\\
     \label{zigzagC}
\begin{tikzpicture}[anchorbase,scale=.8,color=\clr]
        \draw[-,thick] (0,-0.5) to (0,0) to
        [out=up,in=left] (.25,.25) to [out=right,in=up] (.5,0) 
        to [out=down,in=left] (.75,-0.25) to [out=right,in=down] (1,0)
        to (1,.5);
        \node at (0,-.65){$\scriptstyle{1}$};
    \end{tikzpicture}
=
       \begin{tikzpicture}[anchorbase,scale=.8,color=\clr]
        \draw[-,thick] (0,-.5) to (0,.5);
        \node at (0,-.65){$\scriptstyle{1}$};
    \end{tikzpicture} ,
    \qquad
\begin{tikzpicture}[anchorbase,scale=.8,color=\clr]
        \draw[-,thick] (0,-0.5) to (0,0) to
        [out=up,in=right] (-.25,.25) to [out=left,in=up] (-.5,0) 
        to [out=down,in=right] (-.75,-0.25) to [out=left,in=down] (-1,0)
        to (-1,.5);
        \node at (0,-.65){$\scriptstyle{1}$};
    \end{tikzpicture}
    =
       \begin{tikzpicture}[anchorbase,scale=.8,color=\clr]
        \draw[-,thick] (0,-.5) to (0,.5);
        \node at (0,-.65){$\scriptstyle{1}$};
    \end{tikzpicture},
      \qquad
	\begin{tikzpicture}[anchorbase,color=\clr]
		\draw[-,thick] (0.6,-0.3) to[out=140, in=0] (0.1,0.2);
		\draw[-,thick] (0.1,0.2) to[out = -180, in = 90] (-0.2,-0.3);
        \draw[-,thick](.6,.4) to (.1,-.3);
        \node at (-.2,-.45){$\scriptstyle{1}$};
        \node at (.6,-.45){$\scriptstyle{1}$};
        \node at (.1,-.45){$\scriptstyle{1}$};
	\end{tikzpicture}
	 =
	\begin{tikzpicture}[anchorbase,color=\clr]
		\draw[-,thick] (0.3,-0.3) to[out=90, in=0] (0,0.2);
		\draw[-,thick] (0,0.2) to[out = -180, in = 40] (-0.5,-0.3);
         \draw[-,thick](-.5,.4) to (0,-.3);
        \node at (-.5,-.45){$\scriptstyle{1}$};
        \node at (.3,-.45){$\scriptstyle{1}$};
        \node at (0,-.45){$\scriptstyle{1}$};
	\end{tikzpicture},
\\
\label{bubbleC}
   \begin{tikzpicture}[anchorbase, scale=1, color=\clr]
 \draw[thick] (-0.15,0) arc(180:0:0.15) to[out=down,in=60] (-0.15,-0.4);
        \draw[thick] (0.15,-0.4) to[out=120,in=down] (-0.15,0);
        \node at (-.15,-.5) {$\scriptstyle 1$};
        \node at (.15,-.5) {$\scriptstyle 1$};
\end{tikzpicture}
= -\!\begin{tikzpicture}
    [anchorbase,color=\clr]
    \draw[-,thick] (0,0) to [out=up,in=left] (.2,.2) to [out=right,in=up] (.4,0);
    \node at (0,-.15){$\scriptstyle 1$};
    \node at (.4,-.15){$\scriptstyle 1$};
 \end{tikzpicture} ,
 \qquad
\begin{tikzpicture}[anchorbase,scale=.8,color=\clr]
\draw[-,line width=1.6pt] (0.1,.3) to (0.1,.6);
\draw[-,thick] (0.1,-.51) to [out=right,in=-45] (0.1,.31);
\draw[-,thick] (0.1,-.51) to [out=left,in=-135] (0.1,.31);
\node at (-.33,-.05) {$\scriptstyle 1$};
\node at (.45,-.05) {$\scriptstyle 1$};
\end{tikzpicture}
 =0= 
\begin{tikzpicture}[yscale=-1, anchorbase,scale=.8,color=\clr]
\draw[-,line width=1.6pt] (0.1,.3) to (0.1,.6);
\draw[-,thick] (0.1,-.51) to [out=right,in=-45] (0.1,.31);
\draw[-,thick] (0.1,-.51) to [out=left,in=-135] (0.1,.31);
\node at (-.33,-.05) {$\scriptstyle 1$};
\node at (.45,-.05) {$\scriptstyle 1$};
\end{tikzpicture},
\qquad 
\bubble 
    =x.
\end{gather}

The affine B-web category $\AWCb$ is the strict $\C[x]$-linear monoidal category with the same generating objects $a\in \mathbb Z_{\ge 1}$ as for $\WCb$ and one extra generating morphism 
$
\begin{tikzpicture}[anchorbase, scale=0.5, color=\clr]
\draw[-,thick] (0,0) to[out=up, in=down] (0,1.4);
\draw(0,0.6) \bdot;
\node at (0,-.3) {$\scriptstyle 1$};
\end{tikzpicture}
$
besides \eqref{morphismsC}, subject to relations \eqref{webassocC}--\eqref{bubbleC} and the following additional relations: 
\begin{gather}
\label{dotmovecapcupC}
\begin{tikzpicture}[anchorbase, scale=0.5, color=\clr]
\draw[-,thick] (0,-.2) to  (1,1.6);
\draw[-,thick] (1,-.2) to  (0,1.6);
\draw(0.25,1.2)\bdot;
\node at (0, -.5) {$\scriptstyle 1$};
\node at (1, -.5) {$\scriptstyle 1$};
\end{tikzpicture}
-
\begin{tikzpicture}[anchorbase, scale=0.5, color=\clr]
\draw[-, thick] (0,-.2) to (1,1.6);
\draw[-,thick] (1,-.2) to (0,1.6);
\draw(.8,0.2)\bdot;
\node at (0, -.5) {$\scriptstyle 1$};
\node at (1, -.5) {$\scriptstyle 1$};
\end{tikzpicture}
~=~
\begin{tikzpicture}[anchorbase, scale=0.5, color=\clr]
\draw[-,thick] (0,0) to[out=up, in=down] (0,1.6);
\node at (0,-.3) {$\scriptstyle 1$};
\draw[-,thick] (1,0) to[out=up, in=down] (1,1.6);
\node at (1,-.3) {$\scriptstyle 1$};
\end{tikzpicture}
 ~-~
\begin{tikzpicture}[anchorbase, scale=0.5, color=\clr]
\draw[-,thick] (2,1.6) to[out=down,in=right] (1.5,1.1) to[out=left,in=down] (1,1.6);
\draw[-,thick] (2,0) to[out=up, in=right] (1.5,0.5) to[out=left,in=up] (1,0);
\node at (1,-.3) {$\scriptstyle 1$};
\node at (1,1.9) {$\scriptstyle 1$};
\node at (2,-.3) {$\scriptstyle 1$};
\node at (2,1.9) {$\scriptstyle 1$};
\end{tikzpicture},
\qquad
\begin{tikzpicture}[anchorbase, scale=0.3, color=\clr]
\draw[-,thick] (0,0) to [out=up,in=left] (1,1) to [out=right,in=up] (2,0);
\node at (0,-.5){$\scriptstyle 1$};
\node at (2,-.5){$\scriptstyle 1$};
\draw (.25,.58)\bdot;
\end{tikzpicture}
= -
\begin{tikzpicture}[anchorbase, scale=0.3, color=\clr]
\draw[-,thick] (0,0) to [out=up,in=left] (1,1) to [out=right,in=up] (2,0);
\node at (0,-.5){$\scriptstyle 1$};
\node at (2,-.5){$\scriptstyle 1$};
\draw (1.75,.58)\bdot;
\end{tikzpicture}\,\,, 
\qquad
\begin{tikzpicture}[yscale=-1,anchorbase, scale=0.3, color=\clr]
\draw[-,thick] (0,0.5) to [out=up,in=left] (1,1.5) to [out=right,in=up] (2,0.5);
\node at (0,-.1){$\scriptstyle 1$};
\node at (2,-.1){$\scriptstyle 1$};
\draw (.1,.88)\bdot;
\end{tikzpicture}
=-
\begin{tikzpicture}[yscale=-1,anchorbase, scale=0.3, color=\clr]
\draw[-,thick] (0,.5) to [out=up,in=left] (1,1.5) to [out=right,in=up] (2,0.5);
\node at (0,0){$\scriptstyle 1$};
\node at (2,0){$\scriptstyle 1$};
\draw (1.95,.88)\bdot;
\end{tikzpicture} \;\; .
\end{gather}  
\end{definition}

Note the relations \eqref{webassocC}--\eqref{bubbleC} above form part of the relations \eqref{webassoc}--\eqref{bubblecap} for $\Wb$, while the relations \eqref{dotmovecapcupC} form part of \eqref{dotmovecrossing1}--\eqref{dotmovecupcaps} for $\AWCb$. There is a natural functor $\imath: \WCb \rightarrow \AWCb$, which identifies the generating objects and morphisms denoted by the same symbols. 

Using the zigzag relation in \cref{zigzagC} we derive that 
\[
\begin{tikzpicture}[anchorbase,yscale=-1,color=\clr]
		\draw[-,thick] (0.6,-0.3) to[out=140, in=0] (0.1,0.2);
		\draw[-,thick] (0.1,0.2) to[out = -180, in = 90] (-0.2,-0.3);
        \draw[-,thick](.6,.4) to (.1,-.3);
        \node at (-.2,-.45){$\scriptstyle{1}$};
        \node at (.6,-.45){$\scriptstyle{1}$};
        \node at (.1,-.45){$\scriptstyle{1}$};
	\end{tikzpicture}
	 =
	\begin{tikzpicture}[anchorbase,yscale=-1,color=\clr]
		\draw[-,thick] (0.3,-0.3) to[out=90, in=0] (0,0.2);
		\draw[-,thick] (0,0.2) to[out = -180, in = 40] (-0.5,-0.3);
         \draw[-,thick](-.5,.4) to (0,-.3);
        \node at (-.5,-.45){$\scriptstyle{1}$};
        \node at (.3,-.45){$\scriptstyle{1}$};
        \node at (0,-.45){$\scriptstyle{1}$};
	\end{tikzpicture},
    \qquad
     \begin{tikzpicture}[anchorbase, yscale=-1, color=\clr]
 \draw[-,thick] (-0.15,0) arc(180:0:0.15) to[out=down,in=60] (-0.15,-0.4);
        \draw[-,thick] (0.15,-0.4) to[out=120,in=down] (-0.15,0);
        \node at (-.15,-.5) {$\scriptstyle 1$};
        \node at (.15,-.5) {$\scriptstyle 1$};
\end{tikzpicture}
= -\!
\begin{tikzpicture}[yscale=-1,anchorbase, scale=0.3, color=\clr]
\draw[-,thick] (0,.5) to [out=up,in=left] (1,1.5) to [out=right,in=up] (2,0.5);
\node at (0,0){$\scriptstyle 1$};
\node at (2,0){$\scriptstyle 1$};
\end{tikzpicture}.
\]
Composing the first relation in \cref{dotmovecapcupC} with $\begin{tikzpicture}[baseline=-.5mm,scale=.8,color=\clr]
	\draw[-,thick] (-0.3,-.3) to (.3,.4);
	\draw[-,thick] (0.3,-.3) to (-.3,.4);
        \node at (0.3,-.4) {$\scriptstyle 1$};
        \node at (-0.3,-.4) {$\scriptstyle 1$};
         \node at (0.3,.55) {$\scriptstyle 1$};
        \node at (-0.3,.55) {$\scriptstyle 1$};
\end{tikzpicture}$ both on top and bottom, we obtain
\begin{align}
\begin{tikzpicture}[anchorbase, scale=0.5, color=\clr]
\draw[-, thick] (0,-.2) to (1,1.6);
\draw[-,thick] (1,-.2) to(0,1.6);
\draw(.2,0.2)\bdot;
\node at (0, -.5) {$\scriptstyle 1$};
\node at (1, -.5) {$\scriptstyle 1$};
\end{tikzpicture}
-
\begin{tikzpicture}[anchorbase, scale=0.5, color=\clr]
\draw[-,thick] (0,-.2) to  (1,1.6);
\draw[-,thick] (1,-.2) to  (0,1.6);
\draw(0.75,1.2)\bdot;
\node at (0, -.5) {$\scriptstyle 1$};
\node at (1, -.5) {$\scriptstyle 1$};
\end{tikzpicture}
=~
\begin{tikzpicture}[anchorbase, scale=0.5, color=\clr]
\draw[-,thick] (0,0) to[out=up, in=down] (0,1.6);
\node at (0,-.3) {$\scriptstyle 1$};
\draw[-,thick] (1,0) to[out=up, in=down] (1,1.6);
\node at (1,-.3) {$\scriptstyle 1$};
\end{tikzpicture}
 ~-~
\begin{tikzpicture}[anchorbase, scale=0.5, color=\clr]
\draw[-,thick] (2,1.6) to[out=down,in=right] (1.5,1.1) to[out=left,in=down] (1,1.6);
\draw[-,thick] (2,0) to[out=up, in=right] (1.5,0.5) to[out=left,in=up] (1,0);
\node at (1,-.3) {$\scriptstyle 1$};
\node at (1,1.9) {$\scriptstyle 1$};
\node at (2,-.3) {$\scriptstyle 1$};
\node at (2,1.9) {$\scriptstyle 1$};
\end{tikzpicture}.
\end{align}

Hence, by the symmetry of the defining relations \cref{webassocC}--\cref{dotmovecapcupC}, there exists a strict $\C[x]$-linear monoidal isomorphism
\[
  \div \colon \WCb \longrightarrow {\WCb}^{\operatorname{op}} ,
\qquad
\div \colon \AWCb \longrightarrow {\AWCb}^{\operatorname{op}} ,
\]
which sends $a\mapsto a$ on generating objects and reflects the generating morphism diagrams across a horizontal axis.

\subsection{Isomorphism of $\WCb$ and $\Wb$}

We view $\Wb$ as a $\C[x]$-linear category by a base change in this subsection. 
We define (thick) merges, splits, crossings, caps and cups in $\WCb$ and $\AWCb$ inductively by
\begin{align}
    \label{thickmergesplitC}
    \begin{tikzpicture}[anchorbase,color=\clr]
	\draw[-,line width=1pt] (0.28,-.3) to (0.08,0.04);
	\draw[-,line width=1pt] (-0.12,-.3) to (0.08,0.04);
	\draw[-,line width=1.4pt] (0.08,.4) to (0.08,0);
        \node at (-0.22,-.4) {$\scriptstyle a$};
        \node at (0.35,-.4) {$\scriptstyle b$};\node at (0,.55){$\scriptstyle a+b$};\end{tikzpicture}
      &  :=
    \frac{1}{b}
    \begin{tikzpicture}[anchorbase,color=\clr]
	\draw[-,line width=1pt] (0.4,-.6) to (0.08,0.04);
    \draw[-,line width=1pt] (.35,-.5) -- (-.05,-.2);
    \node at (.4,-.2){$\scriptstyle 1$};
	\draw[-,line width=1pt] (-0.2,-.6) to (0.08,0.04);
	\draw[-,line width=1.4pt] (0.08,.4) to (0.08,0);
        \node at (-0.2,-.7) {$\scriptstyle a$};
        \node at (0.5,-.7) {$\scriptstyle b$};\node at (0,.55){$\scriptstyle a+b$};\end{tikzpicture},
        \qquad
    \begin{tikzpicture}[anchorbase,color=\clr]
	\draw[-,line width=1.4pt] (0.08,-.3) to (0.08,0.04);
	\draw[-,line width=1pt] (0.28,.4) to (0.08,0);
	\draw[-,line width=1pt] (-0.12,.4) to (0.08,0);
        \node at (-0.22,.5) {$\scriptstyle a$};
        \node at (0.36,.5) {$\scriptstyle b$};
        \node at (0.1,-.45){$\scriptstyle a+b$};
\end{tikzpicture} 
:=
\frac{1}{b}
 \begin{tikzpicture}[anchorbase,color=\clr,yscale=-1]
	\draw[-,line width=1pt] (0.4,-.6) to (0.08,0.04);
    \draw[-,line width=1pt] (.35,-.5) -- (-.05,-.2);
    \node at (.4,-.2){$\scriptstyle 1$};
	\draw[-,line width=1pt] (-0.2,-.6) to (0.08,0.04);
	\draw[-,line width=1.4pt] (0.08,.4) to (0.08,0);
        \node at (-0.2,-.8) {$\scriptstyle a$};
        \node at (0.4,-.8) {$\scriptstyle b$};\node at (0,.55){$\scriptstyle a+b$};\end{tikzpicture},
        \qquad
        \begin{tikzpicture}
            [anchorbase,color=\clr,line width=1pt]
            \draw[-] (0,0) -- (.6,.6);
            \draw[-] (.6,0) -- (0,.6);
             \node at (0,-.15){$\scriptstyle a$};
        \node at (.6,-.15){$\scriptstyle b$};
        \end{tikzpicture}
        = \frac{1}{ab}\begin{tikzpicture}
        [anchorbase,color=\clr,line width=1pt]
        \draw (0,.6) -- (.1,.5) to[out=right,in=up] (.9,-.5)
              (.1,.5) to[out=down,in=left] (.9,-.5) -- (1,-.6);
              (.9,.5) to[out=down,in=right] (.1,-.5) -- (0,-.6);
        \draw (1,.6) -- (.9,.5) to[out=left,in=up] (.1,-.5)
              (.9,.5) to[out=down,in=right] (.1,-.5) -- (0,-.6);
        \node at (0,-.75){$\scriptstyle a$};
        \node at (1,-.75){$\scriptstyle b$};
        \node at (0.3,-.6){$\scriptstyle 1$};
        \node at (1.1,-.3){$\scriptstyle 1$};
    \end{tikzpicture} , 
\\
\notag
 \capa &=
 \frac{1}{a}\;
 \begin{tikzpicture}[anchorbase,scale=0.2,color=\clr]
\draw [-,line width=1.4pt] (0,-.1) to (0,.75);
\draw [-,line width=1pt] (0,.75) to [out=30,in=270] (1,2.5);
\draw [-,line width=1pt] (0,.75) to [out=150,in=270] (-1,2.5); 
\draw [-,line width=1pt] (1,2.5) to [out=90,in=180] (2,3.5) to [out=0,in=90] (3,2.5);
\draw [-,line width=1pt] (-1,2.5) to [out=90,in=180] (2,5.5) to [out=0,in=90] (5,2.5);
\draw [-,line width=1.4pt] (4,-.1) to (4,.75);
\draw [-,line width=1pt] (4,.75) to [out=30,in=270] (5,2.5);
\draw [-,line width=1pt] (4,.75) to [out=150,in=270] (3,2.5);
\node at (2,2.8) {$\scriptstyle{1}$};
\node at (2,6.4) {$\scriptstyle{a-1}$};
\end{tikzpicture}, 
\qquad 
\cupa= 
\frac{1}{a} \;
\begin{tikzpicture}[anchorbase,scale=0.2,color=\clr,yscale=-1]
\draw [-,line width=1.4pt] (0,-.1) to (0,.75);
\draw [-,line width=1pt] (0,.75) to [out=30,in=270] (1,2.5);
\draw [-,line width=1pt] (0,.75) to [out=150,in=270] (-1,2.5); 
\draw [-,line width=1pt] (1,2.5) to [out=90,in=180] (2,3.5) to [out=0,in=90] (3,2.5);
\draw [-,line width=1pt] (-1,2.5) to [out=90,in=180] (2,5.5) to [out=0,in=90] (5,2.5);
\draw [-,line width=1.4pt] (4,-.1) to (4,.75);
\draw [-,line width=1pt] (4,.75) to [out=30,in=270] (5,2.5);
\draw [-,line width=1pt] (4,.75) to [out=150,in=270] (3,2.5);
\node at (2,2.8) {$\scriptstyle{1}$};
\node at (2,6.4) {$\scriptstyle{a-1}$};
\end{tikzpicture}.
\end{align}

\begin{proposition}
\label{prop:Wiso}
    The $\mathbb C[x]$-linear monoidal categories $\WCb$ and $\Wb$ are isomorphic by matching the generating objects and generating morphisms in the same symbols. 
\end{proposition}

The proof of \cref{prop:Wiso}, which amounts to verifying that all the defining relations of $\Wb$ hold in $\WCb$, can be found in Appendix \ref{app:Wiso}.

\subsection{Isomorphism of $\AWCb$ and $\AWb$}

By a base change, we view $\AWb$ as a $\C[x]$-linear category in this subsection. 
Recall the definitions of thick generating morphisms in $\AWCb$ from \eqref{thickmergesplitC}. 
For $r \ge 1$, we further define the thick dots:
\begin{align}
\label{newgendot}
  \begin{tikzpicture}[anchorbase, scale=0.5, color=\clr]
\draw[-,line width=1.4pt] (0,0) to[out=up, in=down] (0,1.4);
\draw(0,0.6) \bdot;
\node at (0,-.3) {$\scriptstyle r$};
\end{tikzpicture}~ 
:=   \frac{1}{r!} \,
\begin{tikzpicture}[baseline = 1.5mm, scale=.5, color=\clr]
\draw[-, line width=1.4pt] (0.5,2) to (0.5,2.5);
\draw[-, line width=1.4pt] (0.5,0) to (0.5,-.4);
\draw[-,thin]  (0.5,2) to[out=left,in=up] (-.5,1)
to[out=down,in=left] (0.5,0);
\draw[-,thin]  (0.5,2) to[out=left,in=up] (0,1)
 to[out=down,in=left] (0.5,0);      
\draw[-,thin] (0.5,0)to[out=right,in=down] (1.5,1)
to[out=up,in=right] (0.5,2);
\draw[-,thin] (0.5,0)to[out=right,in=down] (1,1) to[out=up,in=right] (0.5,2);
\node at (0.5,.6){$\scriptstyle \cdots$};
\draw (-0.5,1) \bdot; 
\draw (0,1) \bdot; 
\node at (0.5,-.6) {$\scriptstyle r$};
\draw (1,1) \bdot;
\draw (1.5,1) \bdot; 
\node at (-.22,0) {$\scriptstyle 1$};
\node at (1.2,0) {$\scriptstyle 1$};
\node at (.3,0.3) {$\scriptstyle 1$};
\node at (.7,0.3) {$\scriptstyle 1$};
\end{tikzpicture} \; .
\end{align}

\begin{theorem}
\label{thm:AWiso}
    The $\mathbb C[x]$-linear monoidal categories $\AWCb$ and $\AWb$ are isomorphic by matching the generating objects and generating morphisms in the same symbols. 
\end{theorem}

\begin{proof}
     We have a natural functor $\natural^\bullet: \AWCb \rightarrow \AWb$ which matches the generating objects and generating morphisms (see  \cref{thickmergesplitC,newgendot} for $\AWCb$) in the same symbols, extending the functor $\natural: \WCb \rightarrow \Wb$ from \cref{prop:Wiso}. Note that the defining relations for $\AWCb$ are part of the defining relations for $\AWb$, \cref{webassoc}--\cref{bubblecap} and \cref{dotmovecrossing1}--\cref{dotmovecupcaps}. In the proof of \cref{prop:Wiso} that $\natural$ is an isomorphism, we have already verified the relations \cref{webassoc}--\cref{bubblecap} in $\WCb$ (and hence in $\AWCb$). To show that $\natural^\bullet$ is an isomorphism, it suffices to verify the remaining relations \cref{dotmovecrossing1}--\cref{dotmovecupcaps} in $\AWCb$; the detailed verifications can be found in \cref{dotmovecrossingC,lem:dotmoveSM,dotmovethickcupcap} below, respectively.
\end{proof}

\begin{rem}
By specializations, we obtain $\mathbb C$-linear categories  $\WCb(\delta)$ and $\AWCb(\delta)$ from $\WCb$ and $\AWCb$ by specializing $x$ to any $\delta\in \C$. Similarly, we have $\C$-linear categories  $\Wb(\delta)$ and $\AWb(\delta)$ via specializations from $\Wb$ and $\AWb$. The same arguments also establish the isomorphisms  $\WCb(\delta)\cong \Wb(\delta)$ and $\AWCb(\delta)\cong \AWb(\delta)$. 
\end{rem}

\begin{lemma} 
\label{dotmovecrossingC}
The following relations hold in $\AWCb$: 
\begin{align}
\label{dotdowncross}
\sum_{0\leq t \leq \min(a,b)}t!~

.
\label{dotupcross}
\end{align}
\end{lemma}
\begin{proof}
    We only verify \cref{dotdowncross} since \cref{dotupcross} can be checked similarly.
 We proceed by induction on $b$ and $a$. The case for  $a=1=b$ is the first relation in \cref{dotmovecapcupC}.
Suppose that $b=1$ and  $a\ge 2$. 
We have 
\begin{equation}
\label{r=1}
\begin{aligned}
%
,
\quad \text{by inductive assumption on $a-1$}.
    \end{aligned}
\end{equation}

Suppose $b\ge 2$. We have 
\begin{align*}
 %
\end{align*}
which is equal to
\begin{align*}
&\overset{\cref{newgendot}\cref{splitmerge}}{\underset{\cref{sliders}\cref{swallows}}{=}}\frac{b-t}{b}
\sum t!~\:\:
%
 \quad \text{ by induction on $b$}
\end{align*}
\begin{align*}
&
=
\sum t!~
%
 .
\end{align*}
Using the result for $b=1$ we compute each summand $%
 $ of the last summation as follows. 
\begin{align*}
&
%
.
\end{align*}
Then \cref{dotdowncross} follows by summing up all the above together. 
\end{proof}

We define the following morphisms (in $\AWCb$)
\begin{equation}
\label{equ:defofnewdot}
 %
.   
\end{equation}

\begin{lemma}
\label{lem:dotmoveSM}
  The following identities hold, for $a,b\ge 1$:
    \begin{equation}
  \label{updotmovesplits+merge-simplify}
%
.
\end{equation}
     
\end{lemma}

\begin{proof}
We will only prove the first formula and skip a similar proof for the second one. We proceed by induction on $t=a+b$. For the base case $a=1=b$, we have 
    \begin{align*}
    %
 \ .
\end{align*}
We use Lemma \ref{dotmovecrossingC} twice to move the dots through the crossing and  obtain the first diagram as follows:
\begin{align*}
  %
.
\end{align*}
Substituting this formula into the calculation above and using Lemma \ref{Lem:newdotsss}, we have proved the first formula in the lemma. 
\end{proof}

\begin{lemma}
 \label{dotmovethickcupcap}  
The following relations hold in $\AWCb$, for $a\ge 1$:
\[ 
%
 \;.
\]
\end{lemma}

\begin{proof}
We prove the first formula only, and the second one can be proved similarly. The base case for $a=1$ is the second defining relation in \cref{dotmovecapcupC}. For $a>1$, we have 
    \[
    %
 \; .
\]
The lemma is proved.
\end{proof}
This concludes the proof of \cref{thm:AWiso}. 

\section{Bubble relations}
  \label{sec:thick bubbles}

In this section, we show that odd bubbles are explicit polynomials in even bubbles. This gives rise to admissibility conditions for any representation of $\AWb$ in which bubbles act as scalars.

\subsection{Some preparatory relations}

\begin{lemma}
   For $0\leq t\leq a$, we have
    \begin{equation}
    \label{bubblecross}
        %
 .
    \end{equation}
\end{lemma}
\begin{proof}
    Let $s=a-t$, then we have
    \begin{gather*}
         %
    \overset{\cref{bubblecap}}{=}0$ unless $t=r_1$, we have proved the identity \cref{bubblecross}.
\end{proof}

Recall by definition \eqref{bubble+dot+par} that $\dotbubble_a =\bubbledota$, for $a\ge 1$. 

\begin{lemma}
    For $0\leq t\leq a$, we have
    \begin{equation}
        \label{fanrecursion}
        %
.
    \end{equation}
\end{lemma}

\begin{proof}
By \cref{dotmovecrossing1}, we have the following identity:
\begin{equation}
    \label{dotmovesym}
   %
\overset{\cref{bubblecap}}{=}0
$ for any $1\leq k\leq \min\{t,a-t\}$. Thus we have
\begin{gather*}
    %
 .
\end{gather*}
Since $\binom{x-t}{a-t-k}=0$ if $k>a-t$ and
\begin{gather*}
    \binom{x-a+k}{k}\binom{x-t}{a-t-k}=\binom{x-t}{a-t}\binom{a-t}{k},
\end{gather*}
the desired identity \cref{fanrecursion} follows.
\end{proof}

\begin{lemma}
   For $0\leq t\leq a-1$, we have
   \begin{equation}
        \label{fanandbubble}
        %
=\binom{x-t}{a-t}\sum_{k=0}^t(-1)^k\binom{a-k-1}{t-k}\binom{x-k}{t-k}(t-k)!\dotbubble_k.
\end{equation}
\end{lemma}

\begin{proof}
We shall proceed via an induction on $a$. The case when $a=1$ is trivially true. 
    Assume for all $ a' < a $ and $ 0 \leq s < a' $, we have
\[
      %
= \binom{x - s}{a' - s} \sum_{k=0}^{s} (-1)^k \binom{a' - k - 1}{s - k} \binom{x - k}{s - k} (s - k)! \dotbubble_k.
\]
Plug into the recurrence \cref{fanrecursion}, we obtain
\begin{align*}
%
 \\
&= \binom{x-t}{a-t}\sum_{m=0}^{t} m! \binom{ a - t}{m}
\binom{x - t + m}{m}\cdot \\ &\qquad
\sum_{k = 0}^{t - m} (-1)^k \binom{t - k - 1}{t - m - k} \binom{x - k}{t - m - k} (t - m - k)! \dotbubble_k.
\end{align*}
We rearrange the sums such that
\begin{gather*}
    %
 = \binom{x-t}{a-t}\sum_{k = 0}^{t} (-1)^k \dotbubble_k 
    \sum_{m = 0}^{t - k} m! \binom{a - t }{m} \binom{x - t + m}{m}\cdot \\\binom{t - k - 1}{t - m - k} \binom{x - k}{t - m - k} (t - m - k)!\\
    = \binom{x-t}{a-t}\sum_{k = 0}^{t} (-1)^k \dotbubble_k (a-t)\binom{x-k}{t-k}(t-k-1)!
    \sum_{m=0}^{t-k}\binom{a-t-1}{m-1}\binom{t-k}{m}.
\end{gather*}
Thus, to show  \cref{fanandbubble} holds for $a$, it suffices to verify the following equation:
\begin{equation}
\label{fanbubblebinomial}
    (a-t)(t-k-1)!\sum_{m=0}^{t-k}\binom{a-t-1}{m-1}\binom{t-k}{m}=\binom{a-k-1}{t-k}(t-k)!.
\end{equation}
One checks directly that \cref{fanbubblebinomial} follows from (and is equivalent to) the following standard binomial identity:
\begin{equation*}
    \begin{gathered}
 \sum_{m=0}^{t-k}\binom{a-t-1}{m-1}\binom{t-k}{t-k-m}=\binom{a-k-1}{t-k-1}.
\end{gathered}
\end{equation*}
\end{proof}

\begin{lemma}
    For $a\geq 1$, we have
    \begin{equation}
        \label{bubblecrosssum}
        \sum_{t=0}^a t!\begin{tikzpicture}
        [anchorbase,color=\clr,scale=.7]
        \draw[-,thick] (0.6,-.2) to (0.6,0) to (-.5,.8) to (-.5,1);
\draw[-,thick] (-.5,-.2) to (-.5,0) to (.6,.8) to (.6,1);
\draw[-,thick] (-.5,-.2) to (-.5,1);
\draw[-,thick] (0.6,-.2) to (0.6,1);
\draw(.35,0.2)\bdot;
\draw[-,thick] (.6,1) to [out=up,in=up] (-.5,1);
\draw[-,thick] (.6,-.2) to [out=down,in=down] (-.5,-.2);
\node at (0,-.7) {$\scriptstyle a$};
\node at (0.77,.5) {$\scriptstyle t$};
\node at (-.75,.5) {$\scriptstyle t$};
    \end{tikzpicture}
    =
    \dotbubble_a+(x-a+1)\dotbubble_{a-1}.
    \end{equation}
\end{lemma}

\begin{proof}
   We have
    \begin{gather*}
        \sum_{t=0}^a t!\begin{tikzpicture}
        [anchorbase,color=\clr,scale=.7]
        \draw[-,thick] (0.6,-.2) to (0.6,0) to (-.5,.8) to (-.5,1);
\draw[-,thick] (-.5,-.2) to (-.5,0) to (.6,.8) to (.6,1);
\draw[-,thick] (-.5,-.2) to (-.5,1);
\draw[-,thick] (0.6,-.2) to (0.6,1);
\draw(.35,0.2)\bdot;
\draw[-,thick] (.6,1) to [out=up,in=up] (-.5,1);
\draw[-,thick] (.6,-.2) to [out=down,in=down] (-.5,-.2);
\node at (0,-.7) {$\scriptstyle a$};
\node at (0.77,.5) {$\scriptstyle t$};
\node at (-.75,.5) {$\scriptstyle t$};
    \end{tikzpicture}
    ~\overset{\cref{bubblecross}}{=}~
    \sum_{t=0}^a(-1)^{a-t}t! \begin{tikzpicture}
    [anchorbase,color=\clr,scale=.7]
    \draw[-,thick] (-.5,0) to [out=down,in=left] (0,-.5) to [out=right,in=down] (.5,0);
    \draw[-,thick] (-.5,0) -- (-.25,.25)
                   (-.5,0) -- (-.75,.25)
                   (.5,0) -- (.25,.25)
                   (.5,0) -- (.75,.25)
                   (-.25,.25) to [out=45,in=135] (.25,.25)
                   (-.75,.25) to [out=135,in=left]
                   (0,.75) to [out=right,in=45] (.75,.25);
    \node at (0,-.7) {$\scriptstyle a$};
    \node at (.95,.25){$\scriptstyle t$};
    \draw (.25,.25)\bdot;
\end{tikzpicture}
\overset{\cref{fanandbubble}}{=}\\\dotbubble_a+ \sum_{t=1}^a(-1)^{a-t}t!
\Bigg(
\binom{x-a+t}{t}\sum_{k=0}^{a-t}(-1)^k\binom{a-k-1}{a-t-k}\binom{x-k}{a-t-k}(a-t-k)!\dotbubble_k\Bigg) \\
~=~\dotbubble_a+\sum_{t=1}^a\sum_{k=0}^{a-t}(-1)^{a-t-k}{(a-k)!}{\binom{x-k}{a-k}}\binom{a-k-1}{t-1}\dotbubble_k\\
=\dotbubble_a   +\sum_{k=0}^{a-1}{(a-k)!}{\binom{x-k}{a-k}}\dotbubble_k\sum_{t=1}^{a-k}(-1)^{a-t-k}\binom{a-k-1}{t-1}.
    \end{gather*}
Since $\sum_{t=1}^{a-k}(-1)^{a-t-k}\binom{a-k-1}{t-1}=0$ unless $k=a-1$ and $t=1$, we conclude the desired identity \cref{bubblecrosssum}.
\end{proof}

\subsection{Relations between bubbles}

Recall $\dotbubble_0=1$ and  $\dotbubble_1=\binom{x}{2}$.

\begin{proposition}  \label{RelationBubble}
For any $a\geq 1$, we have
    \[
    \sum_{0\le t\le a}(-1)^{a-t}t!\binom{x-a+t}{t}^2 \dotbubble_{a-t} =\dotbubble_a+(x-a+1)\dotbubble_{a-1}.
    \]
\end{proposition}

\begin{proof}
  By \cref{dotmovecrossing1}, we have
  \[
  \begin{tikzpicture}[anchorbase, scale=0.4, color=\clr]
\draw[-,line width=1.2pt] (0,-.2) to  (1,2.2);
\draw[-,line width=1.2pt] (1,-.2) to  (0,2.2);
\draw(0.2,1.6)\bdot;
\node at (0, -.5) {$\scriptstyle a$};
\node at (1, -.5) {$\scriptstyle a$};
\end{tikzpicture}
= 
 \sum_{0 \leq t \leq a}t!~
\begin{tikzpicture}[anchorbase,scale=.8, color=\clr]
\draw[-,thick] (0.58,-.2) to (0.58,0) to (-.48,.8) to (-.48,1);
\draw[-,thick] (-.48,-.2) to (-.48,0) to (.58,.8) to (.58,1);
\draw[-,thick] (-.5,-.2) to (-.5,1);
\draw[-,thick] (0.605,-.2) to (0.605,1);
\draw(.35,0.2)\bdot;
\node at (-.5,-.3) {$\scriptstyle a$};
\node at (0.6,-.3) {$\scriptstyle a$};
\node at (0.77,.5) {$\scriptstyle t$};
\end{tikzpicture}
-
\sum_{1\leq t \leq a}t!~
\begin{tikzpicture}[anchorbase,scale=.8, color=\clr]
\draw[-,thick] (0.58,-.2) to (0.58,0) to (-.48,.8) to (-.48,1);
\draw[-,thick] (-.48,-.2) to (-.48,0) to (.58,.8) to (.58,1);
\draw[-,thick] (-.48,.95) to [out=-46,in=left](0.1,.7) to [out=right, in =45] (.58,.95);
\draw[-,thick] (-.48,.-.15) to [out=46,in=left](0.1,.1) to [out=right, in =-45] (.58,-.15);
\draw[-,thick] (-.5,-.2) to (-.5,0);
\draw[-,thick] (-.5,0.8) to (-.5,1);
\draw[-,thick] (0.605,-.2) to (0.605,0);
\draw[-,thick] (0.605,.8) to (0.605,1);
\draw(-.2,0.6)\bdot;
\node at (-.5,-.3) {$\scriptstyle a$};
\node at (0.6,-.3) {$\scriptstyle a$};
\node at (0.1,.9) {$\scriptstyle t$};
\node at (0.1,-.1) {$\scriptstyle t$};
\end{tikzpicture} .
  \]
  Composing the above equation with $\cupa$ and $\capa$, then we obtain
  \begin{gather*}
      (-1)^a\dotbubble_a=\sum_{t=0}^a t!\begin{tikzpicture}
        [anchorbase,color=\clr,scale=.7]
        \draw[-,thick] (0.6,-.2) to (0.6,0) to (-.5,.8) to (-.5,1);
\draw[-,thick] (-.5,-.2) to (-.5,0) to (.6,.8) to (.6,1);
\draw[-,thick] (-.5,-.2) to (-.5,1);
\draw[-,thick] (0.6,-.2) to (0.6,1);
\draw(.35,0.2)\bdot;
\draw[-,thick] (.6,1) to [out=up,in=up] (-.5,1);
\draw[-,thick] (.6,-.2) to [out=down,in=down] (-.5,-.2);
\node at (0,-.7) {$\scriptstyle a$};
\node at (0.77,.5) {$\scriptstyle t$};
\node at (-.75,.5) {$\scriptstyle t$};
    \end{tikzpicture}
    ~-\sum_{t=1}^at!~
    \begin{tikzpicture}[anchorbase,scale=.8, color=\clr]
\draw[-,thick] (0.58,-.2) to (0.58,0) to (-.48,.8) to (-.48,1);
\draw[-,thick] (-.48,-.2) to (-.48,0) to (.58,.8) to (.58,1);
\draw[-,thick] (-.48,.95) to [out=-46,in=left](0.1,.7) to [out=right, in =45] (.58,.95);
\draw[-,thick] (-.48,.-.15) to [out=46,in=left](0.1,.1) to [out=right, in =-45] (.58,-.15);
\draw[-,thick] (-.5,-.2) to (-.5,0);
\draw[-,thick] (-.5,0.8) to (-.5,1);
\draw[-,thick] (0.605,-.2) to (0.605,0);
\draw[-,thick] (0.605,.8) to (0.605,1);
\draw[-,line width=1.2pt] (.605,1) to [out=up,in=up] (-.48,1);
\draw[-,line width=1.2pt] (.605,-.2) to [out=down,in=down] (-.48,-.2);
\draw(-.2,0.6)\bdot;
\node at (0,-.7) {$\scriptstyle a$};
\node at (0.1,.9) {$\scriptstyle t$};
\node at (0.1,-.1) {$\scriptstyle t$};
\end{tikzpicture} \\
\overset{\cref{sidewaydumbell}}{=}\sum_{t=0}^a t!\begin{tikzpicture}
        [anchorbase,color=\clr,scale=.7]
        \draw[-,thick] (0.6,-.2) to (0.6,0) to (-.5,.8) to (-.5,1);
\draw[-,thick] (-.5,-.2) to (-.5,0) to (.6,.8) to (.6,1);
\draw[-,thick] (-.5,-.2) to (-.5,1);
\draw[-,thick] (0.6,-.2) to (0.6,1);
\draw(.35,0.2)\bdot;
\draw[-,thick] (.6,1) to [out=up,in=up] (-.5,1);
\draw[-,thick] (.6,-.2) to [out=down,in=down] (-.5,-.2);
\node at (0,-.7) {$\scriptstyle a$};
\node at (0.77,.5) {$\scriptstyle t$};
\node at (-.75,.5) {$\scriptstyle t$};
    \end{tikzpicture}
    ~-~\sum_{t=1}^a(-1)^{a-t}\binom{x-a+t}{t}^2\dotbubble_{a-t}.
  \end{gather*}
  Then the desired equation follows from \cref{bubblecrosssum} directly.
\end{proof}

From \cref{RelationBubble} it follows that $\dotbubble_{2m+1}$ is a linear combination of $\dotbubble_{2a}$ with $0\leq a\leq m$. We establish an explicit formula for such a linear combination.

\begin{theorem}  \label{bubbleOdd}
    For $m\in \N$, we have
    \begin{align}  \label{Delta:odd}
        \dotbubble_{2m+1} = \sum_{k=0}^m (-1)^k(2k)! d_k \binom{x-2m+2k}{2k+2} \binom{x-2m+2k-1}{2k} \dotbubble_{2m-2k},
    \end{align}
    where the scalars $d_k \in \Z_{>0}$ are determined by $d_0=1$ and the recursion
    \begin{align}
    \label{recursiondm}
        d_k = (-1)^k (k+1) + \frac12 \sum_{i=1}^k (-1)^{i-1}  \binom{2k+2}{2i}d_{k-i}. 
    \end{align}
\end{theorem}

For the proof of \cref{bubbleOdd} we prepare another recursion formula for the sequence $(d_k)_{k\ge 0}$. 

\begin{lemma}
\label{recursiondm2}
The following identity holds for $k\ge 1$:
    \begin{align} \label{identity2}
    \sum_{r=0}^{k-1}(-1)^r\binom{2k+1}{2r+2}d_r=2k+1.
\end{align}
\end{lemma}

\begin{proof}
    We prove the identity \eqref{identity2} by induction on $k$, with the base case $k=1$ being clear.

Let $k\ge 2$ and set $S(k) =\sum_{r=0}^{k-1}(-1)^r\binom{2k+1}{2r+2}d_r$ to be the left-hand side of \eqref{identity2}. Using \eqref{recursiondm}, we express $S(k)$ as a sum of two summands:
\begin{align*}
    S(k) &= \sum_{r=0}^{k-1}(-1)^r\binom{2k+1}{2r+2}
    \left((-1)^r (r+1) + \frac12 \sum_{i=1}^r (-1)^{i -1}  \binom{2r+2}{2i}d_{r-i} \right)
    \\
    &=\underbrace{\sum_{r=0}^{k-1}(r+1)\binom{2k+1}{2r+2}}_{S_1} + \underbrace{\frac12\sum_{r=0}^{k-1}(-1)^r\binom{2k+1}{2r+2}
     \sum_{i=1}^r (-1)^{i -1}  \binom{2r+2}{2i}d_{r-i}}_{S_2}.
\end{align*}
We have
\begin{align*}
    S_1 
    &= \sum_{r=0}^{k-1}\frac{2k+1}2 \binom{2k}{2r+1}
    =(2k+1) 2^{2k-2}.
\end{align*}
On the other hand, we have
\begin{align*}
    S_2 
&=\frac12\sum_{r=0}^{k-1}\sum_{i=1}^r (-1)^{r+i-1}\binom{2k+1}{2r+2}\binom{2r+2}{2i} d_{r-i}.
\end{align*}
Noting $\binom{2k+1}{2r+2}\binom{2r+2}{2i} =\binom{2k+1}{2k-2i+1}\binom{2k-2i+1}{2(r-i)+2}$, interchanging the double summation above and changing variables $j=r-i$, we obtain
\begin{align*}
    S_2 
&= -\frac12\sum_{i=1}^{k-1} \binom{2k+1}{2k-2i+1} \sum_{j=0}^{k-i-1}(-1)^{j}\binom{2(k-i)+1}{2j+2} d_{j}
\\
&=-\frac12\sum_{i=1}^{k-1} \binom{2k+1}{2k-2i+1} (2k-2i+1)
\\
&=-\frac12\sum_{i=1}^{k-1} (2k+1)\binom{2k}{2k-2i}
\\
&=-\frac12 (2k+1) \left(\sum_{i=0}^{k} \binom{2k}{2k-2i} -2\right) \\
&=-\frac12 (2k+1) (2^{2k-1}-2) 
=(2k+1)(1-2^{2k-2}).
\end{align*}
We have used the inductive assumption for the second equality above.

Summing the above formulas for $S_1$ and $S_2$, we have proved $S(k) =2k+1$, whence the identity \eqref{identity2}. 
\end{proof}

Now we are ready to derive \cref{bubbleOdd} from \cref{RelationBubble}.

\begin{proof} [Proof of \cref{bubbleOdd}]
We prove the identity \eqref{Delta:odd} by induction on $m$. For the base case $m=0$, we have $\dotbubble_1 =\binom{x}{2} \dotbubble_0$, which follows from \cref{RelationBubble} for $a=1$. 

Let $m\ge 1$. By the inductive assumption, we have
\begin{equation}
\label{bubblem-t}
      \dotbubble_{2m-2s+1} = \sum_{k=0}^{m-s} (-1)^k(2k)! d_k \binom{x-2m+2s+2k}{2k+2} \binom{x-2m+2s+2k-1}{2k} \dotbubble_{2m-2s-2k}.
\end{equation}
for $s=m,m-1,\ldots,1$. By \cref{RelationBubble} for $a=2m+1$ we have
\begin{equation}
\label{bubble2m+1}
    2\dotbubble_{2m+1}=(x-2m)(x-2m-1)\dotbubble_{2m}+\sum_{t=1}^{2m}(-1)^t(t+1)!\binom{x-2m+t}{t+1}^2\dotbubble_{2m-t}.
\end{equation}

 Substituting \cref{bubblem-t} into \cref{bubble2m+1}, we obtain that 
 \begin{align}  \label{Delta:odd2}
        \dotbubble_{2m+1} = \sum_{k=0}^m c_k \dotbubble_{2m-2k},
\end{align}
where
\begin{align*}
  {\textstyle c_k} &=  {\textstyle (2k+1)!\binom{x-2m+2k}{2k+1}^2+ }\\
   &\quad {\textstyle \sum_{s=0}^{k-1} (-1)^{s+1}(2s)!(2k-2s)! \binom{x-2m+2k-2s-1}{2k-2s}^2 \binom{x-2m+2k}{2s+2}\binom{x-2m+2k-1}{2s}d_{s}. }
\end{align*}
By comparing the coefficients of $\dotbubble_{2m-2k}$ in \eqref{Delta:odd2} and the desired identity \eqref{Delta:odd}, we reduce the proof of \eqref{Delta:odd} to verifying the following identity, for $k\ge 1$ (where we have changed variables $z=x-2m+2k$):
\begin{align} 
 \label{identity}
      &   (2k+1)!\binom{z}{2k+1}^2+ 
    \\
   & {\tiny  \sum_{s=0}^{k-1} (-1)^{s+1}(2s)!(2k-2s)! \binom{z-2s-1}{2k-2s}^2 \binom{z}{2s+2}\binom{z-1}{2s}d_{s} } 
    \notag \\
    &=2\cdot (-1)^k(2k)!\binom{z}{2k+2}\binom{z-1}{2k}d_k. 
    \notag 
\end{align}
After canceling the common factor $z(z-1)^2(z-2)^2\cdots (z-2k)^2 / (2k+2)!$, we see that \eqref{identity} is equivalent to the following identity
\begin{gather}
\label{keybinomial}
    (2k+2)z+  \sum_{s=0}^{k-1}(-1)^{s+1}(z-2s-1)\binom{2k+2}{2s+2}d_s=2\cdot (-1)^k(z-2k-1)d_k.
\end{gather}
It remains to prove the identity \eqref{keybinomial} on linear polynomials in $z$. The equality of the coefficients of $z$ on both sides of \eqref{keybinomial} is a simple variant of the identity \eqref{recursiondm} with $m=k$. 
Hence by comparing the constant terms, the proof of the identity \cref{keybinomial} is reduced to verifying
\begin{gather*}
    \sum_{s=0}^{k-1}(-1)^s(2s+1)\binom{2k+2}{2s+2}d_s=2(-1)^{k+1}(2k+1)d_k,
\end{gather*}
which is equivalent to (by plugging in \cref{recursiondm} to RHS)
\begin{gather}
\label{keybinomiald}
    \sum_{i=1}^k(-1)^ii\binom{2k+2}{2i}d_{k-i}=(-1)^k(2k+1)(k+1).
\end{gather}
By the fact that $2i\binom{2k+2}{2i}=(2k+2)\binom{2k+1}{2i-1}$, we see that
\cref{keybinomiald} reduces to  
\begin{gather*}
    \sum_{i=1}^k(-1)^i\binom{2k+1}{2i-1}d_{k-i}=(-1)^k(2k+1),
\end{gather*}
which is equivalent to the identity \eqref{identity2} established in \cref{recursiondm2}.

This proves the theorem. 
\end{proof}

\begin{rem}
\label{rem:Genocchi}
One computes $(d_0, d_1, d_2, d_3, d_4, d_5, \ldots) = (1, 1, 3, 17, 155, 2073, \ldots)$. Let $\{G_{2k},k\geq 1\}$ be the {\it unsigned Genocchi numbers} determined by the following generating function (see Wikipedia or The Online Encyclopedia of Integer Sequences A110501)
\[
G(t):=\sum_{n=1}^{\infty} (-1)^nG_{2n}\frac{t^{2n}}{(2n)!}=-t\tanh(\frac{t}{2})=-t\frac{e^t-1}{e^t+1}.
\]
They satisfy the recursion
\[
\sum_{j=0}^{k-1}(-1)^j\binom{2k+1}{2j+2}G_{2j+2}=2k+1.
\]
Comparing with \cref{identity2}, we see that $\{d_r\mid r\geq 0\}$ and $\{G_{2r+2}\mid r\geq 0\}$ satisfy the same recursion, and $d_0=G_2=1$. Hence we conclude that 
\[
d_r =G_{2r+2},
\qquad
\text{ for } r \ge 0.
\]
\end{rem}

\vspace{2mm}
%
{\bf Acknowledgement.} 
W.W. thanks Department of Mathematics and IMS at National University of Singapore, Academia Sinica and the NSTC in Taipei for support and providing excellent research environment. L.S. is partially supported by NSFC (Grant No. 12071346). W.W. is partially supported by DMS--2401351. Y.S. is partially supported by the Fields Institute and the Science and Technology Commission of Shanghai Municipality (No. 22DZ2229014).

\appendix

\section{Proofs of \cref{lollipopspan,thm:basisAWb,prop:Wiso}}
  \label{app:technical}

\subsection{Proof of \cref{lollipopspan}}
  \label{app:lollipop}

Recall from \eqref{eq:ballstick} that 
$
\lol_{a,\lambda} =
\begin{tikzpicture}[yscale=-1, anchorbase,scale=.8,color=\clr]
\draw[-,line width=1.4pt] (0.08,.3) to (0.08,.6);
\draw[-,thick] (0.1,-.51) to [out=right,in=-45] (0.1,.31);
\draw[-,thick] (0.1,-.51) to [out=left,in=-135] (0.1,.31);
\node at (0.08,.8) {$\scriptstyle 2a$};
\draw (-.1,-.2)\bdot;
\node at (-.4,-.2) {$\scriptstyle \omega_\lambda$};
\end{tikzpicture}
$
for $\lambda\in \Par_a$ and $\venus_{2a}$ denotes the subspace of $\Hom_{\AWb}(2a,0)$ spanned by $\wkdotamu$, for all  $\mu\in\EPar_{2a}$. We first prepare some lemmas. 

\begin{lemma}
\label{lem:diamondnonzero}
    We have $\lol_{a,\lambda}\equiv 0$ unless $\lambda\vdash a+2b$ for some $b\in \N$; in this case $\lol_{a,\lambda}$ lies in the subspace spanned by $
\left\{\lol_{a,(a,\mu)}\mid \mu\in \Par_a(2b),\ l(\mu)<l(\lambda),\ (a,\mu)\unrhd \lambda\right\}    
$ up to lower degree terms.
\end{lemma}

\begin{proof}
    Let $k=l(\lambda)$. We proceed by induction on $a$. When $a=1$, it is straightforward to check the lemma via \eqref{dotmovesplitss}, \eqref{dotmovemergess} and \eqref{bubblecap}.  Now we suppose $a>1$.  If $\lambda_k=a$, then $\lambda=(a^k)$ and by \eqref{dotmovesplitss}, \eqref{dotmovemergess} and \eqref{bubblecap} we have
    \[
    \lol_{a,\lambda}\equiv
    \begin{cases}
        0, & \text{ for } k\equiv 0\pmod 2  \\
        2^a \;
       \lol\!_{a, ((2a)^{\frac{k-1}{2}})} & \text{ for } k\equiv 1\pmod 2.
    \end{cases} 
    \]
    Hence the lemma holds in this case. If $\lambda_k=r<a$, then we have
    \[
    \lol_{a,\lambda}\equiv \sum_{\vartheta+\varsigma=\overline{\lambda}}
    \begin{tikzpicture}
[anchorbase,scale=.9,color=\clr]
\draw[-,line width=1pt] (0.08,-.7) to (0.08,-.5);
\draw[-,line width=1pt] (0.08,.3) to (0.08,.5) to [out=up,in=up] (1,.5) -- (1,-.7) -- (.54,-.9) -- (.08,-.7);
\draw[-,line width=1.2pt] (.54,-.9) -- (.54,-1.1);
\draw[-,thick] (0.1,-.51) to [out=45,in=-45] (0.1,.31);
\draw[-,thick] (0.06,-.51) to [out=135,in=-135] (0.06,.31);
\draw(-.1,0) \bdot;
\draw(-.4,0)node {$\scriptstyle \omega_{\vartheta}$};
\draw(.25,0) \bdot;
\draw(.6,0)node {$\scriptstyle \omega_{\varsigma}$};
\draw(-.09,-.35) \bdot;
\node at (-.35,-.35) {$\scriptstyle {r}$};
\node at (.52,-.35) {$\scriptstyle a-r$};
\end{tikzpicture}
~=~\sum_{\vartheta+\varsigma=\overline{\lambda}}~
\begin{tikzpicture}
    [anchorbase,scale=.9,color=\clr,yscale=-1]
\draw[-,line width=1pt] (-.35,-.2) to (0.08,.4);
\draw[-,line width=1pt] (.45,-.2) to (0.08,.4);
\draw[-,line width=1.4pt] (0.08,.6) to (0.08,.4);
\draw[-,line width=1pt]  (.45,-1) to [out=left,in=-135] (.45,-.2)
 (.45,-1) to [out=right,in=-45] (.45,-.2);
 \draw[-,line width=1pt]  (-.35,-1) to [out=left,in=-135] (-.35,-.2)
 (-.35,-1) to [out=right,in=-45] (-.35,-.2);
\node at (-.55,0.1) {$\scriptstyle 2r$};
\node at (1.1,0.1) {$\scriptstyle 2(a-r)$};
\draw (.25,-.6) \bdot;
\node at (.15,-.4){$\scriptstyle {\omega_{\varsigma}}$};
\draw (-.53,-.5) \bdot;
\node at (-.85,-.5){$\scriptstyle r$};
\draw (-.53,-.7) \bdot;
\node at (-.85,-.7){$\scriptstyle \omega_{\vartheta}$};
\end{tikzpicture}
    \]
    where $\overline{\lambda}=(\lambda_1,\ldots,\lambda_{k-1})$, $\ell(\varsigma),\ell(\vartheta)\leq k-1$. By induction on $a$, we see that each summand in the above equation is equivalent to zero unless $\vartheta\vdash 2b_1$ and $\varsigma\vdash (a-r+2b_2)$; in this case, we have $\lambda\vdash (a+2b_1+2b_2)$. Moreover, up to lower degree terms, by the inductive hypothesis we see that $\lol_{r,(r,\vartheta)}$ lies in the subspace spanned by  $
\left\{\lol_{r,(r,\mu)}\mid \mu\in \Par_a(2b_1),\ l(\mu)\leq k-1,\ (r,\mu)\unrhd (r,\vartheta)\right\}    
$  while  $\lol_{a-r,\varsigma}$ lies in the subspace spanned by  $
\left\{\lol_{a-r,(a-r,\nu)}\mid \nu\in \Par_a(2b_2),\ \ell(\nu)\leq k-2,\ (a-r,\nu)\unrhd \varsigma\right\}    
$. However, we have
\[
\begin{tikzpicture}
    [anchorbase,scale=.9,color=\clr,yscale=-1]
\draw[-,line width=1pt] (-.35,-.2) to (0.08,.4);
\draw[-,line width=1pt] (.45,-.2) to (0.08,.4);
\draw[-,line width=1.4pt] (0.08,.6) to (0.08,.4);
\draw[-,line width=1pt]  (.45,-1) to [out=left,in=-135] (.45,-.2)
 (.45,-1) to [out=right,in=-45] (.45,-.2);
 \draw[-,line width=1pt]  (-.35,-1) to [out=left,in=-135] (-.35,-.2)
 (-.35,-1) to [out=right,in=-45] (-.35,-.2);
\node at (-.55,0.1) {$\scriptstyle 2r$};
\node at (1.1,0.1) {$\scriptstyle 2(a-r)$};
\draw(.25,-.8)\bdot;
\node at (.25,-1.1){$\scriptstyle a-r$};
\draw (.3,-.4) \bdot;
\node at (0.05,-.4){$\scriptstyle {\omega_{\nu}}$};
\draw (-.53,-.5) \bdot;
\node at (-.85,-.5){$\scriptstyle r$};
\draw (-.53,-.7) \bdot;
\node at (-.85,-.7){$\scriptstyle \omega_{\mu}$};
\end{tikzpicture}
~=~
\begin{tikzpicture}
[anchorbase,scale=.9,color=\clr]
\draw[-,line width=1pt] (0.08,-.7) to (0.08,-.5);
\draw[-,line width=1pt] (0.08,.3) to (0.08,.5) to [out=up,in=up] (1,.5) -- (1,-.7) -- (.54,-.9) -- (.08,-.7);
\draw[-,line width=1.2pt] (.54,-.9) -- (.54,-1.1);
\draw[-,thick] (0.1,-.51) to [out=45,in=-45] (0.1,.31);
\draw[-,thick] (0.06,-.51) to [out=135,in=-135] (0.06,.31);
\draw(-.1,0) \bdot;
\draw(-.4,0)node {$\scriptstyle \omega_{\mu}$};
\draw(.25,0) \bdot;
\draw(.6,0)node {$\scriptstyle \omega_{\nu}$};
\draw(-.09,-.35) \bdot;
\node at (-.35,-.35) {$\scriptstyle {r}$};
\draw(.2,-.35)\bdot;
\node at (.55,-.35) {$\scriptstyle a-r$};
\end{tikzpicture}
~\equiv~
\begin{tikzpicture}
[anchorbase,scale=.9,color=\clr]
\draw[-,line width=1pt] (0.08,-.7) to (0.08,-.5);
\draw[-,line width=1pt] (0.08,.3) to (0.08,.5) to [out=up,in=up] (1,.5) -- (1,-.7) -- (.54,-.9) -- (.08,-.7);
\draw[-,line width=1.2pt] (.54,-.9) -- (.54,-1.1);
\draw[-,thick] (0.1,-.51) to [out=45,in=-45] (0.1,.31);
\draw[-,thick] (0.06,-.51) to [out=135,in=-135] (0.06,.31);
\draw(-.1,0) \bdot;
\draw(-.4,0)node {$\scriptstyle \omega_{\mu}$};
\draw(.08,.45) \bdot;
\draw(.25,0) \bdot;
\draw(.6,0)node {$\scriptstyle \omega_{\nu}$};
\node at (.5,.45) {$\scriptstyle {a}$};
\end{tikzpicture} \, .
\]
By Lemma~\ref{wkdotsspan}, the right hand side above lies in the subspace spanned by $\lol_{a,(a,\eta)}$ such that $(a,\eta)\unrhd (a,\mu+\nu)\unrhd (r,\vartheta)+(\varsigma)\unrhd \lambda$. This concludes the proof of the case when $a>1$ and hence the lemma.
\end{proof}

By Lemma~\ref{lem:diamondnonzero}, we only need to consider $\lol_{a,\lambda}$ for $\lambda\in \Par_a$ satisfies that
\begin{itemize}
    \item[(i)] $\lambda_1=a$,
    \item[(ii)]  $\underline{\lambda}:=(\lambda_2,\ldots,\lambda_k)\vdash2b$ for some $b\in\N$.
\end{itemize}
We call such $\lambda$ admissible, and denote $\Par_a(a+2b)^{\text{adm}}=\{\lambda\in\Par_a(a+2b)\mid \lambda \text{ satisfies } (i)\text{-}(ii) \text{ above}\}$. Moreover, let $\venus_{2a}^{\lambda}$ denote the subspace spanned by $\{\lol_{a,(a,\mu)}\mid \mu\unrhd \underline{\lambda},\mu\vdash(\lambda-a),l(\mu)=2l,\mu_{2i-1}=\mu_{2i}, \forall i=1,\ldots,l\}$.
\begin{lemma}
\label{diamondadmissible}
    For any $\lambda\in\Par_a(a+2b)^{\text{adm}}$, we have $\lol_{a,\lambda}\in \venus_{2a}^{\lambda}$ up to lower degree terms.
\end{lemma}

\begin{proof}
    We prove  by induction on $k=l(\lambda)$ and an inverse induction on the order with respect to $\rhd$. We first consider the case for $k=1$. In this case, we must have $b=0$ and $\lambda=(a)$, and the lemma clearly holds.

    We then consider the case when $\lambda$ is maximal with respect to $\rhd$. If $2b<a$, then we have $\lambda=(a,2b)$. Thus
    \begin{equation}
    \label{eq:diamonda2b}
    \lol_{a,(a,2b)}\equiv
    \sum_{0\leq k\leq \min\{a-2b,2b\}}
    \begin{tikzpicture}
[anchorbase,scale=.9,color=\clr]
\draw[-,line width=1pt] (0.08,-.7) to (0.08,-.5);
\draw[-,line width=1pt] (0.08,.3) to (0.08,.5) to [out=up,in=up] (1,.5) -- (1,-.7) -- (.54,-.9) -- (.08,-.7);
\draw[-,line width=1.2pt] (.54,-.9) -- (.54,-1.1);
\draw[-,thick] (0.1,-.51) to [out=45,in=-45] (0.1,.31);
\draw[-,thick] (0.06,-.51) to [out=135,in=-135] (0.06,.31);
\draw(-.1,0) \bdot;
\draw(-.4,0)node {$\scriptstyle k$};
\draw(.25,0) \bdot;
\draw(.6,0)node {$\scriptstyle a-k$};
\draw(-.09,-.35) \bdot;
\node at (-.35,-.35) {$\scriptstyle {2b}$};
\node at (.55,-.35) {$\scriptstyle a-2b$};
\node at (.53,-1.35){$\scriptstyle a$};
\end{tikzpicture}\ .
    \end{equation}
    The right hand side above is zero unless $2b=0$, in which case we have $\lol_{a,\lambda}=\lol_{a,a}$. On the other hand, if $2b\geq a$, then we have $\lambda=(a,a,\nu)$ with $\ell(\nu)<\ell(\underline{\lambda})$. Thus $\lol_{a,\lambda}\equiv (-1)^a\begin{tikzpicture}[yscale=-1, anchorbase,scale=.8,color=\clr]
\draw[-,line width=1.4pt] (0.1,.3) to (0.1,.7);
\draw[-,thick] (0.1,-.51) to [out=right,in=-45] (0.1,.31);
\draw[-,thick] (0.1,-.51) to [out=left,in=-135] (0.1,.31);
\node at (0.08,.8) {$\scriptstyle 2a$};
\draw (-.1,-.2)\bdot;
\draw (.1,.45)\bdot;
\node at (-.4,-.2) {$\scriptstyle \omega_\nu$};
\end{tikzpicture}$ and the lemma in this case follows from induction on $k$ and Lemma~\ref{lem:diamondnonzero}. This concludes the proof when $\lambda$ is maximal.

In general, suppose $k>1$ and there exists $1\leq i\leq k-1$ such that $\lambda_i=\lambda_{i+1}$, then we have
\begin{equation}
    \label{eq:casei=i+1}
\lol_{a,\lambda}
~\equiv~
\begin{tikzpicture}[yscale=-1, anchorbase,scale=.8,color=\clr]
\draw[-,line width=1.4pt] (0.1,.3) to (0.1,.7);
\draw[-,thick] (0.1,-.51) to [out=right,in=-45] (0.1,.31);
\draw[-,thick] (0.1,-.51) to [out=left,in=-135] (0.1,.31);
\node at (0.08,.8) {$\scriptstyle 2a$};
\draw (-.1,-.2)\bdot;
\draw (.1,.45)\bdot;
\node at (.5,.45){$\scriptstyle 2\lambda_i$};
\node at (-.6,-.2) {$\scriptstyle {\omega_{\mu^{(0)}}}$};
\end{tikzpicture}
~+\sum_{0<t\leq \lambda_i}\lol_{a,\mu^{(t)}}~
\end{equation}
where $\mu^{(0)}$ is obtained from $\lambda$ by deleting $\lambda_i$ and $\lambda_{i+1}$ and $\mu^{(t)}$  is obtained from $\lambda$ by replacing $(\lambda_i,\lambda_{i+1})$ by $(\lambda_i+t,\lambda_i-t)$. Thus we have $\mu^{(t)}\rhd \lambda$ for $0<t\leq \lambda_i$. Hence the second summand of \eqref{eq:casei=i+1} lie in $\venus_{2a}^{\lambda}$ via inverse induction on the order with respect to $\rhd$. For the first summand of \eqref{eq:casei=i+1}, by induction on $k$ we have $\lol_{a,\mu^{(0)}}\in \venus_{2a}^{\mu^{(0)}}$. For any $\mu\in \Par_a$ with $(a,\mu)\unrhd \mu^{(0)}$, $l(\mu)=2l$ and $\mu_{2i-1}=\mu_{2i}\ \forall i$, we have
\[
\begin{tikzpicture}[yscale=-1, anchorbase,scale=.8,color=\clr]
\draw[-,line width=1.4pt] (0.1,.3) to (0.1,.7);
\draw[-,thick] (0.1,-.51) to [out=right,in=-45] (0.1,.31);
\draw[-,thick] (0.1,-.51) to [out=left,in=-135] (0.1,.31);
\node at (0.08,.8) {$\scriptstyle 2a$};
\draw (-.1,-.2)\bdot;
\draw (.1,.45)\bdot;
\node at (.5,.45){$\scriptstyle 2\lambda_i$};
\node at (-.6,-.2) {$\scriptstyle {\omega_{\mu}}$};
\end{tikzpicture}
~\equiv \sum_{0\leq j\leq \lambda_i}f_j\lol_{a, [a,\mu,\lambda_i+j,\lambda_i-j]}
\]
where $f_j\in \mathbb J'$ and $[a,\mu,\lambda_i+j,\lambda_i-j]$ is the partition obtained from reordering the composition $(a,\mu,\lambda_i+j,\lambda_i-j)$. Since $(a,\mu)\unrhd \mu^{(0)}$ and $(\lambda_i+j,\lambda_i-j)\unrhd \lambda$, we have $[a,\mu,\lambda_i+j,\lambda_i-j]\unrhd \lambda$. The equality holds if and only if $j=0$ and $(a,\mu)=\mu^{(0)}$, in which case $\lambda$ is already of the desired form. If not, then it follows by induction on $\rhd$ that each $\lol_{a, [a,\mu,\lambda_i+j,\lambda_i-j]}\in \venus_{2a}^{\lambda}$. This completes the proof of the lemma when there exists $i$ such that $\lambda_i=\lambda_{i+1}$.

It remains to consider the case when $\lambda=(a,\lambda_2,\ldots,\lambda_k)$ with $\lambda_2>\cdots>\lambda_k=r>0$. In this case we have  \[
    \lol_{a,\lambda}
    \equiv \sum_{\vartheta+\varsigma=\overline{\lambda}}
    \begin{tikzpicture}
[anchorbase,scale=.9,color=\clr]
\draw[-,line width=1pt] (0.08,-.7) to (0.08,-.5);
\draw[-,line width=1pt] (0.08,.3) to (0.08,.5) to [out=up,in=up] (1,.5) -- (1,-.7) -- (.54,-.9) -- (.08,-.7);
\draw[-,line width=1.2pt] (.54,-.9) -- (.54,-1.1);
\draw[-,thick] (0.1,-.51) to [out=45,in=-45] (0.1,.31);
\draw[-,thick] (0.06,-.51) to [out=135,in=-135] (0.06,.31);
\draw(-.1,0) \bdot;
\draw(-.4,0)node {$\scriptstyle \omega_{\vartheta}$};
\draw(.25,0) \bdot;
\draw(.6,0)node {$\scriptstyle \omega_{\varsigma}$};
\draw(-.09,-.35) \bdot;
\node at (-.35,-.35) {$\scriptstyle {r}$};
\node at (.52,-.35) {$\scriptstyle a-r$};
\end{tikzpicture}
~=~\sum_{\vartheta+\varsigma=\overline{\lambda}}~
\begin{tikzpicture}
    [anchorbase,scale=.9,color=\clr,yscale=-1]
\draw[-,line width=1pt] (-.35,-.2) to (0.08,.4);
\draw[-,line width=1pt] (.45,-.2) to (0.08,.4);
\draw[-,line width=1.4pt] (0.08,.6) to (0.08,.4);
\draw[-,line width=1pt]  (.45,-1) to [out=left,in=-135] (.45,-.2)
 (.45,-1) to [out=right,in=-45] (.45,-.2);
 \draw[-,line width=1pt]  (-.35,-1) to [out=left,in=-135] (-.35,-.2)
 (-.35,-1) to [out=right,in=-45] (-.35,-.2);
\node at (-.55,0.1) {$\scriptstyle 2r$};
\node at (1,0.1) {$\scriptstyle 2(a-r)$};
\draw (.25,-.6) \bdot;
\node at (.15,-.4){$\scriptstyle {\omega_{\varsigma}}$};
\draw (-.53,-.5) \bdot;
\node at (-.85,-.5){$\scriptstyle r$};
\draw (-.53,-.7) \bdot;
\node at (-.85,-.7){$\scriptstyle \omega_{\vartheta}$};
\end{tikzpicture}
    \]
where $\overline{\lambda}=(\lambda_1,\ldots,\lambda_{k-1})$, $\vartheta_1=r$ and $\varsigma_1=a-r$. By induction on $a$, we see that
    \[
    \begin{tikzpicture}
    [anchorbase,scale=.9,color=\clr,yscale=-1]
\draw[-,line width=1pt] (-.35,-.2) to (0.08,.4);
\draw[-,line width=1pt] (.45,-.2) to (0.08,.4);
\draw[-,line width=1.4pt] (0.08,.6) to (0.08,.4);
\draw[-,line width=1pt]  (.45,-1) to [out=left,in=-135] (.45,-.2)
 (.45,-1) to [out=right,in=-45] (.45,-.2);
 \draw[-,line width=1pt]  (-.35,-1) to [out=left,in=-135] (-.35,-.2)
 (-.35,-1) to [out=right,in=-45] (-.35,-.2);
\node at (-.4,0.1) {$\scriptstyle 2r$};
\node at (.9,0.1) {$\scriptstyle 2(a-r)$};
\draw (.25,-.6) \bdot;
\node at (.15,-.4){$\scriptstyle {\omega_{\varsigma}}$};
\draw (-.53,-.5) \bdot;
\node at (-.8,-.5){$\scriptstyle r$};
\draw (-.53,-.7) \bdot;
\node at (-.9,-.7){$\scriptstyle \omega_{\vartheta}$};
\end{tikzpicture}
    ~\equiv\sum_{\mu,\nu}~
    \begin{tikzpicture}
    [anchorbase,scale=.9,color=\clr,yscale=-1]
\draw[-,line width=1pt] (-.35,-.2) to (0.08,.4);
\draw[-,line width=1pt] (.45,-.2) to (0.08,.4);
\draw[-,line width=1.4pt] (0.08,.6) to (0.08,.4);
\draw[-,line width=1pt]  (.45,-1) to [out=left,in=-135] (.45,-.2)
 (.45,-1) to [out=right,in=-45] (.45,-.2);
 \draw[-,line width=1pt]  (-.35,-1) to [out=left,in=-135] (-.35,-.2)
 (-.35,-1) to [out=right,in=-45] (-.35,-.2);
\node at (-.4,0.1) {$\scriptstyle 2r$};
\node at (.9,0.1) {$\scriptstyle 2(a-r)$};
\draw(.25,-.8)\bdot;
\node at (.2,-1.1){$\scriptstyle a-r$};
\draw (.3,-.4) \bdot;
\node at (0.05,-.4){$\scriptstyle {\omega_{\nu}}$};
\draw (-.53,-.5) \bdot;
\node at (-.8,-.5){$\scriptstyle r$};
\draw (-.53,-.7) \bdot;
\node at (-.85,-.7){$\scriptstyle \omega_{\mu}$};
\end{tikzpicture}
~=\sum_{\mu,\nu}~
\begin{tikzpicture}
[anchorbase,scale=.9,color=\clr]
\draw[-,line width=1pt] (0.08,-.7) to (0.08,-.5);
\draw[-,line width=1pt] (0.08,.3) to (0.08,.5) to [out=up,in=up] (1,.5) -- (1,-.7) -- (.54,-.9) -- (.08,-.7);
\draw[-,line width=1.2pt] (.54,-.9) -- (.54,-1.1);
\draw[-,thick] (0.1,-.51) to [out=45,in=-45] (0.1,.31);
\draw[-,thick] (0.06,-.51) to [out=135,in=-135] (0.06,.31);
\draw(-.1,0) \bdot;
\draw(-.4,0)node {$\scriptstyle \omega_{\mu}$};
\draw(.25,0) \bdot;
\draw(.6,0)node {$\scriptstyle \omega_{\nu}$};
\draw(-.09,-.35) \bdot;
\node at (-.3,-.35) {$\scriptstyle {r}$};
\draw(.2,-.35)\bdot;
\node at (.6,-.35) {$\scriptstyle a-r$};
\end{tikzpicture}
~\equiv\sum_{\mu,\nu}~
\begin{tikzpicture}
[anchorbase,scale=.8,color=\clr]
\draw[-,line width=1pt] (0.08,-.7) to (0.08,-.5);
\draw[-,line width=1pt] (0.08,.3) to (0.08,.5) to [out=up,in=up] (1,.5) -- (1,-.7) -- (.54,-.9) -- (.08,-.7);
\draw[-,line width=1.2pt] (.54,-.9) -- (.54,-1.1);
\draw[-,thick] (0.1,-.51) to [out=45,in=-45] (0.1,.31);
\draw[-,thick] (0.06,-.51) to [out=135,in=-135] (0.06,.31);
\draw(-.1,0) \bdot;
\draw(-.4,0)node {$\scriptstyle \omega_{\mu}$};
\draw(.08,.45) \bdot;
\draw(.25,0) \bdot;
\draw(.6,0)node {$\scriptstyle \omega_{\nu}$};
\node at (.3,.45) {$\scriptstyle {a}$};
\end{tikzpicture}
    \]
    where $\mu\unrhd \vartheta,\ \nu\unrhd (\varsigma_2,\ldots,\varsigma_{\ell(\varsigma)})$. Hence we have $\mu+\nu\unrhd \underline{\lambda}$. Moreover, by Lemma~\ref{wkdotsspan}, each summand $\begin{tikzpicture}
[anchorbase,scale=.8,color=\clr]
\draw[-,line width=1pt] (0.08,-.7) to (0.08,-.5);
\draw[-,line width=1pt] (0.08,.3) to (0.08,.5) to [out=up,in=up] (1,.5) -- (1,-.7) -- (.54,-.9) -- (.08,-.7);
\draw[-,line width=1.2pt] (.54,-.9) -- (.54,-1.1);
\draw[-,thick] (0.1,-.51) to [out=45,in=-45] (0.1,.31);
\draw[-,thick] (0.06,-.51) to [out=135,in=-135] (0.06,.31);
\draw(-.1,0) \bdot;
\draw(-.4,0)node {$\scriptstyle \omega_{\mu}$};
\draw(.08,.45) \bdot;
\draw(.25,0) \bdot;
\draw(.6,0)node {$\scriptstyle \omega_{\nu}$};
\node at (.3,.45) {$\scriptstyle {a}$};
\end{tikzpicture}$ 
lies in the subspace spanned by $\{\lol_{a,(a,\gamma)}\mid \gamma\unrhd \mu+\nu\}$. Thus $\gamma\unrhd \underline{\lambda}$ as well. However, since $\mu$ and $\nu$ are partitions of even lengths such that $\mu_{2i-1}=\mu_{2i}$ and $\nu_{2j-1}=\nu_{2j}$ for all $1\leq i \leq l(\mu)$ and $1\leq j \leq \ell(\nu)$, we must have $\mu+\nu\rhd \underline{\lambda}$ by the assumption on $\lambda$. Hence the lemma in this case now follows from the inverse induction on the order with respect to $\rhd$. This proves the lemma.
\end{proof}

\begin{proof} [Proof of \cref{lollipopspan}]
  (1)  By Lemma~\ref{lem:diamondnonzero}, we may assume that $\lambda$ is admissible. We proceed by induction on $\lambda$ with respect to the reversed dominance order. 
    
    Suppose $\lambda$ is maximal with respect to $\rhd$. If $2b<a$, then we have $\lambda=(a,2b)$ and this case follows from the same analysis as for \eqref{eq:diamonda2b}. If $2b\geq a$, then we have $\lambda=(a,a,\nu)$ and $\lol_{a,\lambda}\equiv (-1)^a\begin{tikzpicture}[yscale=-1, anchorbase,scale=.8,color=\clr]
\draw[-,line width=1.4pt] (0.1,.3) to (0.1,.7);
\draw[-,thick] (0.1,-.51) to [out=right,in=-45] (0.1,.31);
\draw[-,thick] (0.1,-.51) to [out=left,in=-135] (0.1,.31);
\node at (0.08,.8) {$\scriptstyle 2a$};
\draw (-.1,-.2)\bdot;
\draw (.1,.45)\bdot;
\node at (-.4,-.2) {$\scriptstyle \omega_\nu$};
\end{tikzpicture}$. Hence lemma in this case follows from induction on $k$ and Lemma~\ref{lem:diamondnonzero}. This concludes the proof of (1) when $\lambda$ is maximal.

In general, by Lemma~\ref{diamondadmissible}, we may assume that $\lambda$ is of the form $(a,\underline{\lambda})$ where $\ell(\underline{\lambda})$ is even with $\underline{\lambda}_{2i-1}=\underline{\lambda}_{2i}$ for all $1\leq i \leq \ell(\underline{\lambda})/2$. Then Part~ (1) in this case follows from the same reasoning as in the proof of Lemma~\ref{diamondadmissible} for \eqref{eq:casei=i+1}, i.e. the case when there exists $c\in\N$ such that $\lambda_c=\lambda_{c+1}$.

  (2)  We have
    \[
    \begin{tikzpicture}[anchorbase,color=\clr]
    \draw[line width=1.4pt] (0,-.3) -- (0,0);
    \draw[line width=1pt] (0,0) -- (.3,.6)
                          (0,0) -- (-.3,.6);
    \draw (.3,.6)\bdot (-.3,.6)\bdot (.15,.3)\bdot (-.15,.3)\bdot;
    \node at (-.45,.35){$\scriptstyle \omega_\lambda$};
    \node at (.5,.3){$\scriptstyle \omega_\mu$};
    \node at (-.6,.1){$\scriptstyle 2a-2b$};
     \node at (.35,.1){$\scriptstyle 2b$};
\end{tikzpicture}
~\overset{\eqref{dotmovemergess}}{\equiv}~
\sum_{\vartheta,\varsigma}
\begin{tikzpicture}   [anchorbase,scale=.9,color=\clr,yscale=-1]
\draw[-,line width=1pt] (-.35,-.2) to (0.08,.4);
\draw[-,line width=1pt] (.45,-.2) to (0.08,.4);
\draw[-,line width=1.4pt] (0.08,.6) to (0.08,.4);
\draw[-,line width=1pt]  (.45,-1) to [out=left,in=-135] (.45,-.2)
 (.45,-1) to [out=right,in=-45] (.45,-.2);
 \draw[-,line width=1pt]  (-.35,-1) to [out=left,in=-135] (-.35,-.2)
 (-.35,-1) to [out=right,in=-45] (-.35,-.2);
\node at (-.7,0.1) {$\scriptstyle 2a-2b$};
\node at (.6,0.1) {$\scriptstyle 2b$};
\draw (.25,-.6) \bdot;
\node at (.1,-.4){$\scriptstyle {\omega_{\varsigma}}$};
\draw (-.53,-.6) \bdot;
\node at (-.85,-.6){$\scriptstyle \omega_{\vartheta}$};
\end{tikzpicture}
=
\sum_{\vartheta,\varsigma}
    \begin{tikzpicture}
[anchorbase,scale=.9,color=\clr]
\draw[-,line width=1pt] (0.08,-.7) to (0.08,-.5);
\draw[-,line width=1pt] (0.08,.3) to (0.08,.5) to [out=up,in=up] (1,.5) -- (1,-.7) -- (.54,-.9) -- (.08,-.7);
\draw[-,line width=1.2pt] (.54,-.9) -- (.54,-1.1);
\draw[-,thick] (0.1,-.51) to [out=45,in=-45] (0.1,.31);
\draw[-,thick] (0.06,-.51) to [out=135,in=-135] (0.06,.31);
\draw(-.1,0) \bdot;
\draw(-.4,0)node {$\scriptstyle \omega_{\vartheta}$};
\draw(.25,0) \bdot;
\draw(.6,0)node {$\scriptstyle \omega_{\varsigma}$};
\node at (-.6,-.35) {$\scriptstyle {2a-2b}$};
\node at (.5,-.35) {$\scriptstyle 2b$};
\end{tikzpicture}
~\overset{\cref{wkdotsspan}}{\equiv}~
\sum_\nu
\begin{tikzpicture}[yscale=-1, anchorbase,scale=.8,color=\clr]
\draw[-,line width=1.4pt] (0.1,.3) to (0.1,.7);
\draw[-,thick] (0.1,-.51) to [out=right,in=-45] (0.1,.31);
\draw[-,thick] (0.1,-.51) to [out=left,in=-135] (0.1,.31);
\node at (0.08,.9) {$\scriptstyle 2a$};
\draw (-.1,-.2)\bdot;
\node at (-.4,-.2) {$\scriptstyle \omega_\nu$};
\end{tikzpicture}.
    \]
It follows by (1) that the rightmost terms belong to $\venus_{2a}$, whence (2).
\end{proof}

\subsection{Proof of \cref{thm:basisAWb}}
  \label{app:independence}

To prove the basis \cref{thm:basisAWb}, it remains to prove the linear independence of $\bPSMat_{\lambda,\mu}$.
We first reduce the proof of linear independence to the case when $\mu=\emptyset$. 

To this end, we note that $\AWb$ is a rigid monoidal category by \cref{zigzag}. For any $\lambda=(\lambda_1,\ldots,\lambda_k) \in \Lambda_{\text{st}}(m)$, the dual of $\lambda$ is the strict composition $\lambda^*=(\lambda_k,\ldots,\lambda_1)$: 
\begin{equation}
    \label{adjunction}
    \eta_{\lambda}=
    \begin{tikzpicture}
        [baseline=0,color=\clr]
        \draw[line width=1pt] (-.4,.2) to [out=-60,in=-120] (1.5,.2)
                              (.2,.2) to [out=-60,in=-120] (.8,.2);
        \node at (-.1,.2){$\scriptstyle\cdots$};
         \node at (1.1,.2){$\scriptstyle\cdots$};
         \node at (-.4,.4){$\scriptstyle \lambda_1$};
          \node at (1.5,.4){$\scriptstyle \lambda_1$};
        \node at (.2,.4){$\scriptstyle \lambda_k$};
          \node at (.8,.4){$\scriptstyle \lambda_k$};
    \end{tikzpicture},\qquad
    \epsilon_{\lambda}=
    \begin{tikzpicture}
        [baseline=0,color=\clr,yscale=-1]
        \draw[line width=1pt] (-.4,.2) to [out=-60,in=-120] (1.5,.2)
                              (.2,.2) to [out=-60,in=-120] (.8,.2);
        \node at (-.1,.2){$\scriptstyle\cdots$};
         \node at (1.1,.2){$\scriptstyle\cdots$};
         \node at (-.4,.4){$\scriptstyle \lambda_1$};
          \node at (1.5,.4){$\scriptstyle \lambda_1$};
        \node at (.2,.4){$\scriptstyle \lambda_k$};
          \node at (.8,.4){$\scriptstyle \lambda_k$};
    \end{tikzpicture}.
\end{equation} 

By the standard rigidity argument (cf. \cite[\S 2.4]{BCNR}), the assignment $f\mapsto (f\otimes 1_{\mu^*})\circ (1_\unit\otimes\eta_{\mu})$ gives an isomorphism of $\mathbb J'$-modules
\begin{equation}
\label{isomorphismspace}
   \phi_\mu^\lambda\colon \Hom_{\AWb}(\mu,\lambda)\longrightarrow \Hom_{\AWb}(\unit, \lambda\otimes \mu^*) .
\end{equation}
Any $g \in \bPSMat_{\lambda,\mu}$ 
can be written in the form 
\[
g=\Delta(g)f, 
\qquad \text{where } \Delta(g)=\dotbubble_{2}^{k_2}\dotbubble_4^{k_4}\cdots\in \End_{\AWb}(\unit).
\]

\begin{lemma}
\label{phicorrespond}
    Suppose $\lambda\in\Lambda_{\text{st}}(m)$ and $\mu\in \Lambda_{\text{st}}(n)$
    with $m\equiv n\mod 2$. For any $f\in\bPSMat_{\mu,\lambda}$, we have (modulo lower degree terms)
    $\phi_\mu^{\lambda}(f)\equiv f'$, 
    for some $f'\in \bPSMat_{\unit,\lambda\otimes \mu^*}$ such that $\Delta(f)=\Delta(f')$.
\end{lemma}
\begin{proof}
    By definition, we have
    \(
    \phi_\mu^{\lambda}(f)=\begin{tikzpicture}[baseline =0,color=\clr]
        \draw[line width=1.2pt] (0,.6) \toplabel{\lambda} -- (0,0) to [out=down,in=left] (.2,-.2) to [out=right,in=down] (.4,0) -- (.4,.6)\toplabel{\mu^*};
        \coupon{0,.3}{f};
    \end{tikzpicture}.
    \) Through \cref{mergesplitslideleft}, every split in the bottom part of $f$ becomes a merge on the top part. We then apply \cref{dotmovefreely} such that each elementary dot packet sits in proper positions modulo lower degree terms. Then we can assume the resulting diagram $f'$ is elementary. In this process, the floating bubbles of $f$ remain unchanged, and hence we have $\Delta(f)=\Delta(f')$.
\end{proof}

It follows from \cref{phicorrespond} that each elementary diagram of type $\mu \to \lambda$ corresponds to a unique elementary diagram of type $\unit\to \lambda\otimes\mu^*$ via $\phi^{\lambda}_{\mu}$.

 By \cref{phicorrespond} and \eqref{isomorphismspace}, it remains to show that $\bPSMat_{\emptyset,\gamma}$ is linearly independent in $\Hom_{\AWb}(\emptyset,\gamma)$, for $\gamma=(\gamma_1,\ldots,\gamma_t)$. 
 
Let us first outline the proof strategy below. We associate to each diagram $g=(A_g,P_g)\in \PSMat_{\emptyset,\gamma}$ a carefully chosen label $\theta(g)$ of the exterior tensor factors.  This label records the undotted shape of $g$, hence the corresponding $v_{\theta(g)}$-component separates diagrams with different matrices $A_g$.  Therefore, in any putative linear relation in $\PSMat_{\emptyset,\gamma}$, after passing to terms of maximal degree, it suffices to compare diagrams with the same undotted shape.  For such diagrams, the explicit action formulas for bubbles, lollipops and dot packets identify the top-degree contribution in the relevant $v_{\theta(g)}$-component with a product of elementary symmetric polynomials in Cartan variables.  A suitable monomial order then shows that these leading terms are linearly independent, forcing all maximal-degree coefficients to vanish and any putative linear relation to be trivial.  This proves linear independence over $\C$, and the general case follows by base change.
 
Now let us get into the details. Recall from \cref{eq:MSMat} that any $g\in \PSMat_{\emptyset,\gamma}$ is of the form $g=(A_g,P_g)$, for some $A_g=(a_{ij})\in \SMat_{\emptyset,\gamma}$ and $P_g=(\vartheta_{ij})$ where $\vartheta_{ij}\in \Par_{a_{ij}}$ for $i\neq j$ and $\vartheta_{ii}\in \EPar_{2a_{ii}}$.
We enumerate the endpoints at the top row of $g$ by $1,\ldots,t$ from left to right. The shape of $g$ (after removing the dots) is uniquely determined by $A_g=(a_{ij})$. Namely, there is a lollipop  attached at each thick strand $\gamma_i$ with thickness $2a_{ii}$ and there is a cup with thickness $a_{ij}$ connecting $\gamma_i$ and $\gamma_j$, for $1\leq i\neq j\leq t$. Suppose there are $k$ lollipops and bubbles in total, we label them by a sequence
\begin{equation}
\label{bubbleshapeg}
\sh(g):=\big((L(g)_1,R(g)_1),\ldots, (L(g)_k,R(g)_k)\big)
\end{equation}
such that
\begin{itemize}
    \item $(L(g))_u$ (resp. $(R(g))_u$) represents the left (resp. right) endpoint of $(L(g)_u,R(g)_u)$,
    \item $L(g)_1\leq L(g)_2\leq \cdots \leq L(g)_k\leq t$,
    \item if $L(g)_i=L(g)_j$ with $i<j$, then $R(g)_i<R(g)_j$.
\end{itemize}
Note that if $(L(g)_u,R(g)_u)$ is a lollipop attached at $\gamma_i$, then it should be understood that $L(g)_u=R(g)_u=i$, and the thickness of $(L(g)_u,R(g)_u)$ is encoded by $a_{ii}$.

Suppose $N\gg 0$, we now construct $\theta(g)\in [N]^{|\gamma|}$ in the following way:
\begin{itemize}
    \item[Step 1:] Suppose we have $|\gamma|$ boxes as follows:
    
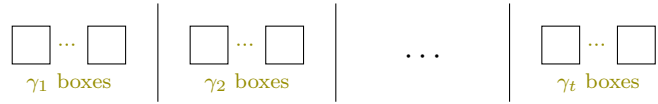
\begin{equation*}
    \begin{tikzpicture}[baseline=0,scale=.5]
    \BoxRow{0}{0};
    \node at (1.5,0.5){$\scriptstyle\cdots$};
    \BoxRow{0}{2};
    \node at (1.5,-.5){$\scriptstyle \gamma_1 \text{ boxes}$};
\end{tikzpicture} \quad
\begin{tikzpicture}
    [baseline=0]
    \draw (0,-.5) -- (0,.8);
\end{tikzpicture}
\quad
\begin{tikzpicture}[baseline=0,scale=.5]
    \BoxRow{0}{0};
    \node at (1.5,0.5){$\scriptstyle\cdots$};
    \BoxRow{0}{2};
    \node at (1.5,-.5){$\scriptstyle \gamma_2 \text{ boxes}$};
\end{tikzpicture}\quad \begin{tikzpicture}
    [baseline=0]
    \draw (0,-.5) -- (0,.8);
\end{tikzpicture} \qquad
\cdots\qquad \begin{tikzpicture}
    [baseline=0]
    \draw (0,-.5) -- (0,.8);
\end{tikzpicture}
\quad
\begin{tikzpicture}[baseline=0,scale=.5]
    \BoxRow{0}{0};
    \node at (1.5,0.5){$\scriptstyle\cdots$};
    \BoxRow{0}{2};
    \node at (1.5,-.5){$\scriptstyle \gamma_t \text{ boxes}$};
\end{tikzpicture}
    \end{equation*}
\item[Step 2:] We consider $(L(g)_1,R(g)_1)$. Let $z_1=a_{L(g)_1,R(g)_1}$. By definition, we have $L(g)_1=1$. We then fill $1,2,\ldots,z_1$ into the stack of $\gamma_1$ boxes:
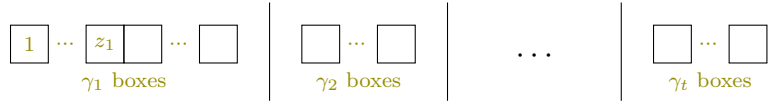
\begin{equation*}
    \begin{tikzpicture}[baseline=0,scale=.5]
    \BoxRow{0}{-3};
    \BoxRow{1}{-1};
    \node at (-2.5,.5){$\scriptstyle {1}$};
    \node at (-.5,.5){$\scriptstyle {z_1}$};
     \node at (-1.5,0.5){$\scriptstyle\cdots$};
    \node at (1.5,0.5){$\scriptstyle\cdots$};
    \BoxRow{0}{2};
    \node at (0,-.5){$\scriptstyle \gamma_1 \text{ boxes}$};
\end{tikzpicture} \quad
\begin{tikzpicture}
    [baseline=0]
    \draw (0,-.5) -- (0,.8);
\end{tikzpicture}
\quad
\begin{tikzpicture}[baseline=0,scale=.5]
    \BoxRow{0}{0};
    \node at (1.5,0.5){$\scriptstyle\cdots$};
    \BoxRow{0}{2};
    \node at (1.5,-.5){$\scriptstyle \gamma_2 \text{ boxes}$};
\end{tikzpicture}\quad \begin{tikzpicture}
    [baseline=0]
    \draw (0,-.5) -- (0,.8);
\end{tikzpicture} \qquad
\cdots\qquad \begin{tikzpicture}
    [baseline=0]
    \draw (0,-.5) -- (0,.8);
\end{tikzpicture}
\quad
\begin{tikzpicture}[baseline=0,scale=.5]
    \BoxRow{0}{0};
    \node at (1.5,0.5){$\scriptstyle\cdots$};
    \BoxRow{0}{2};
    \node at (1.5,-.5){$\scriptstyle \gamma_t \text{ boxes}$};
\end{tikzpicture}
    \end{equation*}

If $R(g)_1=L(g)_1=1$, then we fill $-{1},{-2},\ldots,{-z_1}$ in the stack of  $\gamma_1$ from the rightmost empty box to the left: 
\begin{equation}
       \label{step2boxesa}
    \begin{tikzpicture}[baseline=0,scale=.5]
     \BoxRow{0}{-6};
    \BoxRow{1}{-4};
    \BoxRow{1}{-1};
    \node at (2.5,.5){$\scriptstyle {-1}$};
    \node at (.5,.5){$\scriptstyle {-z_1}$};
    \node at (-3.5,.5){$\scriptstyle {z_1}$};
    \node at (-5.5,.5){$\scriptstyle {1}$};
     \node at (-1.5,0.5){$\scriptstyle\cdots$};
     \node at (-4.5,0.5){$\scriptstyle\cdots$};
    \node at (1.5,0.5){$\scriptstyle\cdots$};
    \BoxRow{0}{2};
    \node at (-1,-.5){$\scriptstyle \gamma_1 \text{ boxes}$};
\end{tikzpicture} \quad
\begin{tikzpicture}
    [baseline=0]
    \draw (0,-.5) -- (0,.8);
\end{tikzpicture}
\quad
\begin{tikzpicture}[baseline=0,scale=.5]
    \BoxRow{0}{0};
    \node at (1.5,0.5){$\scriptstyle\cdots$};
    \BoxRow{0}{2};
    \node at (1.5,-.5){$\scriptstyle \gamma_2 \text{ boxes}$};
\end{tikzpicture}\quad \begin{tikzpicture}
    [baseline=0]
    \draw (0,-.5) -- (0,.8);
\end{tikzpicture} \qquad
\cdots\qquad \begin{tikzpicture}
    [baseline=0]
    \draw (0,-.5) -- (0,.8);
\end{tikzpicture}
\quad
\begin{tikzpicture}[baseline=0,scale=.5]
    \BoxRow{0}{0};
    \node at (1.5,0.5){$\scriptstyle\cdots$};
    \BoxRow{0}{2};
    \node at (1.5,-.5){$\scriptstyle \gamma_t \text{ boxes}$};
\end{tikzpicture}
    \end{equation}
If $R(g)_1=i>1$, then we have $a_{1,1}=a_{1,2}=\ldots=a_{1,i-1}=0$. We then fill ${-1},{-2},\ldots,{-z_1}$ in the stack of $\gamma_i$ boxes from the rightmost empty box to the left:  
\begin{equation}
       \label{step2boxesb}
    \begin{tikzpicture}[baseline=0,scale=.5]
    \BoxRow{0}{-3};
    \BoxRow{1}{-1};
    \node at (-2.5,.5){$\scriptstyle {1}$};
    \node at (-.5,.5){$\scriptstyle {z_1}$};
     \node at (-1.5,0.5){$\scriptstyle\cdots$};
    \node at (1.5,0.5){$\scriptstyle\cdots$};
    \BoxRow{0}{2};
    \node at (0,-.5){$\scriptstyle \gamma_1 \text{ boxes}$};
\end{tikzpicture} \quad
\begin{tikzpicture}
    [baseline=0]
    \draw (0,-.5) -- (0,.8);
\end{tikzpicture}
\quad
\cdots
\quad
\begin{tikzpicture}
    [baseline=0]
    \draw (0,-.5) -- (0,.8);
\end{tikzpicture}
\quad
 \begin{tikzpicture}[baseline=0,scale=.5]
    \BoxRow{0}{-3};
    \BoxRow{1}{-1};
    \node at (2.5,.5){$\scriptstyle {-1}$};
    \node at (.5,.5){$\scriptstyle {-z_1}$};
     \node at (-1.5,0.5){$\scriptstyle\cdots$};
    \node at (1.5,0.5){$\scriptstyle\cdots$};
    \BoxRow{0}{2};
    \node at (0,-.5){$\scriptstyle \gamma_i \text{ boxes}$};
\end{tikzpicture}
\quad
\begin{tikzpicture}
    [baseline=0]
    \draw (0,-.5) -- (0,.8);
\end{tikzpicture}
\quad
\cdots
\quad
\begin{tikzpicture}
    [baseline=0]
    \draw (0,-.5) -- (0,.8);
\end{tikzpicture}
\quad
\begin{tikzpicture}[baseline=0,scale=.5]
    \BoxRow{0}{0};
    \node at (1.5,0.5){$\scriptstyle\cdots$};
    \BoxRow{0}{2};
    \node at (1.5,-.5){$\scriptstyle \gamma_t \text{ boxes}$};
\end{tikzpicture}
    \end{equation}
\item[Step 3:]
    Now consider $(L(g)_2,R(g)_2)$. Let $z_2=z_1+a_{L(g)_2,R(g)_2}$. We first fill $z_1+1,z_1+2,\ldots ,z_2$ into the stack of $\gamma_{L(g)_2}$ boxes of \cref{step2boxesa} or \cref{step2boxesb} from the leftmost empty box to the right. After that, we fill  $-(z_1+1),-(z_1+2),\ldots,-z_2$ to the stack of  $\gamma_{L(g)_2}$ boxes from the rightmost empty box to the left as in Step 2.
\item[Step 4:] Repeat the above steps for $(L(g)_3,R(g)_3),\ldots, (L(g)_k,R(g)_k)$ one by one, then every box will be filled with a number and we get
  \begin{equation}
       \label{filledbox}
    \begin{tikzpicture}[baseline=0,scale=.6]
    \BoxRow{0}{0};
    \node at (1.5,0.5){$\scriptstyle\cdots$};
    \node at (.5,.5){$\scriptstyle g^{(1)}_{1}$};
    \node at (2.5,.5){$\scriptstyle g^{(1)}_{\gamma_1}$};
    \BoxRow{0}{2};
    \node at (1.5,-.5){$\scriptstyle \gamma_1 \text{ boxes}$};
\end{tikzpicture} \quad
\begin{tikzpicture}
    [baseline=0]
    \draw (0,-.5) -- (0,.8);
\end{tikzpicture}
\quad
\begin{tikzpicture}[baseline=0,scale=.6]
       \node at (.5,.5){$\scriptstyle g^{(2)}_{1}$};
    \node at (2.5,.5){$\scriptstyle g^{(2)}_{\gamma_2}$};
    \BoxRow{0}{0};
    \node at (1.5,0.5){$\scriptstyle\cdots$};
    \BoxRow{0}{2};
    \node at (1.5,-.5){$\scriptstyle \gamma_2 \text{ boxes}$};
\end{tikzpicture}\quad \begin{tikzpicture}
    [baseline=0]
    \draw (0,-.5) -- (0,.8);
\end{tikzpicture} \qquad
\cdots\qquad \begin{tikzpicture}
    [baseline=0]
    \draw (0,-.5) -- (0,.8);
\end{tikzpicture}
\quad
\begin{tikzpicture}[baseline=0,scale=.6]
       \node at (.5,.5){$\scriptstyle g^{(t)}_{1}$};
    \node at (2.5,.5){$\scriptstyle g^{(t)}_{\gamma_t}$};
    \BoxRow{0}{0};
    \node at (1.5,0.5){$\scriptstyle\cdots$};
    \BoxRow{0}{2};
    \node at (1.5,-.5){$\scriptstyle \gamma_t \text{ boxes}$};
\end{tikzpicture}
    \end{equation}
    where $-|\gamma|/2\leq g_{\gamma_i}^{(j)}\leq |\gamma|/2$.
\end{itemize}
Our labeling ensures that within each stack of boxes, all positive numbers appear in strictly increasing order, and all negative numbers also appear in strictly increasing order. Moreover, all positive numbers precede the negative numbers in each stack.

Then we define
\begin{equation}
\label{thetag}
\theta(g):=(g^{(1)}_{1},\ldots,g^{(1)}_{\gamma_1}, \ldots,g^{(t)}_{1},\ldots,g^{(t)}_{\gamma_t})\in [N]^{|\gamma|}.
\end{equation}
For any ${\bf i}=({\bf i}^{(1)}_{1},\ldots,{\bf i}^{(1)}_{\gamma_1}, \ldots,{\bf i}^{(t)}_{1},\ldots,{\bf i}^{(t)}_{\gamma_t})\in [N]^{|\gamma|}$, we define a corresponding element 
\[
v_{\bf i}:= v_{{\bf i}^{(1)}_{1}}\wedge \cdots\wedge v_{{\bf i}^{(1)}_{\gamma_1}} \otimes \cdots \otimes v_{{\bf i}^{(t)}_{1}}\wedge \cdots\wedge v_{{\bf i}^{(t)}_{\gamma_t}} \in \wedge{}^{\gamma}V
\]
We define the $v_{\bf i}$-component of $M^{\text{gen}}\otimes \wedge^{\gamma}V$ to be the subspace spanned by $y\otimes v_{\bf i}$ for all $y\in M^{\text{gen}}$. In particular, we have
\[
v_{\theta(g)}:= v_{g^{(1)}_{1}}\wedge \cdots\wedge v_{g^{(1)}_{\gamma_1}} \otimes \cdots \otimes v_{g^{(t)}_{1}}\wedge \cdots\wedge v_{g^{(t)}_{\gamma_t}}.
\]

\begin{lemma}
\label{seperatewedgecomponent}
    Suppose $g=(A_g,P_g),g'=(A_{g'},P_{g'})\in \PSMat_{\emptyset,\gamma}$ and $\hm=1\otimes1 \in M^{\text{gen}}$. If there is a nonzero term of $\cG_N(g')_{M^{\text{gen}}}(\hm)$ in the $v_{\theta(g)}$-component of $M^{\text{gen}}\otimes \bigwedge^{\gamma}V$, then $A_g=A_g'$.
\end{lemma}
\begin{proof}
    Since the actions of the elementary dot packets do not change the $v_{\bf i}$ components of $\cG_N(g')_{M^{\text{gen}}}(\hm)$, we may assume that there are no dot packets on any strand (other than in the heads of the lollipops). By the above construction, we see that $\theta(g)=\theta(g')$ if and only if $\sh(g)=\sh(g')$. However, $\sh(g)$ is uniquely determined by $A_g$. Thus the lemma follows by \cref{cuparep} and \cref{actionlollipoprev}; cf. \cite[Proof of Theorem A]{RS2019}.
\end{proof}

Assume now $\kk=\C$. For any $\tilde g\in \bPSMat_{\emptyset,\gamma}$, define $\beta^{\tilde g}=(\beta^{\tilde g}_{2},\beta^{\tilde g}_{4},\ldots,)$ to be the sequence of non-negative integers such that  $\tilde g=\Delta^{\beta^{\tilde g}} g$  
 where $\Delta^{\beta^{\tilde g}}=\dotbubble_2^{\beta^{\tilde g}_2}\dotbubble_4^{\beta^{\tilde g}_4}\cdots$ and  $g\in \PSMat_{\emptyset,\gamma}$. We shall prove by contradiction that $\bPSMat_{\emptyset,\gamma}$ is linearly independent. Suppose there exists a nonempty finite subset $W\subset\bPSMat_{\emptyset,\gamma}$ such that
\begin{equation}
\label{sum=0hypo}
    \sum_{\tilde g\in W}c_{\tilde g} \Delta^{\beta^{\tilde g}}g=0
\end{equation}
where $c_{\tilde{g}}\neq 0$ for all $\tilde{g}\in W$. Then
$W_0:=\big\{\tilde{f}\in W\mid \deg(\tilde{f}) = \max \{ \deg \tilde{g} ~|~ \tilde{g}\in W\}\big\}\not=\emptyset$. 

Fix a $\tilde f=\Delta^{\beta^{\tilde f}} f \in W_0$. We consider $\cG_N(\tilde g)_{M^{\text{gen}}}(\hm)$ for every $\tilde g\in W_0$ and compare the highest degree terms in the $v_{\theta(f)}$-component of $M^{\text{gen}}\otimes \bigwedge ^\gamma V$.  By \cref{seperatewedgecomponent}, the highest degree $v_{\theta(f)}$-component of $\cG(\tilde g)_{M^{\text{gen}}}(\hm)$ is nonzero if and only if $A_{g}=A_{f}$. Thus we may assume all elements in $W_0$ have the same shape as $\sh(f)$. Let 
\[
\hslash_{\tilde g}(h_1^2,h_2^2,\cdots,h_n^2)=  e_1^{\beta^{\tilde{g}}_2}(h_1^2,\ldots,h_n^2)e_2^{\beta^{\tilde{g}}_4}(h_1^2,\ldots,h_n^2)\cdots
\] 
be the symmetric polynomial appearing in $\cG(\Delta^{\beta^{\tilde{g}}})_{M^{\text{gen}}}(\hm)$ (up to a sign). Then by \cref{cuparep}, \cref{actionlantern} and \cref{actionlollipoprev}, the $M^{\text{gen}}$-part of the $v_{\theta(f)}$-component with degree $\deg(\tilde{f})$ in $\cG(\tilde{g})_{M^{\text{gen}}}(\hm)$ up to a sign equals
\begin{equation}
    \label{highestdegreetermh}
    \begin{aligned}
        \hslash_{\tilde g}(h_1^2,h_2^2,\cdots,h_n^2)\prod_{i=1}^t
e_{\vartheta_{ii}/2}\Big(h_{f^{(i)}_{1}}^2,\cdots,h_{f^{(i)}_{a_{ii}}}^2\Big)h_{f^{(i)}_{1}}\cdots h_{f^{(i)}_{a_{ii}}}
    \prod_{j=i+1}^{\gamma_i} 
    e_{\vartheta_{ij}}\Big(h_{f^{(i)}_{b_{i,j-1}+1}},\cdots,h_{f^{(i)}_{b_{ij}}}\Big)
    \end{aligned}
\end{equation}
where $A_g=(a_{ij}), P_g=(\vartheta_{ij})$, $b_{ij}=\sum_{k=i}^ja_{ik}$ and $\theta(f):=(f^{(1)}_{1},\ldots,f^{(1)}_{\gamma_1}, \ldots,f^{(t)}_{1},\ldots,f^{(t)}_{\gamma_t})$ as defined in \cref{thetag}.

Fix a total order on the monomials of $\U(\mathfrak h)$ such that $h_n^{r_n}\cdots h_2^{r_2}h_1^{r_1}<h_n^{l_n}\cdots h_2^{l_2}h_1^{l_1}$ if either $\sum_{i=1}^n r_i<\sum_{i=1}^n l_i$ or $\sum_{i=1}^n r_i=\sum_{i=1}^n l_i$ and there is a $1\leq u\leq n$ such that $r_u<l_u$ and $r_j>l_j$ whenever $j>u$. Since we choose $N\gg 0$, we may assume the leading monomial of $\hslash_{\tilde{g}}(h_1^2,\cdots,h_n^2)$ with respect to the total order does not contain the factor $h_{f^{(i)}_j}$ for any $i=1,\ldots,t$ and $j=1,\ldots,{\gamma_i}$. 

It follows by \cref{sum=0hypo} and the definition of $W_0$ that 
\begin{align} \label{W0}
\cG_N\Big( \sum_{\tilde g\in W_0}c_{\tilde g} \Delta^{\beta^{\tilde g}}g\Big)_{M^{\text{gen}}}(\hm)=0.
\end{align}

We now compare the leading monomials of the
$v_{\theta(f)}$-component of degree $\deg(\tilde f)$.  Since
$N\gg 0$, by the choice of the total order, the leading monomial of
\(\hslash_{\tilde g}(h_1^2,\ldots,h_n^2)\) involves none of the variables
\(h_{f^{(i)}_j}\), where \(1\leq i\leq t\) and \(1\leq j\leq \gamma_i\). For a fixed undotted shape A, the variables appearing in the dot packets on different legs form disjoint blocks. The bubble factor \(\hslash_{\tilde g}\) has leading monomial supported on variables outside these blocks, by our choice of \(N\gg0\). Within each block, the products of elementary symmetric functions \(e_{\vartheta_{ij}}\) have distinct leading monomials as \(\vartheta_{ij}\) varies. Hence the full products in \eqref{highestdegreetermh} have pairwise distinct leading monomials for distinct \(\tilde g\in W_0\), and any nonzero linear combination of these expressions is nonzero.
This contradicts \eqref{W0} as $W_0 \neq \emptyset$.

The linear independence of $\bPSMat_{\emptyset,\gamma}$ for $\kk=\Z$ and then for an arbitrary commutative ring $\kk$ with $1$ now follows by standard base change arguments. 

This completes the proof of \cref{thm:basisAWb}.

\subsection{Proof of \cref{prop:Wiso}} 
\label{app:Wiso}

\begin{proof} [Proof of \cref{prop:Wiso}]
There is a natural functor $\natural: \Wb \rightarrow \WCb$ which matches the generating objects and morphisms in the same symbols. The defining relations for $\WCb$ are a subset of those for $\Wb$. To show that $\natural$ is an isomorphism it remains to verify that all defining relations \eqref{webassoc}--\eqref{bubblecap} of $\Wb$ hold in $\WCb$. They are verified below in \cref{lem:ass:zz:cc,lem:MSmovecap,lem:sp:m,lem:curl,lem:bub,lem:loll}, respectively.
\end{proof}

\begin{lemma}
\label{lem:ass:zz:cc}
    The following relations hold in  $\WCb$ for $a,b,c\in \Z_{\geq 1}$:
    \begin{gather}
\label{thickwebassocC}
 .
    \end{gather} 
\end{lemma}

\begin{proof}
    These relations follow from \cref{webassocC}, \cref{zigzagC}, and \cref{thickmergesplitC} by induction.
\end{proof}

\begin{lemma}
    The following relations hold in $\WCb$ for $a,b\geq 1$:
    \begin{gather}
        %
 .
\label{swallowsC}
    \end{gather}
\end{lemma}

\begin{proof}
The second relation follows from the first one by applying $\div$, hence we only prove the first one.
We proceed by induction on $a+b$. When $a+b=2$, this follows from composing $%
$ on top with \cref{mergesplitC} (setting $k=1$). Now suppose the first relation is true for any labels whose sum is less than $a+b$, then by the inductive hypothesis we have
        \[
%
.
        \]
        This proves the lemma.
\end{proof}

\begin{lemma}
\label{lem:MSmovecap}
    The following relations hold in $\WCb$ for $a,b\geq 1$:
    \begin{gather}
        \label{mergesplitsliderightC}
            %
 \:\: .
    \end{gather}
\end{lemma}

\begin{proof}
    We prove both relations together by induction on $a+b$. When $a=b=1$, we have
    \[
    %
.
    \]
    Similarly one can prove that
    \[
     %
 \:\: .
    \]
    This concludes the base case of the induction. The inductive process is similar to the calculation above using \cref{mergesplitC,zigzagC,swallowsC,thickmergesplitC}.
\end{proof}

\begin{lemma}
    The following relations hold in $\WCb$ for $k\geq 1$:
    \begin{gather}
    \label{reversemergesplitC}
         %
 .
    \end{gather}
\end{lemma}

\begin{proof}
    We prove this lemma by induction on $k$. The case when $k=1$ is a special case of \cref{mergesplitC}. Suppose the relation holds for any label less than $k$, then we have
    \begin{multline*}
    %
\end{multline*}
as desired.
\end{proof}

\begin{lemma}
    The relations \cref{sliders} hold in $\WCb$.
\end{lemma}
\begin{proof}
  We prove the first relation in \cref{sliders}; the remaining relations are proved in the same way.  
The argument proceeds by an induction on \(a+c\).  
We note that the base case is not explicitly treated there, so we include it here. Our base case is for $a=c=1$. To prove this, we see that
\begin{align*}
    %
\end{align*}
as desired.
\end{proof}

\begin{lemma}
    The following relations hold in $\WCb$ for any $a+1=b+d\geq 2$:
     \begin{gather}
    \label{aquaC}
         %
 .
    \end{gather}
\end{lemma}

\begin{proof}
    This follows by induction on $d$; the inductive process is similar to the proof of \cref{reversemergesplitC} and will be skipped.
\end{proof}

\begin{lemma}
\label{lem:sp:m}
    The relations \cref{splitmerge} hold in $\WCb$ for any $a+b=c+d\geq 2$.
\end{lemma}
\begin{proof}
    The first relation in \cref{splitmerge} is straightforward to check using \cref{mergesplitC,thickmergesplitC}. For the second relation in \cref{splitmerge}, we proceed by induction on $c$. When $c=1$, this is exactly \cref{aquaC}. Now suppose the statement is true for any label less than $c$, then we have
    \begin{align*}
    c\begin{tikzpicture}[anchorbase,scale=1, color=\clr]
	\draw[-,line width=1pt] (0,0) to (.29,.3) to (.29,.7) to (0,1);
	\draw[-,line width=1pt] (.6,0) to (.31,.3) to (.31,.7) to (.6,1);
        \node at (0,1.13) {$\scriptstyle b$};
        \node at (0.63,1.13) {$\scriptstyle d$};
        \node at (0,-.1) {$\scriptstyle a$};
        \node at (0.63,-.1) {$\scriptstyle c$};
\end{tikzpicture}
\overset{\cref{thickmergesplitC}}{=}&
  \begin{tikzpicture}[anchorbase,scale=1, color=\clr]
	\draw[-,line width=1pt] (0,0) to (.29,.3) to (.29,.7) to (0,1);
	\draw[-,line width=1pt] (.6,0) to (.31,.3) to (.31,.7) to (.6,1);
    \draw[-,thick] (.5,.1) -- (.2,.2);
    \node at (.55,.3){$\scriptstyle 1$};
        \node at (0,1.13) {$\scriptstyle b$};
        \node at (0.63,1.13) {$\scriptstyle d$};
        \node at (0,-.1) {$\scriptstyle a$};
        \node at (0.63,-.1) {$\scriptstyle c$};
\end{tikzpicture}
\overset{\cref{aquaC}}{=}
\begin{tikzpicture}
[anchorbase,scale=.8,color=\clr]
    \draw[-,thick] (0,0) -- (0,1);
    \draw[-,thick] (.6,0) -- (.6,1);
    \draw[-,thick] (0,.4) -- (.6,.2);
    \draw[-,thick] (0,.6) -- (.6,.8);
    \node at (.75,.5){$\scriptstyle 1$};
    \node at (0,-.15){$\scriptstyle a$};
    \node at (.6,-.15){$\scriptstyle c$};
    \node at (0,1.15){$\scriptstyle b$};
    \node at (.6,1.15){$\scriptstyle d$};
\end{tikzpicture}
~+~
\begin{tikzpicture}
[anchorbase,scale=.8,color=\clr]
    \draw[-,thick] (0,0) -- (0,1);
    \draw[-,thick] (.6,0) to[out=up,in=-45] (0,.9);
    \draw[-,thick] (0,.4) -- (.6,.2);
    \draw[-,thick] (0,.6) -- (.6,.8);
    \node at (.75,.5){$\scriptstyle 1$};
    \node at (0,-.15){$\scriptstyle a$};
    \node at (.6,-.15){$\scriptstyle c$};
    \node at (0,1.15){$\scriptstyle b$};
    \node at (.6,1.15){$\scriptstyle d$};
\end{tikzpicture}\\
=&
\sum_{\substack{0 \leq s \leq \min(a,b)\\0 \leq t \leq \min(c-1,d-1)\\t-s=d-1-a}}
\begin{tikzpicture}[anchorbase,scale=1, color=\clr]
	\draw[-,thick] (0.85,0) to (0.85,.2) to (.02,.8) to (.02,1);
	\draw[-,thick] (0.02,0) to (0.02,.2) to (.85,.8) to (.85,1);
	\draw[-,thick] (0,0) to (0,1);
	\draw[-,line width=1pt] (0.61,.32) to (0.61,.65);
    \draw[-] (0.85,0) to (0.85,1);
        \node at (0,1.13) {$\scriptstyle b$};
        \node at (0.7,1.13) {$\scriptstyle d$};
        \node at (0,-.1) {$\scriptstyle a$};
        \node at (0.7,-.1) {$\scriptstyle c$};
        \node at (-0.1,.5) {$\scriptstyle s$};
        \node at (0.7,.5) {$\scriptstyle t$};
        \node at (0.95,.5) {$\scriptstyle 1$};
\end{tikzpicture}
+\sum_{\substack{0 \leq k \leq \min(a,b-1)\\0 \leq h \leq \min(c-1,d)\\k-h=d-a}}
\begin{tikzpicture}[anchorbase,scale=1, color=\clr]
	\draw[-,thick]  (0.68,-.1) to (0.68,0) to (.02,.6) to (.02,1);
    \draw[-] (.02,.8) to [out=0,in=90](0.68,0);
	\draw[-,thick] (0.02,0) to (0.02,.2) to (.68,1);
	\draw[-,thick] (0,0) to (0,1);
	\draw[-,line width=1pt] (0.38,0.25) to (0.38,.6);
        \node at (0,1.13) {$\scriptstyle b$};
        \node at (0.65,1.13) {$\scriptstyle d$};
        \node at (0,-.1) {$\scriptstyle a$};
        \node at (0.68,-.2) {$\scriptstyle c$};
        \node at (-0.1,.4) {$\scriptstyle k$};
        \node at (0.47,.45) {$\scriptstyle h$};
        \node at (0.75,.5) {$\scriptstyle 1$};
\end{tikzpicture} \quad \text{ by induction on $c$ }\\
\overset{\cref{webassocC},\cref{mergesplitC},\cref{sliders}}=&
(t-1)\sum_{\substack{0 \leq s \leq \min(a,b)\\1 \leq t \leq \min(c,d)\\t-s=d-a}}
\begin{tikzpicture}[anchorbase,scale=1, color=\clr]
	\draw[-,thick] (0.58,0) to (0.58,.2) to (.02,.8) to (.02,1);
	\draw[-,thick] (0.02,0) to (0.02,.2) to (.58,.8) to (.58,1);
	\draw[-,thick] (0,0) to (0,1);
	\draw[-,line width=1pt] (0.61,0) to (0.61,1);
        \node at (0,1.13) {$\scriptstyle b$};
        \node at (0.6,1.13) {$\scriptstyle d$};
        \node at (0,-.1) {$\scriptstyle a$};
        \node at (0.6,-.1) {$\scriptstyle c$};
        \node at (-0.1,.5) {$\scriptstyle s$};
        \node at (0.77,.5) {$\scriptstyle t$};
\end{tikzpicture}
+
(c-t)\sum_{\substack{0 \leq s \leq \min(a,b-1)\\0 \leq t \leq \min(c-1,d)\\t-s=d-a}}
\begin{tikzpicture}[anchorbase,scale=1, color=\clr]
	\draw[-,thick] (0.58,0) to (0.58,.2) to (.02,.8) to (.02,1);
	\draw[-,thick] (0.02,0) to (0.02,.2) to (.58,.8) to (.58,1);
	\draw[-,thick] (0,0) to (0,1);
	\draw[-,line width=1pt] (0.61,0) to (0.61,1);
        \node at (0,1.13) {$\scriptstyle b$};
        \node at (0.6,1.13) {$\scriptstyle d$};
        \node at (0,-.1) {$\scriptstyle a$};
        \node at (0.6,-.1) {$\scriptstyle c$};
        \node at (-0.1,.5) {$\scriptstyle s$};
        \node at (0.77,.5) {$\scriptstyle t$};
\end{tikzpicture}
\\
=&
c\sum_{\substack{0 \leq s \leq \min(a,b)\\0 \leq t \leq \min(c,d)\\t-s=d-a}}
\begin{tikzpicture}[anchorbase,scale=1, color=\clr]
	\draw[-,thick] (0.58,0) to (0.58,.2) to (.02,.8) to (.02,1);
	\draw[-,thick] (0.02,0) to (0.02,.2) to (.58,.8) to (.58,1);
	\draw[-,thick] (0,0) to (0,1);
	\draw[-,line width=1pt] (0.61,0) to (0.61,1);
        \node at (0,1.13) {$\scriptstyle b$};
        \node at (0.6,1.13) {$\scriptstyle d$};
        \node at (0,-.1) {$\scriptstyle a$};
        \node at (0.6,-.1) {$\scriptstyle c$};
        \node at (-0.1,.5) {$\scriptstyle s$};
        \node at (0.77,.5) {$\scriptstyle t$};
\end{tikzpicture}.
    \end{align*}
  Then  \cref{splitmerge} follows.  
\end{proof}

Once \cref{splitmerge} is established for $\WCb$, it follows, together with \cref{thickwebassocC}, that \cref{braid} also holds in $\AWCb$; see for example \cite[Appendix]{BEEO}. 

\begin{lemma}
\label{lem:curl}
    The following relation holds in $\WCb$, for $a\ge 1$:
    \begin{gather*}
         \begin{tikzpicture}[anchorbase, scale=1, color=\clr]
 \draw[line width=1.2pt] (-0.15,0) arc(180:0:0.15) to[out=down,in=60] (-0.15,-0.4);
        \draw[line width=1.2pt] (0.15,-0.4) to[out=120,in=down] (-0.15,0);
        \node at (-.15,-.5) {$\scriptstyle a$};
        \node at (.15,-.5) {$\scriptstyle a$};
\end{tikzpicture}
= (-1)^a\!\capa .
    \end{gather*}
\end{lemma}

\begin{proof}
    We proceed by induction on $a$. The case when $a=1$ is exactly the first relation in \cref{bubbleC}. Now suppose the statement is true for any label less than $a$. Then we have
    \[
    \begin{tikzpicture}[anchorbase, scale=1, color=\clr]
 \draw[line width=1.2pt] (-0.15,0) arc(180:0:0.15) to[out=down,in=60] (-0.15,-0.4);
        \draw[line width=1.2pt] (0.15,-0.4) to[out=120,in=down] (-0.15,0);
        \node at (-.15,-.5) {$\scriptstyle a$};
        \node at (.15,-.5) {$\scriptstyle a$};
\end{tikzpicture}
=~\frac{1}{a}~
         \begin{tikzpicture}
        [anchorbase,color=\clr]
        \draw[line width=1.2pt] (0,0) to [out=70,in=down] (.6,.6) to [out=up,in=right] (.3,.9) to [out=left,in=up] (0,.6) to[out=down,in=110] (.6,0);
        \draw[-] (0,.5) to [out=up,in=up] (.6,.5);
        \node at (0,-.15) {$\scriptstyle a$};
        \node at (.6,-.15) {$\scriptstyle a$};
        \node at (.3,1.2){$\scriptstyle a-1$};
    \end{tikzpicture}
    \overset{\cref{sliders}\cref{braid}}{=}~-\frac{1}{a}~
        \begin{tikzpicture}
    [anchorbase,color=\clr]
        \draw[line width=1.2pt] (-.1,0) to [out=70,in=down] (.4,.6) to [out=up,in=right] (.2,.8) to [out=left,in=up] (0,.6) to[out=down,in=110] (.5,0);
        \draw[-] (-.05,.05) to [out=20,in=160] (.45,.05);
        \node at (-.1,-.15) {$\scriptstyle a$};
        \node at (.5,-.15) {$\scriptstyle a$};
        \node at (.2,1){$\scriptstyle a-1$};
    \end{tikzpicture}
    =(-1)^{a}\!\capa.
    \]
    This finishes the proof.
\end{proof}

\begin{lemma}
\label{lem:bub}
The following relation holds in $\AWCb$, for $a\ge 1$:
\begin{align*}
    \bubblea= \binom{x }{a}.
\end{align*}
\end{lemma}

\begin{proof}
    We proceed by induction on $a$. When $a=1$ the lemma follows from \cref{bubbleC}. Suppose the lemma holds for $a-1$. Then we have
    \[
     a\bubblea
     ~=~
     \begin{tikzpicture}
         [anchorbase,color=\clr]
         \draw[line width=1pt] (0,0) to [out=up,in=left] (.4,.4) to [out=right,in=up] (.8,0) to [out=down,in=right] (.4,-.4) to [out=left,in=down] (0,0);
         \draw (.6,.35) to [out=-120,in=120] (.6,-.35);
         \node at (-.2,0){$\scriptstyle a$};
         \node at (1.1,0){$\scriptstyle a-1$};
     \end{tikzpicture}
     ~=~
      \begin{tikzpicture}
         [anchorbase,color=\clr]
         \draw[line width=1pt] (.2,.3) to [out=45,in=left] (.4,.4) to [out=right,in=up] (.8,0) to [out=down,in=right] (.4,-.4) to [out=left,in=-45] (.2,-.3)
         (.2,.3) -- (.3,.1) -- (.3,-.1) -- (.2,-.3);
         \draw (.3,.1) -- (.4,.3) to [out=45,in=-45] (.4,-.3) -- (.3,-.1);
         \node at (.1,0){$\scriptstyle a$};
       \node at (1.1,0){$\scriptstyle a-1$};
     \end{tikzpicture}
     ~\overset{\cref{mergesplitC}}{=}~
     \begin{tikzpicture}
         [anchorbase,color=\clr]
        \draw[line width=1pt] (0,0) to [out=up,in=left] (.4,.4) to [out=right,in=up] (.8,0) to [out=down,in=right] (.4,-.4) to [out=left,in=down] (0,0);
        \draw (.2,0) to [out=up,in=left] (.4,.2) to [out=right,in=up] (.6,0) to [out=down,in=right] (.4,-.2) to [out=left,in=down] (.2,0);
         \node at (1.1,0){$\scriptstyle a-1$};
     \end{tikzpicture}
     ~+~
      \begin{tikzpicture}
         [anchorbase,color=\clr]
        \draw[line width=1pt] (0,0) to [out=up,in=left] (.4,.4) to [out=right,in=up] (.8,0) to [out=down,in=right] (.4,-.4) to [out=left,in=down] (0,0);
        \draw (.1,.25) to [out=-45,in=left] (.4,-.2) to [out=right,in=right] (.4,.2) to [out=left,in=45] (.1,-.25); 
         \node at (1.1,0){$\scriptstyle a-1$};
     \end{tikzpicture}
     ~=~(x-a+1)\binom{x}{a-1}.
    \]
    This lemma follows from this since $(x-a+1)\binom{x}{a-1}=a\binom{x}{a}$.
\end{proof}

\begin{lemma}
\label{lem:loll}
The following relation holds in $\WCb$,  for $a\ge 1$:
    \begin{align*}
\begin{tikzpicture}[anchorbase,scale=.8,color=\clr]
\draw[-,line width=1.4pt] (0.08,.3) to (0.08,.6);
\draw[-,thick] (0.1,-.51) to [out=right,in=-45] (0.1,.31);
\draw[-,thick] (0.1,-.51) to [out=left,in=-135] (0.1,.31);
\node at (-.33,-.05) {$\scriptstyle a$};
\node at (.45,-.05) {$\scriptstyle a$};
\end{tikzpicture}
~=0=~
\begin{tikzpicture}[yscale=-1, anchorbase,scale=.8,color=\clr]
\draw[-,line width=1.4pt] (0.08,.3) to (0.08,.6);
\draw[-,thick] (0.1,-.51) to [out=right,in=-45] (0.1,.31);
\draw[-,thick] (0.1,-.51) to [out=left,in=-135] (0.1,.31);
\node at (-.33,-.05) {$\scriptstyle a$};
\node at (.45,-.05) {$\scriptstyle a$};
\end{tikzpicture}.
\end{align*} 
\end{lemma}

\begin{proof} 
We prove only the first equation; the second follows by applying \(\div\) to the first. We proceed by induction on $a$. When $a=1$, the lemma follows from \cref{bubbleC}.  Suppose the lemma is true for $a-1$, then we have
\[
a\begin{tikzpicture}[anchorbase,scale=.8,color=\clr]
\draw[-,line width=1.4pt] (0.08,.3) to (0.08,.6);
\draw[-,thick] (0.1,-.51) to [out=right,in=-45] (0.1,.31);
\draw[-,thick] (0.1,-.51) to [out=left,in=-135] (0.1,.31);
\node at (-.33,-.05) {$\scriptstyle a$};
\node at (.45,-.05) {$\scriptstyle a$};
\end{tikzpicture}
~=~
\begin{tikzpicture}[anchorbase,scale=.8,color=\clr]
\draw[-,line width=1.4pt] (0.08,.3) to (0.08,.6);
\draw[-,line width=1pt] (0.1,-.51) to [out=right,in=-45] (0.1,.31);
\draw[-,line width=1pt] (0.1,-.51) to [out=left,in=-135] (0.1,.31);
\draw (-.05,0) to [out=-60,in=30] (-.02,-.5);
\node at (-.55,-.25) {$\scriptstyle a-1$};
\node at (.45,-.05) {$\scriptstyle a$};
\end{tikzpicture}
~=~
\begin{tikzpicture}[anchorbase,scale=.8,color=\clr]
\draw[-,line width=1.4pt] (0.08,.3) to (0.08,.6);
\draw[-,line width=1pt] (0.1,-.51) to [out=right,in=-45] (0.1,.31);
\draw[-,line width=1pt] (0.1,-.51) to [out=left,in=-135] (0.1,.31);
\draw[-] (0.1,-.31) to [out=right,in=-60] (0.1,.31);
\draw[-] (0.1,-.31) to [out=left,in=-120] (0.1,.31);
\node at (.65,-.05) {$\scriptstyle a-1$};
\end{tikzpicture}
=0.
\]
This proves the lemma.
\end{proof}

\bibliographystyle{alpha}
\bibliography{Bweb}

\end{document}